\documentclass[12pt,leqno]{amsart}
\usepackage{amsmath,stmaryrd}
\usepackage{amssymb}
\usepackage{amsthm}
\usepackage{epsf}
\usepackage{graphicx}
\usepackage{esint} 
\usepackage{dsfont}
\usepackage[pagebackref]{hyperref} 
\hypersetup{pdfpagemode=FullScreen,  colorlinks=true} 
\usepackage{enumerate} 
\usepackage{xcolor,bbm,easybmat,bm}

\usepackage{amsfonts}
\usepackage{amsrefs}
\usepackage{verbatim}
\usepackage{booktabs}
\usepackage{color}
\usepackage{graphicx,float,tikz}
\usetikzlibrary{arrows.meta, calc, angles, quotes, patterns, positioning, shapes.geometric}

\numberwithin{equation}{section}

\newcommand\footnoteref[1]{\protected@xdef\@thefnmark{\ref{#1}}\@footnotemark}
\makeatother

\newcommand{\vp}{\varphi}
\newcommand{\dr}{\partial}

\DeclareMathOperator{\divg}{div}
\DeclareMathOperator{\sgn}{sgn}
\DeclareMathOperator{\dist}{dist}
\DeclareMathOperator{\supp}{supp}
\DeclareMathOperator{\diam}{diam}

\DeclareMathOperator{\loc}{loc}

\newcommand{\1}{{\mathds 1}}

\newcommand{\ms}{\medskip}

\newcommand{\R}{\mathbb R}
\newcommand{\N}{\mathbb N}
\newcommand{\cF}{\mathcal F}
\renewcommand{\H}{\mathcal H}
\newcommand{\bp}{\noindent {\em Proof: }}
\newcommand{\ep}{\hfill $\square$ \medskip}

\newcommand{\wt}{\widetilde}
\newcommand{\wh}{\widehat}
\newcommand{\D}{\mathbb D}
\newcommand{\Base}{\mathcal{O}}

\newcommand{\LL}{\mathcal L}
\renewcommand{\L}{L}

\newcommand{\C}{\mathcal C}

\newcommand{\OO}{\mathcal O}
\newcommand{\cS}{\mathcal S}
\newcommand{\cSS}{\mathcal S (q,\tau,r,r_0,\frac16\hbar)}
\newcommand{\W}{\mathcal W}

\newcommand{\Rn}{\mathbb R^n}
\newcommand{\norm}[1]{\left\Vert#1\right\Vert}
\newcommand{\abs}[1]{\left\vert#1\right\vert}
\newcommand{\br}[1]{\left(#1\right)}
\newcommand{\set}[1]{\left\{#1\right\}}
\renewcommand{\div}{\mathrm{div}}
\renewcommand{\d}{\, \mathrm{d}} 

\newcommand{\om}{\Omega}
\newcommand{\pom}{\partial\Omega}
\newcommand{\dtau}{\ensuremath{\, \mathrm{d} \tau}}

\newcommand{\E}{\mathsf{E}} 
\newcommand{\Lloc}{\L_{\operatorname{loc}}} 
\newcommand{\HT}{H_t} 
\newcommand{\dhalf}{D_t^{1/2}} 
\newcommand{\Hdot}{\dot{H}\protect{\vphantom{H}}} 
\newcommand{\mS}{{\mathcal S}} 
\newcommand{\IC}{\mathbb{C}}
\newcommand{\pd}{\partial}
\newcommand{\cl}[1]{\overline{#1}} 

\newcommand{\ree}{{\mathbb{R}^{n}}}

\newcommand{\dint}{\int\!\!\!\!\!\int}
\def\Yint#1{\mathchoice
	{\YYint\displaystyle\textstyle{#1}}%
	{\YYint\textstyle\scriptstyle{#1}}%
	{\YYint\scriptstyle\scriptscriptstyle{#1}}%
	{\YYint\scriptscriptstyle\scriptscriptstyle{#1}}%
	\!\dint}
\def\YYint#1#2#3{{\setbox0=\hbox{$#1{#2#3}{\iint}$}
		\vcenter{\hbox{$#2#3$}}\kern-.51\wd0}}
\def\longdash{\mkern-1.5mu{-}\mkern-7.5mu{-}} 

\def\fiint{\Yint\longdash}

\newcommand{\Z}{{\mathbb Z}}

\theoremstyle{plain}
\newtheorem{theorem}[equation]{Theorem}
\newtheorem{lemma}[equation]{Lemma}
\newtheorem{corollary}[equation]{Corollary}
\newtheorem{proposition}[equation]{Proposition}

\newtheorem{definition}[equation]{Definition}

\theoremstyle{definition}

\theoremstyle{remark}
\newtheorem{remark}[equation]{Remark}

\begin{document}

\title[The $L^p$ Regularity problem for parabolic operators]{The $L^p$ regularity problem for parabolic operators}

\author[Dindo\v{s}]{Martin Dindo\v{s}}
\address{School of Mathematics, 
The University of Edinburgh and Maxwell Institute of Mathematical Sciences, Edinburgh, UK}
\email{M.Dindos@ed.ac.uk}

\author[Li]{Linhan Li}
\address{School of Mathematics, The University of Edinburgh and Maxwell Institute of Mathematical Sciences, Edinburgh, UK}
\email{linhan.li@ed.ac.uk}

\author[Pipher]{Jill Pipher}
\address{Department of Mathematics, 
 Brown University, RI, US}
\email{jill\_pipher@brown.edu}

\thanks{The first author has been supported in part by EPSRC grant EP/Y033078/1.}

\begin{abstract} 
In this paper, we fully resolve the question of whether the Regularity problem for the parabolic PDE
$-\partial_tu + \div(A\nabla u)=0$ on a Lipschitz cylinder $\mathcal O\times\R$ is solvable for some $p\in (1,\infty)$ under the assumption that the matrix $A$ is elliptic, has bounded and measurable coefficients and its coefficients satisfy a natural Carleson condition (a parabolic analog of the so-called DKP-condition). 

We prove that for some $p_0>1$ the Regularity problem is solvable in the range $(1,p_0)$. We note that answer to this question was not known even in the {\it small Carleson case}, that is, when the Carleson norm of coefficients is sufficiently small.  

In the elliptic case the analogous question was only fully resolved recently independently by two groups, with two very different methods:
one involving two of the authors and S. Hofmann, the second by M. Mourgoglou, B. Poggi and X. Tolsa. Our approach in the parabolic case is motivated by that of the first group, but in the parabolic setting there are significant new challenges.

\end{abstract}

\maketitle



\newpage

\tableofcontents

\section{Introduction}


Inspired by recent breakthroughs (\cite{DHP}, \cite{MPT}) that answered some longstanding questions concerning elliptic Regularity boundary value problems, 
this paper proves sharp solvability of the Regularity problem for {\it parabolic differential equations in divergence form} where the coefficients
satisfy a natural and well-studied minimal smoothness assumption.  

The minimal smoothness condition we impose on the coefficients of $A$ has a long history, first formulated in a conjecture
of Dahlberg concerning solvability of the Dirichlet problem for elliptic divergence form equations.  It arises from a 
particular change of 
variable that is relevant to the study of smooth equations, both elliptic and parabolic, defined in Lipschitz domains (Dahlberg, Kenig, Ne\v{c}as, Stein \cites{D,N}).
This smoothness assumption 
is expressed as a {\it Carleson measure} condition of arbitrary magnitude that is satisfied by the coefficients of the matrix 
associated to the operator. 
Up to now, only the Dirichlet problem has been studied for parabolic operators in this general regime.

The operators we consider
have the form 
\begin{equation}\label{E:pde}
			\LL u:= -\dr_t u +\divg (A \nabla u)=0   \quad\text{in } \Omega, 
\end{equation}
and will be defined in a domain $\Omega = \mathcal O \times \R$ where $\mathcal O = \R^n_+$ or $\mathcal O$ is a bounded or unbounded Lipschitz domain (Definition \ref{DefLipDomain}).

The matrix $A= [a_{ij}(X, t)]$ is a $n\times n$ matrix satisfying the uniform ellipticity condition with $X \in \mathcal O$, $t\in \R$.
That is, there exists positive constants $\lambda$ and $\Lambda$ such that
\begin{equation}
	\label{E:elliptic}
	\lambda |\xi|^2 \leq \sum_{i,j} a_{ij}(X,t) \xi_i \xi_j \leq \Lambda |\xi|^2
\end{equation}
for almost every $(X,t) \in \Omega$ and all $\xi \in \R^n$.

Specifically for parabolic equations, we assume that 
the measure defined by
\begin{equation}\label{E:1:carl}
d\mu = \left( \delta(X,t)|\nabla A|^2 + \delta(X,t)^3|\partial_t A|^2  \right) dX dt
\end{equation}
is the density of a Carleson measure on $\Omega$ with Carleson norm $\|\mu\|_C$ and, moreover,
\begin{equation}\label{E:1:bound}
\delta(X,t)|\nabla A| + \delta(X,t)^2|\partial_t A| \leq K<\infty.
\end{equation}
Here and in the sequel, $\delta(X,t)$  denotes the parabolic distance to the boundary of $\om$, which is defined as: 
\[
\delta(X,t)=\inf_{(Y,\tau)\in\pom}\br{\abs{X-Y}^2+\abs{t-\tau}}^{1/2}.
\]

The condition \eqref{E:1:carl} can further be viewed as imposing a constraint on the oscillation of the coefficients of $A$. This related oscillation condition was first formulated and investigated 
in \cite{DPP}, and is especially useful as it does not require differentiability of the coefficients of $A$. We shall also provide solvability results here under this weaker condition. \vglue1mm

The boundary value problems we consider are posed for boundary data and/or certain derivatives of data belonging to an $L^p$ space, for $1 < p < \infty$.
This inquiry has classical
roots in the study of harmonic functions or solutions to the heat equation in smooth domains, where, as in the classical theory, the objective is to prove nontangential convergence of solutions to
their boundary data. (See \ref{DefRpar} and \ref{DefDir} for definitions.)
The $L^p$ Dirichlet problem when $p=2$ can be thought of as being $1/2$ derivative below the natural class of {\it energy solutions}, the weak solutions that arise from applying the Lax-Milgram lemma. These Lax-Milgram solutions have finite energy, meaning that $\iint_{\Omega}|\nabla u|^2<\infty$, and also have traces in the Besov-Sobolev space $\dot {B}^{2,2}_{1/2}(\pom)$. This trace space is an $L^2$-based space of functions having $1/2$-derivative on $\pom$,
appropriately understood. Such a class of energy solutions can also be defined in the parabolic setting, and this will be done in subsection \ref{RwEs}.

We now provide some details which will set the context for this paper's study of the solvability of the parabolic Regularity problem, beginning with a brief review of what's known for elliptic equations. 
The study of elliptic divergence form operators under the analogous assumption - that the gradient of the coefficients of the matrix satisfies a Carleson condition - has
been well developed over more than two decades by numerous authors, and even enjoys fruitful connections 
to the study of geometric properties of the boundary of the domain such as uniform rectifiability (see \cite{Azz,AHMMT,HMT,HMMtrans,HMMTZ,FL}).
In general these smoothness assumptions fall into two different regimes: one where the norm bound $\|\mu\|_C$ in \eqref{E:1:carl} is of arbitrary magnitude and a second where it is assumed to be 
sufficiently small. In the elliptic setting, the ``large Carleson norm" condition implies that the elliptic measure belongs to the Muckenhoupt weight class $A^\infty$
with respect to surface measure, and this in turn implies solvability of the Dirichlet 
problem with data belonging to $L^p$ for some (possibly large) value of $p$. This was shown in \cite{KP01} in 2001. In the ``small Carleson norm" regime, the elliptic measure is provably  
much better behaved with respect to surface measure, 
yielding solvability of the $L^p$ Dirichlet problem for all $1 < p< \infty$. This was the main result of \cite{DPP} in 2007, completing the program of solving these particular 
elliptic Dirichlet boundary value problems.

The Regularity problem imposes an additional smoothness assumption on the Dirichlet data, namely, existence of ``tangential derivatives" in $L^p$, while solvability requires a nontangential estimate on the gradient of solutions. When the elliptic Regularity problem is solvable with data in some $L^p$ space, there is a remarkable duality between the range of solvability of the Dirichlet and Regularity problems; namely, the Dirichlet problem is solvable in the range $(p_0,\infty)$ if and only if the Regularity problem for its adjoint is solvable in the range $(1,p_0')$, where $1/p_0 + 1/p_0' =1 $ (\cites{KP93,S2}). It is an open question whether this duality of ranges holds without the assumption that Regularity is solvable at some point in the $L^p$ scale. This highlights one of the sources of interest in the Regularity problem. Another source is interest in the companion boundary value problem, namely, the Neumann problem, which prescribes normal (or ``co-normal") derivatives of the solution on the boundary, as opposed to tangential derivatives. While apparently closely connected, 
solvability of the Neumann problem is much less well understood. Solvability of both the elliptic $L^p$ Regularity and Neumann
problem in the small Carleson norm regime (on coefficients) was first established for domains $\Omega\subset\R^2$ in \cite{DR}, and then in all dimensions in \cite{DPR}.
Finally, the Regularity problem in $L^p$ for some (possibly small) value of $p$ was established in the large Carleson norm regime recently and independently in \cite{DHP} and \cite{MPT}, by very different methods. We note that the approach of \cite{MPT} solves this Regularity problem in a larger class of domains, beyond Lipschitz. However, the Neumann problem is still open in the large Carleson norm regime, except for dimension two (\cite{DHP}). A survey of current state of the art for elliptic PDE can be found in \cite{DPsurv}. Even in the elliptic setting, there are questions that have yet to be answered; for example, solvability of Regularity in the presence of drift terms. 

Moving to the parabolic setting, the first results for solvability of the Dirichlet problem for the small Carleson norm regime were obtained in \cite{DH18}, for 
$p \geq 2$. This was later improved to $p > 1$ in \cite{DDH}.  The solvability of the Dirichlet problem for the large Carleson norm regime was 
established in \cite{DPP2}. A broader history of solvability of parabolic boundary value problems with coefficients satisfying various mild regularity conditions can be
 found in, for example, the introduction of \cite{Din23}.

Similar to the elliptic setting, the $L^p$ Regularity problem in the parabolic setting for $p=2$ can be thought of as being $1/2$ derivative above the energy class of solutions, where the homogeneity of the parabolic setting entails that solutions restricted to boundary should have one full spatial gradient in $L^p(\pom)$. The time variable 
of a parabolic PDE enjoys a different scaling from that of the spatial variables since the equation tells us that, roughly speaking,  $\partial_t u\sim \nabla^2 u$. That is, one spatial gradient of $u$ corresponds to $1/2$ derivative in time.
Hence the class of energy solutions (c.f. \cite{AEN}) for parabolic PDEs is the space with {\it finite energy}:
$$\iint_{\Omega}|\nabla u|^2+\iint_{\Omega}|D_t^{1/2}u|^2<\infty.$$
Finite energy solutions will have traces on the boundary with $1/2$-spatial and $1/4$-time derivatives. All of these considerations will
be explained with care in 
subsection \ref{RwEs}. 

Therefore, in order to define the parabolic $L^p$-based Regularity problem, one looks for solutions of the parabolic equation with prescribed data having one full spatial derivative and a half-time derivative on $\pom$.
The initial formulation of the Regularity problem for the heat equation is from \cite{FR}, where it was formulated for $C^1$ cylinders. This was then reformulated and solved in
bounded Lipschitz cylinders in \cite{Bro87, Bro89}.  Constant coefficient systems were studied in \cite{Nys06}.
The parabolic Regularity problem with time-independent variable coefficients has been formulated and solved in bounded Lipschitz cylinders in \cite{M} and
\cite{CRS}, but only for data in $L^p$ with $p$ close to 2. 

\subsection{Main Results and Methodology}

There are several optimal results in this paper for parabolic equations whose 
coefficients satisfy a Carleson measure condition. In 
the large Carleson norm regime, we show that the Regularity problem for an operator is solvable in the range dual to solvability of the Dirichlet problem for the
adjoint operator. This is the main goal of the paper, and will allow us to show solvability of Regularity in $L^p$ for some interval $(1, p_0)$.
Moreover, when the Carleson norm is assumed sufficiently small, this result will also give the first and best possible results for solvability of parabolic Regularity in this setting.  

The approach taken in this paper essentially follows the elliptic roadmap of \cite{DHP}; however, nearly every single step of that roadmap was missing from the 
 body of known results in the parabolic setting. Some of these steps represent significant advances in their own right, with
 multiple further applications, and have appeared only recently as separate papers. Specifically, this program requires: (1) the parabolic perturbation results of \cite{Up}, (2) the parabolic Dirichlet-Regularity duality results of \cite{DinSa}, and (3) the
nontangential estimates for the half derivative in time in \cite{Din23}. 

 Here, we establish
two key steps to obtain sharp solvability of the parabolic Regularity problem: first, reducing the general problem to a subclass of operators with a special block-form structure on the upper half-space; and second, leveraging the unique structure of such block-form operators—particularly the presence of a special transversal direction in which the operator exhibits well-behaved properties—to establish the solvability of the $L^2$ Regularity problem for these block-form operators.

We encountered significant new challenges in both steps of the proof. For the reduction to the block-form case, the parabolic strategy required 
an interesting technical refinement on known methodology of estimating the $L^p$ norm of the nontangential maximal function by duality. 
To solve the Regularity problem for the block-form parabolic operator, we had to overcome major additional challenges that arose 
in handling the new appearance of a non-local operator, the half-derivative in time $D_t^{1/2}u$. We introduce and study the PDE satisfied by $D_t^{1/2}u$, which involves a non-local term. The results of this study constitute a novel contribution of this paper. 
It turns out that the complexity of the mechanisms and the intricacy of the computations required in both of these steps necessitate the length of this paper.

We briefly address our assumption on the domain of the form $\Omega = \mathcal O \times \R$, i.e., the domain does not vary in time. This means that we do not include (unlike the elliptic case \cite{DHP}),
the class of operators arising from the pull-backs of the heat equation from time varying domains. The reason for this omission is the current lack of theory of solvability of the Regularity problem for operators that  include a singular drift (as in the work of Hofmann-Lewis \cite{HL01}). This theory does not exist even in the elliptic case (as noted in \cite{DPR}) and it currently prevents  making any inroads in the time-varying case. We expect that when/if any advance is made on the solvability of the elliptic Regularity problem
with singular drift terms, a similar advance will be possible on the parabolic Regularity problem on time varying domains satisfying the usual assumptions as in \cite{HL01}. This distinguishes the Regularity problem from the Dirichlet one, where the presence of drift terms can be usually handled. \vglue1mm

\medskip

One important aspect of the methodology here is the existence of a preferred direction - an essential feature of the upper half 
space, or the domain above a Lipschitz graph. 
This enables the use of a variant of the measure \eqref{E:1:carl} in directions parallel to the boundary, namely,
\begin{equation}\label{def.mu11}
    d\mu_{||}(X,t) = \left( \delta(X,t)|\nabla_x A|^2 + \delta(X,t)^3|\partial_t A|^2  \right) dX dt.
\end{equation}
The only difference between $\mu$ and $\mu_{||}$ is that $\mu_{||}$ does not contain $\dr_{x_n}A$, which is important because the Carleson norm of the measure $\mu_{||}$ can be 
made small, unlike that of $\mu$.

\medskip

We state our main results below.

\begin{theorem}\label{MainT} Consider the parabolic PDE $\LL u= -\dr_t u +\divg (A \nabla u)=0$ on the domain  $\Omega=\mathcal O\times\R$, where $\mathcal O\subset\R^n$ is a bounded or unbounded Lipschitz domain, the matrix $A$ is uniformly elliptic \eqref{E:elliptic} with bounded and measurable coefficients. Assume that in addition one of the the following two conditions holds:
\begin{enumerate}
\item matrix A satisfies \eqref{E:1:carl} and \eqref{E:1:bound}, or
\item the measure
$d\nu=\delta(X,t)^{-1}\left(\sup_{B((X,t),\delta(X,t)/2)}\mbox{\rm osc } A\right)^2\,dX\,dt$
satisfies the Carleson measure condition.
\end{enumerate}
Then there exists $p_0>1$ such that for all $1<p<p_0$ the $L^p$ Regularity problem for the equation $\LL u=0$ in $\Omega$ is solvable (c.f. Definition \ref{DefRpar}). Moreover, this interval of solvability $(1,p_0)$ is dual to the interval of solvability of the $L^q$ Dirichlet problem (c.f. Definition \ref{DefDir}) for the adjoint equation 
\begin{equation}\label{E:adjpde}
     \LL^*u:=\dr_tu+\divg(A^T\nabla u)=0
\end{equation} in $\Omega$ which is solvable for all $q>p_0'$.
\end{theorem}

A special case of Theorem \ref{MainT} is the following result for block-form operators.

\begin{theorem}\label{thm.BL}
    Assume that $\Omega=\R^{n}_+\times\R$ and  the matrix $A$ is of the block-form, that is $
A=\left[ \begin{array}{c|c}
   A_\parallel & 0 \\
   \midrule
   0 & 1 \\
\end{array}\right] 
$, for some $(n-1)\times(n-1)$ matrix $A_\parallel$ and that the matrix $A$ is uniformly elliptic with bounded measurable coefficients. 
Then the $L^p$ Dirichlet problem $\mathcal Lu=0$ for the operator $\LL=-\dr_t+\divg A\nabla$ is solvable for all $1<p<\infty$.\vglue1mm

\noindent Furthermore, assume that the matrix $A_\parallel$
satisfies the condition (1) or (2) as in Theorem \ref{MainT}. Then the $L^p$ Regularity problem $\mathcal Lu=0$
is solvable for all $1<p<\infty$.
\end{theorem} 

The elliptic analog of the first part statement in Theorem \ref{thm.BL} concerning solvability of the Dirichlet problem is well known (and was 
observed by S. Mayboroda). However we are not aware that such a result is known for the parabolic case. Because we need this result as a component 
in solving the Regularity problems addressed in this paper, the proof of this fact is given in Appendix \ref{APA}. The second part of Theorem \ref{thm.BL}, the solvability of the Regularity problem,
is the aforementioned key step in the proof of Theorem \ref{MainT}. 

The proof of Theorems \ref{MainT} and \ref{thm.BL} is  presented in Section~\ref{sec.pf}, building on results established throughout the rest of the paper.\medskip

\noindent {\bf Acknowledgment:} We would like to thank Joseph~Feneuil for some helpful discussions, especially for explaining their argument regarding the `$N<S$' estimate in \cite{DFM}. 

\section{Formulation of the Regularity problem}
\setcounter{equation}{0}

\subsection{Parabolic Sobolev Spaces on {$\partial\Omega$}}%
\label{S:paraSobolev}

When considering the appropriate function space for our boundary data we want it to have the same homogeneity as the PDE\@.
As a rule of thumb, one derivative in time behaves like two derivatives in space and so the correct order of our time derivative should be $1/2$ if we impose data with one derivative in spatial variables.
This has been studied previously in~\cites{HL96,HL99,Nys06}, who have followed~\cite{FJ68} in defining the homogeneous parabolic Sobolev space $\dot{L}^p_{1,1/2}$ in the following way. 

\begin{definition}%
	\label{D:paraSob}
	The \textit{homogeneous parabolic Sobolev space} $\dot{L}^p_{1,1/2}(\R^n)$, for $1 < p < \infty$, is defined to consist of equivalence classes of functions $f$ with distributional derivatives satisfying $\|f\|_{\dot{L}^p_{1,1/2}(\R^n)} < \infty$, where
	\begin{equation}
		\|f\|_{\dot{L}^p_{1,1/2}(\R^n)} = \|\D f\|_{L^p(\R^n)}
	\end{equation}
	and
	\begin{equation}
		(\D f)\,\widehat{\,}\,(\xi,\tau) := \|(\xi,\tau)\| \widehat{f}(\xi,\tau).
	\end{equation}
Here $\|(\xi,\tau)\|$ on $\R^{n-1} \times \R$ is defined as the unique positive solution $\rho$ to the following equation
\begin{equation}
	\label{E:par-norm}
	\frac{|\xi|^2}{\rho^2} + \frac{\tau^2}{\rho^4} = 1.
\end{equation}
One can easily show that $\|(\xi,\tau)\| \sim (|\xi|^2 + |\tau|)^{1/2}$ and that this norm scales correctly according to the parabolic nature of the PDE\@.
\end{definition}

\noindent{\it Remark.}
In the definition above we consider $(x,t)\in{\R^n}=\R^{n-1}\times \R$. This is due to the fact that the boundary $\partial\Omega$ is a lower dimensional set and our aim is to define the parabolic Sobolev space on $\partial\Omega$.
\vglue2mm

In addition, following~\cite{FR67}, we define a parabolic half-order time derivative by
\begin{equation}
	(\D_n f)\,\widehat{\,}\,(\xi,\tau) := \frac{\tau}{\|(\xi,\tau) \|} \widehat{f}(\xi,\tau).
\end{equation}

If $0 < \alpha \leq 2$, then for $g \in C^\infty_c(\R)$ the \textit{one-dimensional fractional differentiation operators} $D_\alpha$ are defined by
\begin{equation}
	(D^\alpha g)\,\widehat{\,}\,(\tau) := |\tau|^\alpha \widehat{g}(\tau).
\end{equation}
It is also well known that if $0 < \alpha < 1$ then
\begin{equation}
	D^\alpha g(s) = c\int_{\R} \frac{g(s) - g(\tau)}{|s-\tau|^{1 + \alpha}} \dtau
\end{equation}
whenever $s \in \R$.
If $h(x,t) \in C^\infty_c(\R^n)$ then by $D^\alpha_t h: \R^n \rightarrow \R$ we mean the function $D^\alpha h(x, \cdot)$ defined a.e.\ for each fixed $x \in \R^{n-1}$.
We now recall the established connections between $\D$, $\D_n$ and $D_t^{1/2}$. From ~\cites{FR66,FR67, DinD} we have that
\begin{align}\label{eqrn}
	\|\D f\|_{L^p(\R^n)} \sim  \|\D_n f\|_{L^p(\R^n)}+ \|\nabla f\|_{L^p(\R^n)}\sim  \|D_t^{1/2} f\|_{L^p(\R^n)}+\|\nabla f\|_{L^p(\R^n)},
\end{align}
for all $1<p<\infty$. Here $\nabla$ denotes the usual gradient in the variables $x\in\mathbb R^{n-1}$.\vglue1mm

Recall the definition of Lipschitz domain.

\begin{definition}
$\Z \subset \R^n$ is an $\ell$-cylinder of diameter $d$ if there
exists an orthogonal coordinate system $(x',x_n)$  with $x'\in\mathbb R^{n-1}$ and $x_n\in\mathbb R$ such that
\[
\Z = \{ (x',x_n)\; : \; |x'|\leq d, \; -(\ell+1) d \leq x_n \leq (\ell+1) d \}
\]
and for $s>0$,
\[
s\Z:=\{(x',x_n)\;:\; |x'|\le sd, -(\ell +1)s d \leq x_n \leq (\ell +1)s d \}.
\]
\end{definition}

\begin{definition}\label{DefLipDomain}
$\mathcal O\subset \R^n$ is a Lipschitz domain with Lipschitz
`character' $(\ell,N,C_0)$ if there exists a positive scale $r_0\in (0,
\infty]$ and
at most $N$ $\ell$-cylinders $\{{\Z}_j\}_{j=1}^N$ of diameter $d$, with
$\frac{r_0}{C_0}\leq d \leq C_0 r_0$ such that 
\vglue2mm

\noindent (i) $8{\Z}_j \cap {\partial\mathcal O}$ is the graph of a Lipschitz
function $\phi_j$, $\|\nabla\phi_j \|_\infty \leq \ell \, ;
\phi_j(0)=0$,\vglue2mm

\noindent (ii) $\displaystyle {\partial\mathcal O}=\bigcup_j ({\Z}_j \cap {\partial\mathcal O}
)$,

\noindent (iii) $\displaystyle{\Z}_j \cap \mathcal O \supset \left\{
(x',x_n)\in\mathcal O \; : \; |x'|<d, \; \mathrm{dist}\left( (x',x_n),{\partial\mathcal O}
\right) \leq \frac{d}{2}\right\}$.

\noindent (iv) Each cylinder $\displaystyle{\Z}_j$ contains points from $\mathcal O^c={\mathbb R^n}\setminus\mathcal O$.

\noindent (v) If $r_0<\infty$ the domain $\mathcal O$ is a bounded set.
\vglue1mm
\end{definition}

\noindent{\it Remark.} If the scale $r_0$ is finite, then the domain $\mathcal O$ from the definition above is a bounded Lipschitz domain. However, we shall also allow both $r_0,\, d$ to be infinite, and in this case, since $\displaystyle{\Z}=\mathbb R^n$,
we have that  $\mathcal O$ can be written in some coordinate system as
$$\mathcal O = \{(x',x_n): x_n > \phi(x')\}\quad\mbox{ where $ \phi:\mathbb R^n \rightarrow \mathbb R$ is a Lipschitz function.}$$

The set $\mathcal O\times\R$ will be called a parabolic Lipschitz cylinder with Lipschitz base $\mathcal O$.
We will need an extension of the definition of spaces $\dot L^p_{1,1/2}$ to Lipschitz cylinders.

\begin{definition}%
	\label{D:paraSob2}
	The \textit{homogeneous parabolic Sobolev space.} 
Let $1<p<\infty$ and $\Omega=\mathcal O\times\R$	where $\mathcal O$ is a Lipschitz domain as in Definition \ref{DefLipDomain}. The space $\dot{L}^p_{1,1/2}(\pom)$  consists of an equivalence class of functions $f$ with distributional derivatives satisfying $\|\nabla _T f\|_{L^p(\pom)}+\|D_t^{1/2} f\|_{L^p(\pom)}< \infty$, where
$\nabla _T f$ denotes the tangential spatial gradient of $f$ (defined for a.e. point of $(q,\tau)\in \pom\times\R$) by $\nabla f(q,\tau)\Big|_{\pom}-\nu(\nabla f(q,\tau)\cdot \nu(q))$, where $\nu$ is the outer normal vector to $\partial\mathcal O$.
The $\dot{L}^p_{1,1/2}(\pom)$ norm of $f$ is defined as 
 $$\|f\|_{\dot{L}^p_{1,1/2}(\pom)}=\|\nabla _T f\|_{L^p(\pom)}+\|D_t^{1/2} f\|_{L^p(\pom)}.$$
\end{definition}

In light of the equivalences of norms in \eqref{eqrn} for $\Omega=\R^{n}_+\times\R$, this coincides with our previous definition for $\pom=\R^{n-1}\times\R$.

\subsection{Reinforced weak and Energy solutions}\label{RwEs}

We recall the paper \cite{AEN} that neatly presents the concept of reinforced weak solutions for the parabolic problem of interest here. In \cite{Din23}, it was shown that the definition given in \cite{AEN} can be weakened a bit further by only asking for the \lq\lq local $1/2$ derivative" in the time variable. We explain below.
\vglue1mm

If $\mathcal O$ is an open subset of $\mathbb R^{n} $, we let $\H^1(\mathcal O)=\W^{1,2}(\mathcal O)$ be the standard Sobolev space of real valued functions $v$ defined on $\mathcal O$, such that $v$ and $\nabla v$ are in $\L^{2}(\mathcal O;\R)$ and $\L^{2}(\mathcal O;\R^n)$, respectively. A subscripted `$\loc$' will indicate that these conditions hold locally.

We shall say that $u$ is a \emph{reinforced weak solution} of $-\partial_t u + \divg(A\nabla u)=0$ on $\Omega=\mathcal O\times \R$ if
\begin{align*}
 u\in \dot {\E}_{\loc}(\Omega):= \H_{\loc}^{1/2}(\R; \L^2_{\loc}(\mathcal O)) \cap \Lloc^2(\R; \W^{1,2}_{\loc}(\mathcal O))
\end{align*}
and if for all $\phi,\psi \in \C_0^\infty(\Omega)$,
\begin{equation}\label{2.10}
\iint_{\Omega}\left[
 A\nabla u\cdot{\nabla (\phi\psi)}+ \HT\dhalf (u\psi)\cdot {\dhalf \phi}+ \HT\dhalf (u\phi)\cdot {\dhalf \psi}\right]\, \d X \d t=0.
 \end{equation}
Here, $\dhalf$ is the half-order derivative and $\HT$ is the Hilbert transform with respect to the $t$ variable, normalized so that $\partial_{t}= \dhalf \HT \dhalf$. The space $\Hdot^{1/2}(\R)$ is the homogeneous Sobolev space of order 1/2 - the completion of $\C_0^\infty(\R)$ in the norm $\|\dhalf (\cdot)\|_{2}$ -  and it embeds into the space $\mS'(\R)/\IC$ of tempered distributions, modulo constants).
The local space  $\H_{\loc}^{1/2}(\R)$ consists of functions $u$ such that $u\phi\in \Hdot^{1/2}(\R)$ for all $\phi\in \C_0^\infty(\Omega)$. As shown in \cite{Din23} the space $\H_{\loc}^{1/2}(\R)$ is larger than $\Hdot^{1/2}(\R)$.
For  $u\in\Hdot^{1/2}(\R; \L^2_{\loc}(\mathcal O)), $ \eqref{2.10} simplifies to 
\begin{equation}\nonumber
\iint_{\Omega}\left[
 A\nabla u\cdot{\nabla \phi}+ \HT\dhalf u\cdot {\dhalf \phi}\right]\, \d X \d t=0.
 \end{equation}
  
 Our definition has the advantage that, in taking a cut-off of the function $u$,  we might potentially improve the decay of $D^{1/2}_tu$ at infinity. In particular, this matters when considering $u\big|_{\partial\Omega}\in L^p_{1,1/2}(\partial\Omega)$ for $p>2$ as such $u$ might not decay fast enough to belong to
$\Hdot^{1/2}(\R; \L^2_{\loc}(\mathcal O))$.

At this point we remark that for any $u\in \Hdot^{1/2}(\R)$ and $\phi,\psi\in \C_0^\infty(\R)$ the formula
\begin{align*}
 \int_{\R} \left[\HT\dhalf (u\psi)\cdot {\dhalf\phi}+\HT\dhalf (u\phi)\cdot {\dhalf\psi}\right]\d t = - \int_{\R} u \cdot {\partial_{t}(\phi\psi)} \d t
\end{align*}
holds, where on the right-hand side we use the duality form extension of the complex inner product of $\L^2(\R)$, between $\Hdot^{1/2}(\R)$ and its dual $\Hdot^{-1/2}(\R)$. By taking $\psi=1$ on the set where $\phi$ is supported, it follows that a reinforced weak solution is a weak solution in the usual sense on $\Omega$ since it satisfies
$u\in \Lloc^2(\R; \W^{1,2}_{\loc}(\mathcal O))$ and for all $\phi\in \C_0^\infty(\Omega)$,
\begin{align*}
 \iint_{\Omega} A\nabla u\cdot{\nabla \phi} \d X \d t - \iint_{\Omega} u \cdot {\partial_{t}\phi} \d X  \d t=0.
\end{align*}
 This implies $\pd_{t}u\in \Lloc^2(\R; \W^{-1,2}_{\loc}(\mathcal O))$. Conversely, any  weak solution $u$ in    $ \H_{\loc}^{1/2}(\R; \L^2_{\loc}(\mathcal O))$ is a reinforced weak solution.\vglue2mm

Specializing to the case $\Omega=\mathcal O\times\mathbb R$, where $\mathcal O$ is either bounded or unbounded Lipschitz domain,
 we say that a 
reinforced weak solution $v\in \dot{\E}_{\loc}(\mathcal O\times\mathbb R)$ belongs to the \emph{energy class} $\dot \E(\mathcal O\times\mathbb R)$ if
\begin{align*}
 \|v\|_{\dot \E} := \bigg(\|\nabla v\|_{\L^2(\mathcal O\times\mathbb R)}^2 + \|\HT \dhalf v\|_{\L^2(\mathcal O\times\mathbb R)}^2 \bigg)^{1/2} < \infty.
\end{align*}
Consequently, these are called \emph{energy solutions}. When considered modulo constants, $\dot \E $ is a Hilbert space and it is in fact the closure of $\C_0^\infty\!\big(\,\cl{\mathcal O\times\mathbb R}\,\big)$ for the homogeneous norm $\|\cdot\|_{\dot \E}$.  As shown in \cite{AEN} (with a small generalization), functions from $\dot \E $ have well defined traces with values in
 the \emph{homogeneous parabolic Sobolev space} $\Hdot^{1/4}_{\pd_{t} - \Delta_x}(\partial\mathcal O\times\mathbb R)$. Here, $\Hdot^{s}_{\pm \pd_{t} - \Delta_x}(\R^n)$ is defined as the closure of Schwartz functions $v \in \mS(\ree)$ with Fourier support away from the origin in the norm $\|\cF^{-1}((|\xi|^2 \pm i \tau)^s \cF v)\|_2$. This yields a space of tempered  distributions modulo constants in $\Lloc^2(\ree)$ if $0 < s \leq 1/2$. 
 Conversely, any $g \in \Hdot^{1/4}_{\pd_{t} - \Delta_x}$ can be extended to a function $v \in  \dot \E$ with trace $v\big|_{\partial\mathbb R^{n+1}_+} = g$.  For this via partition of unity we can define 
 $\Hdot^{1/4}_{\pd_{t} - \Delta_x}(\partial\mathcal O\times\mathbb R)$ for $\mathcal O$ Lipschitz.
  
Hence, by the energy solution to $-\partial_tu + \divg(A\nabla u)=0$ with Dirichlet boundary datum $u\big|_{\partial\mathcal O\times\R} = f \in \Hdot^{1/4}_{\pd_{t} - \Delta_x}$ (understood in the trace sense) we mean $u \in \dot\E$ such that
\begin{align*}
 \iint_{\mathcal O\times\R} \left[A \nabla u \cdot{\nabla v} + \HT \dhalf u \cdot {\dhalf v}\right] \d X \d t = 0,
\end{align*}
holds for all $v \in \dot \E_0$, the subspace of $\dot \E$ with zero boundary trace.

By \cite{AEN} the key to solving these problems is the introduction of the modified sesquilinear form
\begin{equation}\label{eq-sesq}
 a_\delta(u,v) := \iint_{\mathcal O\times\R} \left[A \nabla u \cdot {\nabla (1-\delta \HT) v} + \HT \dhalf u \cdot {\dhalf (1-\delta \HT) v}\right] \d X \d t,
\end{equation}
where $\delta$ is a  real number yet to be chosen. The Hilbert transform $\HT$ is a skew-symmetric isometric operator with inverse $-\HT$ on both $\dot \E$ and $\Hdot^{1/4}_{\pd_{t} - \Delta_x}$.  Hence, $1-\delta \HT$ is invertible on these spaces for any $\delta \in \R$. Hence for a fixed $\delta>0$  small enough, $a_\delta$ is coercive on $\dot \E$ since
\begin{equation}\label{eq:coer}
 a_\delta(u,u) \ge (\lambda-\Lambda\delta )\|\nabla u\|_2^2 + \delta \|\HT \dhalf u \|_2^2.
\end{equation} 
In particular $\delta=\lambda/(\Lambda+1)$ would work.
To solve the Dirichlet problem we take an extension $w \in \dot \E$ of the data $f$ and apply the Lax-Milgram lemma to $a_\delta$ on $\dot \E_0$ to obtain some $u \in \dot \E_0$ such that 
\begin{align*}
 a_\delta(u,v) = - a_{\delta}(w,v) \qquad (v \in \dot \E_0).          \end{align*}
Hence, $u + w$ is an energy solution with data $f$. Should there exist another solution $v$, then $a_\delta(u+w-v,u+w-v) = 0$ and hence by coercivity $\|u +w - v \|_{\dot \E} = 0$. Thus the two solutions only differ by a constant. It means that  Dirichlet problem associated with our parabolic PDE is  \emph{well-posed} in the energy class.

\subsection{$(R)_p$ boundary value problem}

 Let $\Omega=\OO\times \mathbb{R}$, where $\OO\subset\Rn$ is a Lipschitz domain. Assume that $A:\Omega\to M_{n\times n}(\mathbb R)$ is a bounded uniformly elliptic matrix-valued function. Let $\LL=-\dr_t+\divg A\nabla$. 
 \vglue1mm

 To give the definition of the parabolic Regularity problem, we introduce notation for parabolic balls and cubes, and then define nontangential parabolic cones and parabolic nontangential maximal functions.  
 
 \begin{definition}
A parabolic cube on $\R^n\times\R$  centered at $(X,t)$ with sidelength $r$ as 
$$    Q_r(X,t):=\{ (Y, s) \in \R^{n}\times\R : |x_i - y_i| < r \ \text{ for } 1 \leq i \leq n, \ | t - s |^{1/2} < r \}.
$$
When writing a lower case point $(x,t)$ we shall mean a boundary parabolic cube on $\R^{n-1}\times\R$
which has an analogous definition but in one less spatial dimension:
\begin{equation}\label{eqdef.bdypcube}
    Q_r(x,t):=\{ (y, s) \in \R^{n-1}\times\R : |x_i - y_i| < r \ \text{ for } 1 \leq i \leq n-1, \ | t - s |^{1/2} < r \}.
\end{equation}
A parabolic ball on $\R^n\times\R$  centered at $(X,t)$ with radius $r$ is the ball
\begin{equation}\label{eqdef.ball}
    B_r(X,t):=\{ (Y, s) \in \R^{n}\times\R : \|(X-Y,t-s)\|<r \},
\end{equation}
where $\|\cdot\|$ has been defined earlier by \eqref{E:par-norm}. Recall that $\|\cdot\|$ scales as the parabolic distance function
\[
d_p((X,t),(Y,s)) := \br{\abs{X-Y}^2+\abs{t-s}}^{1/2}\sim \|(X-Y,t-s)\|.\]

For parabolic balls at the boundary we use notation $\Delta_r(X,t)=B_r(X,t)\cap \pom$. In the special case $\Omega=\R^n_+\times\R$ we drop the last coordinate and also write $\Delta_r(x,t)$ with understanding that the ball is centered at $(X,t)=(x,0,t)$.
\end{definition}

\begin{definition}
For $a>0$ and $(q,\tau)\in\pom$, unless otherwise defined, we denote the nontangential parabolic cones by
\begin{equation}\label{Gamma2.11}
    \Gamma_a(q,\tau):=\set{(X,t)\in\om: d_p((X,t),(q,\tau))<(1+a)\delta(X,t)},
\end{equation}

and $\delta(\cdot)$ is the parabolic distance to the boundary: 
\[
\delta(X, t) = \inf_{(q,\tau)\in\pom}
d_p((X, t),(q, \tau)).\]
We also use the truncated parabolic cones: for $r>0$,
\[
\Gamma_a^r(q,\tau):=\set{(X,t)\in\om: d((X,t),(q,\tau))<(1+a)\delta(X,t),\, \delta(X,t)<r}
\]
is the parabolic cone with vertex $(q,\tau)$ truncated at height $r$.
\end{definition}

\begin{definition}
For $w\in L^\infty_{\loc}(\om)$, we define the nontangential maximal function of $w$ as
\[
 N_a(w)(q,\tau):=\sup_{(X,t)\in\Gamma_a(q,\tau)}\abs{w(X,t)} \quad\text{for }(q,\tau)\in\pom.
\]
If $w\in L^p_{\loc}(\om)$, $p\in(0,\infty)$, we need the modified nontangential maximal function
\begin{equation}\label{def.Nap}
    \wt N_{a,p}(w)(q,\tau):=\sup_{(X,t)\in\Gamma_a(q,\tau)}\br{\fiint_{B_{\delta(X,t)/2}(X,t)}\abs{w(Y,s)}^pdYds}^{1/p} \quad\text{for }(q,\tau)\in\pom,
\end{equation}
where $B_r(X,t)$ is a parabolic ball centered at $(X,t)$ with ``radius" $r$, that is, 
\[B_r(X,t):= \set{(Y,s)\in\Rn\times\R: d((Y,s),(X,t))<r}.\]
We simply denote 
\begin{equation}\label{def.N2}
    \wt N_{a}(w):=\wt N_{a,2}(w)
\end{equation}
when $p=2$ in \eqref{def.Nap}.
\end{definition}

\noindent We now introduce the parabolic square and area functions. 

\begin{definition}
For a function $w$ with $\nabla w\in L^2_{\loc}(\om)$, define the square function of $w$ as 
\begin{equation}\label{DefSquare}
S_a(w)(q,\tau):=\br{\iint_{\Gamma_a(q,\tau)}\abs{\nabla w(X,t)}^2\delta(X,t)^{-n}dXdt}^{1/2} \quad\text{for }(q,\tau)\in\pom.
\end{equation}
We also need the following two versions of the area function defined as in \cite{DH18} for a function $w: \om\to\R$ defined by
\begin{eqnarray}\label{DefArea}
    A_a(w)(q,\tau)&=&\left(\iint_{\Gamma_a(q,\tau)}|\partial_t w(X,t)|^2\delta(X,t)^{-n+2}\,dX\,dt\right)^{1/2}\\\nonumber
    \tilde{A}_a(w)(q,\tau)&=&\left(\iint_{\Gamma_a(q,\tau)}|\nabla^2 w(X,t)|^2\delta(X,t)^{-n+2}\,dX\,dt\right)^{1/2}.
\end{eqnarray}

\end{definition}

In virtue of the parabolic PDE heuristic that $\partial_t w\sim \nabla^2w$,  the operators $A_a$ and $\tilde{A}_a$ are closely related, but for different purposes there is an advantage of using one over the other.
\medskip

It is well-known (using level-sets argument) that for $p\in(0,\infty)$, the $L^p$ norms of $N_a(w)$, $\wt{N}_a(w)$, $S_a(w)$, $A_a(w)$ and $\tilde A_a(w)$ are invariant under changes of $a$ up to a constant multiple. For this reason, we  omit the dependence on the aperture $a$ of the cones when there is no need for the specificity. 
\medskip

 We are now ready to formulate our definition of the parabolic Regularity problem $(R)_p$. 
\begin{definition}\label{DefRpar}
Let $1<p<\infty$.
The Regularity problem for the parabolic operator $\LL$ with boundary data in
$\dot{L}^{p}_{1,1/2}(\partial\Omega)$ is solvable (abbreviated $(R)_{p}$), if
for every $f\in \dot{L}^{p}_{1,1/2}(\partial\Omega)\cap \Hdot^{1/4}_{\partial_t-\Delta_x}(\partial\Omega)$ 
the {\rm energy solution} $u\in  \dot \E(\Omega)$ (as defined above) to the problem
\begin{align}\label{eq-pp}
\begin{cases}
\LL u&
=0 \quad\text{ in } \Omega\\
u|_{\partial \Omega} &= f \quad\text{ on } \partial \Omega
\end{cases}
\end{align}
satisfies
\begin{align}
\label{RPpar} \quad\|\widetilde{N}(\nabla u)\|_{L^p(\partial
\Omega)}\lesssim \|\nabla_T f\|_{L^{p}(\partial\Omega)}+\|D^{1/2}_t f\|_{L^{p}(\partial\Omega)}\approx\| f\|_{\dot{L}_{1,1/2}^{p}(\partial\Omega)}.
\end{align}
The implied constant depends only the matrix $A$, $p$ and $n$.
\end{definition}

When we need to specify the operator, we will use the abbreviation $(R)_p^{\LL}$.

Observe that this definition does not require control of  $\|\widetilde{N}(D^{1/2}_tu)\|_{L^p}$ or of $\|\widetilde{N}(D^{1/2}_tH_tu)\|_{L^p}$ in 
 \eqref{RPpar} because these bounds will follow from \eqref{RPpar} and do not need to be established separately.

\begin{theorem}(c.f. \cite{Din23})\label{timp} Let $\Omega$ be an infinite Lipschitz cylinder of the form $\mathcal O\times\mathbb R$, $\LL$ be as above and assume that for some $1<p<\infty$ the Regularity problem $(R)_p$ for $\LL$ on the domain $\Omega$ is solvable. Then in addition to \eqref{RPpar} we also have the bounds
\begin{align}
\label{RPpar2} \quad\|\widetilde{N}(D^{1/2}_t u)\|_{L^p(\partial
\Omega)}+\|\widetilde{N}(D^{1/2}_tH_t u)\|_{L^p(\partial
\Omega)}\lesssim \| f\|_{\dot{L}_{1,1/2}^{p}(\partial\Omega)}.
\end{align}
\end{theorem}

Thus by Theorem \ref{timp} our definition coincides with the previous notions of solvability of the Regularity problem as in \cite{Bro87, Bro89, Nys06, M, CRS} which did include one or both of the terms on the left-hand side of \eqref{RPpar2}. This theorem significantly simplifies our task as we need 
only to focus on bounds for the spatial gradient of a solution $u$.\vglue1mm

For completeness we also state the definition of the $L^p$ Dirichlet problem $(D)_p$:

\begin{definition}\label{DefDir} Let $p\in (1,\infty)$. 
We say that the Dirichlet problem for the operator $\LL$ with boundary data in
${L}^{p}(\partial\Omega)$ is solvable (abbreviated $(D)_{p}$), if
for every $f\in {L}^{p}(\partial\Omega)\cap \Hdot^{1/4}_{\partial_t-\Delta_x}(\partial\Omega)$ 
the {\rm energy solution} $u\in  \dot \E(\Omega)$ (as defined above) to the problem \eqref{eq-pp} satisfies the estimate
\begin{align}
\label{RPdir} \quad\|\widetilde{N}( u)\|_{L^p(\partial
\Omega)}\lesssim \| f\|_{L^{p}(\partial\Omega)}.
\end{align}
The implied constant again depends only the matrix $A$, $p$ and $n$.
\end{definition}
We will also denote it by $(D)_p^{\LL}$ when we need to specify the operator.

\section{Preliminaries}

\subsection{Overview of basic parabolic results}

We recall some foundational estimates that will be used. 

\begin{lemma}[Caccioppoli's inequality, see~\cite{Aro68}]%
	\label{L:Caccio}
	Let $A$ satisfy \eqref{E:elliptic} and  suppose that $u$ is a weak solution of \eqref{E:pde} in $Q_{4r}(X,t)\subset\Omega$.
	Then there exists a constant $C=C(\lambda,\Lambda, n)$ such that
	\begin{multline*}
	    r^{n} \left(\sup_{Q_{r/2}(X,t)} u \right)^{2}
			\leq C 
   \sup_{t-r^2 \leq s \leq t+r^2} \int_{Q_r(X,t) \cap \{t = s\}} u^{2}(Y,s) dY\\
			+ C\iint_{Q_{r}(X, t)} |\nabla u|^{2} dYds 
			\leq \frac{C^2}{r^2} \iint_{Q_{2r}(X, t)} u^{2}(Y, s) dYds.
	\end{multline*}
\end{lemma}

\begin{lemma}[Reverse H\"older inequality for gradient]\label{lem.gradL2-L1}
    Let $u$ and $Q_{4r}(X,t)\subset\Omega$ be as in Lemma \ref{L:Caccio}. Then there exists a constant $C=C(\lambda,\Lambda, n)$ such that
    \[
    \br{\fiint_{Q_r(X,t)}\abs{\nabla u(Y,s)}^2dYds}^{1/2}\le C\fiint_{Q_{2r}}\abs{\nabla u(Y,s)}dYds.
    \]
\end{lemma}

The theorem as stated is proven in {\cite[Lemma 3.6]{DinSa}} and is a consequence of higher $p>2$ regularity of gradients together with a real variable argument of Shen \cite{S2}.\medskip

Lemmas 3.4 and 3.5 in~\cite{HL01} give the following estimates for weak solutions of~\eqref{E:pde}.

\begin{lemma}[Interior H\"{o}lder continuity]%
	\label{L:int_Holder}
	Let $A$ satisfy~\eqref{E:elliptic} and suppose that $u$ is a weak solution of~\eqref{E:pde} in $Q_{4r}(X,t)\subset\Omega$ with $0 < r < \delta(X,t)/8$.
	Then for any $(Y,s), (Z,\tau) \in Q_{2r}(X,t)$
	\[
		\left|u(Y, s) - u(Z, \tau)\right| \leq C \left( \frac{|Y-Z| + |s - \tau|^{1/2}}{r}\right)^{\alpha} \fiint_{Q_{4r}(X,t)} |u|,
	\]
	where $C=C(\lambda, \Lambda, n)$, $\alpha=\alpha(\lambda,\Lambda,n)$, and $0 < \alpha < 1$.
\end{lemma}

\begin{lemma}[Maximum Principle, see {\cite[Theorem 1.4]{Ny97}, \cite[Lemma 3.38]{HL01}}]
	\label{L:MP}
	Let $A$ satisfy \eqref{E:elliptic}, and let $u$, $v$ be bounded continuous weak solutions to \eqref{E:pde} in $\Omega$.
	If $|u|,|v|\to 0$ uniformly as $t \to -\infty$ and
	\[
		\limsup_{(Y,s)\to (X,t)} (u-v)(Y,s) \leq 0
	\]
	for all  $(X,t)\in\partial\Omega$, then $u\leq v$ in $\Omega$. 
 
 If $\om'$ is a bounded subdomain of $\om$, and if $u-v\le 0$ on $\dr_p\om'$, then $u\le v$ in $\om'$.
\end{lemma}

\begin{lemma}[Harnack inequality]%
	\label{L:Harnack}
	Let $A$ satisfy~\eqref{E:elliptic} and suppose that $u$ is a weak non-negative solution of~\eqref{E:pde} in $Q_{4r}(X,t)\subset\Omega$, with $0 < r < \delta(X,t)/8$.
	Suppose that $(Y,s), (Z,\tau) \in Q_{2r}(X,t)$ then there exists $C=C(\lambda, \Lambda, n)$ such that, for $\tau < s$,
	\[
		u(Z, \tau) \leq u(Y, s) \exp \left[ C\left( \frac{|Y-Z|^2}{|s-\tau|} + 1 \right) \right].
	\]
 If $u\ge0$ is a weak solution to \eqref{E:adjpde}, then the inequality is valid when $\tau>s$.
\end{lemma}

The continuous Dirichlet problem for \eqref{E:pde} is solvable for any compactly supported, continuous function $f\in C_0^\infty(\pom)$ (\cite{HL01})
and that solution is given by 
\begin{equation}\label{eq.parmeasure}
    u(X,t)=\int_{\pom} f(y,s)d\omega(X,t,y,s),
\end{equation}
where $d\omega(X,t,y,s)$ denotes the boundary parabolic measure.  \medskip

Under the assumption \eqref{E:elliptic}, it is known (see e.g. \cite{HL01}) that there exists a Green’s function $G$ for \eqref{E:pde} in $\om$ with corresponding
parabolic, adjoint parabolic measures $\omega(X, t, \cdot)$, $\omega^T(X, t, \cdot)$, satisfying for
each $(X, t)\in \om$ the condition:
\smallskip
\begin{equation}\label{eq.Riesz}
    \phi(X,t)
=\iint_\om \br{A\nabla\phi\cdot\nabla_Y G(X,t, Y,s)+ G(X,t,Y,s)\dr_s\phi}dYds
+\int_{\pom}\phi(y,s)d\omega(X,t,y,s),
\end{equation}
\[
\phi(X,t)
=\iint_\om\br{ A^T\nabla\phi\cdot\nabla_Y G(Y,s, X,t)-G(Y,s,X,t)\dr_s\phi} dYds
+\int_{\pom}\phi(y,s)d\omega^T(X,t,y,s),
\]
for any $\phi\in C_0^\infty(\R^{n+1})$. 
The Green's function for \eqref{E:pde} has the following properties.
	\begin{enumerate}
		\item $G(X,t,Y,s) = 0$ for $s > t$, $(X,t)$, $(Y,s) \in \Omega$.

		\item For fixed $(Y,s) \in \Omega$, $G(\cdot, Y,s)$ is a solution to \eqref{E:pde} in $\om \setminus \{(Y,s)\}$.

		\item For fixed $(X,t) \in \Omega$, $G(X,t, \cdot)$ is a solution to \eqref{E:adjpde},
  the adjoint equation in $\Omega \setminus \{(X,t)\}$.

		\item If $(X,t)$, $(Y,s) \in \Omega$ then $G(X,t, \cdot)$ and $G(\cdot, Y,s)$ extend continuously to $\overline{\Omega}$ provided both functions are defined to be zero on $\partial \Omega$.
	\end{enumerate}

 It was proven in \cite{Aro68} that Green's function on $\R^n\times\R$ satisfies the bounds
 \begin{equation}\label{eq.Grnf_upbd}
     G(X,t,Y,s)\le C(t-s)^{-n/2}\exp\br{-\frac{\abs{X-Y}^2}{C(t-s)}} \quad\text{ for all } t>s,
 \end{equation}
 for some $C=C(\lambda,\Lambda,n)\ge 1$. \medskip

\begin{lemma}[Boundary Caccioppoli's inequality]\label{lem.bdyCaccio} 
 Let $A$ satisfy~\eqref{E:elliptic} and suppose that $u$ is a weak solution of~\eqref{E:pde} or \eqref{E:adjpde} in $(0,2r)\times Q_{2r}$, where $Q_{2r}=Q_{2r}(y,s)$ is a parabolic cube on $\pom$ defined as in \eqref{eqdef.bdypcube}. If $u$ vanishes continuously on $Q_{2r}$, then there exist constant $C=C(\lambda,\Lambda,n)$ such that 
 \[
 \int_{x_n=0}^{r}\int_{Q_r} \abs{\nabla u}^2dtdX\le Cr^{-2}\int_{x_n=0}^{2r}\int_{Q_{2r}}u^2dtdX.
 \]
\end{lemma}
The proof is standard; see for example \cite{Ny97} Remark 2.3.

A consequence of Caccioppoli's inequality on the boundary is the following:

\begin{lemma}\label{lemma:Gradient L2toL1}
    Let $u$, $Q_{2r}$ be as in Lemma \ref{lem.bdyCaccio}.
    Then
    \begin{align}\label{eq:Gradient_L2toL1}
        \br{\int_0^r\int_{Q_r} |\nabla u|^2}^{1/2} 
        \le C r^{-\frac{n+2}{2}} \int_0^{2r}\int_{Q_{2r}} |\nabla u|  .
    \end{align}
\end{lemma}

\begin{lemma}[Boundary H\"older continuity, see \cite{HL01} Lemma 3.9]\label{lem.bdyHolder}
    Let $A$, $u$, $Q_{2r}$ be as in Lemma \ref{lem.bdyCaccio}. If $u$ vanishes continuously on $Q_{2r}$, then there exist constants $C=C(\lambda,\Lambda,n)$ and $\alpha=\alpha(\lambda,\Lambda,n)\in(0,1)$, such that for any $(X,t)\in (0,r/2)\times Q_{r/2}$,
    \[
    u(X,t)\le C\br{\frac{\delta(X,t)}{r}}^\alpha\sup_{(0,r)\times Q_r}u.
    \]

    If in addition, $u\ge0$ in $(0,2r)\times Q_{2r}$, then there exist constants $C=C(\lambda,\Lambda,n)$ and $\alpha=\alpha(\lambda,\Lambda,n)\in(0,1)$, such that for any $(X,t)\in (0,r/2)\times Q_{r/2}$,
    \begin{equation}\label{eq.bdyHar}
        u(X,t)\le C\br{\frac{\delta(X,t)}{r}}^\alpha u(A_r^\pm(y,s)),
    \end{equation}
    where the plus sign is taken when $u$ is a weak solution to \eqref{E:pde} and the minus sign
for $u$ satisfying \eqref{E:adjpde}.
\end{lemma}

We now introduce the definition of Carleson measures.

\begin{definition}
A measure $\mu$ is a Carleson measure on $\Omega = \mathcal O \times \R$  if there exists a constant $M$ such that for $(X,t) \in \partial \Omega$ and
$0<r<\diam(\Omega)$,
\begin{equation}\label{Cmeasure}
\mu(B_r(X,t) \cap \Omega) \leq  M r^{n+1}.
\end{equation}
The infimum over all $M$ such that \eqref{Cmeasure} holds is the Carleson norm, $\norm{\mu}_C$, of $\mu$.
\end{definition}

\begin{lemma}{\cite[Theorem 1.1]{Up}}\label{lem.pert}
       Let $\OO\subset\Rn$ be a Lipschitz domain, and let $\LL_0=\dr_t- \divg(A_0\nabla)$ and $\LL_1=\dr_t-\divg(A_1\nabla)$ be two parabolic operators such that $A_i$ satisfies \eqref{E:elliptic} on $\om=\OO\times\R$, $i=1,2$. Set
        \[
d\nu =
\sup_{Q(X,t,\delta(X,t)/2)} |A_0 - A_1|^2\frac{dXdt}{\delta}. \]
Assume that the Regularity problem $(R)_q$ for $\LL_0$ is solvable for $q>1$, and that $\norm{\nu}_C<\infty$. Then the Regularity problem $(R)_{p}$ for $\LL_1$ is solvable for some $1 < p < q$.
\end{lemma}

\begin{remark}\label{re.pert}
    ~\\ \vspace{-0.5cm}
    \begin{enumerate}
        \item \cite[Theorem 1.1]{Up} also includes a stronger result in the setting of small Carleson constants.
        \item The theorem still holds if one replaces $d\nu$ with 
        \begin{equation}\label{def.mu'}
            d\nu':=|A_0(X,t) - A_1(X,t)|^2\frac{dXdt}{\delta(X,t)}
        \end{equation}
            and assumes in addition that $A_0$ satisfies $x_n\abs{\nabla A_0(X,t)}
            \le C$ for $(X,t)\in \Rn_+\times\R$. This is justified in \cite{Up}; the reasoning is similar to that of the elliptic case (see \cite{F24}).
  
    \end{enumerate}
\end{remark}

\begin{lemma}[\cite{DinD}]\label{lem.RtoD} Assume that $\OO$ is a bounded Lipschitz domain and $\LL$ is a
parabolic operator of the form \eqref{E:pde} on the domain $\om  =\OO \times\R$\footnote{\cite{DinD} treats $Lip(1,1/2)$ domains, which is more general.} with $A$ satisfying the ellipticity condition \eqref{E:elliptic}. Suppose $(R)^{\LL}_p$ is solvable for some $p\in(1,\infty)$, then $(D)^{\LL^*}_{p'}$ is solvable for $p'=\frac{p}{p-1}$.
\end{lemma}

\begin{lemma}[{\cite[Theorem 1.5]{DinSa}}]\label{lem.DtoR}
    Assume that $\OO$ is Lipschitz (either bounded or unbounded) and $\LL$ is a
parabolic operator of the form \eqref{E:pde} on the domain $\om  =\OO \times\R$, with $A$ satisfying the ellipticity condition \eqref{E:elliptic}.
Suppose that $(D)^{\LL^*}_{p'}$ is solvable for some $p\in(1,\infty)$. Then either $(R)^{\LL}_p$ is solvable
or $(R)^{\LL}_q$ is not solvable for all $1 < q < \infty$.
\end{lemma}

\begin{lemma}[{\cite[Theorem 1.4]{DinSa}}]\label{lem.Rinterp}
Let $\om$ and $\LL$ be as in Lemma \ref{lem.DtoR}. Suppose that $(R)_p^{\LL}$ is solvable for some $1 < p < \infty$. Then $(R)_q^{\LL}$ is solvable for $1 < q < p$.
\end{lemma}

In order to use the averaged version $\tilde{N}$ of the nontangential maximal function instead of ${N}$ in \eqref{eq.Carl1}-\eqref{eq.Carl2}, we will use the following estimate for nontangential maximal functions of solutions $u$ as well as gradients $\nabla u$ of solutions. 

\begin{lemma}\label{spCarl} Let $u$ be a solution to $\partial_t-\mbox{\rm div}(A\nabla u)=0$ in $\Omega\subset\R^n_+\times\R$. Then for any $a>0$ there exists $\tilde{a}>0$ such that we have pointwise for all $(q,\tau)\in\partial\Omega$:
$$N_a(u)(q,\tau)\le C\tilde N_{\tilde{a}}(u)(q,\tau).$$
 Additionally, if $\delta(X,t)|\nabla A(X,t)|\le K$ then also
\begin{eqnarray}\label{NtildeN}
N_a(\nabla u)(q,\tau)&\le& C(1+K)\tilde N_{\tilde{a}}(\nabla u)(q,\tau),
\end{eqnarray}
for some $C=C(n,\lambda,\Lambda)$.
\end{lemma}
\begin{proof}
    The first inequality is a consequence of the interior H\"older inequality (Lemma \ref{L:int_Holder}): notice that the  aperture of the cone $\tilde N_{\tilde{a}}$)
    has been enlarged. The second inequality follows from the interior H\"older inequality for gradients \cite{L87}, which requires the extra assumption that $\delta(X,t)|\nabla A(X,t)|\le K$.
\end{proof}

\subsection{Some real analysis results}

$\mbox{ }$
\medskip

We introduce some special notation for a class of Carleson measures whose density is given by a function. 

\begin{definition} For a function $f$ defined on $\Omega$, we say that $f\in CM_\Omega(M)$ when 
$|f|^2 \delta^{-1}$ is the density of a Carleson measure, that is,  if there exists $M>0$ such that for any $(X,t)\in \dr \Omega$ and $0<r<\diam(\Omega)$,
\begin{equation} \label{defCM}
\iint_{B_r(X,t) \cap \Omega} |f(Y,\tau)|^2 \delta(Y,\tau)^{-1} dYd\tau \leq M r^{n+1}.
\end{equation}
\end{definition}

We now recall some well known properties of these measures.

\begin{lemma}\label{lem.Carl}
    Let $\mathfrak{c}\in CM(M)$ and let $f$ and $g$ be functions on $\om\subset\Rn_+\times\R$. Then 
    \begin{equation}\label{eq.Carl1}
        \iint_\om f(X,t)\mathfrak{c}^2(X,t)\frac{dXdt}{\delta(X,t)}\le C_n M\int_{\pom}N(f)\, dq\,dt,
    \end{equation}
    and 
    \begin{equation}\label{eq.Carl2}
        \iint_\om \mathfrak c(X,t)f(X,t)g(X,t)\frac{dXdt}{\delta(X,t)}\le C_n M\int_{\pom}N(f)\mathcal{A}(g)\, dq\,dt,
    \end{equation}
    where 
    \[
    \mathcal{A}(g)(q,t):=\br{\iint_{\Gamma(q,t)}\abs{g(Y,s)}^2\delta(Y,s)^{-n-2}dYds}^{1/2}\quad \text{for }(q,t)\in\pom.
    \]
\end{lemma}
\begin{proof}
    The first inequality \eqref{eq.Carl1} is classical, which is originally due to Carleson; see for example \cite[Corollary 3.3.6]{Graf}. The second inequality can be obtained by a stopping time argument; see for instance the proof of (3.33) in \cite{DHP}. Note that $\mathcal{A}(\delta\nabla g)=S(g)$. See also \cite[Theorem 1]{CMS85} or \cite[Proposition 2.10]{Fen}.
\end{proof}

\begin{lemma}\label{lem.NM<N}
Let $f$ be a continuous function defined on $\om=\Rn_+\times\R$. For any $(X,t)\in \om$, let $F(X,t):=M(f(X,\cdot))(t)$ be the maximal function of $f$ in $t$. Then for any $p\in (1,\infty)$,
    \[
\int_{\pom}N_a(F)(x,t)^pdtdx\le C\int_{\pom}N_a(f)(x,t)^pdtdx,
\]
where the constant $C$ depends only on $n$ and $p$.
\end{lemma}
\bp
We claim that for $(x,t)\in\R^{n-1}\times\R$,
\begin{equation}\label{eqN(Mgradu)}
    N_a(F)(x,t)\le N_{2a}(f)(x,t)+2M(N_a(f)(x,\cdot))(t).
\end{equation}
Then the lemma would follow from the $L^p$ boundedness of the maximal function. 
To show \eqref{eqN(Mgradu)}, we fix $(x,t)$ and write
\begin{multline}\label{NFsplit}
    N_a(F)(x,t)=\sup_{(Y,s)\in\Gamma_a(x,t)}\sup_{r>0}\fint_{s-r}^{s+r}\abs{f(Y,\tau)}d\tau\\
=\sup_{(Y,s)\in\Gamma_a(x,t)}\br{\sup_{0<r<(ay_n)^2}\fint_{s-r}^{s+r}\abs{f(Y,\tau)}d\tau+
\sup_{r\ge(ay_n)^2}\fint_{s-r}^{s+r}\abs{f(Y,\tau)}d\tau}.
\end{multline}
Fix any $(Y,s)\in\Gamma_a(x,t)$; that is, $\abs{y-x}+\abs{s-t}^{1/2}\le ay_n$.
In the case when $r\ge(ay_n)^2$, $s+r\le t+\abs{t-s}+r\le t+(ay_n)^2+r\le t+2r$, and similarly, $s-r\ge t-2r$. So 
\[
    \int_{s-r}^{s+r}\abs{f(Y,\tau)}d\tau\le \int_{t-2r}^{t+2r}\abs{f(Y,\tau)}d\tau
    \le  \int_{t-2r}^{t+2r}N(f)(x,\tau)d\tau,
\]
as $(Y,\tau)\in\Gamma(x,\tau)$. Therefore, for $r\ge(ay_n)^2$, 
\[
\fint_{s-r}^{s+r}\abs{f(Y,\tau)}d\tau\le2\fint_{t-2r}^{t+2r}N(f)(x,\tau)d\tau\le2M(N(f)(x,\cdot))(t).
\]
When $0<r<(ay_n)^2$, we observe that $(Y,\tau)\in\Gamma_{2a}(x,t)$ for any $\tau\in[s-r,s+r]$. This entails that 
\[
\fint_{s-r}^{s+r}\abs{f(Y,\tau)}d\tau\le N_{2a}(f)(x,t),
\]
which together with the first case proves \eqref{eqN(Mgradu)}.
\ep

\begin{lemma}\label{lem.CarImprov}
    If the matrix $A$ satisfies \eqref{E:elliptic} and \begin{equation}\label{gradACarl}
        \abs{\nabla A}\delta+\abs{\dr_t A}\delta^2\in CM_{\R_+^{n+1}}(M),
    \end{equation} then there exist matrices $B$ and $D$ such that $A= B+D$, where $B$ satisfies \eqref{E:elliptic} with the same constants $\lambda$ and $\Lambda$. In addition, the following estimates hold:
    \begin{equation}\label{B+CCM}
        \abs{\nabla B}\delta+\abs{\dr_t B}\delta^2+\abs{ D}\in CM_{\R_+^{n+1}}(CM),
    \end{equation}
    \begin{equation}\label{D2B_CM}
        \abs{\dr_t\nabla B}\delta^3\in CM_{\R_+^{n+1}}(CM),
    \end{equation}
    and 
    \begin{equation}\label{D2Bbdd}
        \abs{x_n\nabla B(X,t)}+\abs{x_n^2\dr_t B(X,t)}+\abs{x_n^3\dr_t\nabla B(X,t)}\le C,
    \end{equation}
    where the constant $C$ depends only on $n$ and $\Lambda$ (and not on the Carleson constant $M$).
\end{lemma}
\begin{remark}\label{re.A=B+C} ~ \newline \vspace{-0.5cm}
    \begin{enumerate}
        \item An analog of this lemma in the elliptic case was given in \cite[Lemma 2.1]{FLM}. We note that as in \cite{FLM}, the condition \eqref{gradACarl} can be weakened and the (analogous) result holds for a general domain $\om$. The idea of the lemma in the elliptic case appeared in \cite{DPP} and various weaker versions of the result have been used in the literature (e.g. \cite{DH18}). 
        \item One can derive estimates for higher derivatives on $B$, but  \eqref{D2B_CM} and \eqref{D2Bbdd} are all that is needed in the rest of the paper.
        \item It is easy to see from the proof that if $A$ is in block form \eqref{block}, then $B$ is in block form too. 
    \end{enumerate}
\end{remark}

\smallskip

\bp
Let $\theta \in C^\infty_0(\R^n\times\R)$ be a nonnegative function such that $\supp \theta \subset B(0,\frac1{10})$, $\iint_{\R^{n+1}} \theta(X,t) dXdt = 1$. Define
\[\theta_{X,t}(Y,s) := \frac{1}{x_n^{n+2}}\theta\br{\frac{Y-X}{x_n},\frac{s-t}{x_n^2}},\] 
\begin{equation} \label{defB}
B(X,t) := \iint_{\R^{n+1}}A(Y,s)  \, \theta_{X,t}(Y,s)  \, dYds, \quad \text{ and }  \quad C := A - B.
\end{equation}
We compute the derivatives of $\theta_{X,t}$ and get the following.
\begin{equation}\label{eq.dktheta}
    \dr_{x_k}\theta_{X,t}(Y,s)=\frac{1}{x_n}(\dr_k\theta)_{X,t}(Y,s)=-\dr_{y_k}\theta_{X,t}(Y,s)\quad \text{for }k=1,\dots,n-1,
\end{equation}
\begin{multline}\label{eq.dntheta}
    \dr_{x_n}\theta_{X,t}(Y,s)=-\frac{n+2}{x_n}\theta_{X,t}(Y,s)+\frac{1}{x_n}\br{\dr_n\theta}_{X,t}(Y,s)-\frac{1}{x_n^2}(X-Y)\cdot(\nabla\theta)_{X,t}(Y,s)\\
    -\frac{2(t-s)}{x_n^3}(\dr_t\theta)_{X,t}(Y,s)\\
    =\frac{1}{x_n}\left[\divg_Y\br{(X-Y)\cdot\theta_{X,t}(Y,s)}+\dr_s\br{2(t-s)\theta_{X,t}(Y,s)}\right]
    -\dr_{y_n}\theta_{X,t}(Y,s),
\end{multline}
and
\begin{equation}\label{eq.dttheta}
    \dr_t\theta_{X,t}(Y,s)=\frac{1}{x_n^2}(\dr_t\theta)_{X,t}(Y,s)=-\dr_s\theta_{X,t}(Y,s).
\end{equation}
Then it is easy to see that 
\[x_n^{n+3}\abs{\nabla_X\theta_{X,t}} +
x_n^{n+4}\abs{\dr_t\theta_{X,t}}+x_n^{n+5}\abs{\dr_t\nabla_X\theta_{X,t}}\lesssim \1_{B_{x_n/4}(X,t)},\]
and thus,
\[
\abs{x_n\nabla B(X,t)}+
+\abs{x_n^2\dr_tB(X,t)}+
\abs{x_n^3\dr_t\nabla B(X,t)}
\le \frac{C}{x_n^{n+2}}\iint_{B_{x_n/4}(X,t)}\abs{A}dYds\le C\Lambda,
\]
which is \eqref{D2Bbdd}. Here and in the sequel, $B_r(X,t)=\set{(Y,s)\in\Rn\times\R: \abs{Y-X}^2+\abs{t-s}<r^2}$ is the parabolic ball.

By \eqref{eq.dktheta}, we can move $\dr_{x_k}$  from $\theta_{X,t}$ in \eqref{defB} for $k=1,\dots,n-1$ to get that
\begin{equation}\label{eq.dkB}
    \dr_k B(X,t)=\iint\dr_{y_k}A(Y,s)\theta_{X,t}(Y,s)dYds.
\end{equation}
By \eqref{eq.dntheta}, 
\begin{multline}\label{eq.dnB}
    \dr_nB(X,t)=\iint\dr_{y_n}A(Y,s)\theta_{X,t}(Y,s)dYds
    -\frac{1}{x_n}\iint\nabla_YA(Y,s)(X-Y)\theta_{X,t}(Y,s)dYds\\
    -\frac{2}{x_n}\iint\dr_sA(Y,s)(t-s)\theta_{X,t}(Y,s)dYds.
\end{multline}
So by the support property of $\theta_{X,t}$, 
\[\abs{x_n\nabla B(X,t)}\lesssim x_n\fiint_{B_{x_n/4}(X,t)}\abs{\nabla A(Y,s)}dYds+x_n^2\fiint_{B_{x_n/4}}\abs{\dr_sA(Y,s)}dYds.\]
By Fubini's theorem and our assumption \eqref{gradACarl}, we get 
\[\iint_{B\cap\om}\abs{x_n\nabla B(X,t)}^2\frac{dXdt}{x_n}\le CM\sigma(2B\cap\pom)\]
for any parabolic ball $B$ centered at $\pom$. Using \eqref{eq.dttheta} and a similar argument as above one can get that 
\[\iint_{B\cap\om}\abs{x_n^2\dr_t B(X,t)}^2\frac{dXdt}{x_n}\le C\iint_{2B\cap\om}x_n^3\abs{\dr_t A}^2dXdt\le CM\sigma(2B\cap\pom).\]
This proves the desired estimate for the first two terms in \eqref{B+CCM}.

To show \eqref{D2B_CM}, we take $\dr_t$ to  \eqref{eq.dkB} and \eqref{eq.dnB}. We obtain
\[
\dr_t\dr_k B(X,t)=\frac{1}{x_n^2}\iint \dr_{y_k}A(Y,s)\br{\dr_t\theta}_{X,t}(Y,s)dYds \quad\text{for }k=1,\dots,n-1,
\]
and so
\[
    \abs{x_n^3\dr_t\nabla_xB(X,t)}\le C\fiint_{B_{x_n/4}(X,t)}y_n\abs{\nabla_yA(Y,s)}dYds.
\]
A similar computation shows that 
\[
 \abs{x_n^3\dr_t\nabla_nB(X,t)}\le C\fiint_{B_{x_n/4}(X,t)}y_n\abs{\nabla A(Y,s)}dYds+C\fiint_{B_{x_n/4}(X,t)}y_n^2\abs{\nabla_sA(Y,s)}dYds.
\]
Then \eqref{D2B_CM} follows from Fubini's theorem and \eqref{gradACarl}.

We turn to the estimate for the matrix $D=A-B$. We can write 
\begin{multline*}
    \abs{D(X,t)}=\abs{\iint(A(X,t)-A(Y,s))\theta_{X,t}(Y,s)dYds}\\
    \le C\fiint_{B_{x_n/4}(X,t)}\abs{A(X,t)-A(Y,s)}dYds \quad \text{ for any }(X,t)\in\om.
\end{multline*}
Fix $(Z,\tau)\in B\cap\om$, we have 
\begin{multline*}
    \fiint_{B_{z_n/10}(Z,\tau)}\abs{D(X,t)}^2dXdt\le C\fiint_{B_{z_n/10}(Z,\tau)}\fiint_{B_{x_n/4}(X,t)}\abs{A(X,t)-A(Y,s)}^2dYds\\
    \le C\fiint_{B_{z_n/10}(Z,\tau)}\fiint_{B_{z_n/2}(Z,\tau)}\abs{A(X,t)-A(Y,s)}^2dYds
\end{multline*}
as $B_{x_n/4}(X,t)\subset B_{z_n/2}(Z,\tau)$ for $(X,t)\in B_{z_n/10}(Z,\tau)$. 
By the Poincar\'e inequality, one can see that the above is bounded by
\[
 C\fiint_{B_{z_n/2}(Z,\tau)}x_n^2\abs{\nabla A(X,t)}^2dXdt+C\fiint_{B_{z_n/2}(Z,\tau)}x_n^4\abs{\dr_t A(X,t)}^2dXdt,
\]
and hence, by Fubini's theorem,
\begin{multline*}
    \iint_{B\cap\om}\abs{D}^2dXdt\le C\iint_{B\cap\om}\fiint_{B_{z_n/10}(Z,\tau)}\abs{D}^2dXdt\\
    \le C\iint_{2B\cap\om}\br{\abs{\nabla A}^2x_n+\abs{\dr_t A}^2x_n^3}dXdt\le CM\sigma(B\cap\pom),
\end{multline*}
which completes the proof of \eqref{B+CCM}, and the lemma.
\ep

\begin{lemma}\label{lem.locMaxCarl}
    Let $p>1$, $f$ be a non-negative function on $\Rn_+\times\R$, and $g$ be a $t$-independent non-negative function on $\Rn_+$. Assume that the measure $\nu$ given by $d\nu=f(X,t)^pg(X)dXdt$ is a Carleson measure on $\Rn_+\times\R$. Then for any $\kappa>0$, the measure $\wt\nu$ given by $d\wt\nu=M_{\kappa x_n^2}(f(X,\cdot))(t)^pg(X)dXdt$ is a Carleson measure, where \begin{equation}\label{def.trcmaxfunc}
    M_R(f)(t):=\sup\limits_{t\in I, \ell(I)<R}\fint_{I}\abs{f(s)}ds
\end{equation}
is the localized uncentered Hardy-Littlewood maximal function at scale at most $R$. Moreover, there exists $C>0$ such that
    \(\norm{\wt\nu}_C\le C\max\set{1,\kappa}\norm{\nu}_C\).
\end{lemma}

\bp Let $p>1$, $R>0$, and $(a,b)$ be any interval in $\R$. One can show that for any $L^p$ integrable function $h$ on $(a-R,b+R)$, 
\begin{equation}\label{eq.MaxR}
    \norm{M_R(h)}_{L^p(a,b)}\le \norm{M_R(h\1_{(a-R,b+R)})}_{L^p(\R)}\le C\norm{h}_{L^p(a-R,b+R)}.
\end{equation}
Let $Q_r=Q_r(x_0,t_0)$ be a parabolic cube on $\pom=\R^{n-1}\times\R$. Then by \eqref{eq.MaxR}, 
\begin{multline*}
    \int_{x_n=0}^r\int_{(x,t)\in Q_r}M_{\kappa x_n^2}(f(X,\cdot))(t)^pg(X)dxdtdx_n\\
    \le C\int_{x_n=0}^r\int_{x\in Q_r(x_0)}\int_{t_0-(\kappa +1)x_n^2}^{t_0+(\kappa+1)x_n^2}f(X,t)^pg(X)dtdxdx_n\\
    \le C\int_{x_n=0}^r\int_{x\in Q_r(x_0)}\int_{t_0-(\kappa+1)r^2}^{t_0+(\kappa+1)r^2}f(X,t)^pg(X)dtdxdx_n=:T.
\end{multline*}
If $\kappa<1$, then 
\[
T\le \int_{x_n=0}^{2r}\int_{(x,t)\in Q_{2r}}f(X,t)^pg(X)dtdxx_n\le C_n\norm{\nu}_Cr^{n+1}.
\]
If $\kappa\ge 1$, then we can cover $Q_r(x_0)\times(t_0-(\kappa+1)r^2,t_0+(\kappa+1)r^2)$ by $[\kappa]+1$ parabolic cubes $Q_i=Q_r(x_0,t_i)$ of equal size, and therefore,
\[
T\le \sum_{i=1}^{[\kappa]+1}\int_{x_n=0}^{r}\int_{(x,t)\in Q_i}f(X,t)^pg(X)dtdxx_n \le C\kappa\norm{\nu}_C r^{n+1}.
\]
This proves that   \(\norm{\wt\nu}_C\le C\max\set{1,\kappa}\norm{\nu}_C\).
\ep

With Lemma \ref{lem.locMaxCarl} in hand, we can construct many Carleson measures on $\om$ if we assume that our coefficients satisfy the Carleson condition \eqref{E:1:carl}. In particular, we have the following. 
\begin{corollary}\label{cor.locMaxCarl}
    Let $K\ge 1$ and $\nu$ given by 
    $d\nu=x_n^3M_{K^2x_n^2}(\abs{\dr_tA(X,\cdot)})(t)^2dXdt$. Then $\nu$ 
    is a Carleson measure on $\om$ if $x_n^3\abs{\dr_tA(X,t)}^2dXdt$ is so. Moreover, there exists $C>0$ such that
    \[\norm{\nu}_C\le CK^2\norm{x_n^3\abs{\dr_tA(X,t)}^2dXdt}_C\le CK^2 \norm{\mu_{||}}_C.\]
\end{corollary}
\bp
Simply take $p=2$, $\kappa=K^2$, $f(X,t)=\abs{\dr_tA(X,t)}$, and $g(X)=x_n^3$ in Lemma \ref{lem.locMaxCarl}.
\ep

\section{Proof of the main results}\label{sec.pf}

In this section, we outline how it all fits together and then prove our main theorems, relying on the results of Sections 5 - 11.
\subsection{Proof of Theorem \ref{MainT} assuming Theorem \ref{thm.BL}}
$\mbox{ }$\newline

We first prove Theorem \ref{MainT} for $\om=\Rn_+\times \R$.\medskip

{\bf Step 1}. It suffices to prove Theorem \ref{MainT} under assumption (1), that is that \eqref{E:1:carl} and \eqref{E:1:bound} hold. If, instead, we assume condition (2) of Theorem \ref{MainT}, then, as in 
\cite{DPP}, we construct a new matrix $B$ such that $A=B+D$ where $B$ and $D$ are as in Lemma \ref{lem.CarImprov}. In fact the proof is analogous to the one given in Lemma \ref{lem.CarImprov}. The matrix $B$ will satisfy \eqref{E:1:carl} and \eqref{E:1:bound} (and also ellipticity when $A$ is elliptic). Thus, if we can prove Theorem \ref{MainT} for $B$, then by Lemma \ref{lem.pert} we also obtain solvability of the 
Regularity problem $(R)^{\LL}_q$ for the original matrix $A$ for some $q>1$. This is sufficient to conclude that 
$(R)_p^{\LL}$ holds if and only if $(D)_{p'}^{\LL^*}$ by Lemma \ref{lem.DtoR}.
\medskip

{\bf Step 2}. Hence for the remainder of this section we assume that $A$ is elliptic and \eqref{E:1:carl} and \eqref{E:1:bound} hold.
By Theorem \ref{t1}, $(R)_p^{\LL}$ will hold if and only if $(D)_{p'}^{\LL^*}$ holds, assuming that 
 \begin{itemize}
 \item[(i)] $(R)_q^{\LL_0}$ holds for all $q\in(1,\infty)$, where $\LL_0$ is the parabolic operator in block form given by \eqref{e0a}.
 \item[(ii)] For any $q\in(1,\infty)$ and any $f\in \dot L^q_{1,1/2}(\pom)$, the energy solution $\wt u$ to $\LL_0 \wt u=0$ in $\om$ with $\wt u|_{\pom}=f$ satisfies \[\norm{S(\nabla \wt u)}_{L^q}\le C\norm{f}_{\dot L^q_{1,1/2}}.\]
\end{itemize}
Theorem \ref{thm.BL} (proven below) ensures that (i) holds. For (ii), the estimate $\norm{S(\nabla \wt u)}_{L^q}\le C\norm{f}_{\dot L^q_{1,1/2}(\pom)}$ follows from a combination of Theorem \ref{thm.S<Np}, which shows that $\norm{S(\nabla \wt u)}_{L^q}\le C\norm{N(\nabla \wt u)}_{L^q}$, and solvability of $(R)^{\LL_0}_q$, which implies that $\norm{N(\nabla \wt u)}_{L^q}\le C\norm{f}_{\dot L^q_{1,1/2}(\pom)}$. 
\medskip

\noindent{\bf Step 3}. Assume now $\om=\OO\times\R$ for a Lipschitz domain $\OO$ as defined in Definition \ref{DefLipDomain}. We distinguish two cases: 
\begin{itemize}
 \item[(i)] $\OO$ is unbounded, 
 \item[(ii)] $\OO$ is bounded. 
 \end{itemize}
 In both cases, it suffices to show that $(R)_p^{\LL}$ holds at least one $p\in(1,\infty)$, as the rest of our claim then follows from Lemma \ref{lem.DtoR}. \medskip

In case (i), $\OO$ can be written as $\set{(x',x_n)\in\Rn: x_n>\vp(x')}$ where $\vp$ is a Lipschitz function. Define the map  $\rho(x',x_n,t)=(\Psi(x',x_n),t)$ as in \eqref{map_rho}, where $\Psi(x',x_n)=(x',c_0x_n+(\theta_{x_n}*\phi)(x'))$,
where $(\theta_s)_{s>0}$ is smooth compactly supported approximate identity and $c_0$ can be chosen large enough (depending only on
$\|\nabla\phi\|_{L^\infty({\mathbb R}^{n-1})}$) so that $\Psi$ is one to one. One can show that $\rho$ is bi-Lipschitz and that the pull-back of the operator $\LL$ under the map $\rho$ (denoted by $\LL_1$) is an elliptic operator on $\Rn_+\times\R$ that satisfies the Carleson condition (1) or (2) with respect to the domain $\Rn_+\times\R$ with a different constant (see the discussion below \eqref{map_rho}). Therefore, we can apply Step 1 to conclude that the $L^p$ Regulairty problem is solvable for some $p>1$ for $\LL_1$ on $\Rn_+\times\R$, which yields that $(R)^{\LL}_p$ is solvable on $\om$ as $\rho$ is bi-Lipschitz.\medskip

The case (ii)  is carried out in detail in Section 10. More precisely, it is shown there that  matters can be reduced to to proving Theorem \ref{tlocal}. The proof of Theorem \ref{tlocal} consists of a similar argument as in case (i), as well as a comparison between solutions with boundary data with compact support. 

We remark that in both cases, it is possible to track the exponent $p$ in the argument to show directly using Step 1 that the range of $p$ for which $(R)^{\LL}_p$ is solvable is the same as that of $(D)^{\LL^*}_{p'}$, but thanks to Lemma \ref{lem.DtoR}, this is not necessary. In particular, we only need to establish solvability of the Regularity problem $(R)^{\LL}_q$ for one value of $q>1$.

\subsection{Proof of Theorem \ref{thm.BL}} 
Let $\om$, $A$, and $\LL$ be as in Theorem \ref{thm.BL}. Let us denote by $\mu^A$ the measure given by the density \eqref{E:1:carl}- we emphasize the matrix $A$ in the notation because we will need to introduce more matrices. Let  $M:=\norm{\mu^A}_{C}$.  Since $(D)_{p'}^{\LL^*}$ holds for all $p\in(1,\infty)$ (c.f. Appendix 
\ref{APA}), it suffices to show that $(R)_{q_0}^{\LL}$ is solvable for some $q_0\in (1,\infty)$ and then by  Lemma \ref{lem.DtoR} we can deduce solvability of $(R)_{p}^{\LL}$ for all $p\in(1,\infty)$.\vglue1mm

Let $B$ be the matrix found in Lemma \ref{lem.CarImprov}, and so the measure $\mu^B$ given by 
\[
d\mu^B:=\br{x_n\abs{\nabla B}^2+x_n^3\abs{\dr_t B}^2+x_n^5\abs{\dr_t\nabla B}^2}dxdx_ndt
\]
is a Carleson measure on $\om$, with $\norm{\mu^B}_C\le C\norm{\mu^A}_C=CM$. Also, $B\in C^\infty(\om)$ satisfies the bound \eqref{D2Bbdd} where the constant $C$ depends only on $n$ and $\Lambda$.
 Moreover, by Remark \ref{re.A=B+C} (3), $B$ is in block form, and we write $B=\left[ \begin{array}{c|c}
   B_\parallel & 0 \\
   \midrule
   0 & 1 \\
\end{array}\right] 
$.

For $k>1$, we define a new matrix, $B^k(x,x_n,t):=B(x,kx_n,t)$ for $(x,x_n,t)\in\Rn_+\times\R$. This dilation in $x_n$ is a key step in the proof, and 
appeared at a similar point in the proof of the elliptic Regularity problem in \cite{DHP}. As we shall see in a moment, the important observation is that this matrix gives rise
to a new Carleson measure that has small norm for sufficiently large $k$.  Let us denote
\[
\LL_0:=\divg(B\nabla\cdot)-\dr_t= \divg_x(B_\parallel \nabla_x \cdot)+\partial^2_{nn}-\partial_t,
\]
and 
\begin{equation}\label{ellk}
\LL_k:=\divg(B^k\nabla\cdot)-\dr_t=\divg_x(B^k_\parallel \nabla_x \cdot)+\partial^2_{nn}-\partial_t \qquad\text{for }k>0.
\end{equation}
We observe that $B^k$ enjoys the following properties: for $k>1$, $B^k\in C^\infty(\Rn_+\times\R)$, and
\begin{enumerate}[(a)]
    \item for all $(x,x_n,t)\in\Rn_+\times\R$, \[x_n\abs{\nabla_x B^k(x,x_n,t)}+x_n^2\abs{\dr_t B^k(x,x_n,t)}\le C\,k^{-1}, \quad x_n\abs{\dr_n B^k(x,x_n,t)}\le C \] 
    for some $C$ depending only on $n$ and $\Lambda$;
    \item the measure $\mu_{||}^k$ given by
    \[
    d\mu_{||}^k:=\br{x_n\abs{\nabla_x B^k}^2+x_n^3\abs{\dr_t B^k}^2}dxdx_ndt
    \]
    is a Carleson measure on $\om$, with $\norm{\mu_{||}^k}_C\le \br{\norm{\mu^B}_C+C\log k}k^{-2}\le C(M+\log k)k^{-2}$;
    \item the measure $\mu^k$ given by 
    \[
    d\mu^k:=\br{x_n\abs{\nabla B^k}^2+x_n^3\abs{\dr_t B^k}^2+x_n^5\abs{\dr_t\nabla B^k}^2}dxdx_ndt
    \]
    is a Carleson measure on $\om$, with $\norm{\mu^k}_C\le \norm{\mu^B}_C+C\log k\le C(M+\log k)$;
    \item $\LL_k$ is a Carleson perturbation of $\LL_0$, and more precisely, the measure $\nu_k$ given by 
    \[
    d\nu_k:=\abs{B^k(x,x_n,t)-B(x,x_n,t)}^2\frac{dxdx_ndt}{x_n}
    \]
    is a Carleson measure on $\om$, with $\norm{\nu_k}_C\le Ck(M+\log k)$.
\end{enumerate}
Property (a) is a direct consequence of the bound \eqref{D2Bbdd}. For (b) and (c), we only show the computation of the Carleson norm of $\delta\abs{\dr_n B^k}^2$ to illustrate the idea. Fix any $R>0$ and any parabolic cube $Q_R$ with sidelength $R$ on $\partial\br{\Rn_+\times\R}$. We write
\begin{multline*}
    \int_{x_n=0}^R\int_{Q_R}\abs{\dr_nB^k(x,x_n,t)}^2x_ndxdtdx_n
    =k^2\int_{x_n=0}^R\int_{Q_R}\abs{(\dr_nB)(x,kx_n,t)}^2x_ndxdtdx_n\\
    =\int_0^{kR}\int_{Q_R}\abs{\dr_{y_n}B(x,y_n,t)}^2y_ndxdy_ndt
    =\int_0^{R}\int_{Q_R}
    +\int_R^{kR}\int_{Q_R}.
\end{multline*}
Using the Carleson condition on $\mu^B$, the first integral on the right-hand side is bounded by $\norm{\mu^B}_C\sigma(Q_R)$. For the second integral, we use the bound $\abs{\dr_{y_n}B(x,y_n,t)}\le Cy_n^{-1}$ to get that 
\[
\int_R^{kR}\int_{Q_R}\abs{\dr_{y_n}B(x,y_n,t)}^2y_ndxdy_ndt\le C\sigma(Q_R)\log k.
\]
From this it follows that $\norm{\delta\abs{\dr_nB^k}^2}_C\le C(M+\log k)$.

The smallness of the Carleson norm for the expression $x_n\abs{\nabla_x B^k(x,x_n,t)}$ in (b) comes from the fact that
the $k^2$ factor that multiplies $\abs{\dr_nB^k(x,x_n,t)}^2x_ndxdtdx_n$ above will not appear in that computation.

Property (d) can be verified using the Fundamental Theorem of Calculus and Fubini's theorem: fix $R$ and $Q_R$ as before, we write
\begin{multline*}
    \int_{x_n=0}^R\int_{Q_R}\abs{B^k-B}^2\frac{dxdtdx_n}{x_n}
    = \int_{x_n=0}^R\int_{Q_R}\abs{\int_{x_n}^{kx_n}\dr_sB(x,s,t)ds}^2\frac{dxdtdx_n}{x_n}\\
    \le k\int_{x_n=0}^R\int_{Q_R}\int_{x_n}^{kx_n}\abs{\dr_sB(x,s,t)}^2ds\, dxdtdx_n\le k\int_{s=0}^{kR}\int_{Q_R}\abs{\dr_sB(x,s,t)}^2s\,dxdtds.
\end{multline*}
Then we can obtain the estimate for $\norm{\nu_k}_C$ by splitting the integral in $s$ into $\int_0^R$ and $\int_R^{kR}$.

\medskip
We are now done with all the preparations and ready to prove Theorem \ref{thm.BL}.\medskip

{\bf Step 1.} We check that the matrix $B^k$ meets the conditions of Theorem \ref{thm.S<NL2}, and so there exist constants $C>0$ and $K_0\ge 1$ depending only on the dimension and the ellipticity constants, such that for any $f\in \dot L^2_{1,1/2}(\pom)$ and  any $K\ge K_0$,
    the energy solution $u$ to $\LL_k u=0$ in $\om$, $u|_{\pom}=f$ satisfies  
\begin{multline*}
    \norm{S(\nabla u)}_{L^2(\pom)}^2
    \le
       C\br{K^5\norm{\mu^k_{||}}_C+K^{-1}\norm{\mu^k}_C}\norm{N(\nabla u)}_{L^2(\pom)}^2 +C\norm{f}_{\dot L_{1,1/2}^2(\pom)}^2\\
       \le C(M+\log k)\br{K^5k^{-2}+K^{-1}}\norm{N(\nabla u)}_{L^2(\pom)}^2 +C\norm{f}_{\dot L_{1,1/2}^2(\pom)}^2
\end{multline*}
by using properties (b) and (c) in the second inequality.\medskip

{\bf Step 2.} Let $\delta>0$ be as in Theorem \ref{thm.NlessS}, and let $k>1$ to be large enough so that 
\begin{equation}\label{eq.kc1}
\br{\norm{\mu_{||}^k}_C+\norm{\delta\nabla_xB^k}_{L^\infty(\om)}^2}\br{1+\norm{\delta\nabla B^k}_{L^\infty(\om)}^2}\le \frac{C(1+M+\log k)}{k^2} <\delta.
\end{equation}
With this requirement on $k$, $B^k$ satisfies the conditions of Theorem \ref{thm.NlessS}, and so if we a priori know that $\norm{\tilde{N}(\nabla u)}_{L^2(\pom)}<\infty$ then
\begin{multline}\label{eq.N<S+A0}
    \norm{\tilde{N}(\nabla u)}_{L^2(\pom)}\le C\br{1+\|\mu^k\|_C+\norm{\delta|\nabla B^k|}_{L^\infty}^2}\left[\norm{S(\nabla u)}_{L^2(\pom)} +\norm{\nabla_x f}_{L^2(\pom)}\right]\\+C\norm{A(\nabla u)}_{L^2(\pom)}.
 \end{multline}   
where $u$ is as in Step 1, and $C$ depends only on $n$, $\delta$, and the ellipticity constants.
By \eqref{eq.sqrAAL2}, we have 
\[
\|A(\nabla u)\|_{L^2}\le C\|S(\nabla u)\|_{L^2}+C\norm{\mu_{||}^k}_C^{1/2}\|N(\nabla u)\|_{L^2}.\]
Plugging this into \eqref{eq.N<S+A0}, and then using properties (a) - (c) of $B^k$, one obtains that
\begin{multline*}
    \norm{\tilde{N}(\nabla u)}_{L^2(\pom)}\le C(1+M+\log k)\br{\norm{S(\nabla u)}_{L^2(\pom)} +\norm{\nabla_x f}_{L^2(\pom)}}\\
    +\frac{C(M+\log k)^{1/2}}{k}\norm{N(\nabla u)}_{L^2(\pom)}.
\end{multline*}
{\bf Step 3.} We combine the estimates in Steps 1 and 2. Note that $\norm{N(\nabla u)}_{L^2}\lesssim\norm{\wt N(\nabla u)}_{L^2}$ with an implicit constant depending only on $n$ and $\Lambda$ thanks to Lemma \ref{spCarl} and property (a) of the matrix $B^k$. We get that for $k$ satisfying \eqref{eq.kc1} and for $K\ge K_0$,
\begin{equation}\label{Stp3.1}
    \norm{\wt N(\nabla u)}_{L^2}^2\le C_1(K^5k^{-2}+K^{-1})\norm{\wt N(\nabla u)}_{L^2}^2+C_2\norm{f}_{\dot L^2_{1,1/2}}^2,
\end{equation}
where $C_1=O((\log k)^3)$ and $C_2=O((\log k)^2)$.

We take $K=k^{1/5}$ and let $k$ be sufficiently large so that all the following estimates are satisfied: \eqref{eq.kc1}, $k^{1/5}\ge K_0$, and $C_1(k^{-1}+k^{-1/5})\le 1/2$. With this choice of $K$ and $k$, we can hide the first term on the right-hand side of \eqref{Stp3.1} to obtain that 
$ \norm{\wt N(\nabla u)}_{L^2}^2\le C\norm{f}_{\dot L^2_{1,1/2}}^2$, which implies that $(R)_2^{\LL_k}$ is solvable.\medskip

{\bf Step 4.} By property (d) of $\LL_k$ and Lemma \ref{lem.pert}, we deduce that $(R)_q^{\LL_0}$ is solvable for some $1<q<2$. By our definition of $B$ and Lemma \ref{lem.CarImprov}, $\abs{A-B}^2\frac{dxdx_ndt}{x_n}$ is a Carleson measure on $\om$. Thus, by Lemma \ref{lem.pert} again, we conclude that  $(R)_{q_0}^{\LL}$ is solvable for some $1<q_0<2$. Appendix \ref{APA} shows that the Dirichlet problem for operators with a block form matrix is solvable for all $1<p<\infty$, so from
this and Lemma \ref{lem.DtoR} and we obtain
$(R)_{p}^{\LL}$ is also solvable for all $1<p<\infty$.

This completes the proof, modulo an argument that $\|\tilde{N}(\nabla u)\|_{L^2(\pom)}<\infty$ in the first place, since we used finiteness of this term in the above inequalities.
\medskip

{\bf Step 5. Further approximation.} Consider a matrix $B^k$ as in step 3. Let us drop dependence of this matrix on $k$, as we have already fixed $k$ and call this matrix just $B$. We approximate $B$ by a sequence of matrices $B^i$, $i=1,2,3,\dots$ as follows. 

Let $\eta^i$ be a smooth non-negative cutoff function in $\R^n\times\R$ that is equal to $1$ on a parabolic ball $B(0,i)$
and vanishing outside the ball $B(0,2i)$. Consider
\begin{equation}\label{Bapprox}
B^i(x',x_n,t)=\eta^i(x,x_n,t)B(x',x_n+1/i,t)+(1-\eta^{i})I,
\end{equation}
where $I$ is the $n\times n$ identity matrix. It follows that all matrices $B^i$ on the domain $\Omega=\R^n_+\times\R$ are uniformly elliptic with the same ellipticity constants, each  $B^i$ is the identity matrix outside a large ball centered at zero and that $\|\nabla B^i\|_{L^\infty(\Omega)}+\|\partial_t B^i\|_{L^\infty(\Omega)}<\infty$. This latter
bound is a consequence of the fact that $|\nabla B|\le Cx_n^{-1}$ and thus $|\nabla B^i|\le C(x_n+1/i)^{-1}\le Ci$, with similar argument for $\partial_t B$.
Furthermore each $B^i$ satisfies the same Carleson measure conditions as $B$ does and thus we have properties (a)-(c) as stated just below \eqref{ellk} 
and with same constants as $B$. Finally, we observe that on all compact subsets $K\subset \Omega$ we have that $B^i\to B$ uniformly in the $L^\infty$ norm as $i\to\infty$.

Consider the approximate energy solutions $u^i$ to the PDE problems
$$\mathcal L^iu^i=\div(B^i\nabla u^i)-\partial_tu^i=0\quad\mbox{in }\Omega\quad\mbox{with}\quad u^i\Big|_{\partial\Omega}=f,$$
where $f$ is in the trace space 
 $\Hdot^{1/4}_{\pd_{t} - \Delta_x}(\partial\Omega)$. As for all $i$ the ellipticity constants have the same bounds we can conclude that the energy norm $\dot\E(\Omega)$ of all $u^i$ is uniformly bounded by the same constant which only depends on the $\Hdot^{1/4}_{\pd_{t} - \Delta_x}(\partial\Omega)$ norm of $f$.
 
As the solutions $u^i$ are determined up to a constant, and enjoy interior H\"older continuity, we may without loss of generality assume that 
$u^i(0,1,0)=0$ for all $i$. With this in hand, and the fact that $\sup_i \|u^i\|_{\dot\E(\Omega)}<\infty$
we conclude that there exists a subsequence of the sequence $(u^i)$ such that this subsequence
$u^{i_n}$ converges weakly in $\dot\E(\Omega)$ to a function $u\in\dot\E(\Omega)$ with
$u(0,1,0)=0$. With the normalization $u^i(0,1,0)=0$ in place, 
for any $R>0$
the space $L^2(\Omega\cap B(0,R))$ is compactly embedded in $\dot\E$ with uniformly bounded 
$\sup_{i}\|u^i\|_{L^2(B(0,R))}$ (by the Sobolev embedding theorem), allowing us to conclude that we have the convergence
$$u^{i_n}\to u\quad\mbox{as }n\to \infty\quad \mbox{in the $L^2$ norm on all balls $B(0,R)$, $R>0$}.$$

Fix now a compact subset $K\subset\Omega$ and consider a test function $\varphi\in C_0^\infty(K)$. Since each $u^{i_n}$ is a weak solution of its corresponding PDE we get that:
$$\iint_\Omega u^{i_n}\partial_t\varphi-\iint_{\Omega}B\nabla u^{i_n}\cdot\nabla\varphi
=\iint_\Omega(B^{i_n}-B)\nabla u^{i_n}\cdot\nabla\varphi.$$
As $n\to \infty$ we obtain that the RHS converges to zero as $B^i\to B$ uniformly on $K$ and hence 
the weak limit $u$ solves the PDE 
$$\mathcal Lu=\div(B\nabla u)-\partial_tu=0\quad\mbox{in }\Omega.$$
Because the  $\dot\E(\Omega)$ norm of $u$ is bounded by $\sup_i \|u^i\|_{\dot\E(\Omega)}<\infty$, we conclude that $u$ is also an energy solution in $\Omega$.
The PDE for the difference $u^{i_n}-u$, which is
$$\mathcal L(u^{i_n}-u)=\div(B\nabla (u^{i_n}-u))-\partial_t(u^{i_n}-u)=\div f^i:=\div((B-B^{i_n})\nabla u^{i_n})$$
yields a Caccioppoli inequality of the form:
$$\iint_{B((X,t),r)}|\nabla (u^{i_n}-u)|^2\le \frac{C}{r^2}\iint_{B((X,t),2r)}|u^{i_n}-u|^2+C\iint_{B((X,t),2r)}\left(f^i\right)^2.
$$
Since $f^i\to 0$ in $L^2$ norm on compact subsets of $\Omega$, from weak convergence $\nabla u^{i_n}\to \nabla u$ in $L^2$, we actually have strong convergence in $L^2$ on compact subsets of $\Omega$.

We will prove below in step 6 that for the approximate solutions $u^i$ we have the a priori bound $\|\tilde{N}(\nabla u^i)\|_{L^2(\pom)}<\infty$, as required above for $f$ with finite $\dot{L}^2_{1,1/2}(\partial\Omega)$ norm. 
Assuming that for the moment, by step 3 we have that for each $i$, the estimate 
$$\|\tilde{N}(\nabla u^i)\|_{L^2(\pom)}\le C\|f\|_{\dot{L}^2_{1,1/2}(\partial\Omega)},$$
holds with a constant $C$ which is independent of $i$, depending only on ellipticity of the matrix and the Carleson norm of its coefficients. 
Fix now a compact subset $K\subset \Omega$ and let $\tilde N_K$ be a variant of the nontangential maximal 
function $\tilde N$ defined with respect to the modified nontangential cones $\Gamma_K(q,\tau)=\Gamma(q,\tau)\cap K$. Because of the strong convergence $\nabla u^{i_n}\to \nabla u$ in $L^2$ on compact sets we have
that $\tilde{N}_K(\nabla u^i)\to \tilde{N}_K(\nabla u)$ uniformly on $\partial\Omega$ and furthermore the functions $\tilde{N}_K(\nabla u^i)$ and $\tilde{N}_K(\nabla u)$ vanish outside of a compact subset of the boundary. Hence $\tilde{N}_K(\nabla (u^{i_n}-u))\to 0$ in $L^2(\partial\Omega)$. Finally, this implies that
$$\norm{\tilde{N}_K(\nabla u)}_{L^2(\pom)}=\lim_{n\to\infty}\norm{\tilde{N}_K(\nabla u^{i_n})}_{L^2(\pom)}
\le\limsup_{n\to\infty}\norm{\tilde{N}(\nabla u^{i_n})}_{L^2(\pom)}\le C\|f\|_{\dot{L}^2_{1,1/2}(\partial\Omega)}.
$$
Taking the supremum over all compact subsets $K\subset\Omega$ then implies that
\begin{equation}\label{R_2solest}
\|\tilde{N}(\nabla u)\|_{L^2(\pom)}\le C\|f\|_{\dot{L}^2_{1,1/2}(\partial\Omega)}.
\end{equation}
To see that $u$ attains the boundary value $f$ we reason as follows. Consider an extension $\tilde{f}$ of $f$ into $\Omega$ (for example, the 
solution of the heat equation in $\Omega$ with datum $f$). Then each $u^i$ can be written also as $\tilde f+v^i$ where $v^i\in\dot\E_0(\Omega)$ solves the inhomogeneous PDE:
$$\mathcal L^iv^i=\div(B^i\nabla v^i)-\partial_tv^i=-\div(B^i\nabla\tilde f).$$
As before, the $\dot\E_0(\Omega)$ norms are uniformly bounded (by $C\|\nabla \tilde f\|_{L^2(\Omega)}$) and hence there is a subsequence (which we may again take to be $(i_n)$) such that $v^i\to v$ weakly in $\dot\E_0(\Omega)$ and $v\in \dot\E_0(\Omega)$. Given that $u^{i_n}\to u$ strongly on compact sets we must have that $u=\tilde f+v$ and therefore Tr$\, u=f$. Thus $u$ is the energy solution of our PDE with boundary datum $f$ as desired. It follows that such $u$ solves the $L^2$ Regularity problem for the operator with matrix $B$ and boundary data $f$ thanks to \eqref{R_2solest}.
\medskip

{\bf Step 6. An a priori estimate for $\|\tilde{N}(\nabla u^i)\|_{L^2(\pom)}$.} It follows that we only need to prove that  $\|\tilde{N}(\nabla u^i)\|_{L^2(\pom)}<\infty$ 
for solutions $u^i$, as defined above associated to matrices $B^i$
having the properties listed in step 5. We fix $i\in\N$ and consider a matrix $B^i$ as above. Dropping the index $i$ we may assume that $B$ is uniformly elliptic, satisfies the Carleson measure properties listed as (a)-(c)  under \eqref{ellk}, equals the identity matrix outside a large ball centered at zero, and also has the property
that $\|\nabla B\|_{L^\infty(\Omega)}+\|\partial_t B\|_{L^\infty(\Omega)}<\infty$.
Let $u$ be an energy solution to $\LL u=0$ where $\LL=\div(B\nabla\cdot)-\partial_t$. Without loss of generality assume also that $u\Big|_{\pom}=f$ for some compactly supported Lipschitz function $f$ (in all variables). 
It suffices to establish solvability for a class of function that is dense in our underlying space since the solvability operator then extends continuously to the whole space thanks to the estimate \eqref{R_2solest}.

As explained in section \ref{SS:43} we have for such $u$: $\nabla_x u=\vec V+\vec\eta$, where components of $\vec V$ solve the Dirichlet problem (for which nontangential estimates are proven in Appendix \ref{APA}) and $\vec\eta=(\eta_1,\eta_2,\dots,\eta_{n-1})$ solves the PDE system \eqref{system2}.
Our first objective is to prove that $\eta$ is bounded in $\Omega$. This can be proven as follows: Initially we prove that $\|\vec\eta\|_{\dot\E(\Omega)}<\infty$. This follows from \eqref{system2}, provided we show that we can apply Lax-Milgram to this system in the spirit of subsection \ref{RwEs}. This requires establishing coercivity of the system \eqref{system2}, as well as proving bounds for the right-hand side terms of the form  
$$\partial_i((\partial_mb_{ij})v_j),$$
where for $v_j$ we have an $L^\infty$ bound by the maximum principle. Also, since $B=I$ outside a bounded set, we have $\nabla B=0$ outside a bounded set and hence $(\partial_mb_{ij})v_j$ belongs to $L^p(\Omega)$ for all $1\le p\le\infty$. It follows this term is not problematic. The second term on the righthand side of 
\eqref{system2} includes components of $\vec\eta$ and when paired with $\eta_m$ enjoys the estimate:
$$\left|\iint_\Omega\partial_i((\partial_mb_{ij})\eta_j)\eta_m\right|=\left|\iint_\Omega((\partial_mb_{ij})\eta_j)\partial_i\eta_m\right|\le C\left(\iint_\Omega|\nabla\vec\eta|^2\right)^{1/2}\left(\iint_\Omega\frac{|\vec\eta|^2}{x_n^2}\right)^{1/2},
$$
by Cauchy-Schwarz for some small $C$ (c.f. assumption (a) stated under  \eqref{ellk}). Finally, by the Hardy-Sobolev inequality, which applies here as $\vec\eta$ vanishes on the boundary, the right hand side of the above inequality is bounded by $C\iint_\Omega|\nabla\vec\eta|^2$.
Thus this term can be absorbed (thanks to smallness of $C$) by the coercivity the lefthand side enjoys (c.f. subsection \ref{RwEs}).
It follows that $\vec\eta\in \dot\E(\Omega)$. Given that $\vec\eta|_{\pom}=0$ we get by the embedding theorem
and the fact that $\nabla B$ vanishes outside bounded sets
 that  $(\partial_mb_{ij})\eta_j\in L^{2n/(n-2)}(\Omega)$. By a bootstrap argument we may then further improve regularity of $\vec\eta$, via a Moser type iteration argument, eventually obtaining that $\vec\eta\in L^\infty(\Omega)$.
Hence we may conclude that $\nabla_x u=\vec V+\vec\eta\in L^\infty(\Omega)$. In particular this implies that
$\tilde{N}(\nabla_x u)$ is $L^p$ integrable on bounded subsets of $\pom$ for all $1\le p\le\infty$.

Consider a ball $B(0,R)$ large enough such that the operator $B$ equals $I$ outside this set. It follows that on the set $\Omega\setminus B(0,R)$ the PDE for $\vec\eta$ is just the heat equation
$$\Delta \vec\eta-\partial_t\vec\eta=0\quad\mbox{and}\quad \vec\eta\Big|_{\pom\cap B(0,R)^c}=0.$$
Hence on this set components of $\vec\eta$ are comparable with the Green's function for the heat equation in the upper half-space.  We pick the pole  $(Y,s)$ with time $s$ 
in the past of the ball $B(0,R)$. We can find a large constant $M>0$ for which $0\le |\vec\eta|\le MG$
on $\Omega\cap\partial B(0,R)$ and hence by the comparison principle 

$$0<|\vec\eta|(x,x_n,t)\le M G((x,x_n,t),(Y,s)),\qquad\mbox{for all } (x,x_n,t)\notin B(0,R).$$
with $G$ being the Green's function of the whole space given by \eqref{eq.Riesz}. By \eqref{eq.Grnf_upbd}, 
$G$ has sufficient decay - exponential in spatial variables, and polynomial in $t$ of order at least $C(2R+t)^{-1}$ - so
which is $L^p$ integrable, for $p>1$ on an interval $(-R,\infty)$,
 we conclude that $\tilde N(\vec\eta)$ will be integrable in $L^p(\pom)$ for all $p>1$.
 Thus we have established the required a priori bound for $\tilde{N}(\nabla_x u)$. 
 
 It remains to establish a bound for $\tilde{N}(\partial_n u)$.
The PDE satisfied by $w_n$ is (using that $B$ is in block form) 
 \begin{equation}\label{eqwn}
 -\dr_tw_n+\divg(B\nabla w_n)=-\divg_x\br{(\dr_n B_{||})\nabla_x u}:=\div_x\vec F. 
  \end{equation}
  Given what we know about $\nabla_x u$, the righthand side of the PDE has the form $\div_x \vec F$ for a bounded function $\vec F$ with bounded support, and hence such $\vec F\in L^q(\Omega)$ for all $q>n$. It follows that 
we can write $w_n$ as $w^0+\widetilde{w_n}$, where $w^0$ solves the above PDE with zero boundary data
and (by results of, say, \cite[Theorem 8.17]{GT}) is therefore bounded on $\Omega$. The other term, $\widetilde{w_n}$,
solves the homogeneous PDE $\mathcal L\widetilde{w_n}=0$ with boundary datum $w_n\big|_{\pom}$.
If   $w_n\big|_{\pom}\in L^2(\pom)$, then by results of Appendix \ref{APA} we have the required nontangential bounds for it. For $w^0$ the rest of the argument continues as above for $\vec\eta$, since once again outside a bounded set this function is comparable to the Green's function of the heat equation.\medskip

Hence it remains to establish that $w_n\big|_{\pom}\in L^2(\pom)$. The argument is similar to that given in \eqref{TTBBMMx},
but is a bit easier since  \eqref{TTBBMMx} is a local estimate and in this calculation we need to take $\theta\hbar=0$ and $c=0$. It follows that for any $h>0$ we have:

\begin{eqnarray}\nonumber
\mathcal A(h):&=&\int_{\R^n}w_n^2(x,0,t)\,dx\,dt-\int_{\R^n}w_n^2(x,h,t)\,dx\,dt
+h\int_{\R^n}(\partial_{x_n}w_n^2)(x,h,t)\,dx\,dt\\\nonumber
&=&2\iint_{\R^{n-1}\times(0,h)\times\R}|\partial_{x_n}w_n|^2x_n\,dX\,dt+
2\iint_{\R^{n-1}\times(0,h)\times\R}w_n\partial^2_{x_nx_n}w_n\,x_n\,dX\,dt\\
&=& 2\iint_{\R^{n-1}\times(0,h)\times\R}|\partial_{x_n}w_n|^2x_n\,dX\,dt-2\sum_{i,j<n}\iint_{\R^{n-1}\times(0,h)\times\R}
w_n \partial_i(b_{ij}\partial_jw_n)\,x_n\,dX\,dt\\\nonumber
&-&  2\iint_{\R^{n-1}\times(0,h)\times\R} w_n\,\div_x\vec F\,x_n\,dX\,dt-\iint_{\R^{n-1}\times(0,h)\times\R}\partial_t(w_n^2)\,dX\,dt.
\end{eqnarray}
Here we have used \eqref{eqwn}. Integration by parts,
together with the properties of $\vec F$, yields
\begin{eqnarray}\label{diffest}
&&|\mathcal A(h)|\lesssim 
\|S(\nabla_x u)\|^2_{L^2(\Omega)}+\iint_{\R^{n-1}\times(0,h)\times\R}|\partial^2_{x_nx_n}u|^2x_n\,dX\,dt.
\end{eqnarray}
Hence by \eqref{eq.BlckSdk<N} the first term of the righthand side of \eqref{diffest} is finite because we know that $\|N(\nabla_x u)\|^2_{L^2(\Omega)}<\infty$.
For the second term, we make use of the PDE that $u$ satisfies. This yields another term, $\|S(\nabla_x u)\|^2_{L^2(\Omega)}$, that we have already handled, as well as the term $\iint_{\R^{n-1}\times(0,h)\times\R}|\nabla_x u|^2x_n$
which is bounded by $h^2\|N(\nabla_x u)\|^2_{L^2(\Omega)}$
and lastly the term
$\iint_{\R^{n-1}\times(0,h)\times\R}|\partial_tu|^2x_n$, which we consider now. Setting $v=\partial_t u$,
differentiating the PDE for $u$ gives 
\begin{equation}\label{eqnv}
 -\dr_tv+\divg(B\nabla v)=-\divg_x\br{(\partial_t B_{||})\nabla_x u}:=\div_x\vec F^1,
\end{equation}
with $v\big|_{\pom}=\partial_t f$, where $f$ has bounded derivatives (including the $t$-variable) and $\partial_t f$  
 is compactly supported. As with $w_n$ the solution $v$ can be written as $v^0+\tilde{v}$ with $v^0$ solving the inhomogeneous problem with zero boundary data and $\tilde{v}$ solving $\mathcal L\tilde v=0$ with boundary datum $\partial_t f$. Again, Appendix \ref{APA} applies to $\tilde v$ since, by our assumptions on $f$, $\partial_t f\in L^2(\pom)$), and so $\|N(\tilde v)\|_{L^2(\pom)}$ is finite.
The same argument for $w^0$ applies to $v^0$ since the functions $\vec F$ and $\vec F^1$ have similar properites. Thus,
$$\iint_{\R^{n-1}\times(0,h)\times\R}|\partial_tu|^2x_n\,dX\,dt\le h^2\left[\|N(v^0)\|_{L^2(\pom)}^2+\|N(\tilde v)\|_{L^2(\pom)}^2\right]<\infty.$$
Hence, by combining all terms of \eqref{diffest} we have proven that $|\mathcal A(h)|\le C(1+h^2)$. Fix now $H>0$. We average $\mathcal A(h)$ over the interval $(O,H)$ and see that:

$$
H^{-1}\int_0^H\mathcal A(h)dh=\int_{\R^n}w_n^2(x,0,t)\,dx\,dt+\int_{\R^n}w_n^2(x,H,t)\,dx\,dt
-H^{-1}\iint_{\R^{n-1}\times(0,H)\times \R}w_n^2\,dX\,dt.$$
Combining this with our earlier estimate we see that
$$\int_{\R^n}w_n^2(x,0,t)\,dx\,dt\le C+CH^2+H^{-1}\iint_{\R^{n-1}\times(0,H)\times \R}w_n^2\,dX\,dt
\le C+\iint_{\Omega}|\nabla u|^2\,dX\,dt,$$
when $H=1$. Finally, recall that $u$ is the energy solution, and therefore $\iint_{\Omega}|\nabla u|^2\,dX\,dt<\infty$, so that
$w_n\in L^2(\pom)$. This completes the proof.\qed
\medskip

In the remainder of the paper we establish the results we have used in this section, starting with Theorem \ref{t1}.

\section{Solvability of the Regularity problem - reduction to the block form case}
\setcounter{equation}{0}

The following section establishes a result that reduces the question of solvability of the Regularity problem for parabolic PDEs with coefficients satisfying a large Carleson condition to solvability of the same problem for elliptic 
matrices of special block form (see below).

\begin{theorem}\label{t1} Let $\LL=\div(A\nabla\cdot)-\partial_t$ be an operator in ${\mathbb R}^n_+\times\mathbb R$ where matrix $A$ is uniformly elliptic, with bounded real coefficients such both \eqref{E:1:carl} and \eqref{E:1:bound} hold.
Suppose that for all $1<p<\infty$ the $L^p$, the Regularity problem for the block form operator
\begin{equation}\label{e0a}
{\LL}_0 u= -\partial_tu+\mbox{\rm div}_\parallel(A_\parallel \nabla_\parallel u)+\partial^2_{nn}u,
\end{equation}
is solvable in ${\mathbb R}^n_+\times\mathbb R$. Here  $A_{\parallel}$ is the matrix $(a_{ij})_{1\le i,j\le n-1}$.
In addition, assume that for any $p\in(1,\infty)$, any $f\in \dot L^p_{1,1/2}(\R^{n-1}\times \R)$, the energy solution $\wt u$ to $\LL_0 \wt u=0$ in $\Rn_+\times\R$ with $\wt u=f$ on $\partial(\Rn_+\times\R)$ satisfies
\begin{equation}\label{eq.addEst}
    \norm{S(\nabla \wt u)}_{L^p}\le C\norm{f}_{\dot L^p_{1,1/2}}.
\end{equation}
\vglue1mm

Then we have the following: For any $1<q<\infty$ the $L^q$ Regularity problem for the operator $\LL$ is solvable in ${\mathbb R}^n_+\times\mathbb R$ if and only if the $L^{q'}$ Dirichlet problem for the adjoint operator $\LL^*$ is solvable in ${\mathbb R}^n_+\times\mathbb R$.
\end{theorem}

\noindent {\it Proof of Theorem \ref{t1}.}
We start by adapting useful results from \cite{KP2} to parabolic settings. Let us denote by $\tilde{N}_{1,\varepsilon}$ the $L^1$-averaged version of the nontangential maximal function for {\it doubly truncated} parabolic cones. That is, for $\vec{u}:{\mathbb R}^n_+\times\mathbb R\to\mathbb R^m$, we set
$\Gamma_\varepsilon(Q):=\Gamma(Q)\cap
\{(X,t): \varepsilon < \delta(X,t)< 1/\varepsilon\}$, 
and
$$\tilde{N}_{1,\varepsilon}(\vec{u})(Q)=\sup_{(X,t)\in\Gamma_\varepsilon(Q)}\fiint_{(Z,\tau)\in B_{\delta(X,t)/4}(X,t)}|\vec{u}(Z,\tau)|dZd\tau.$$

Lemma 2.8 of \cite{KP2}, stated below, provides a way to estimate the $L^q$ norms of $\tilde{N}_{1,\varepsilon}(\nabla F)(Q)$ via duality (based on tent-spaces). We state these results in the following subsection.

\subsection{Test function $\vec h$.}

\begin{lemma}\label{l1bb} There exist $\vec{\alpha}(X,Z)$ with $\vec{\alpha}(X,\cdot):B(X,\delta(X)/2)\to{\mathbb R}^{n}$ and \newline $\|\vec\alpha(X,\cdot)\|_{L^\infty(B(X,\delta(X)/2))}=1$, a nonnegative scalar function $\beta(X,Q)\in L^1(\Gamma_\varepsilon(Q))$ with \newline $\int_{\Gamma_\varepsilon(Q)}\beta(X,Q)\,dX=1$ and a nonnegative $g\in L^{q'}(\partial{(\mathbb R^n_+\times\mathbb R)},d\sigma)$ with $\|g\|_{L^{q'}}=1$ such that
\begin{equation}
\left\|\tilde{N}_{1,\varepsilon}(\nabla F)\right\|_{L^q(\partial{(\mathbb R^n_+\times \mathbb R)},d\sigma)}\lesssim \iint_{{\mathbb R^n_+\times \mathbb R}}\nabla F(Z)\cdot \vec{h}(Z)\, dZ,
\label{e1a}
\end{equation}
where 
$$\vec{h}(Z)=\int_{\partial{(\mathbb R^n_+\times \mathbb R)}}\iint_{\Gamma(Q)}g(Q)\vec{\alpha}(X,Z)\beta(X,Q)\frac{\chi_{B(X,\delta(X)/4)}(Z)}{\delta(X)^{n+2}}\,dX\,dQ,$$
and $\chi_A$ is the characteristic function of the set $A$.

Moreover, for any $G:{\mathbb R^n_+\times\mathbb R}\to\mathbb R$ with $\tilde{N}_{1}(\nabla G)$ we also have an upper bound
\begin{equation}
\left|\iint_{{\mathbb R^n_+\times\mathbb R}}\nabla G(Z)\cdot \vec{h}(Z)\, dZ\right|\lesssim \left\|\tilde{N}_{1}(\nabla G)\right\|_{L^q(\partial({\mathbb R^n_+}\times\mathbb R),d\sigma)},
\label{e2}
\end{equation}
where $\wt N_1$ is the modified nontangential maximal function with $p=1$ in \eqref{def.Nap}.
The implied constants in \eqref{e1a}-\eqref{e2} do not depend on $\varepsilon$, only on the dimension $n$.
\end{lemma}

The function $\vec{h}$ defined above works perfectly for the purposes outlined in the Lemma, but unfortunately has one flaw which will become important later. Namely, at one point we will need to differentiate $\vec{h}$ (since the derivatives on $\vec{h}$ will not be removable by the usual methods such as integration by parts). Thus, we claim that under certain mild assumptions we can improve this function and make it differentiable with $|\nabla\vec{h}(Z)|\lesssim \delta(Z)^{-1}{h_1}(Z)$, where ${h_1}$ is a function of same type as $\vec{h}$. 

\begin{proposition}\label{betterh} Fix $\varepsilon>0$.
Assume that $\vec{u}$ is a real vector (or scalar) valued function defined on $\mathbb R^n_+\times\mathbb R$ and that the interior H\"older continuity condition holds. That is for some $C>0$ and $\alpha\in (0,1)$ we have for any interior ball $B_{4r}\subset \mathbb R^n_+\times\mathbb R$ the estimate:
\begin{equation}\label{hoa}
\sup_{Z,Z'\in B_r}|\vec{u}(Z)-\vec{u}(Z')|\le C\left(\frac{\|Z-Z'\|}{r}\right)^\alpha\fiint_{B_{2r}}|\vec{u}|.
\end{equation}
Here $\|\cdot\|$ is as in the Definition \ref{D:paraSob}.
Fix a non-negative function $\eta\in C_0^\infty(\R^n)$ supported in $B(0,1)$ with $\iint\eta=1$.
Then there exists $\rho=\rho(n,\alpha,C)\in (0,1/4)$ (independent of $\vec{u}$), 
$\vec{\alpha}(\cdot)$ with  $|\vec\alpha(X)|\le1$, a nonnegative scalar function $\beta(X,Q)\in L^1(\Gamma_\varepsilon(Q))$ with $\int_{\Gamma_\varepsilon(Q)}\beta(X,Q)\,dX=1$, 
 a mapping $Y:\mathbb R^n_+\times\mathbb R\to \R^n_+\times\mathbb R$, $X\mapsto Y(X)\in B(X,\delta(X)/4)$ where $B(X,R)$ is a parabolic ball centered at $X$ with ``radius'' $R$,
 and a nonnegative $g\in L^{q'}(\partial{(\mathbb R^n_+\times\mathbb R)})$ with $\|g\|_{L^{q'}}=1$ such that
for
\begin{equation}\label{newh}
\vec{h}(Z)=\int_{\partial{(\mathbb R^n_+\times \mathbb R)}}\iint_{\Gamma(Q)}g(Q)\vec{\alpha}(X)\beta(X,Q)
\eta_{\rho\delta(X)}(Z-Y(X))\,dX\,dQ,
\end{equation}
we have that
\begin{equation}
\left\|\tilde{N}_{1,\varepsilon}(\vec{u})\right\|_{L^q(\partial{(\mathbb R^n_+\times \mathbb R)},d\sigma)}\approx \iint_{{\mathbb R^n_+\times \mathbb R}}\vec{u}(Z)\cdot \vec{h}(Z)\, dZ,
\label{e1aa}
\end{equation}
where \eqref{e1aa} holds with the implied constant independent of $\varepsilon>0$ and $\vec u$. Here $\eta_\delta$ denotes the parabolic rescaling of the function $\eta$ defined by
$$\eta_\delta(Z)=\eta_{\delta}(z,\tau)=\delta^{-n-2}\eta(z/\delta,t/\delta^2).$$

Moreover, for any $\vec{G}:{\mathbb R^n_+\times\mathbb R}\to\mathbb R^n$ with $\tilde{N}_{1}(\vec G)$ we also have that the upper bound
\begin{equation}
\left|\iint_{{\mathbb R^n_+\times\mathbb R}}\vec G(Z)\cdot \vec{h}(Z)\, dZ\right|\lesssim \left\|\tilde{N}_{1}(\vec G)\right\|_{L^q(\partial({\mathbb R^n_+}\times\mathbb R),d\sigma)}.
\label{e2aa}
\end{equation}

 This estimate holds even if we drop the assumption that $\eta\in C_0^\infty(\R^n)$ is non-negative, as the assumed non-negativity plays no role in the upper bound.
\end{proposition}

\begin{proof} Clearly, \eqref{e2aa} must hold as the comparison between $\vec{h}$ in Lemma \ref{l1bb} and \eqref{newh} reveals that the only difference between these functions is that $\vec{\alpha}$ not does not depend on $Z$ and we have replaced the function $\frac{\chi_{B(X,\delta(X)/4)}(Z)}{\delta(X)^{n+2}}$ by its smoothened-out version $\eta_{\rho\delta(X)}$ shifted to the point $Y(X)$ (and thus supported in $B(Y(X),\rho\delta(X))$). Hence we have not changed the nature of our function in any way. For the upper bound we also can take absolute values inside the integral. Therefore, 

$$\iint_{\R^n_+\times\R} |\vec G(Z)||\vec{h}(Z)|\, dZ\le $$
$$\int_{\partial{(\mathbb R^n_+\times \mathbb R)}}g(Q)\iint_{\Gamma(Q)}|\beta(X,Q)|
\iint_{\R^n_+\times\R}|\vec{\alpha}(X)||\eta_{\rho\delta(X)}(Z-Y(X))||\vec{G}(Z)|\,dZ\,dX\,dQ.
$$
As $\eta_{\rho\delta(X)}$ is supported in $B(Y(X),\rho\delta(X))\subset B(X,\delta(X)/4)$ and 
$|\eta_{\rho\delta(X)}|\le\sup|\eta|(\rho\delta(X))^{-n-2}$ the innermost integral has an upper bound:
$$\iint_{\R^n_+\times\R}|\vec{\alpha}(X)||\eta_{\rho\delta(X)}(Z-Y(X))||\vec{G}(Z)|\,dZ\le \sup|\eta|\omega_n^{-1}\fiint_{B(Y(X),\rho\delta(X))}|\vec G|$$
$$\le \sup|\eta|\omega_n^{-1}C_n^{-\log_2(\rho/4)}\fiint_{B(X,\delta(X)/4)}|\vec G|\le \sup|\eta|\omega_n^{-1}C_n^{-\log_2(\rho/4)}\tilde {N}_1(\vec G)(Q).
$$
Here $\omega_n$ is the volume of the $n$-dimensional unit ball, $C_n$ is the doubling constant and $X$ is an arbitrary point in $\Gamma(Q)$. As a result, the second innermost integral has an upper bound of:
$$ \sup|\eta|\omega_n^{-1}C_n^{-\log_2(\rho/4)}\tilde {N}_1(\vec G)(Q)\iint_{\Gamma_\varepsilon(Q)}|\beta(X,Q)|\,dX=
\sup|\eta|\omega_n^{-1}C_n^{-\log_2(\rho/4)}\tilde {N}_1(\vec G)(Q).$$
Finally, it follows that the outside integral enjoys, by H\"older's inequality, an upper bound of
$$\sup|\eta|\omega_n^{-1}C_n^{-\log_2(\rho/4)}\|\tilde {N}_1(\vec G)\|_{L^q}\|g\|_{L^{q'}}.$$
From this the conclusion follows since $\|g\|_{L^{q'}}=1$. Observe that the constant in our estimate only depends on the dimension $n$, $\rho>0$ and $\sup|\eta|$ which is finite as $\eta$ is a compact $C^\infty$ function. We have not used any other property of $\eta$ apart from its boundedness and the fact that it is supported in the unit ball.\medskip

We now focus on \eqref{e1aa}.
Assume for now that $u$ is a scalar function in $L^1_{loc}$. We will explain the modification needed for our argument to apply in the vector valued case later. Let us define the two versions of average functions
\begin{equation}\label{wW}
W(X)=\fiint_{Z\in B(X,\delta(X)/2)}|{u}|dZ, \quad  w(X)=\fiint_{Z\in B(X,\delta(X)/8)}|{u}|dZ.
\end{equation}
Let $C_n$ be a doubling constant on the space $\mathbb R^n_+\times \mathbb R$, that is we have for any ball
$$|B_{2r}|\le C_n|B_r|.$$
Clearly, 
$$w(X)=|B_{\delta(X)/8}|^{-1}\iint_{B(X,\delta(X)/8)} |u|\le (C_n)^2 |B_{\delta(X)/2}|^{-1}\iint_{B(X,\delta(X)/2)}|u|=(C_n)^2W(X).$$
What about the reverse inequality? With $N$ defined by
$$N(v)(Q)=\sup_{X\in\Gamma_{\varepsilon} (Q)}|v(X)|,$$
it follows by the level-set argument that for all $1\le q<\infty$ we have
$$\|N(w)\|_{L^q}\approx \|N(W)\|_{L^q},$$
where the implied constants are independent of the function $u$ and of $\epsilon>0$. Hence, let $K$ be a universal constant
such that 
$$\|N(W)\|_{L^q}\le K \|N(w)\|_{L^q}.$$

Fix $\gamma>0$ (to be determined later) and define the sets
$$B_{good}=\{Q\in \partial(\R^n_+\times\mathbb R):\, N(W)(Q)\le\gamma N(w)(Q)\},\qquad B_{bad}=\partial(\R^n_+\times\mathbb R)\setminus B_{good}.$$
We have that
$$\int_{B_{bad}}N(w)^q\le \gamma^{-q} \int_{B_{bad}}N(W)^q.$$
Therefore if 
$$\int_{B_{bad}}N(w)^q>\frac12 \int_{\partial(\R^n_+\times\mathbb R)}N(w)^q,$$ 
it would follow that
$$\int_{\partial(\R^n_+\times\mathbb R)}N(w)^q< 2\int_{B_{bad}}N(w)^q\le
2\gamma^{-q} \int_{B_{bad}}N(W)^q \le 2\gamma^{-q}K^q\int_{\partial(\R^n_+\times\mathbb R)}N(w)^q,$$
which is a contradiction if we choose $\gamma>0$ such that $2\gamma^{-q}K^q=1$. Let us make such  choice of $\gamma$. Notice that this choice is independent of $u$ and $\varepsilon$.

With $\gamma$ chosen as above, we must therefore have 
$$\|N(w)\|_{L^q(\partial(\R^n_+\times\mathbb R))}\le 2^{-q}\|N(w)\|_{L^q(B_{good})}.$$
We make our choice of the function $g$. We pick $g$ supported on the set $B_{good}$ such that $\|g\|_{L^{q'}}=1$ and 
$$\|N(w)\|_{L^q(B_{good})}=\int_{\partial(\R^n_+\times\mathbb R)}N(w)g.$$
Clearly we can pick $g$ nonnegative (since $N(w)\ge 0$) and we also have
$$\|N(W)\|_{L^q(\partial(\R^n_+\times\mathbb R))}\le K\|N(w)\|_{L^q(\partial(\R^n_+\times\mathbb R))}\le 2^{-q}K\int_{\partial(\R^n_+\times\mathbb R)}N(w)g.$$

Fix now $Q\in\mbox{supp } g$. We have for such $Q$ the two inequalities:
$$\gamma^{-1} N(W)(Q)\le N(w)(Q)\le (C_n)^2N(W)(Q)$$
and hence the qualities $N(w)(Q)$ and $N(W)(Q)$ are comparable. By the definition of $N$ there exists a point
$X\in\Gamma_\varepsilon(Q)$ such that
$$w(X)\ge\frac12 N(w)(Q)\ge (2\gamma)^{-1}N(W)(Q).$$

For such point $X$ consider the ball used to define $w(X)$, i.e. $B(X,\delta(X)/8)$. Simple geometric
considerations imply that for all points
$$Z\in B(X,\delta(X)/16)\cap\Gamma_\varepsilon(Q)\quad\mbox{ we have that }\quad B(X,\delta(X)/8)\subset B(Z,\delta(Z)/4)$$
and therefore
$$\fiint_{ B(Z,\delta(Z)/4)}|{u}|\ge C_n^{-1}w(X)\ge (2C_n\gamma)^{-1}N(W)(Q),$$
and trivially $\fiint_{ B(Z,\delta(Z)/4)}|{u}|\le C_nN(W)(Q)$. It follows that $\fiint_{ B(Z,\delta(Z)/4)}|{u}|\approx N(W)(Q)$.
Given the geometry of the set $\Gamma_\varepsilon(Q)$, clearly $B(X,\delta(X)/16)\cap\Gamma_\varepsilon(Q)$
is a set of positive measure. We now choose the function $\beta(\cdot,Q)$ to be
$$\beta(\cdot,Q)=\frac{1}{|B(X,\delta(X)/16)\cap\Gamma_\varepsilon(Q)|}\chi_{B(X,\delta(X)/16)\cap\Gamma_\varepsilon(Q)}.$$
Observe that it follows that such $\beta$ satisfies all properties stated in our Proposition. To make our choice complete for $Q\in B_{bad}$ (and thus not in the support of $g$), we may pick $\beta$ completely arbitrarily.

In summary, we have so far defined $g$ and $\beta$ such that whenever $Q\in\mbox{supp }g$ and $X\in\mbox{supp }\beta(\cdot,Q)$ we must have that 
$$(2C_n\gamma)^{-1}N(W)(Q)\le \fiint_{ B(X,\delta(X)/4)}|{u}|,$$
and therefore also 
\begin{equation}\label{goodp}
(2C_n\gamma)^{-1}\fiint_{ B(X,\delta(X)/2)}|{u}|\le \fiint_{ B(X,\delta(X)/4)}|{u}|.
\end{equation}
Up to this point there is no difference in the argument for $u$ that is vector valued.\medskip

Because of the choices we have made above, we only care about the points $X\in \R^n_+\times\mathbb R$ for which \eqref{goodp} holds. Thus for points $X$ where \eqref{goodp} does not hold we simply choose 
 $Y(X)=X$ and $\vec{\alpha}(X)=0$. Observe that this means that 
$$\vec{\alpha}(X)\eta_{\rho\delta(X)}(\cdot-Y(X))=0,$$
and thus there is no contribution from this term in the definition of $\vec{h}$. \medskip

It remains to define $Y$ and $\alpha$ when \eqref{goodp} holds for a point $X$. We start with scalar $u$.
We apply our assumption of H\"older regularity of $u$. As follows from \eqref{hoa} and \eqref{goodp} that on the ball $B(X,\delta(X)/4)$ we have
\begin{equation}\label{hoa1}
\sup_{Z,Z'\in B(X,\delta(X)/4)}|{u}(Z)-{u}(Z')|\le C'\left(\frac{\|Z-Z'\|}{\delta(X)}\right)^\alpha\fiint_{B(X,\delta(X)/2)}|{u}|
\end{equation}
$$\le 2C'C_n\gamma\left(\frac{\|Z-Z'\|}{\delta(X)}\right)^\alpha\fiint_{ B(X,\delta(X)/4)}|{u}|\le 2C'C_n\gamma\left(\frac{\|Z-Z'\|}{\delta(X)}\right)^\alpha\sup_{B(X,\delta(X)/4)}|u|.$$
Consider a point $Y'\in \overline{B(X,\delta(X)/4)}$ where $\sup_{B(X,\delta(X)/4)}|u|$ is attained (possibly at the boundary of the ball). Applying \eqref{hoa1} to this point we see that for all $Z\in B(Y',\rho'\delta(X))\cap B(X,\delta(X)/4)$ we have that
$$|u(Z)-u(Y')|\le 2C'C_n\gamma (\rho')^\alpha|u(Y')|.$$
We choose $\rho'$ so that $2C'C_n\gamma (\rho')^\alpha=1/2$ which ensures that 
$$|u(Z)-u(Y')|\le \frac12|u(Y')|,\qquad\mbox{for all }Z\in B(Y',\rho'\delta(X))\cap B(X,\delta(X)/4).$$
As $u$ is scalar, this inequality ensures that $u(Z)$ has the same sign as $u(Y')$.
Let $\alpha(X)=1$ when $u(Y')>0$ and $\alpha(X)=-1$ otherwise, that is, $\alpha(X)=\sgn u(Y')$. 
The geometry of the situation implies that there exits some $\rho=\rho(\rho',n)>0$ (the worst-case is when $Y'$ is on the boundary of $B(X,\delta(X)/4)$) such that for some point 
$Y\in B(Y',\rho'\delta(X))\cap B(X,\rho(X)/4)$ we have that 
$$B(Y,\rho\delta(X))\subset B(Y',\rho'\delta(X))\cap B(X,\delta(X)/4).$$
This point $Y$ is the point we set to be the outcome of the mapping $Y(X)$ and the $\rho$ we have found above is the one  used to define $\eta_{\rho\delta(X)}$.

We note that for all $Z\in B(Y,\rho\delta(X))$ the scalar function $u(Z)$ has the same sign and thus
$u(Z)\alpha(X)>0$. Furthermore since $|u(Z)|\ge \frac12|u(Y')|$ and the function 
$\eta_{\rho\delta(X)}(\cdot-Y(X)))$ is supported in the ball $B(Y,\rho\delta(X))$ and has integral equal to $1$, we have that
$$|u(Y')|\ge\iint_{\R^n_+\times\mathbb R}u(Z)\alpha(X)\eta_{\rho\delta(X)}(Z-Y(X)))\,dZ\ge \frac12|u(Y')|=\frac12\sup_{B(X,\delta(X)/4)}|u|.
$$
Since also
$$\sup_{B(X,\delta(X)/4)}|u|\le 2\fiint_{ B(Y,\rho\delta(X))}|{u}|\le 2C_n^{-\log_2(\rho/4)}
 \fiint_{ B(X,\delta(X)/4)}|{u}|,$$
this is only possible if
$$\iint_{\R^n_+\times\mathbb R}u(Z)\alpha(X)\eta_{\rho\delta(X)}(Z-Y(X)))\,dZ\approx \fiint_{ B(X,\delta(X)/4)}|{u}|.$$
Thus integrating this portion of the function $h$ against $u$ captures a good portion of the average of $|u|$ on $B(X,\delta(X)/4)$.

There are minor differences in the definition if $\vec{u}$ is a vector. 
Suppose that $\vec{u}$ has $k$ components. We choose $Y'$ in the way same as above. By \eqref{hoa1} there will be $j\in\{1,2,\dots,k\}$ such that
$$|{u_j}(Z)-{u_j}(Y')|\le |\vec{u}(Z)-\vec{u}(Y')|\le 2C'C_nk^{-1/2}\gamma (\rho')^\alpha|u_j(Y')|,$$
since, for at least one $j$,  $|u_j(Y')|\ge \sqrt{\frac1k}|u(Y')|$. We choose $\alpha_j(X)=\mbox{sgn }u_j(Y')$ and
$\alpha_i(X)=0$ for $i\ne j$. The rest of the argument follows as above replacing $u$ by $u_j$ which is a scalar function. We choose $\rho'>0$ such that $2C'C_nk^{-1/2}\gamma (\rho')^\alpha=1/2$, i.e., now $\rho'$  also depends on $k$ - the dimension of our vector.
We again obtain the desired conclusion that 
\begin{equation}\label{zas}
\iint_{\R^n_+\times\mathbb R}\vec{u}(Z)\vec{\alpha}(X)\eta_{\rho\delta(X)}(\cdot-Y(X)))\approx \fiint_{ B(X,\delta(X)/4)}|{\vec{u}}|.
\end{equation}
\medskip

In summary, we have chosen $g$ such that 
$\|N_{1,\varepsilon}(\vec{u})\|_{L^q(\partial(\R^n_+\times\mathbb R))}\approx\int_{\partial(\R^n_+\times\mathbb R)}N(w)g$
and for any $Q\in\mbox{supp }g$ we have that for all $X\in\mbox{supp }\beta(\cdot, Q)$ \eqref{zas} holds.
Thus the function $\vec{h}$ integrated against $\vec{u}$ picks up a sufficiently large portion of the $L^q$ norm of 
$\|N_{1,\varepsilon}(\vec{u})\|_{L^q}$.
\end{proof}

\noindent{\bf Remark on differentiability of $\vec h$:}  Given the definition of $\vec h$, if we take a partial (spatial derivative $\partial_i$) of this function we see that
$$\partial_i[\vec{\alpha}(X)\beta(X,Q) \eta_{\rho\delta(X)}(\cdot-Y(X))]=\vec{\alpha}(X)\beta(X,Q) \partial_i[ \eta_{\rho\delta(X)}(\cdot-Y(X))],$$
since the derivative only acts on the smooth function. Given the scaling we have applied to $\eta$,
differentiating in Z gives
$$\partial_i[ \eta_{\rho\delta(X)}(Z-Y(X))]=(\rho\delta(X))^{-n-3}\partial_i\eta((z-y(X))/(\rho\delta(X)),(t-\tau(X))/(\rho\delta(X))^2).$$
Recall that the point $Z\in B(X,\delta(X)/4)$ and hence $\delta(Z)\approx\delta(X)$, implying that
$$\left|\partial_i[\vec{\alpha}(X)\beta(X,Q) \eta_{\rho\delta(X)}(\cdot-Y(X))]\right|\lesssim \delta(Z)^{-1}|\vec\alpha(X)|\beta(X,Q)(|\partial_i\eta|)_{\rho\delta(X)}(\cdot-Y(X)),$$
which is a function of the same type as $\vec h$ which we shall call ${h_1}$.
After integrating, we conclude that
\begin{equation}\label{diffh}
|\partial_i\vec h(Z)|\lesssim \delta^{-1}(Z){h_1}(Z),\qquad\forall Z,
\end{equation}
and similarly for the time derivative
\begin{equation}\label{diffht}
|\partial_t\vec h(Z)|\lesssim \delta^{-2}(Z){h_2}(Z),\qquad\forall Z.
\end{equation}
Here both ${h_1}$ and ${h_2}$ enjoy the estimates \eqref{e2aa}, as they are of the same type as $\vec{h}$. Similar estimates hold for higher order derivatives with each spatial derivative bringing an extra $\delta^{-1}(Z)$ and the time derivative 
$\delta^{-2}(Z)$.
\vglue4mm

 Let $u$ be the solution of the following boundary value problem
\begin{equation}
\mathcal{L}u=-\partial_t u+\div(A\nabla u)=0\mbox{ in }{\mathbb R}^n_+,\qquad u\Big|_{\partial{\mathbb R}^n_+}=f,\label{e5xb}
\end{equation}

We find $\vec{h}$ as in Proposition \ref{betterh} for the vector $\nabla u$. The reason we can apply this result to $\nabla u$ is that under our assumption  $x_n|\nabla A|+x_n^2\abs{\dr_t A}\lesssim 1$ the gradient of $u$ does enjoy interior H\"older estimates (\cite{L87}), and given the scaling $|\nabla A|\lesssim x_n^{-1}$  the constant in the condition \eqref{hoa} does not depend on the size of the ball $B_r$ - the condition $B_{4r}\subset\R^n_+\times \R$ ensures we stay in the interior of the domain of at least distance $r$ away from the boundary.

\subsection{The adjoint inhomogenous Dirichlet problem}

We let $v:{\mathbb R}^n_+\times \R\to\mathbb R$ be the solution of the inhomogenous Dirichlet problem for the operator ${\LL}^*$ (adjoint to $\LL$) which is a backward in time parabolic operator:
\begin{equation}
{ \LL}^*v=\partial_tv+\div(A^*\nabla v)=\div(\vec{h})\mbox{ in }{\mathbb R}^n_+\times\mathbb R,\qquad v\Big|_{\partial({\mathbb R}^n_+\times\mathbb R)}=0.\label{e3}
\end{equation}

Fix $\varepsilon>0$. Our aim is to estimate $N_{1,\varepsilon}(\nabla u)$  in $L^q$ using Proposition \ref{newh}. Then since $\vec{h}\big|_{\partial({{\mathbb R}^n_+\times\mathbb R})}=0$ and $\vec{h}$ vanishes as $x_n\to\infty$, we have by integration by parts
\begin{equation}\label{e6}
\|\wt N_{1,\varepsilon}(\nabla u)\|_{L^q}\lesssim \iint_{{{\mathbb R}^n_+\times\mathbb R}}\nabla u\cdot\vec{h}\,dZ
=-\iint_{{\mathbb R}^n_+\times\mathbb R} u\,\div\vec{h}\,dZ=-
\iint_{{{\mathbb R}^n_+\times\mathbb R}}u\,{\LL}^*v\,dZ
\end{equation}
We now move $u$ inside the divergence operator and apply the divergence theorem.
\begin{equation}\nonumber
\mbox{RHS of \eqref{e6}}=-\iint_{{{\mathbb R}^n_+\times\mathbb R}}\div(uA^*\nabla v)\,dZ+\iint_{{\mathbb R}^n_+\times\mathbb R} A\nabla u\cdot\nabla v\,dZ+\iint_{{\mathbb R}^n_+\times\mathbb R} (\partial_tu)v\,dZ
\end{equation}
$$=
\int_{\partial{{(\mathbb R}^n_+\times\mathbb R)}}u(\cdot,0,\cdot)a^*_{nj}\partial_jv\,d\sigma.$$
since
$$\iint_{{\mathbb R}^n_+\times\mathbb R} A\nabla u\cdot\nabla v\,dZ+\iint_{{\mathbb R}^n_+\times\mathbb R} (\partial_tu)v\,dZ=-\iint_{{\mathbb R}^n_+\times\mathbb R} \mathcal Lu\,v\,dZ=0.$$
Here there is no boundary integral since $v$ vanishes on the boundary of ${{\mathbb R}^n_+\times\mathbb R}$. It follows that
\begin{equation}\label{e8c}
\|N_{1,\varepsilon}(\nabla u)\|_{L^q}\lesssim
\int_{\partial{{(\mathbb R}^n_+\times\mathbb R)}}u(\cdot,0,\cdot)a^*_{nj}\partial_jv\,d\sigma,
\end{equation}
where the implied constant in \eqref{e8c} is independent of $\varepsilon>0$. Now, we use the fundamental theorem of calculus and the decay of $\nabla v$ at infinity to write \eqref{e8c} as
\begin{equation}\label{e9}
\|N_{1,\varepsilon}(\nabla u)\|_{L^q}\lesssim-\int_{\partial({{\mathbb R}^n_+}\times\mathbb R)}u(x,0,t)\left(\int_0^\infty \frac{d}{ds}\left(a^*_{nj}(x,s,t)\partial_jv(x,s,t)\right)ds\right)dx\,dt.
\end{equation}
Recall that $\div(A^*\nabla v)=\div(\vec{h})-\partial_tv$ and hence the righthand side of \eqref{e9} equals 
\begin{equation}\label{e10}
\int_{\partial{{(\mathbb R}^n_+}\times\mathbb R)}u(x,0,t)\Bigg(\int_0^\infty \Big[\sum_{i<n}\partial_i(a^*_{ij}(x,s,t)\partial_jv(x,s,t))-\div\vec{h}(x,s,t)
\end{equation}
\begin{equation}\nonumber
+\partial_tv(x,s,t)\Big]ds\Bigg)dx\,dt.
\end{equation}
We integrate by parts moving $\partial_i$ for $i<n$ onto $u(\cdot,0,\cdot)$. We also peel off half-derivative from $\partial_t v$ onto $u$.

The integral term containing $\partial_nh_n(x,s,t)$ does not need to be considered as it equals to zero by the fundamental theorem of calculus since $\vec{h}(\cdot,0,\cdot)=\vec{0}$ and
$\vec{h}(\cdot,s,\cdot)\to\vec{0}$ as $s\to\infty$. 
It follows that
\begin{eqnarray}\nonumber
&&\|N_{1,\varepsilon}(\nabla u)\|_{L^q}\\\nonumber&\lesssim& \int_{\partial({{\mathbb R}^n_+}\times\mathbb R)} \nabla_\parallel f(x,t)\cdot\left(\int_0^\infty \left[\vec{h}_\parallel(x,s,t)-(A^*\nabla v)_\parallel(x,s,t)\right]ds\right)dx\,dt
\\\nonumber&-& \int_{\partial({{\mathbb R}^n_+}\times\mathbb R)} D^{1/2}_t f(x,t)\left(\int_0^\infty H_tD^{1/2}_tv(x,s,t)\,ds\right)dx\,dt
\\
&=&I+II+III.\label{e11}
\end{eqnarray}
Here $I$ is the term containing $\vec{h}_\parallel$ and $II$ contains $(A^*\nabla v)_\parallel$ and $III$ term $H_tD^{1/2}_tv$. Here $H_t$ is the Hilbert transform in the $t$-variable.
The notation we are using here  is that, for a vector $\vec{w}=(w_1,w_2,\dots,w_n)$, the vector $\vec{w}_\parallel$ denotes the first $n-1$ components of $\vec{w}$, that is $(w_1,w_2,\dots,w_{n-1})$, and $\nabla_{||}=\nabla_x$.

As shall see below we do not need worry about term $I$. This is because what we are going to do next is essentially undo the integration by parts we have done above but we swap the function $u$ with another (better behaving) function $\tilde{u}$ with the same boundary data. Doing this we eventually arrive at $\|\tilde N(\nabla\tilde{u})\|_{L^q}$ plus some error terms (solid integrals) that arise from the fact that $u$ and $\tilde{u}$ disagree inside the domain. This explains why we get the same boundary integral as $I$  but with the opposite sign, as this \lq\lq reverse process" will undo and eliminate this boundary term. 

We solve a new auxiliary PDE problem to define $\tilde{u}$. Let $\tilde{u}$ be the solution
of the following boundary value problem for the operator $\LL_0$ whose matrix $A_0$ has the block-form
$
A_0=\left[ \begin{array}{c|c}
   A_\parallel & 0 \\
   \midrule
   0 & 1 \\
\end{array}\right] 
$ and

\begin{equation}
{ \LL}_0 \tilde{u}=-\partial_t\tilde{u}+\div(A_0\nabla\tilde{u})=0\mbox{ in }\Omega,\qquad \tilde{u}\Big|_{\partial\Omega}=f.\label{e5a}
\end{equation}
Recall that we have assumed that the $L^q$ Regularity problem for the operator 
${L_0}$ is solvable. We look the terms $II$ and $III$. 
Let
\begin{equation}
\vec{V}(x,x_n,t)=-\int_{x_n}^\infty (A^*\nabla v)_\parallel(x,s,t)ds,\quad {W}(x,x_n,t)=-\int_{x_n}^\infty H_tD^{1/2}_tv(x,s,t)ds.
\end{equation}
It follows that by the fundamental theorem of calculus and the observation that $\wt u=u$ on $\pom$, where $\wt u $ is same as in \eqref{e5a},  
\begin{multline}
    II=\int_{\partial{{(\mathbb R}^n_+\times\mathbb R)}}\nabla_\parallel \wt u(x,0,t)\cdot \vec{V}(x,0,t)dx\,dt\\
    =-\iint_{\Rn_+\times\R}\dr_n\br{\nabla_{||}\wt u(x,x_n,t)\cdot\vec V(x,x_n,t)}\dr_n(x_n)dxdx_ndt\\
    =\iint_{{{\mathbb R}^n_+\times\mathbb R}}\partial^2_{nn}\left[\nabla_\parallel \tilde{u}(x,x_n,t)\cdot \vec{V}(x,x_n,t)\right]x_n\,dx\,dx_n\,dt,
\end{multline}
and therefore,
\begin{eqnarray}
II&=&\iint_{{{\mathbb R}^n_+\times\mathbb R}}\partial^2_{nn}(\nabla_\parallel \tilde{u})\cdot \vec{V}x_n\,dx\,dx_n\,dt+2\iint_{{{\mathbb R}^n_+\times\mathbb R}}\partial_{n}(\nabla_\parallel \tilde{u})\cdot \partial_n(\vec{V})x_n\,dx\,dx_n\,dt\nonumber+\\&&+\iint_{{{\mathbb R}^n_+}\times\mathbb R}\nabla_\parallel \tilde{u}\cdot \partial^2_{nn}(\vec{V})x_n\,dx\,dx_n\,dt
=II_1+II_2+II_3.
\end{eqnarray}
Since $ \partial_n\vec{V}(x,x_n,t)=(A^*\nabla v)_\parallel$ the term $II_2$ is easiest to handle and can be estimated as a product of two square functions
\begin{equation}\label{125}
|II_2|
\le C\int_{\partial(\Rn_+\times\R)}\iint_{\Gamma(y,s)}\abs{\partial_{n}(\nabla_\parallel \tilde{u})}\abs{\partial_n(\vec{V})}x_n^{-n}dxdx_ndtdyds
\le C\|S(\partial_n\tilde{u})\|_{L^q}\|S(v)\|_{L^{q'}}.
\end{equation}

Next we look at $II_1$. We integrate by parts moving $\nabla_\parallel$ from $\tilde{u}$. This gives us
\begin{equation}\label{eq335}
II_1 =\iint_{{{\mathbb R}^n_+}\times\mathbb R}\partial^2_{nn}\tilde{u}\cdot\left(\int_{x_n}^\infty\div_{\parallel}(A^*\nabla v)_{\parallel}ds\right)x_n\,dx\,dx_n\,dt.
\end{equation}
Using the PDE $v$ satisfies we get that
$$\int_{x_n}^\infty\div_{\parallel}(A^*\nabla v)_{\parallel}ds=(a_{nj}\partial_jv)(x,x_n,t)+\int_{x_n}^\infty[\div\vec{h}-\partial_tv]\,ds.$$
Using this in \eqref{eq335} we see that
\begin{equation}\label{eq335a}
II_1 =\iint_{{{\mathbb R}^n_+\times\mathbb R}}(\partial^2_{nn}\tilde{u})(a_{nj}\partial_jv)x_n\,dx\,dx_n\,dt
+\iint_{{{\mathbb R}^n_+}\times\mathbb R}\partial^2_{nn}\tilde{u}
\cdot\left(\int_{x_n}^\infty[\div\vec{h}-\partial_tv]\,ds\right)x_n\,dx\,dx_n\,dt.
\end{equation}
Here the first term enjoys the same estimate as $II_2$, namely \eqref{125}. We work more with the second term which we call $II_{12}$. We integrate by parts in $\partial_n$. 
\begin{eqnarray}\nonumber
II_{12} &=&\iint_{{{\mathbb R}^n_+}\times\mathbb R}(\partial_{n}\tilde{u})[
\div\vec{h}-\partial_tv]x_n\,dx\,dx_n\,dt\\\nonumber
&\qquad -&\iint_{{{\mathbb R}^n_+}\times\mathbb R}\partial_{n}\tilde{u}
\cdot\left(\int_{x_n}^\infty[\div\vec{h}-\partial_tv]\,ds\right)\,dx\,dx_n\,dt\\\nonumber
&=&\iint_{{{\mathbb R}^n_+}\times\mathbb R}(\partial_{n}\tilde{u})[
\div\vec{h}]x_n\,dx\,dx_n\,dt
-\iint_{{{\mathbb R}^n_+}\times\mathbb R}(\partial_{n}\tilde{u})\partial_tvx_n\,dx\,dx_n\,dt\\\nonumber
&\qquad-&\iint_{{{\mathbb R}^n_+}\times\mathbb R}\partial_{n}\tilde{u}
\cdot\left(\int_{x_n}^\infty[\div\vec{h}-\partial_tv]\,ds\right)\,dx\,dx_n\,dt\\\nonumber
&=:&II_{121}+II_{122}-\iint_{{{\mathbb R}^n_+}\times\mathbb R}\partial_{n}\tilde{u}
\cdot\left(\int_{x_n}^\infty[\div\vec{h}-\partial_tv]\,ds\right)\,dx\,dx_n\,dt\nonumber.
\end{eqnarray}
For the last term, we separate it into two integrals and apply integration by parts to the first integral to get that
\begin{multline*}
    -\iint_{{{\mathbb R}^n_+}\times\mathbb R}\partial_{n}\tilde{u}
\cdot\left(\int_{x_n}^\infty[\div\vec{h}-\partial_tv]\,ds\right)\,dx\,dx_n\,dt
=\int_{\partial({\mathbb R^n_+}\times\mathbb R)}\tilde{u}(x,0,t)\left(\int_0^\infty \div\vec{h}\right)dx\,dt\\
-\iint_{{{\mathbb R}^n_+}\times\mathbb R}\nabla\tilde{u}\cdot\vec{h}\,dx\,dx_n\,dt+
\iint_{{{\mathbb R}^n_+\times\mathbb R}}\partial_{n}\tilde{u}\left(\int_{x_n}^\infty \partial_tv\,ds\right)\,\,dx\,dx_n\,dt\\
=:\int_{\partial({\mathbb R^n_+}\times\mathbb R)}\tilde{u}(x,0,t)\left(\int_0^\infty \div\vec{h}\right)dx\,dt+II_{123}+II_{124}.
\end{multline*}
Note that 
\[
\int_0^\infty \divg\vec h(x,s,t)ds=\int_0^\infty\br{\divg_{||}\vec h_{||}(x,s,t)+\dr_s\vec h(x,s,t)}ds=\int_0^\infty\divg_{||}\vec h_{||}(x,s,t)ds.
\]
We integrate by parts again and get that
\begin{equation*}
    \int_{\partial({\mathbb R^n_+}\times\mathbb R)}\tilde{u}(x,0,t)\left(\int_0^\infty \div\vec{h}\right)dx\,dt=
    -\int_{\partial(\mathbb R^n_+\times\mathbb R)}\nabla_\parallel \tilde{u}(x,0,t)\left(\int_0^\infty \vec{h}_\parallel\, ds\right)dx\,dt,
\end{equation*}
which is precisely the term $I$ defined by \eqref{e11} with an opposite sign. To summarize, we have obtained that
\[
II_{12}=II_{121}+II_{122}-I+II_{123}+II_{124}.
\]
We return to the term $II_{121}$ later. As we will see below the term $II_{124}$ appears again but with opposite sign and hence will cancel out. The term $II_{123}$ thanks to \eqref{e2} of Lemma \ref{l1bb} have the estimate:
$$|II_{123}|=\left|\iint_{{{\mathbb R}^n_+}\times\mathbb R}\nabla\tilde{u}\cdot\vec{h}\,dx\,dx_n\,dt\right|\le \left\|\tilde{N}_{1}(\nabla \tilde{u})\right\|_{L^q(\partial({\mathbb R^n_+}\times\mathbb R),d\sigma)}.$$

We repeat this approach with the term $III$. By the same logic we have
\begin{eqnarray}\nonumber
III &=& \iint_{{{\mathbb R}^n_+\times\mathbb R}}\partial^2_{nn}\left[D^{1/2}_t \tilde{u}(x,x_n,t) W(x,x_n,t)\right]x_n\,dx\,dx_n\,dt\\\label{w1}
&=&\iint_{{{\mathbb R}^n_+\times\mathbb R}}\partial^2_{nn}(D^{1/2}_t \tilde{u})W\,x_n\,dx\,dx_n\,dt+2\iint_{{{\mathbb R}^n_+\times\mathbb R}}\partial_{n}(D^{1/2}_t \tilde{u})\partial_n(W)x_n\,dx\,dx_n\,dt\nonumber+\\\nonumber&&+\iint_{{{\mathbb R}^n_+}\times\mathbb R}D^{1/2}_t \tilde{u}\, \partial^2_{nn}(W)x_n\,dx\,dx_n\,dt=II_{122}\\\nonumber
&+&\iint_{{{\mathbb R}^n_+\times\mathbb R}}\partial_{n}(D^{1/2}_t \tilde{u})\left(\int_{x_n}^\infty H_tD^{1/2}_tv\,ds\right)\,\,dx\,dx_n\,dt+\iint_{{{\mathbb R}^n_+}\times\mathbb R}D^{1/2}_t \tilde{u}\, \partial^2_{nn}(W)x_n\,dx\,dx_n\,dt\\\nonumber&=&II_{122}-\iint_{{{\mathbb R}^n_+\times\mathbb R}}\partial_{n}\tilde{u}\left(\int_{x_n}^\infty \partial_tv\,ds\right)\,\,dx\,dx_n\,dt+\iint_{{{\mathbb R}^n_+}\times\mathbb R}D^{1/2}_t \tilde{u}\, \partial^2_{nn}(W)x_n\,dx\,dx_n\,dt\\\nonumber&=&
II_{122}-II_{124}+III_1.
\end{eqnarray}

Here $II_{122}$ arises from the second term on the second line by shifting half-derivative from $\tilde{u}$.
 Next we look at $II_3$ and $III_1$ together. We see that
\begin{eqnarray}\nonumber
II_3&+&III_1=
 \iint_{{{\mathbb R}^n_+}\times\mathbb R}\nabla_\parallel \tilde{u}\cdot\, \partial_n(A^*\nabla v)_\parallel x_n\,dx\,dx_n\,dt +
 \iint_{{{\mathbb R}^n_+}\times\mathbb R}D_t^{1/2}\tilde{u}\, \partial_n(H_tD^{1/2}_t v) x_n\,dx\,dx_n\,dt \\\nonumber
&=&
\iint_{{{\mathbb R}^n_+}\times\mathbb R}\nabla_\parallel \tilde{u}\cdot ((\partial_n A^*)\nabla v)_\parallel x_n\,dx\,dx_n\,dt
+\iint_{{{\mathbb R}^n_+}\times\mathbb R}\nabla_\parallel \tilde{u}\cdot (A^*\nabla (\partial_n v))_\parallel x_n\,dx\,dx_n\,dt
\end{eqnarray}
\begin{equation}\nonumber
+\iint_{{{\mathbb R}^n_+}\times\mathbb R}D_t^{1/2}\tilde{u}\, \partial_n(H_tD^{1/2}_t v) x_n\,dx\,dx_n\,dt=IV_{1}+IV_{2}+IV_{3}.
\end{equation}

In order to handle the term $IV_{1}$ we will use the fact that the matrix $A$ satisfies the Carleson measure condition, \eqref{eq.Carl2}, and H\"older's inequality to obtain
\begin{equation}
|IV_{1}|\lesssim \|S(v)\|_{L^{q'}}\|N(\nabla\tilde{u})\|_{L^q}.
\end{equation}
For the term $ IV_{2}$ we separate the parallel and tangential parts of the gradient, to get 
\begin{eqnarray}\nonumber
IV_{2}&+&IV_{3}=\iint_{{{\mathbb R}^n_+}\times\mathbb R}\nabla_\parallel \tilde{u}\cdot (A_\parallel^*\nabla_\parallel (\partial_n v))x_n\,dx\,dx_n\,dt+
\iint_{{{\mathbb R}^n_+}\times\mathbb R}\nabla_\parallel \tilde{u}\cdot (a_{in}^*\partial^2_{nn} v)_{i<n}x_n\,dx\,dx_n\,dt\\\nonumber
&+&\iint_{{{\mathbb R}^n_+}\times\mathbb R}D_t^{1/2}\tilde{u}\, \partial_n(H_tD^{1/2}_t v) x_n\,dx\,dx_n\,dt=
-\iint_{{{\mathbb R}^n_+}\times\mathbb R}\div_\parallel(A_\parallel\nabla \tilde{u})(\partial_nv)x_ndx\,dx_n\,dt
\\\nonumber
&+&\iint_{{{\mathbb R}^n_+}\times\mathbb R}\partial_t \tilde{u}(\partial_nv)x_ndx\,dx_n\,dt
+\iint_{{{\mathbb R}^n_+}\times\mathbb R}\nabla_\parallel \tilde{u}\cdot (a_{in}^*\partial^2_{nn} v)_{i<n}x_n\,dx\,dx_n\,dt\\\nonumber
&=& \iint_{{{\mathbb R}^n_+}\times\mathbb R}(\partial^2_{nn}\tilde{u})(\partial_nv)x_ndx\,dx_n\,dt+IV_{4}.\nonumber
\end{eqnarray}
Here we have integrated the first and third term by parts and then used the equation that $\tilde{u}$ satisfies.
The first term on the last line enjoys a square function estimate by $\|S(v)\|_{L^{q'}}\|S(\partial_n\tilde{u})\|_{L^q}$.
 For $IV_{4}$
we write $\partial^2_{nn}v$ as
$$\partial^2_{nn}v=\partial_n\left(\frac{a^*_{nn}}{a^*_{nn}}\partial_n v\right)=\frac1{a^*_{nn}}\partial_n(a^*_{nn}\partial_nv)-\frac{\partial_n a^*_{nn}}{a^*_{nn}}\partial_n v$$
$$=-\frac1{a^*_{nn}}\left[\div_\parallel(A^*_\parallel\nabla_\parallel v)+\sum_{i<n}\left[\partial_i(a^*_{in}\partial_n v)+\partial_n(a^*_{ni}\partial_i v)\right]+\partial_n(a^*_{nn})\partial_n v+\partial_tv -\div\vec{h}\right],$$
where the final line follows from the equation that $v$ satisfies. It therefore follows that the term $IV_{4}$ can be written as a sum of six terms (which we shall call $IV_{41},\dots,IV_{46}$).

Terms $IV_{41}$ and $IV_{42}$ are similar and we deal with then via integration by parts (in $\partial_i$, $i<n$):
\begin{equation}\label{eq133}
|IV_{41}|+|IV_{42}|\le C\iint_{{\mathbb R}^n_+\times\mathbb R}|\nabla^2 \tilde{u}||\nabla v|x_n+C\iint_{{\mathbb R}^n_+\times\mathbb R}|\nabla A||\nabla\tilde{u}||\nabla v|x_n.
\end{equation}
The first term of \eqref{eq133}  can be seen to be a product of two square functions and hence by H\"older it has an estimate by $\|S(\nabla\tilde{u})\|_{L^q}\|S(v)\|_{L^{q'}}$. The second term of \eqref{eq133} is similar to the term $IV_1$ with analogous estimate. For the third term $IV_{43}$ we observe that $\partial_n(a^*_{ni}\partial_i v)=\partial_i(a^*_{ni}\partial_n v)+(\partial_na^*_{ni})\partial_iv-(\partial_ia^*_{ni})\partial_nv$ which implies that it again can be estimated by the right-hand side of \eqref{eq133}. 
The same is true for the term $IV_{44}$ which has a bound by the second term on the right-hand side of \eqref{eq133}.  
It remains to consider the term  $IV_{45}-IV_{46}$ which are
\begin{equation}\label{eq134}
IV_{45}+IV_{46}=\sum_{i<n}\iint_{{{\mathbb R}^n_+\times\mathbb R}}\frac{a_{ni}}{a_{nn}}(\partial_i\tilde{u})\,[-v_t+\div\vec{h}]\,x_n\,dx\,dx_n\,dt.
\end{equation}
Notice the similarity of this term with $II_{121}$, hence the calculation below also applies to it. For the terms
 containing $\div\vec{h}$, since $|\div\vec{h}|x_n\lesssim H$ with $H$ of the same type as $\vec{h}$ we have that

\begin{equation}\nonumber
\left|\sum_{i}\iint_{{{\mathbb R}^n_+\times\mathbb R}}\frac{a_{ni}}{a_{nn}}(\partial_i\tilde{u})\div\vec{h}\,x_n\,dx\,dx_n\,dt\right|\le C\iint_{{\mathbb R}^n_+\times\mathbb R}|\nabla \tilde{u}|H
\lesssim \|N_1(\nabla\tilde u)\|_{L^q},
\end{equation}
by \eqref{e2aa}.

The remaining two terms to consider are the term $II_{122}$ (twice) and the terms containing $\partial_t$ in \eqref{eq134}. Taken together these terms are equal to
\begin{equation}\label{zmos}
-\sum_{i<n}\iint_{{{\mathbb R}^n_+\times\mathbb R}}\frac{a_{ni}}{a_{nn}}(\partial_i\tilde{u})v_t\,x_n\,dx\,dx_n\,dt-2\iint_{{{\mathbb R}^n_+\times\mathbb R}}(\partial_i\tilde{u})v_t\,x_n\,dx\,dx_n\,dt:=V_{i<n}+V_{i=n}.
\end{equation}
Starting with $i=n$ we have
\begin{eqnarray}\nonumber
V_{i=n}&=&\iint_{{{\mathbb R}^n_+\times\mathbb R}}\partial_n\tilde{u}\,\partial_tv\,\partial_n(x_n^2)\,dx\,dx_n\,dt\\\nonumber
&=&-\iint_{{{\mathbb R}^n_+\times\mathbb R}}\partial^2_{nn}\tilde{u}\,\partial_tv\,x_n^2\,dx\,dx_n\,dt
-\iint_{{{\mathbb R}^n_+\times\mathbb R}}\partial_n\tilde{u}\,(\partial_n\partial_tv)\,x_n^2\,dx\,dx_n\,dt\\\nonumber
&=&-\iint_{{{\mathbb R}^n_+\times\mathbb R}}\partial^2_{nn}\tilde{u}\,\partial_tv\,x_n^2\,dx\,dx_n\,dt
+\iint_{{{\mathbb R}^n_+\times\mathbb R}}(\partial_t\partial_n\tilde{u})\,\partial_nv\,x_n^2\,dx\,dx_n\,dt.
\end{eqnarray}

Recalling the two versions of the Area function introduced in \eqref{DefArea}, 
we see that for the second term of $V_{i=n}$ we have a bound by
$$\left|\iint_{{{\mathbb R}^n_+\times\mathbb R}}(\partial_t\partial_n\tilde{u})\,\partial_nv\,x_n^2\,dx\,dx_n\,dt\right|\lesssim \|A(\nabla\tilde{u})\|_{L^q}\|S(v)\|_{L^{q'}}.$$
We work a bit more on the first term. Using the PDE for $v$:
\begin{eqnarray}\label{zmet1}
&&-\iint_{{{\mathbb R}^n_+\times\mathbb R}}\partial^2_{nn}\tilde{u}\,\partial_tv\,x_n^2\,dx\,dx_n\,dt\\\nonumber
&=&\iint_{{{\mathbb R}^n_+\times\mathbb R}}\partial^2_{nn}\tilde u\mbox{ div}(A^*\nabla v)x_n^2\,dx\,dx_n\,dt
-\iint_{{{\mathbb R}^n_+\times\mathbb R}}(\partial^2_{nn}\tilde u)\mbox{div}\vec h\,x_n^2\,dx\,dx_n\,dt.
\end{eqnarray}
For the second term we use the remark on differentiability of the function $\vec{h}$ (namely \eqref{diffh}) which implies that
$|\mbox{div} \vec h|x_n\lesssim {H}$, for some function ${H}$ with the same properties as $\vec{h}$. Hence
$$\left|\iint_{{{\mathbb R}^n_+\times\mathbb R}}(\partial^2_{nn}\tilde u)(\mbox{div }\vec h)\,x_n^2\,dx\,dx_n\,dt\right|\lesssim \iint_{{{\mathbb R}^n_+\times\mathbb R}}|\nabla ^2\tilde{u}||H|x_n\,dx\,dx_n\,dt,$$
where this term enjoys the  bound  by $\|\tilde N(x_n\nabla^2\tilde u)\|_{L^q}\lesssim \|\tilde N(\nabla\tilde u)\|_{L^q}$ which follows from the Cacciopoli's inequality for second gradient (Lemma 4.4 of \cite{DH18}).
.
The first term on the last line of \eqref{zmet1} can be further estimated by
\begin{eqnarray}\nonumber
&&\left|\iint_{{{\mathbb R}^n_+\times\mathbb R}}\partial^2_{nn}\tilde u\mbox{ div}(A^*\nabla v)x_n^2\,dx\,dx_n\,dt\right|\\\nonumber
&\lesssim& \iint_{{{\mathbb R}^n_+\times\mathbb R}}|\partial^2_{nn}\tilde u||\nabla A||\nabla v|x_n^2\,dx\,dx_n\,dt
+ \iint_{{{\mathbb R}^n_+\times\mathbb R}}|\partial^2_{nn}\tilde u||\nabla^2 v|x_n^2\,dx\,dx_n\,dt\\\nonumber
&\lesssim& \|S(\nabla\tilde u)\|_{L^q}\|S(v)\|_{L^q}+ \|S(\nabla\tilde u)\|_{L^q}\|\tilde{A}(v)\|_{L^{q'}},
\end{eqnarray}
where for the penultimate term we have used boundedness of $|\nabla A|x_n$.

We proceed similarly for the terms $V_{i<n}$:
\begin{eqnarray}\nonumber
V_{i<n}&=&\frac12\iint_{{{\mathbb R}^n_+\times\mathbb R}}\frac{a_{ni}}{a_{nn}}\partial_n\tilde{u}\,\partial_tv\,\partial_n(x_n^2)\,dx\,dx_n\,dt\\\nonumber
&=&-\frac12\iint_{{{\mathbb R}^n_+\times\mathbb R}}\frac{a_{ni}}{a_{nn}}\partial^2_{nn}\tilde{u}\,\partial_tv\,x_n^2\,dx\,dx_n\,dt
-\frac12\iint_{{{\mathbb R}^n_+\times\mathbb R}}\frac{a_{ni}}{a_{nn}}\partial_n\tilde{u}\,(\partial_n\partial_tv)\,x_n^2\,dx\,dx_n\,dt\\\nonumber
&&-\frac12\iint_{{{\mathbb R}^n_+\times\mathbb R}}\partial_n\left(\frac{a_{ni}}{a_{nn}}\right)\partial_n\tilde{u}\,\partial_tv\,x_n^2\,dx\,dx_n\,dt.
\end{eqnarray}
The first term of the second row enjoys bounds identical to the term in \eqref{zmet1}, i.e., 
 by $\|S(\nabla\tilde{u})\|_{L^q}\|\tilde A(v)\|_{L^{q'}}+\|\tilde N(\nabla\tilde u)\|_{L^q}$ as above, and the last term is bounded by $\|\tilde N(\nabla\tilde{u})\|_{L^q}(1+\|\tilde{A}(v)\|_{L^{q'}}+\| S(v)\|_{L^{q'}})$ using the Carleson condition on $\nabla A$, the PDE for $v_t$, and $\abs{\divg \vec h}x_n\lesssim H$.
Meanwhile for the second term of the second row we have
\begin{eqnarray}\nonumber
&&-\frac12\iint_{{{\mathbb R}^n_+\times\mathbb R}}\frac{a_{ni}}{a_{nn}}\partial_n\tilde{u}\,(\partial_n\partial_tv)\,x_n^2\,dx\,dx_n\,dt\\\nonumber&=&\frac12\iint_{{{\mathbb R}^n_+\times\mathbb R}}\partial_t\left(\frac{a_{ni}}{a_{nn}}\right)\partial_n\tilde{u}\,\partial_nv\,x_n^2\,dx\,dx_n\,dt+\frac12\iint_{{{\mathbb R}^n_+\times\mathbb R}}\frac{a_{ni}}{a_{nn}}(\partial_t\partial_n\tilde{u})\,\partial_nv\,x_n^2\,dx\,dx_n\,dt.
\end{eqnarray}
Here the last term clearly is bounded by $\|A(\nabla\tilde{u})\|_{L^q}\|S(v)\|_{L^{q'}}$, while the penultimate term
(using the Carleson condition \eqref{E:1:carl} for $|\partial_tA|^2x_n^3$) is bounded by $\|\tilde N(\nabla\tilde{u})\|_{L^q}\|S(v)\|_{L^{q'}}$.\medskip

We make one more observation before summing up the argument. 
In section \ref{S:Area-par} we will establish \eqref{sqrAfinal} which provides bounds on $\|A(\nabla\tilde{u})\|_{L^q}$.

Therefore, under the assumptions we have made (along with \eqref{sqrAfinal}), namely the solvability of Regularity problem $(R)_q^{\LL_0}$ for the block form-parabolic operator accompanied by the estimates:

\begin{equation}\label{NSA}
\|\tilde{N}(\nabla \tilde{u})\|_{L^q}+\|S(\nabla\tilde{u})\|_{L^q}
\le C(\|\nabla_\parallel f\|_{L^q}+\|D^{1/2}_t f\|_{L^q}),
\end{equation}
as well as the solvability of the inhomogeneous Dirichlet problem \eqref{e3} with the estimates:
\begin{equation}
\|N(v)\|_{L^{q'}}+\|S(v)\|_{L^{q'}}+\|\tilde{A}(v)\|_{L^{q'}}\le C.
\label{e4}
\end{equation}
then:
$$II+III=\int_{\partial({{\mathbb R}^n_+}\times\mathbb R)}\nabla_\parallel u(x,0)\cdot \vec{V}(x,0)dx\,dt+
 \int_{\partial({{\mathbb R}^n_+}\times\mathbb R)} D^{1/2}_t f(x)W(x,0)dx\,dt$$
$$\le C(\|\nabla_\parallel f\|_{L^q}+\|D^{1/2}_t f\|_{L^q})-I.$$

Pulling all the estimates together (since term $I$ cancels out), we have established the following:
$$ \|\tilde{N}_{1,\varepsilon}(\nabla u)\|_{L^q}\le C(\|\nabla_\parallel f\|_{L^q}+\|D^{1/2}_t f\|_{L^q}).$$
\medskip

An argument is required to demonstrate that the control of $\tilde{N}_{1,\varepsilon}(\nabla u)$ of a solution $\LL u=0$ implies the control of $\tilde{N}(\nabla u)$ (the $L^2$ averaged version of the nontangential maximal function as defined in \eqref{def.N2}). Firstly, as the established estimates are independent of $\varepsilon>0$ we obtain
$$\|\tilde{N}_{1}(\nabla u)\|_{L^q}=\lim_{\varepsilon\to 0+}\|\tilde{N}_{1,\varepsilon}(\nabla u)\|_{L^q}\le C(\|\nabla_\parallel f\|_{L^q}+\|D^{1/2}_t f\|_{L^q}).$$
Secondly, 
by Lemma \ref{lem.gradL2-L1}, 
$$\left(\fiint_B|\nabla u|^2\right)^{1/2}\lesssim \left(\fiint_{2B}|\nabla u|\right),$$
which implies a bound of $\tilde{N}(\nabla u)(\cdot)$ defined using cones $\Gamma_a(\cdot)$ of some aperture $a>0$ by $\tilde{N}_1(\nabla u)(\cdot)$ defined using cones $\Gamma_b(\cdot)$ of some slightly larger aperture $b>a$.

\subsection{The area function of the adjoint Dirichlet problem}
\label{S:Area-adj}

As our first task we find an estimate of the $L^{q'}$ norm of $\tilde A(v)$ which is needed for \eqref{e4}. Fix a boundary point $P$ and consider
the nontangential cone $\Gamma(P)$ which can be covered by a collection of parabolic Whitney cubes $Q_i$ with their enlargements $2Q_i$ having finite overlap and $3Q_i\subset \R^n_+\times\R$.

Consider now one such cube $Q_r=Q_i$ (with $r$ being the sidelength of this cube). 
We take the spatial gradient of this PDE. For simplicity, let $u_i = \partial_{i} v$ and $w_i=u_i\zeta^2$, $1\le i\le n$ where $0\leq \zeta \leq 1$ is a smooth cutoff function equal to $1$ on $Q_{r}$ and supported in $Q_{2r}$ satisfying $r|\nabla \zeta| + r^2 |\zeta_{t}| \leq c$ for some $c>0$. It follows that (summing over repeating indices)
\[
\iint_{Q_{2r}} (u_{i})_t w_i \,dX\,dt = \iint_{Q_{2r}} \left( A^T \nabla u_i + (\partial_{i} A^T) \nabla v \right)\cdot\nabla w_i +\iint_{Q_{2r}}\partial_i (\mbox{div }\vec h) w_i\,dX\,dt.\]
which implies that
\begin{multline*}
    \iint A^T\nabla u_i\cdot\nabla u_i\zeta^2dXdt=
    \iint\dr_t\br{\frac{u_i^2\zeta^2}{2}}dXdt-\iint\dr_t\br{\zeta^2}u_i^2dXdt\\
    -2\iint A^T\nabla u_i\cdot\nabla\zeta\, u_i\zeta dXdt
    -\iint \br{\dr_i A^T}\nabla v\cdot\nabla u_i\,\zeta^2dXdt\\
    -2\iint \br{\dr_i A^T}\nabla v\cdot\nabla \zeta\, u_i\,\zeta\,dXdt
    -\iint\dr_i(\divg \vec h) u_i\zeta^2dXdt.
\end{multline*}


Using the ellipticity and boundedness of the coefficients and the Cauchy-Schwarz inequality, it follows that
\[\begin{split}
& \lambda \iint_{Q_{r}} \left|\nabla^2 v\right|^{2}\,dX\,dt \\
&\leq \frac{C}{r^2} \iint_{Q_{2r}} |\nabla v|^{2} \,dX\,dt
 + {C} \iint_{Q_{2r}} |\nabla A|^{2} |\nabla v|^{2}\,dX\,dt
+ Cr^n\fiint_{Q_{2r}} H|\nabla v|\,dX\,dt\\
&\leq \frac{C}{r^2} \iint_{Q_{2r}} |\nabla v|^2 \,dX\,dt+ C\left(\fiint_{Q_{2r}}|\nabla v|^2r^2\,dX\,dt\right)^{1/2}r^{n-1}\left(\fiint_{Q_{2r}}|H|^2\,dX\,dt\right)^{1/2},
\end{split}\]
for some constant $C= C(\lambda, \Lambda)>0$ and $|\nabla^2\vec h|x_n^2\le H$, where $H$ is a function of the same type as $\vec{h}$.

We now multiply above inequality by $r^{-n+2}$, and sum over all indices $i$ for $Q=Q_i$ to obtain:

$$\tilde{A}^2(v)(P)\lesssim S^2(v)(P)
+{\tilde N}(x_n\nabla v)(P)\iint_{\Gamma_{a'}(P)}W(H)x_n^{-n-1}\,dX\,dt.$$
Here $\Gamma_{a'}(P)$ is a cone of wider aperture than the original cone so that $\bigcup_i 2Q_i\subset\Gamma_{a'}(P)$ and the square and nontangential maximal functions $S(v),\, {\tilde N}(x_n\nabla v)$ are also defined using this wider cone.  $W(H)$ denotes the $L^2$ average of $H$ over Whitney balls (c.f. \eqref{wW}). 
And, we have used the estimate of $\left(\fiint_{Q_{2r}}|\nabla v|^2r^2\,dX\,dt\right)^{1/2}$ by ${\tilde N}(x_n\nabla v)(P)$.
Recall the setting of the paper \cite{KP2} which introduced the function $T_p(H)(P)$ as follows:
$$T_p(H)(P)=\left(\iint_{\Gamma(P)}\left(\fiint_{B((X,t),\delta((X,t))/2)}|H|^p\right)^{1/p} x_n^{-n-1}\,dX\,dt. \right)^{1/2}. $$
In particular in $L^{q'}$ we have
\begin{equation}\label{A-v}
\|\tilde{A}(v)\|_{L^{q'}}\lesssim \|{S}(v)\|_{L^{q'}}+\|N(x_n\nabla v)\|_{L^{q'}}+\|T_2(H)\|_{L^{q'}}.
\end{equation}

By the results of Ulmer \cite{Up} (see also \cite{U}), Lemma 3.7 provides bounds $L^{q'}$ norms of $N(v)$ and $N(x_n \nabla v)$, while  Corollary 3.9 bounds $S(v)$ and finally Lemma 3.10 established bounds for $T_p(H)$. In summary, we have that there exists $C>0$ independent of $v$ such that
$$\|N(v)\|_{L^{q'}}+\|N(x_n\nabla v)\|_{L^{q'}}+\|{S}(v)\|_{L^{q'}}+\|T_2(H)\|_{L^{q'}}\le C,$$
for all $q'$ for which the parabolic Dirichlet problem for the operator $\mathcal L^*$ is solvable. Thanks to \eqref{A-v} we therefore can conclude that $\|\tilde{A}(v)\|_{L^{q'}}\le C'$ as well. This establishes all required bounds we have asked for  in \eqref{e4} for the inhomogeneous parabolic Dirichlet problem.\medskip

It follows we have reduced the problem of solvability of the Regularity problem for the parabolic PDE
$u_t-\mbox{div}(A\nabla u)=0$ with elliptic matrix $A$ to a related problem of 
solvability of the Regularity problem for the parabolic PDE
$\tilde u_t-\mbox{div}(\tilde{A}\nabla \tilde u)=0$ where the matrix $\tilde{A}$ has an additional feature, namely it is of the block form. It remain to establish \eqref{NSA} for such solutions $\tilde u$.
This concludes the proof of Theorem \ref{t1}.\qed

\bigskip

\section{The area function of $\nabla u$ for the block-form parabolic PDEs.}
\label{S:Area-par}

Let us recall the subsection \ref{S:Area-adj} where we have established estimates for the area function of the adjoint inhomogenuous solution $v$. We now establish similar estimates for the area function of $u$; the solution of the block-form parabolic PDEs. Let us denote $w=\partial_t u$, then $w$ will solve the PDE
\begin{equation}
-\partial_tw+\mbox{div}_x(A_\parallel (\nabla w)_\parallel)+\partial^2_{x_nx_n}w=\mbox{div}_x((\partial_tA_\parallel)(\nabla u)_\parallel),
\end{equation}
where, consistent with our previous notation, $(\nabla w)_\parallel$ and $A_\parallel$ denotes the first $n-1$ components of $\nabla w$ and the $n-1\times n-1$ subblock of $A$, respectively.
Mimicking the calculation for $\tilde{A}(v)$ in subsection \ref{S:Area-adj} we multiply both sides by $w\zeta^2$
where $\zeta$ is a cutoff function equal to one on $Q$ for an interior Witney cube and zero outside $2Q$. After integration by parts we obtain:

\[\begin{split}
&\frac{1}{2}\iint_{2Q} \left[(w\zeta)^2\right]_{t} \,dX\,dt
+ \iint_{2Q} A \nabla(w\zeta)\cdot \nabla(w\zeta)\,dX\,dt 
\\
&=  \iint_{2Q} w^{2}\zeta\zeta_{t} \,dX\,dt  + \iint_{2Q} w^2 A\nabla\zeta\cdot\nabla \zeta \,dX\,dt + \iint_{2Q} A^T \nabla (w\zeta)\cdot w\nabla\zeta\,dX\,dt\\
&\quad- \iint_{2Q} A \nabla (w\zeta)\cdot w\nabla\zeta\,dX\,dt
 - \iint_{2Q}  (\partial_{t} A)\nabla u\zeta\cdot (\nabla w\zeta)\,dX\,dt \\ &\quad- \iint_{2Q} (\partial_{t} A)\nabla u( w\cdot( \zeta\nabla\zeta)) \,dX\,dt.
\end{split}\]

Using the ellipticity and boundedness of the coefficients and the Cauchy-Schwarz inequality, it follows that
\[\begin{split}
& \lambda \iint_{Q} \left|\nabla w \right|^{2}\,dX\,dt \leq \frac{C}{r^2}  \iint_{2Q} |w|^{2} \,dX\,dt
 + C \iint_{2Q} |\partial_t A|^{2} |\nabla u|^{2}\,dX\,dt.
\end{split}\]
for some constant $C= C(\Lambda)$.

As before in \ref{S:Area-adj} , we cover $\Gamma_a(P)$ by a collection of Whitney cubes $(Q_i)_{i}$, 
multiply above inequality by $r_i^{-n+2}$, and sum over all indices $i$ for $Q=Q_i$ to obtain:

\begin{equation}\label{sqrA}
{A}^2(\nabla u)(P)\lesssim \iint_{\Gamma_{a'}(P)}|u_t|^2x_n^{-n}\,dX\,dt
+\iint_{\Gamma_{a'}(P)} |\partial_t A|^{2} |\nabla u|^{2}x_n^{-n+2}\,dX\,dt.
\end{equation}
Here $\Gamma_{a'}(P)$ is a cone of wider aperture (smaller slope) than the original cone so that $\bigcup_i 2Q_i\subset\Gamma_{a'}(P)$. We note that \eqref{sqrA} also holds if the cones are truncated.

We note that the first term can be further estimated using the equation for $u$ and hence
$$\iint_{\Gamma_{a'}(P)}|u_t|^2x_n^{-n}\,dX\,dt\le S^2(\nabla u)(P)+
\iint_{\Gamma_{a'}(P)} |\nabla_\parallel A|^{2} |\nabla u|^{2}x_n^{-n}\,dX\,dt,$$
and therefore
\begin{equation}\label{sqrAA}
{A}^2(\nabla u)(P)\lesssim S^2(\nabla u)(P)
+\iint_{\Gamma_{a'}(P)} (|\partial_t A|^{2}x_n^2+|\nabla_\parallel A|^{2}) |\nabla u|^{2}x_n^{-n}\,dX\,dt.
\end{equation}
Again, a truncated version of this inequality also holds. Using truncated cones at the height $r$ and integrating over a boundary ball $\Delta_r$ we obtain that:
$$\|A(\nabla u)\|^2_{L^2(\Delta_r)}\lesssim \|S(\nabla u)\|^2_{L^2(2\Delta_r)}
+\iint_{T(2\Delta_r)} (|\partial_t A|^{2}x_n^2+|\nabla_\parallel A|^{2}) |\nabla u|^{2}x_n\,dX\,dt
$$
$$\le \|S(\nabla u)\|^2_{L^2(2\Delta_r)}
+\|\mu_\parallel\|_C\|N(\nabla u)\|^2_{L^2(2\Delta_r)},$$
using the Carleson condition for $\mu_\parallel$.
In particular after taking $r\to\infty$ we obtain
\begin{equation}\label{eq.sqrAAL2}
    \|A(\nabla u)\|_{L^2(\R^{n-1}\times\R)}\le C \|S(\nabla u)\|_{L^2(\R^{n-1}\times\R)}+C\|\mu_\parallel\|_C^{1/2}\|N(\nabla u)\|_{L^2(\R^{n-1}\times\R)}.
\end{equation}
We are also interested in an $L^q$ version of this inequality, which will be obtained via a good-$\lambda$ argument.

Denote by $S_A(u)(P)$ the square root of the second term of \eqref{sqrAA}; that is,
$$S_A(u)(P)=\left(\iint_{\Gamma_{a'}(P)} (|\partial_t A|^{2}x_n^2+|\nabla_\parallel A|^{2}) |\nabla u|^{2}x_n^{-n}\,dX\,dt\right)^{1/2}.$$
It follows that if we establish $\|S_A(u)\|_{L^q}\le C\|N(\nabla u)\|_{L^q}$ for all $q>1$ it would follow from 
 \eqref{sqrAA} that
\begin{equation}\label{sqrAfinal}
\|A(\nabla u)\|_{L^q(\R^{n-1}\times\R)}\lesssim \|S(\nabla u)\|_{L^q(\R^{n-1}\times\R)}+\|\mu_\parallel\|_C^{1/2}\|N(\nabla u)\|_{L^q(\R^{n-1}\times\R)},\quad\mbox{for }q>1.
\end{equation}
Denote by $E_\lambda$ the set $\{P:\,S_A(u)(P)>2\lambda,\,N(\nabla u)(P)\le\gamma\lambda\}$ for some $\gamma>0$ to be determined later. We decompose the open set $\{P:\,S_A(u)(P)>\lambda\}$ and write it as a union of countable collection of open balls $\Delta_i$ of finite overlap such that each $\Delta_i$ has the property that $2\Delta_i\cap E_\lambda^c\ne\emptyset$.

If we prove that for each $i$ we have $|\Delta_i\cap E_\lambda|\le C\gamma^2|\Delta_i|$ it will follow that 
$$|E_\lambda|\le C\gamma^2|\{P:\,S_A(u)(P)>\lambda\}|.$$
From this the claim $\|S_A(u)\|_{L^q}\lesssim C\|N(\nabla u)\|_{L^q}$ for all $q>1$ follows.

We further decompose $S_A(u)(P)=S_{A,<r}(u)(P)+S_{A,\ge r}(u)(P)$, where  $S_{A,<r}$ is the truncated version of $S_A$ using cones of height $r=\mbox{diam}(\Delta_i)$, with $S_{A,\ge r}$ being the complementary term. We aim to prove that $S_{A,<r}(u)(P)>\lambda$ for all $P\in\Delta_i\cap E_\lambda$, provided $\gamma$ is chosen small enough.
Assuming this for the moment we have that
\begin{equation}\label{GLas}
|\Delta_i\cap E_\lambda|
\le\frac1{\lambda^2}\int_{\Delta_i\cap E_\lambda}S_{A,<r}^2(u).
\end{equation}
By Fubini, the first term enjoys the estimate
$$\frac1{\lambda^2}\int_{\Delta_i\cap E_\lambda}S_{A,<r}^2(u)dx\,dt\lesssim \frac1{\lambda^2}\iint_{\cup \{\Gamma^r(P):\, P\in \Delta_i\bigcap E_\lambda\}}(|\partial_t A|^{2}x_n^2+|\nabla_\parallel A|^{2}) |\nabla u|^{2}x_n\,dX\,dt$$
$$\le \gamma^2\iint_{T(2\Delta_i)}(|\partial_t A|^{2}x_n^2+|\nabla_\parallel A|^{2})x_n\,dX\,dt\le C\gamma^2\|\mu_\parallel\|_C|\Delta_i|,$$
using the Carleson condition and the fact that on the domain of integration $N(\nabla u)\le\gamma\lambda$.

It remains to show that $S_{A,<r}(u)(P)>\lambda$ for all $P\in\Delta_i\cap E_\lambda$.
We use the fact that $2\Delta_i$ is not contained in $E_\lambda$ and hence for some $Q\in 2\Delta_i\setminus\Delta_i$ we have that $S_{A}(u)(Q)\le\lambda$.
It follows that for all $P\in\Delta_i\cap E_\lambda$
\begin{equation}\label{diffG}
\mathcal I:=\iint_{\Gamma_{a'}(P)\setminus \Gamma_{a'}(Q)} (|\partial_t A|^{2}x_n^2+|\nabla_\parallel A|^{2}) |\nabla u|^{2}x_n^{-n}\,dX\,dt\ge 3\lambda^2.
\end{equation}
We consider the above integral of $\mathcal I$ over the set $(\Gamma_{a'}(P)\setminus \Gamma_{a'}(Q))\cap\{\delta(X,t)\ge r\}$, which
we denote $\mathcal I^{top}$, and $\mathcal I^{bottom}$ will be the difference. We further split $\mathcal I^{top}$ dyadically into regions where $x_n\approx 2^kr$, $k\ge 1$ to get
$$\mathcal I^{top}\le \sum_{k\ge 1}\iint_{(\Gamma_{a'}(P)\setminus \Gamma_{a'}(Q))\cap\{x_n\approx 2^kr\}} (|\partial_t A|^{2}x_n^2+|\nabla_\parallel A|^{2}) |\nabla u|^{2}x_n^{-n}\,dX\,dt$$
$$\le \sum_{k\ge 1}|(\Gamma_{a'}(P)\setminus \Gamma_{a'}(Q))\cap\{x_n\approx 2^kr\}|(2^kr)^{-n-2}\gamma^2\lambda^2,
$$
where we have used the trivial pointwise bounds $|\nabla_\parallel A|\lesssim x_n^{-1}$ and $|\partial_t A|\lesssim x_n^{-2}$ leaving us with $N(\nabla u)(P)$, which 
is bounded by $\gamma\lambda$ for $P\in E_\lambda$.

It remains to estimate the volume of the set $(\Gamma_{a'}(P)\setminus \Gamma_{a'}(Q))\cap\{x_n\approx 2^kr\}$. Being a difference of two cones with vertices of distance $\le r$, this set has at least one \lq\lq short side"
which does not grow as $2^k r$ and instead has size $r$ for each $k$. It follows that
$$|(\Gamma_{a'}(P)\setminus \Gamma_{a'}(Q))\cap\{x_n\approx 2^kr\}|\lesssim (2^kr)^{n+1}r=(2^k)^{n+1}r^{n+2}.$$
Using this and the estimate above we have that
$$\mathcal I^{top}\le C\gamma^2\lambda^2\sum_{k\ge 1}(2^k)^{n+1}r^{n+2}(2^kr)^{-n-2}=C\gamma^2\lambda^2\sum_{k\ge 1}2^{-k}=C\gamma^2\lambda^2.$$
Hence we conclude that for $P\in\Delta_i\cap E_\lambda$ we have that

$$\mathcal I^{bottom}=\mathcal I-\mathcal I^{top}\ge 3\lambda^2-C\gamma^2\lambda^2=(3-C\gamma^2)\lambda^2.$$
We chose $\gamma>0$ such that $(3-C\gamma^2)> 1$ and then we get that
$$S_{A,<r}(u)(P)^2\ge \mathcal I^{bottom}>\lambda^2,$$
for all $P\in\Delta_i\cap E_\lambda$ as desired. Hence we have established \eqref{sqrAfinal}.\qed

\section{Estimates for $D^{1/2}_{t}u$}

\subsection{The PDE for $D^{1/2}_{t}u$.}
In this section, we find the partial differential equation that $D^{1/2}_{t}u$ satisfies, where $u$ is a weak solution of \eqref{E:pde}. Define $w:=D^{1/2}_{t}u$. 
\begin{lemma}\label{lem.D1/2eq}
The function $w$ satisfies 
    \begin{equation}\label{eq.D1/2eq}
    \dr_tw(X,t)=\divg(A\nabla w)(X,t)
    +c_0\divg\br{\int_{s\in\R}\frac{A(X,t)-A(X,s)}{\abs{t-s}^{3/2}}\nabla u(X,s)ds}.
\end{equation}
\end{lemma}

\bp  We can assume that $A$ is smooth, which is valid by the reasoning in Remark \ref{re.pert} (3). We can also assume that $u$ is a Schwartz function, as these are dense in the solution space.

Taking the Fourier transform in $t$, we have 
\[
\dr_t w(X,\cdot)^{\wedge}(\tau)=i\tau D^{1/2}_{t}u(X,\cdot)^\wedge(\tau)=\abs{\tau}^{1/2}\dr_tu(X,\cdot)^\wedge(\tau).\]
Taking the Fourier transform in $t$ on both sides of the PDE for $u$ gives that
\[
\dr_tu(X,\cdot)^\wedge(\tau)=\divg\br{(A(X,\cdot)\nabla u(X,\cdot))^\wedge(\tau)}.
\]
Plugging this into the previous equation yields that  
\[
     \wh{\dr_tw(\tau)}=\abs{\tau}^{1/2}\divg\br{(A\nabla u)^\wedge(\tau)}
    =\divg\br{(D^{1/2}_{t}(A\nabla u))^\wedge(\tau)},
\]
where we have omitted $X$ for simplicity, 
and hence,
\begin{equation}\label{DtD1/2u.1}
    \dr_t w(X,t)=\divg\br{D^{1/2}_{t}(A\nabla u)(X,t)}.
\end{equation}
On the other hand,
\begin{multline*}
   c_0^{-1}D^{1/2}_{t}(A(X,t)\nabla u(X,t))
    = \int_{s\in\R}\frac{A(X,t)\nabla u(X,t)-A(X,s)\nabla u(X,s)}{\abs{t-s}^{3/2}}ds\\
    =\int_{\R}\frac{A(X,t)(\nabla u(X,t)-\nabla u(X,s))}{\abs{t-s}^{3/2}}ds
    +\int_{\R}\frac{(A(X,t)-A(X,s))\nabla u(X,s)}{\abs{t-s}^{3/2}}ds\\
    =c_0^{-1}A(X,t)D^{1/2}_{t}(\nabla u)(X,t)+\int_{\R}\frac{(A(X,t)-A(X,s))\nabla u(X,s)}{\abs{t-s}^{3/2}}ds\\
    =c_0^{-1}A(X,t)\nabla w(X,t)+\int_{\R}\frac{(A(X,t)-A(X,s))\nabla u(X,s)}{\abs{t-s}^{3/2}}ds.
\end{multline*}
Plugging this back into \eqref{DtD1/2u.1} gives \eqref{eq.D1/2eq}.

\ep

\subsection{Caccioppoli-type estimates for $\nabla^2D^{1/2}_{t}u$.}
We will need a Caccioppoli-type estimate for the second derivatives of $D^{1/2}_{t}u$, where $u$ is a weak solution of \eqref{E:pde}. Recall that $w=D^{1/2}_{t}u$ satisfies the equation \eqref{eq.D1/2eq}.

Since the equation contains a non-local term with respect to the time variable, we cannot hope for local estimates in time. Fortunately, we only need to localize the spatial variables in the Caccioppoli-type estimate in application. Let us denote the spatial cube centered at $Y\in\Rn_+$ with sidelength $r$ by $Q_r(Y)$; that is,
\[Q_r(Y)=\set{X\in\Rn_+:\, \abs{x_i-y_i}<r \text{ for all }1\le i\le n }.\]

\begin{lemma}[A Caccioppoli inequality for $\nabla^2D^{1/2}_{t}u$]\label{lem.Caccio}
     Suppose that $A$ satisfies \eqref{E:elliptic} and \eqref{E:1:bound}. Let $w$ be a solution of \eqref{eq.D1/2eq}. Let $Y=(y,y_n)\in\Rn_+$ and $r>0$ be such that $Q_{2r}=Q_{2r}(Y)\subset\Rn_+$. 
    Then 
    \begin{multline}\label{eq.Caccio}
        \int_{Q_r}\int_{\R}\abs{\nabla^2w}^2dtdX\lesssim \frac{1}{r^2}\int_{Q_{2r}}\int_{\R}\abs{\nabla w}^2dtdX+\int_{Q_{2r}}\int_{\R}\abs{\nabla A}^2\abs{\nabla u}^2\frac{dtdX}{x_n^2}\\
        +\int_{Q_{2r}}\int_{\R}\abs{\nabla A}^2M(\nabla u(X,\cdot))(t)^2\frac{dtdX}{x_n^2}
        +\int_{Q_{2r}}\int_{\R}M_{x_n^2}(\dr_\tau\nabla A(X,\cdot))(t)^2\abs{\nabla u}^2x_n^2dtdX\\
    +\int_{Q_{2r}}\int_{\R}\abs{\nabla^2u(X,t)}^2\frac{dtdX}{x_n^2},
    \end{multline}
    where $M_{x_n^2}$ is the localized maximal function at scale $x_n^2$ as defined in \eqref{def.trcmaxfunc}.    
\end{lemma}

\bp
Let $1\le k\le n$ be fixed. We differentiate the PDE \eqref{eq.D1/2eq} with respect to $x_k$ and take $\vp_k=(\dr_kw)\zeta^2$ as a test function, where $\zeta=\zeta(X)\in[0,1]$ is a smooth cutoff function equal to 1 on $Q_r(Y)$ and supported in $Q_{2r}(Y)$ satisfying $r\abs{\nabla\zeta}\le c$ for some $c>0$. It follows that
\begin{multline*}
   \iint\dr_t\dr_kw(\dr_kw) \zeta^2dtdX=-\iint (\dr_kA)\nabla w\cdot\nabla\br{\zeta^2\dr_kw}-\iint A\nabla\dr_kw\cdot\nabla(\zeta^2\dr_kw) \\
   - \iint \int_{s\in\R}\frac{\dr_k A(X,t)-\dr_k A(X,s)}{\abs{t-s}^{3/2}}\nabla u(X,s)ds\cdot\nabla\br{\zeta^2\dr_k w}\\
    -
    \iint \int_{s\in\R}\frac{A(X,t)-A(X,s)}{\abs{t-s}^{3/2}}\nabla(\dr_ku)(X,s)ds\cdot\nabla\br{\zeta^2\dr_k w}.
\end{multline*}
Since  $\zeta$ is $t$-independent,  the left-hand side above is equal to
\[\frac12\iint\dr_t\br{(\dr_kw)^2\zeta^2}dtdX=0.\]
This implies that 
\begin{multline*}
    \iint A\nabla(\dr_kw)\cdot\nabla(\dr_kw)\zeta^2
    =-\iint (\dr_kA)\nabla w\cdot\nabla(\dr_kw)\zeta^2-2\iint(\dr_kA)\nabla w\cdot\nabla\zeta(\dr_kw)\zeta\\
    -2\iint A\nabla(\dr_kw)\cdot\nabla\zeta(\dr_kw)\zeta
     - \iint \int_{s\in\R}\frac{\dr_k A(X,t)-\dr_k A(X,s)}{\abs{t-s}^{3/2}}\nabla u(X,s)ds\cdot\nabla\br{\zeta^2\dr_k w}\\
    -
    \iint \int_{s\in\R}\frac{A(X,t)-A(X,s)}{\abs{t-s}^{3/2}}\nabla(\dr_ku)(X,s)ds\cdot\nabla\br{\zeta^2\dr_k w}.
\end{multline*}
Note that the assumption that $Q_{2r}(Y)\subset\Rn_+$ entails that $y_n>2r$, and thus $\abs{\nabla A}\le Cr^{-1}$ in $Q_{2r}$ (see \eqref{D2Bbdd}). By ellipticity and applying the Cauchy-Schwarz inequality to the first three terms on the right-hand side, we get that 
\begin{equation}\label{Caccio1}
    \lambda\iint\abs{\nabla\dr_kw}^2\zeta^2\le \frac{\lambda}{4}\iint\abs{\nabla\dr_kw}^2\zeta^2
    +\frac{C}{r^2}\int_{Q_{2r}}\int_{\R}\abs{\nabla w}^2dtdX
    +I+J,
\end{equation}
where
\[
I:=\abs{\iint \int_{s\in\R}\frac{\dr_k A(X,t)-\dr_k A(X,s)}{\abs{t-s}^{3/2}}\nabla u(X,s)ds\cdot\nabla\br{\zeta^2\dr_k w}},
\]
and \[
J:=\abs{\iint \int_{s\in\R}\frac{A(X,t)-A(X,s)}{\abs{t-s}^{3/2}}\nabla(\dr_ku)(X,s)ds\cdot\nabla\br{\zeta^2\dr_k w}}.
\]
We only need to focus on $I$ and $J$. For both of them, we split the integral in $s$ into regions $\set{s:\, \abs{t-s}\le x_n^2}$ and $\set{s:\, \abs{t-s}> x_n^2}$,
obtaining four terms in total. 

We start with the terms from $I$.
We claim that 
\begin{multline}\label{CaccioI11}
   \abs{I_{11}}:=\abs{\iint\int_{\abs{s-t}>x_n^2}\frac{\dr_k A(X,t)-\dr_k A(X,s)}{\abs{t-s}^{3/2}}\nabla u(X,s)ds\cdot\nabla(\dr_k w)\zeta^2dtdX}\\
    \le
    \br{\iint\abs{\dr_kA}^2M(\nabla u(X,\cdot))(t)^2\frac{\zeta(X)^2}{x_n^2}dtdX}^{1/2}\br{\iint\abs{\nabla\dr_kw}^2\zeta^2dtdX}^{1/2}\\
    +\br{\iint\abs{\dr_kA}^2\abs{\nabla u}^2\frac{\zeta^2}{x_n^2}dtdX}^{1/2}\br{\iint\abs{\nabla\dr_kw}^2\zeta^2dtdX}^{1/2}
\end{multline}
To see this, use the triangle inequality and bound the left-hand side of \eqref{CaccioI11} by
\begin{multline*}
    \iint \int_{\abs{s-t}>x_n^2}\frac{\abs{\dr_k A(X,t)}}{\abs{t-s}^{3/2}}\abs{\nabla u(X,s)}ds\abs{\nabla\dr_k w}\zeta^2dtdX\\
    +\iint \int_{\abs{s-t}>x_n^2}\frac{\abs{\dr_k A(X,s)}}{\abs{t-s}^{3/2}}\abs{\nabla u(X,s)}ds\abs{\nabla\dr_k w}\zeta^2dtdX
    =:I_{111}+I_{112}
\end{multline*}
For $I_{112}$, use Cauchy-Schwarz inequality and the $L^2$ boundedness of the Hardy-Littlewood maximal function:
\begin{multline*}
    I_{112}\le\sum_{j\ge0}\iint\int_{2^jx_n^2\le\abs{s-t}\le2^{j+1}x_n^2}\frac{\abs{\dr_kA(X,s)}}{\br{2^jx_n^2}^{2/3}}\abs{\nabla u(X,s)}ds\abs{\nabla\dr_k w}\zeta^2dtdX\\
\le \iint\sum_{j\ge0}\br{2^jx_n^2}^{-1/2}M\br{\dr_kA(X,\cdot)\nabla u(X,\cdot)}(t)\abs{\nabla(\dr_k w)}\zeta^2dtdX\\
\lesssim\int_{\Rn_+}\frac{\zeta(X)}{x_n}\br{\int_{\R}M(\dr_kA\nabla u(X,\cdot))^2dt}^{1/2}\br{\int_\R\abs{\nabla\dr_kw}^2\zeta^2dt}^{1/2}dX\\
\lesssim \br{\iint\abs{\dr_kA}^2\abs{\nabla u}^2\frac{\zeta^2}{x_n^2}dtdX}^{1/2}\br{\iint\abs{\nabla\dr_k w}^2\zeta^2\,dtdX}^{1/2}.
\end{multline*}
 The term $I_{111}$ can be treated similarly:
 \begin{multline*}
     I_{111}\le\sum_{j\ge0}\iint\int_{2^jx_n^2\le\abs{s-t}\le2^{j+1}x_n^2}\frac{\abs{\dr_kA(X,t)}}{\br{2^jx_n^2}^{2/3}}\abs{\nabla u(X,s)}ds\abs{\nabla\dr_k w}\zeta^2dtdX\\
     \lesssim\iint\abs{\dr_kA(X,t)}M(\nabla u(X,\cdot))(t)\abs{\nabla\dr_kw}\zeta(X)^2x_n^{-1}dtdX\\
     \le\br{\iint\abs{\dr_kA}^2M(\nabla u(X,\cdot))(t)^2\frac{\zeta^2}{x_n^2}dtdX}^{1/2}\br{\iint\abs{\nabla\dr_k w}^2\zeta^2\,dtdX}^{1/2},
 \end{multline*}
 which gives \eqref{CaccioI11}.
 The other term from $I$, when $\abs{t-s}\ge x_n^2$, is 
 \[
 I_{12}:=\iint\int_{\abs{s-t}>x_n^2}\frac{\dr_k A(X,t)-\dr_k A(X,s)}{\abs{t-s}^{3/2}}\nabla u(X,s)ds\cdot\nabla\zeta(\dr_k w)\zeta\,dtdX.
 \]
 It is straightforward to check that, using a similar argument,  
 \begin{multline}\label{CaccioI12}
   \abs{I_{12}}
    \le
    \br{\iint\abs{\dr_kA}^2M(\nabla u(X,\cdot))(t)^2\frac{\zeta(X)^2}{x_n^2}dtdX}^{1/2}\br{\iint\abs{\dr_kw}^2\abs{\nabla\zeta}^2dtdX}^{1/2}\\
    +\br{\iint\abs{\dr_kA}^2\abs{\nabla u}^2\frac{\zeta^2}{x_n^2}dtdX}^{1/2}\br{\iint\abs{\dr_kw}^2\abs{\nabla\zeta}^2dtdX}^{1/2}.
\end{multline}
Now we turn to the terms from $I$ when $\abs{s-t}\le x_n^2$. We write 
   \begin{multline*}
       I_{21}:=\iint\int_{\abs{s-t}\le x_n^2}\frac{\dr_k A(X,t)-\dr_k A(X,s)}{\abs{t-s}^{3/2}}\nabla u(X,s)ds\cdot\nabla(\dr_k w)\zeta^2dtdX\\
       =\iint\int_{\abs{s-t}\le x_n^2}\frac{\int_s^t\dr_\tau\dr_k A(X,\tau)d\tau}{\abs{t-s}^{3/2}}\nabla u(X,s)ds\cdot\nabla(\dr_kw)\zeta^2dtdX.
   \end{multline*}
Using the localized maximal function defined as in \eqref{def.trcmaxfunc} gives
\begin{multline*}
    \abs{I_{21}}\le2\iint\int_{\abs{s-t}\le x_n^2}\frac{1}{\abs{t-s}^{1/2}}\fint_{s-\abs{s-t}}^{s+\abs{s-t}}\abs{\dr_t\dr_kA(X,\tau)}d\tau\abs{\nabla u(X,s)}ds\abs{\nabla\dr_kw}\zeta^2dtdX\\
    \le 2\iint\int_{\abs{s-t}\le x_n^2}\abs{t-s}^{-1/2}M_{x_n^2}\br{\dr_t\dr_kA(X,\cdot)}(s)\abs{\nabla u(X,s)}ds\abs{\nabla\dr_kw}\zeta^2dtdX
\end{multline*}
Decompose the interval $\set{s:\abs{s-t}\le x_n^2}\subset\bigcup_{j\ge 0}\set{s: 2^{-j-1}x_n^2\le \abs{s-t}\le 2^{-j}x_n^2}$ to get 
an upper bound for $\abs{I_{21}}$ 
\begin{multline*}
C\sum_{j\ge0}2^{-j/2}\int_{\Rn_+}\zeta^2x_n\int_{t\in\R}\fint_{2^{-j-1}x_n^2\le\abs{s-t}\le2^{-j}x_n^2}M_{x_n^2}\br{\dr_t\dr_kA(X,\cdot)}(s)\abs{\nabla u}ds\abs{\nabla\dr_kw}dtdX\\
    \le C\int_{\Rn_+}\zeta^2(X)x_n\int_{t\in\R}M\br{M_{x_n^2}\br{\dr_t\dr_kA(X,\cdot)}\abs{\nabla u(X,\cdot)}}(t)\abs{\nabla\dr_kw}dtdX.
\end{multline*}
By Cauchy-Schwarz and the $L^2$ boundedness of the maximal function (in $t$), we find that
\begin{multline}\label{eq.CaccioI21}
    \abs{I_{21}}\lesssim\int_{\Rn_+}\zeta^2x_n\br{\int_{\R}{M_{x_n^2}\br{\dr_t\dr_kA(X,\cdot)}(t)^2\abs{\nabla u(X,t)}^2}dt}^{1/2}\br{\int_{\R}\abs{\nabla\dr_kw(X,t)}^2dt}^{1/2}dX\\
    \lesssim \br{\iint M_{x_n^2}\br{\dr_t\dr_kA(X,\cdot)}(t)^2\abs{\nabla u(X,t)}^2x_n^2\zeta(X)^2dXdt}^{1/2}\br{\iint\abs{\nabla\dr_kw}^2\zeta(X)^2dXdt}^{1/2}.
\end{multline}
The only remaining term from $I$ is 
\[
 I_{22}:=\iint\int_{\abs{s-t}\le x_n^2}\frac{\dr_k A(X,t)-\dr_k A(X,s)}{\abs{t-s}^{3/2}}\nabla u(X,s)ds\cdot\nabla\zeta(\dr_k w)\zeta\,dtdX,
\]
which can be estimated almost identically, and we see that $I_{21}$ is bounded by
\begin{equation}\label{eq.CaccioI22}
\br{\iint M_{x_n^2}\br{\dr_t\dr_kA(X,\cdot)}(t)^2\abs{\nabla u(X,t)}^2x_n^2\zeta(X)^2dXdt}^{1/2}\br{\iint\abs{\dr_kw}^2\abs{\nabla\zeta}^2dXdt}^{1/2}.
\end{equation}
It remains to estimate 
\begin{multline*}
    J:=\abs{\iint \int_{s\in\R}\frac{A(X,t)-A(X,s)}{\abs{t-s}^{3/2}}\nabla(\dr_ku)(X,s)ds\cdot\nabla\br{\zeta^2\dr_k w}dtdX}\\
    \le \abs{\int_{\Rn_+}\int_{t\in\R} \int_{\abs{s-t}>x_n^2}\dots}+\abs{\int_{\Rn_+}\int_{t\in\R} \int_{\abs{s-t}\le x_n^2}\dots}=:J_1+J_2.
\end{multline*}
For $J_2$, since
$A(X,t)-A(X,s)=\int_s^t\dr_\tau A(X,\tau)d\tau$, the estimate $\abs{\dr_t A(X,t)}\le Cx_n^{-2}$ implies
\[
\abs{A(X,t)-A(X,s)}\le Cx_n^{-2}\abs{t-s}.
\] 
Therefore,
\[
J_2\le C\iint x_n^{-2}\int_{\abs{s-t}\le x_n^2}\frac{\abs{\nabla\dr_ku(X,s)}}{\abs{t-s}^{1/2}}ds\br{\abs{\nabla(\dr_k w)}\zeta^2+\abs{\nabla\zeta}\abs{\dr_k w}\zeta}dtdX.
\]
Break the integral in $s$ dyadically, and use the maximal function, to obtain:
\begin{multline*}
    J_2\le C\sum_{j\ge 0}2^{-j}\iint x_n^{-1}\fint_{\abs{s-t}\sim 2^{-j}x_n^2}\abs{\nabla \dr_k u(X,s)}ds\br{\abs{\nabla(\dr_k w)}\zeta^2+\abs{\nabla\zeta}\abs{\dr_k w}\zeta}dtdX\\
    \le C\int_{\R^n_+} \zeta(X)x_n^{-1}\int_{t\in R}M(\nabla\dr_k u(X,\cdot))(t)\br{\abs{\nabla(\dr_k w)}\zeta+\abs{\nabla\zeta}\abs{\dr_k w}}dtdX.
\end{multline*}
By Cauchy-Schwarz and the $L^2$ boundedness of the maximal function, we obtain that
\begin{multline*}
    J_2\le C\br{\iint\abs{\nabla\dr_k u(X,t)}^2\frac{\zeta^2}{x_n^2}dXdt}^{1/2}\\
    \cdot\bigg\{\br{\iint\abs{\nabla\dr_k w}^2\zeta^2\, dXdt}^{1/2}
+\br{\iint\abs{\dr_k w}^2\abs{\nabla\zeta}^2\, dXdt}^{1/2}\bigg\}.
\end{multline*}
For $J_1$, Schur's test shows that 
\[\int_{t\in\R} \br{\int_{\abs{s-t
     }>x_n^2}\abs{t-s}^{-\frac32}\abs{\nabla\dr_k u(X,s)}\,ds}^2dt
     \le x_n^{-2}\int_{\R}\abs{\nabla\dr_k u(X,t)}^2dt\]
as $\int_{\abs{s-t}>x_n^2}\abs{s-t}^{-3/2}ds=4x_n^{-1}$. Hence, apply Cauchy-Schwarz,  we get that
\[
    J_1\le C\int_{\R^n_+}\zeta(X)x_n^{-1}\br{\int_{\R}\abs{\nabla\dr_k u(X,t)}^2dt}^{1/2}\br{\int_{\R}\zeta^2\abs{\nabla\dr_k w}^2+\abs{\nabla\zeta}^2\abs{\dr_kw}^2dt}^{1/2}dX,
\]
which has the same upper bound as$J_2$ by Cauchy-Schwarz once again.

In conclusion, \eqref{eq.Caccio} follows from combining all the previous estimates, applying Young's inequality, and choosing a sufficient small coefficient of $\iint\abs{\nabla \dr_k w}^2\zeta^2$.

\section{$S<N$ estimates in $L^2$ for block form matrix $A$}
In this section, we focus on the block form case when the matrix $A$ is 
\begin{equation}\label{block}
    A=\begin{bmatrix}
\begin{BMAT}{c.c}{c.c}
A_\parallel  & \mathbf{0}  \\
\mathbf{0} & 1
\end{BMAT}
\end{bmatrix}
\end{equation}
and $A_\parallel=(a_{ij})_{1\le i,j\le n-1}$. Then $L=\divg A\nabla =\divg_xA_\parallel\nabla_x+\dr^2_{nn}$. The goal of this section is to prove the following result.
\begin{theorem}\label{thm.S<NL2}
    Let $\LL=-\dr_t+L$ be a parabolic operator with matrix $A$ in the block form \eqref{block} in $\om=\Rn_+\times\R$. Assume that $A$ satisfies the ellipticity condition \eqref{E:elliptic}, the Carleson condition \eqref{E:1:carl} and the bound \eqref{E:1:bound}. Assume in addition that $\abs{\dr_t\nabla A}\delta^3\in CM(\norm{\mu}_C)$. 
    Then there exist constants $C>0$ and $K_0\ge 1$ depending only on the dimension and the ellipticity constants, such that for any $f\in \dot L^2_{1,1/2}(\pom)$ and  any $K\ge K_0$,
    the energy solution $u$ to \eqref{eq-pp} satsifies  
\begin{equation}\label{eq.S<NL2}
    \iint_\om \abs{\nabla^2u}^2x_n\, dXdt
    \le
       C\br{K^5\norm{\mu_{||}}_C+K^{-1}\norm{\mu}_C}\norm{N(\nabla u)}_{L^2(\pom)}^2 +C\norm{f}_{\dot L_{1,1/2}^2(\pom)}^2,
\end{equation}
where $\mu_{||}$ and $\mu$ are defined in \eqref{def.mu11} and \eqref{E:1:carl}, respectively.
\end{theorem}
\subsection{Outline of the proof} We give an outline of our proof of Theorem \ref{thm.S<NL2}.

\underline{Step 1.} We obtain the $L^2$ estimate for the square function of $\nabla_xu$: for $i\in\N$ such that $1\le i\le n-1$,
\[
\iint_\om\abs{\nabla(\dr_i u)}^2x_ndXdt\le C\norm{\mu_{||}}_C\norm{N(\nabla_x u)}_{L^2(\pom)}^2+C\norm{\nabla_xf}_{L^2(\pom)}^2,
\]
which is a consequence of Lemma \ref{lem.Sdk}.

\underline {Step 2.} We use the PDE of $u$ and the block-form structure to obtain an $L^2$ estimate for $\iint_{\om}\abs{\dr_{nn}u}^2x_ndXdt$, which is left out in Step 1.
\[
\iint_{\om}\abs{\dr_{nn}u}^2x_ndXdt\le C\norm{\mu_{||}}_C\norm{N(\nabla_x u)}_{L^2(\pom)}^2+C\norm{\nabla_xf}_{L^2(\pom)}^2 +C\iint_{\om}\abs{\dr_t u}^2x_ndXdt.
\]
This is proven in \eqref{eq.dnn}. So it remains to control $\iint_{\om}\abs{\dr_t u}^2x_ndXdt$.

\underline{Step 3.} We derive an $L^2$ estimate for the square function of $D_t^{1/2}u$: for any $K\ge 1$,
\[
\iint_{\om}\abs{\nabla D_t^{1/2}u}^2x_ndXdt
\le C\br{K^4\norm{\mu_{||}}_C+K^{-2}\norm{\mu}_{C}}\norm{N(\nabla_xu)}_{L^2(\pom)}^2 +C\norm{f}_{\dot L^2_{1,1/2}(\pom)}^2.
\]
This is shown in Lemma \ref{lem.gradD1/2}.

\underline{Step 4.} We relate $\iint_{\om}\abs{\dr_t u}^2x_ndXdt$ to the square function of $D_t^{1/2}u$: for any $\epsilon>0$ small, 
\begin{multline*}
    \iint_{\om}\abs{\dr_t u}^2x_ndXdt\le \epsilon\iint\abs{\nabla^2(D_t^{1/2}u)}^2x_n^3dXdt+C\br{1+ \epsilon^{-1}}\iint\abs{\nabla(D^{1/2}_tu)}^2x_n\\
    +C\iint_{\om}\abs{\nabla_x^2u}^2x_ndXdt
    +C(1+\epsilon^{-1})\norm{\mu_{||}}_C\norm{N(\nabla_xu)}_{L^2(\pom)}^2.
\end{multline*}
This is Lemma \ref{lem.Dtu}. 

\underline{Step 5.} We use a Caccioppoli type inequality (Lemma \ref{lem.Caccio}) to control the first term on the right-hand side of the equation in Step 4: 
\begin{multline*}
    \iint\abs{\nabla^2(D_t^{1/2}u)}^2x_n^3dXdt
\le C\iint\abs{\nabla(D^{1/2}_tu)}^2x_ndXdt+ C\iint\abs{\nabla^2u}^2x_ndXdt\\
+C\norm{\mu}_C\norm{N(\nabla u)}_{L^2(\pom)}^2.
\end{multline*}
This is shown in Lemma \ref{lem.glbCaccio}.

\underline{Step 6.} We plug the estimate from Step 5 into the estimate from Step 4, and we separate $\abs{\dr_{nn}u}$ from $\abs{\nabla^2u}$ because we know how to treat $\sum_{i=1}^{n-1}\iint\abs{\nabla(\dr_i u)}^2x_ndXdt$ from Step 1. We obtain  that for any $\epsilon\in(0,1)$,
\begin{multline*}
    \iint\abs{\dr_tu}^2x_ndXdt\le\epsilon\,C\iint\abs{\dr_{nn}u}^2x_ndXdt+C\br{1+\epsilon^{-1}}\iint\abs{\nabla(D^{1/2}_t u)}^2x_ndXdt\\
    +C\epsilon^{-1}\norm{\mu_{||}}_C\norm{N(\nabla_x u)}_{L^2(\pom)}^2 +\epsilon\, C\norm{\mu}_C\norm{N(\nabla u)}_{L^2(\pom)}^2
    +C\norm{\nabla_xf}_{L^2(\pom)}^2.
\end{multline*}
Plugging this estimate into Step 2, using the estimate from Step 3, and then rearranging the coefficients yields that there is $C>0$ such that for any $\epsilon\in(0,1)$ and any $K\ge1$,
\begin{multline}\label{eq.S2final}
     \iint\abs{\dr_{nn}u}^2x_ndXdt\le \epsilon\, C\iint\abs{\dr_{nn}u}^2x_ndXdt+C\norm{f}_{\dot L_{1,1/2}^2(\pom)}^2\\
    +C\br{\epsilon^{-1}K^4\norm{\mu_{||}}_C+\epsilon^{-1}K^{-2}\norm{\mu}_C}\norm{N(\nabla u)}_{L^2(\pom)}^2.
\end{multline}

\underline{Step 7.} We let $\epsilon=K^{-1}$ and take $K_0$ large enough so that $CK_0^{-1}<1/2$, and so for $K\ge K_0$, the first term on the right-hand side of \eqref{eq.S2final} can be absorbed into the left-hand side. Hence, for any $K\ge K_0$, we have that 
\[
     \iint\abs{\dr_{nn}u}^2x_ndXdt \le C\norm{f}_{\dot L_{1,1/2}^2(\pom)}^2
+C\br{K^5\norm{\mu_{||}}_C+K^{-1}\norm{\mu}_C}\norm{N(\nabla u)}_{L^2(\pom)}^2.
\]
Combining with Step 1, we conclude that for $K\ge K_0$, 
\[
    \iint\abs{\nabla^2u}^2x_ndXdt\le C\norm{f}_{\dot L_{1,1/2}^2(\pom)}^2
+C\br{K^5\norm{\mu_{||}}_C+K^{-1}\norm{\mu}_C}\norm{N(\nabla u)}_{L^2(\pom)}^2.
\]
as desired.

\subsection{Proof of estimates in Step 1 - 5}
We start with $L^2$ square function estimates for $\nabla_xu$.

\begin{lemma}\label{lem.Sdk}
Let $\LL=-\dr_t+L$ be a parabolic operator with matrix $A$ in the block form \eqref{block}, and let $u$ be a bounded energy solution to $\LL u=0$ in $\om$. Then for $k\in\N$ such that $1\le k\le n-1$, and $r>0$,
\begin{multline}\label{eq.BlckSdk<Nr}
    \lambda\int_0^{\frac{r}2
    }\int_{\pom} \abs{\nabla(\dr_k u)}^2x_n\,dxdtdx_n+\frac{2}{r}\int_0^r\int_{\pom}(\dr_ku)^2dxdtdx_n\\
    \le\int_{\pom}\dr_ku(x,r,t)^2dxdt
    +\int_{\pom}\dr_ku(x,0,t)^2dxdt
    +C\norm{\abs{\dr_k A_{\parallel}}^2x_n}_{C}\int_{\pom}N^r(\nabla_xu)^2dxdt,
\end{multline}
where $C$ depends only on the ellipticity constant $\lambda$.

In particular, since $\int_{\pom}\dr_ku(x,r,t)^2dxdt\to 0$ as $r\to\infty$, taking $r\to\infty$ in \eqref{eq.BlckSdk<Nr} gives that 
\begin{equation}\label{eq.BlckSdk<N}
    \lambda\iint_\om \abs{\nabla(\dr_ku)}^2x_n\, dXdt
    \le \int_{\pom}\dr_ku(x,0,t)^2dxdt +C \norm{\abs{\dr_k A_{\parallel}}^2x_n}_{C}\int_{\pom}\abs{N(\nabla_x u)}^2dxdt.
\end{equation}
\end{lemma}
\smallskip

\bp 
We start with a local estimate on $(0,r)\times Q_r$, where $Q_r=Q_r(y,s)$ is a parabolic cube on the boundary as defined in \eqref{eqdef.bdypcube}. Let $\zeta=\zeta(x,t)$ be a smooth cutoff function that satisfies
\[
\zeta=\begin{cases}
    1\quad \text{in }Q_r,\\
    0\quad \text{outside }Q_{2r},
\end{cases}
\]
and that for some constant $0<c<\infty$
\[r\abs{\dr_{x_i}\zeta}+r^2\abs{\dr_t\zeta}\le c\quad \text{where }1\le i\le n-1.\]

To lighten notation, we denote $w_k:=\dr_k u$ for $1\le k\le n$. Fix a $k\in\set{1,2,\dots, n-1}$,
we write 
\begin{multline}\label{eq.pfSdk1}
    \int_0^r\int_{Q_{2r}}A\nabla w_k\cdot \nabla w_k(\zeta^2x_n)dxdtdx_n
    =\int_0^r\int_{Q_{2r}}\divg\br{w_k\zeta^2x_n A\nabla w_k}\\
    -\int_0^r\int_{Q_{2r}}w_kA\nabla w_k\cdot\nabla\br{\zeta^2x_n}
    -\int_0^r\int_{Q_{2r}}\divg(A\nabla w_k)(w_k\zeta^2x_n).
\end{multline}
We then use the divergence theorem for the first term on the right-hand side, which gives us a boundary term. For the second term, we write 
$\nabla\br{\zeta^2x_n}=\nabla(\zeta^2)x_n+\zeta^2\nabla x_n$, which gives two terms correspondingly. Notice that $A\nabla w_k\cdot\nabla x_n= \dr_nw_k$ thanks to the block-form structure \eqref{block}, and so the corresponding term in the integral is
\begin{multline*}
    -\int_0^r\int_{Q_{2r}}w_k\dr_nw_k\zeta^2
    =-\frac12\int_0^r\int_{Q_{2r}}\dr_n\br{w_k^2\zeta^2}dxdtdx_n\\
    =\frac12\int_{Q_{2r}}w_k(x,0,t)^2\zeta(x,t)^2dxdt-\frac12\int_{Q_{2r}}w_k(x,r,t)^2\zeta(x,t)^2dxdt.
\end{multline*}
To summarize, the left-hand side of \eqref{eq.pfSdk1} is equal to
\begin{multline}\label{eq.pfSdk2}
    \frac{r}{2}\int_{Q_{2r}}\dr_n\br{w_k^2}(x,r,t)\zeta(x,t)^2dxdt
    +\frac12\int_{Q_{2r}}w_k(x,0,t)^2\zeta(x,t)^2dxdt\\
    -\frac12\int_{Q_{2r}}w_k(x,r,t)^2\zeta(x,t)^2dxdt
    -\int_0^r\int_{Q_{2r}}w_kx_nA_\parallel\nabla_xw_k\cdot\nabla_x(\zeta^2)
    -\int_0^r\int_{Q_{2r}}\divg(A\nabla w_k)(w_k\zeta^2x_n).
\end{multline}
Now we use the PDE of $w_k$ to treat the last term. Since $\divg A\nabla u=\dr_tu$, a direct computation shows that 
\[-\divg A\nabla w_k=-\dr_tw_k+\divg\br{(\dr_kA)\vec{w}},\]
where we have used the notation $\vec{w}:=(w_1,w_2,\cdots,w_n)^T=\nabla u$. Therefore,
\begin{multline*}
     -\int_0^r\int_{Q_{2r}}\divg(A\nabla w_k)(w_k\zeta^2x_n)
     =-\int_0^r\int_{Q_{2r}}\dr_tw_k(w_k\zeta^2x_n)
     +\int_0^r\int_{Q_{2r}}\divg\br{(\dr_kA)\vec{w}}w_k\zeta^2x_n\\
     =:I_1+I_2.
\end{multline*}
We continue to compute
\[
    I_1=-\frac12\int_0^r\int_{Q_{2r}}\dr_t\br{w_k^2\zeta^2x_n}+\frac12\int_0^r\int_{Q_{2r}}w_k^2\dr_t(\zeta^2)x_n=\frac12\int_0^r\int_{Q_{2r}}w_k^2\dr_t(\zeta^2)x_n
\]
because the first integral is 0. We apply the divergence theorem to $I_2$ and note that the boundary term is 
\[\int_{Q_{2r}}\left[(\dr_kA)\vec{w}\cdot\vec{e_n} (w_k\zeta^2x_n)\right]_{x_n=r}dxdt=0\]
again thanks to the special structure of the matrix. So we get that
\begin{multline*}
    I_2=-\int_0^r\int_{Q_{2r}}(\dr_kA)\vec{w}\cdot\nabla\br{w_k\zeta^2x_n}dxdtdx_n\\
    =-\int_0^r\int_{Q_{2r}}(\dr_kA)\vec{w}\cdot\nabla w_k\zeta^2x_n
    -\int_0^r\int_{Q_{2r}}(\dr_kA)\vec{w}\cdot\nabla\br{\zeta^2} w_kx_n
\end{multline*}
We now use \eqref{eq.pfSdk2} and our computation of $I_1$ and $I_2$ to obtain a global estimate on $(0,r)\times\pom$.

Let $\set{Q_{2r}^\ell}_{\ell=1}^\infty$ be a collection of disjoint parabolic cubes of sidelength $2r$ that covers $\pom$. Let $\set{\zeta_\ell^2}_\ell$ be a partition of unity subordinate to this collection, that is, $\supp\zeta_\ell\subset Q_{2r}^\ell$ and $\sum_{\ell=1}^\infty\zeta_\ell^2\equiv 1$ on $\pom$. Moreover, for any $\ell$,
$\zeta_\ell=1$ in $Q^\ell_r$,
and for some constant $0<c<\infty$
\[r\abs{\dr_{x_i}\zeta_\ell}+r^2\abs{\dr_t\zeta_\ell}\le c\quad \text{where }1\le i\le n-1.\]
Note that 
\begin{equation}\label{eq.pou0}
    \sum_\ell\dr_{x_i}\zeta_\ell^2=0 \quad\text{for }1\le i\le n, \quad \text{and }\sum_\ell\dr_{t}\zeta_\ell^2=0.
\end{equation}
By \eqref{eq.pfSdk2}, the computation of $I_1$ and $I_2$, and \eqref{eq.pou0}, we can write 
\begin{multline}\label{eq.pfSdk3}
    \int_0^r\int_{\pom}A\nabla w_k\cdot\nabla w_k\,x_n
    =\sum_{\ell=1}^\infty\int_0^r\int_{\pom}A\nabla w_k\cdot\nabla w_k(\zeta_\ell^2x_n)\\
    = \frac{r}{2}\int_{\pom}\dr_n\br{w_k^2}(x,r,t)dxdt
    +\frac12\int_{\pom}w_k(x,0,t)^2dxdt
    -\frac12\int_{\pom}w_k(x,r,t)^2dxdt\\
    \underbrace{-\int_0^r\int_{\pom}(\dr_kA)\vec{w}\cdot\nabla w_k\,x_n\,dxdtdx_n}_J.
\end{multline}
We estimate the last term $J$ using the Cauchy-Schwarz inequality (recalling $\vec w=\nabla u$):
\begin{multline*}
    \abs{J}=\abs{\int_0^r\int_{\pom}(\dr_kA_\parallel)\nabla_xu\cdot\nabla_x w_k\,x_n\,dxdtdx_n}\\
    \le \br{\int_0^r\int_{\pom}\abs{\dr_kA_\parallel}^2\abs{\nabla_xu}^2x_n}^{1/2}
    \br{\int_0^r\int_{\pom}\abs{\nabla_xw_k}^2x_n}^{1/2}\\
    \le \frac{\lambda}{2}\int_0^r\int_{\pom}\abs{\nabla_xw_k}^2x_n
    +C_\lambda\norm{\abs{\dr_kA_\parallel}^2x_n}_C\int_{\pom}N^r(\nabla_xu)^2dxdt.
\end{multline*}
By ellipticity, \eqref{eq.pfSdk3}, and absorbing the term  $\frac{\lambda}2\int_0^r\int_{\pom}\abs{\nabla w_k}^2x_n$ to the left-hand side, one obtains that
\begin{multline}\label{eq.pfSdk4}
    \lambda\int_0^r\int_{\pom}\abs{\nabla w_k}^2x_n
    \le r\int_{\pom}\dr_n\br{w_k^2}(x,r,t)dxdt
    +\int_{\pom}w_k(x,0,t)^2dxdt
    -\int_{\pom}w_k(x,r,t)^2dxdt\\
    +C\norm{\abs{\dr_kA_\parallel}^2x_n}_C\int_{\pom}N^r(\nabla_xu)^2dxdt.
\end{multline}
Fix any $r_0>0$. We integrate \eqref{eq.pfSdk4} in $r$ variable over $[0,r_0]$ and divide both sides by $r_0$. Since $(\dr_n w^2)x_n=\dr_n(w^2x_n)-w^2$, 
\[\int_0^r\int_{\pom}r\dr_n(w_k^2)(x,r,t)dxdtdr=r_0\int_{\pom}w_k(x,r_0,t)^2dxdt-\int_0^{r_0}\int_{\pom}w_k^2(x,x_n,t)dxdtdx_n,\]
and hence
\begin{multline*}
    \lambda\int_0^{r_0}\int_{\pom}\br{x_n-\frac{x_n^2}{r_0}}\abs{\nabla w_k}^2x_n\,dxdtdx_n
    \le \int_{\pom}w_k(x,r_0,t)^2dxdt
    +\int_{\pom}w_k(x,0,t)^2dxdt\\
    -\frac{2}{r_0}\int_0^{r_0}\int_{\pom}w_k(x,x_n,t)^2dxdtdx_n
    +C\norm{\abs{\dr_kA_\parallel}^2x_n}_C\int_{\pom}N^{r_0}(\nabla_xu)^2dxdt.
\end{multline*}
Truncating the integral on the left-hand side to $[0,\frac{r_0}{2}]$ (and replacing $r_0$ by $r$) we obtain \eqref{eq.BlckSdk<Nr} as desired.
\ep

\ms 

Next, we want to get a bound for the square function of $\dr_nu$. 

Let us denote $w_n:=\dr_nu$. We separate the derivatives as follows.
\begin{multline}\label{eq.gradwn}
    \iint_\om \abs{\nabla w_n}^2x_n\,dXdt=\iint_\om \abs{\nabla_x w_n}^2x_n\,dXdt +\iint_\om \abs{\dr_n w_n}^2x_n\,dXdt\\
    =\iint_\om \abs{\dr_n(\nabla_x u)}^2x_n\,dXdt+\iint_\om \abs{\dr_n w_n}^2x_n\,dXdt.
\end{multline}
The first term on the right-hand side of \eqref{eq.gradwn} is fine by Lemma \ref{lem.Sdk}. So we only need to estimate the second term on the right-hand side. By the equation,
\[\dr_nw_n=\dr^2_{nn}u=-\divg_x(A_\parallel\nabla_x u)+\dr_tu,\]
and hence, 
\begin{multline*}
    \iint_\om \abs{\dr_n w_n}^2x_n\,dXdt\lesssim \sum_{1\le i,j\le n-1}\iint_\om\abs{\dr_i(a_{ij}\dr_ju)}^2x_n\, dXdt+\iint_\om \abs{\dr_t u}^2x_n\,dXdt\\
    \lesssim \iint_\om \abs{\nabla_xA_\parallel}^2\abs{\nabla_x u}^2x_n\,dXdt+\sum_{1\le i,j\le n-1}\iint_{\om}\abs{a_{ij}}\abs{\dr^2_{ij}u}^2x_n\,dXdt+\iint_\om \abs{\dr_t u}^2x_n\,dXdt\\
    \le \norm{\abs{\nabla_x A_\parallel}^2\delta}_{C}\int_{\pom}\abs{N(\nabla_x u)}^2dxdt+C\iint_\om \abs{\nabla^2_xu}^2x_n\,dXdt+\iint_\om \abs{\dr_t u}^2x_n\,dXdt.
\end{multline*}
Therefore, by \eqref{eq.gradwn} and Lemma \ref{lem.Sdk}, we obtain that 
\begin{multline}\label{eq.dnn}
    \iint_\om\abs{\nabla(\dr_nu)}^2 x_ndXdt\le C\int_{\pom}\abs{\nabla_xu}^2dxdt+C\norm{\abs{\nabla_x A_\parallel}^2\delta}_{C}\int_{\pom}\abs{N(\nabla_x u)}^2dxdt\\
    +C\iint_\om \abs{\dr_t u}^2x_n\,dXdt.
\end{multline}

\ms 

In the rest of this subsection, we take care of the term $\iint_\om \abs{\dr_t u}^2x_n\,dXdt$. We begin with a square function estimate for $D^{1/2}_{t}u$.

\begin{lemma}\label{lem.gradD1/2}
    Let $u$ be a bounded solution to $\mathcal{L}u=0$ in $\om$, $u=f$ on $\pom$, and $A$ be in block form \eqref{block}. Then there is $C>0$ depending only on the dimension and ellipticity, such that for any $K\ge 1$,
    \begin{multline}\label{eq.gradD1/2.1}
        \iint_\om \abs{\nabla D^{1/2}_{t}u}^2x_ndXdt\le \frac12 \iint_\om \abs{\nabla D^{1/2}_{t}u}^2x_ndXdt
        +C K^{-2}\sum_{i=1}^{n-1}\iint_{\om}\abs{\nabla(\dr_iu)}^2x_ndXdt\\
+C\br{K^4\norm{\mu_{||}}_C+K^{-2}\norm{\abs{\dr_nA_{||}}^2\delta}_C}\int_{\pom}N(\nabla_xu)^2dxdt
+ C\norm{D^{1/2}_{t}f}_{L^2(\pom)}^2.
    \end{multline}
    where $\mu_{||}$ is defined in \eqref{def.mu11}. 
    In particular,  if $\iint_\om |\nabla D^{1/2}_{t}u|^2x_ndXdt<\infty$, then by Lemma \ref{lem.Sdk},
    \begin{equation}\label{gradD12u.bd}
         \iint_\om \abs{\nabla D^{1/2}_{t}u}^2x_ndXdt\le 
   C\br{K^4\norm{\mu_{||}}_C+K^{-2}\norm{\abs{\dr_nA_{||}}^2\delta}_C}\norm{N(\nabla_xu)}_{L^2}^2 + C\norm{f}_{\dot L^2_{1,1/2}}^2.
    \end{equation}
\end{lemma}

\bp
Let us denote $w:=D^{1/2}_{t}u$.
By ellipticity, 
\begin{multline}\label{eq.pfD1/21}
     \iint_\om\abs{\nabla w}^2x_n\,dXdt\lesssim \iint A\nabla w\cdot\nabla w\, x_n\, dXdt\\
     =-\iint \divg (A\nabla w)wx_n\, dXdt -\iint A\nabla w\cdot\nabla x_n\, w\,dXdt.
\end{multline}
  For the second term on the right-hand side, we integrate in $x_n$ and get that
\begin{multline*}
    -\iint A\nabla w\cdot\nabla x_n\, w\,dXdt=-\frac12\iint\dr_n(w^2)dXdt=\frac12\int_{\R^{n-1}}\int_Rw^2(x,0,t)dxdt\\
    =\frac12\norm{D^{1/2}_{t}f}_{L^2(\pom)}^2.
\end{multline*}
  For the first term on the right-hand side of \eqref{eq.pfD1/21}, we use the PDE for $w$. Notice that since $A$ is in block form, $A(X,t)-A(X,s)$ is non-zero only in the first $n-1$ rows (and columns), and so 
\[
\divg\br{\int_{s\in\R}\frac{A(X,t)-A(X,s)}{\abs{t-s}^{3/2}}\nabla u(X,s)ds}=\divg_x\br{\int_{s\in\R}\frac{A_{||}(X,t)-A_{||}(X,s)}{\abs{t-s}^{3/2}}\nabla_x u(X,s)ds}.
\] 
Then by Lemma \ref{lem.D1/2eq}, $w$ satisfies the equation
\begin{equation}\label{eq.BlckD1/2eq}
    \dr_tw(X,t)=\divg(A\nabla w)(X,t)
    +c_0\divg_x\br{\int_{s\in\R}\frac{A_\parallel(X,t)-A_\parallel(X,s)}{\abs{t-s}^{3/2}}\nabla_xu(X,s)ds}.
\end{equation}
Hence,
  \begin{multline}
      -\iint_\om \divg (A\nabla w)wx_n\, dXdt
      =-\iint \dr_tw\,w\,x_ndXdt \\+c_0\iint\divg_x\br{\int_{s\in\R}\frac{A_\parallel(X,t)-A_\parallel(X,s)}{\abs{t-s}^{3/2}}\nabla_xu(X,s)ds}w\,x_n\,dXdt=: I_1+I_2.
  \end{multline}
  For $I_1$, we integrate in $t$ and use  the decay of $D^{1/2}_{t}u$ as $\abs{t}\to\infty$ to get that
\[
    I_1=\frac12\iint_\om\dr_t(w^2x_n)dXdt=0.
\]
For $I_2$, we use the divergence theorem for $x$ and then break the integral into two parts:
\begin{multline*}
     -\frac{1}{c_0}I_2=\iint\int_{s\in\R}\frac{A_\parallel(X,t)-A_\parallel(X,s)}{\abs{t-s}^{3/2}}\nabla_xu(X,s)ds\cdot\nabla_xw(X,t)\,x_n\,dXdt\\
     =\int_{\Rn_+}\int_{t\in\R}\int_{s:\abs{s-t
     }\le K^2x_n^2}\dots 
     +
     \int_{\Rn_+}\int_{t\in\R}\int_{s:\abs{s-t
     }>K^2x_n^2}\dots=: I_{21}+I_{22}.
\end{multline*}
The term $I_{21}$ is easier so we treat it first. We write 
\begin{multline*}
    \abs{I_{21}}
    =\abs{\int_{\Rn_+}x_n\int_{t\in\R}\int_{\abs{s-t}\le K^2x_n^2}\frac{\int_s^t\dr_\tau A_\parallel(X,\tau)d\tau}{\abs{t-s}^{3/2}}\nabla_xu(X,s)ds\cdot\nabla_xw(X,t)dtdX}\\
    \le
    \int_{\Rn_+}x_n\int_{t\in\R}\int_{\abs{s-t}\le K^2x_n^2}\frac{\int_{s-\abs{t-s}}^{s+\abs{t-s}}\abs{\dr_\tau A_\parallel(X,\tau)}d\tau}{\abs{t-s}^{3/2}}\abs{\nabla_xu(X,s)}ds\abs{\nabla_xw(X,t)}dtdX\\
    \le
    \int_{\Rn_+}x_n\int_{t\in\R}\int_{\abs{s-t}\le K^2x_n^2}\frac{M_{K^2x_n^2}(\abs{\dr_\tau A_\parallel(X,\cdot)})(s)}{\abs{t-s}^{1/2}}\abs{\nabla_xu(X,s)}ds\abs{\nabla_xw(X,t)}dtdX,
\end{multline*}
where $M_{K^2x_n^2}(\abs{\dr_\tau A_\parallel(X,\cdot)})$ is the localized maximal function defined as in \eqref{def.trcmaxfunc}.
We further decompose $\abs{s-t}\le K^2x_n^2$ into $\bigcup_{j\ge0}\set{s: 2^{-j-1}K^2x_n^2\le\abs{s-t}\le 2^{-j}K^2x_n^2}$ and denote it by $\bigcup_{j\ge0}\set{s: \abs{s-t}\sim 2^{-j}K^2x_n^2}$. Then
\begin{multline*}
    \abs{I_{21}}\le 
     \int_{\Rn_+}x_n\int_{t\in\R}\sum_{j\ge 0}\int\limits_{\abs{s-t}\sim 2^{-j}K^2x_n^2}\frac{M_{K^2x_n^2}(\abs{\dr_\tau A_\parallel(X,\cdot)})(s)}{\abs{t-s}^{1/2}}\abs{\nabla_xu(X,s)}ds\abs{\nabla_xw(X,t)}dtdX\\
     \lesssim
     K\int_{\Rn_+}x_n^2\int_{t\in\R}\sum_{j\ge 0}2^{-j/2}\fint\limits_{\abs{s-t}\sim 2^{-j}K^2x_n^2}M_{K^2x_n^2}(\abs{\dr_\tau A_\parallel(X,\cdot)})(s)\abs{\nabla_xu(X,s)}ds\abs{\nabla_xw(X,t)}dtdX\\
     \lesssim
     K\int_{\Rn_+}x_n^2\int_{t\in\R} M\br{M_{K^2x_n^2}(\abs{\dr_\tau A_\parallel(X,\cdot)})(\cdot)\abs{\nabla_xu(X,\cdot)}}(t)\abs{\nabla_xw(X,t)}dtdX.
\end{multline*}
By Cauchy-Schwarz and the $L^2$ boundedness of the Hardy-Littlewood maximal function, we get that 
\begin{multline*}
\abs{I_{21}}\le CK\int_{\Rn_+}x_n^2\br{\int_{t\in\R} M\br{M_{K^2x_n^2}(\abs{\dr_\tau A_\parallel(X,\cdot)})(\cdot)\abs{\nabla_xu(X,\cdot)}}(t)^2dt}^{1/2}\\
\cdot\br{\int_{t\in\R}\abs{\nabla_xw(X,t)}^2dt}^{1/2}dX\\
     \le CK
     \int_{\Rn_+}x_n^2\br{\int_{t\in\R}M_{K^2x_n^2}(\abs{\dr_\tau A_\parallel(X,\cdot)})(t)^2\abs{\nabla_xu(X,t)}^2dt}^{1/2}\\
     \cdot\br{\int_{t\in\R}\abs{\nabla_xw(X,t)}^2dt}^{1/2}dX.
\end{multline*}
Applying Cauchy-Schwarz again, we get that 
\begin{multline}\label{eq.SI21}
  \abs{I_{21}}\le CK
    \br{\iint_{\Rn_+\times\R} x_n^3M_{K^2x_n^2}(\abs{\dr_\tau A_\parallel(X,\cdot)})(t)^2\abs{\nabla_xu(X,t)}^2dtdX}^{1/2}\\
    \cdot\br{\iint_{\Rn_+\times\R}\abs{\nabla_xw(X,t)}^2x_ndtdX}^{1/2}\\
    \le 
    \epsilon \iint_{\Rn_+\times\R}\abs{\nabla_xw(X,t)}^2x_ndtdX
+C_\epsilon K^4\norm{\abs{\dr_tA_\parallel}^2\delta^3}_{C}\int_{\pom}N(\nabla_xu)(x,t)^2dxdt,
\end{multline}
where we have used Young's inequality and Corollary \ref{cor.locMaxCarl}.
\smallskip

Now we turn to the term $I_{22}$. Writing $x_n=\frac12\dr_n(x_n^2)$ and integration by parts give that
 \begin{multline}\label{eq.BlockI22}
     2I_{22}=-\int_{\Rn_+}\int_{t}\int_{\abs{s-t
     }>K^2x_n^2}\frac{A_\parallel(X,t)-A_\parallel(X,s)}{\abs{t-s}^{3/2}}\nabla_xu(X,s)ds\cdot (\dr_n\nabla_xw(X,t))\,x_n^2dt\,dX\\
     -\int_{\Rn_+}\int_{t}\dr_n\br{\int_{\abs{s-t
     }>K^2x_n^2}\frac{A_\parallel(X,t)-A_\parallel(X,s)}{\abs{t-s}^{3/2}}\nabla_xu(X,s)ds}\cdot \nabla_xw(X,t)\,x_n^2dt\,dX\\
     =:I_{221}+I_{222}
 \end{multline}
 For $I_{221}$, we move $\nabla_x$ from $w$ to get that it is equal to
 \[
 \int_{\Rn_+}\int_{t}\int_{\abs{s-t
     }>K^2x_n^2}\abs{t-s}^{-3/2}\divg_x\Big\{\br{A_\parallel(X,t)-A_\parallel(X,s)}\nabla_xu(X,s)\Big\}\,ds\, (\dr_nw(X,t))\,x_n^2dt\,dX.
 \]
We further split it as $\abs{I_{221}}\le C(I_{2211}+I_{2212}+I_{2213})$, where
\[
I_{2211}:=K^{-1}\int_{\Rn_+}\int_t\abs{\dr_nw(X,t)}x_n\sum_{j=0}^{\infty}2^{-\frac{j}{2}}\fint\displaylimits_{\abs{s-t}\sim 2^{j}K^2x_n^2}\abs{\nabla_xA_\parallel(X,t)}\abs{\nabla_xu(X,s)}ds\,dtdX,
\]
\[
I_{2212}:=K^{-1}\int_{\Rn_+}\int_t\abs{\dr_nw(X,t)}x_n\sum_{j=0}^{\infty}2^{-\frac{j}{2}}\fint\displaylimits_{\abs{s-t}\sim 2^{j}K^2x_n^2}\abs{\nabla_xA_\parallel(X,s)}\abs{\nabla_xu(X,s)}ds\,dtdX, 
\]
and 
\[
I_{2213}:=\int_{\Rn_+}\int_{t}\abs{\dr_nw(X,t)}x_n^2\int_{\abs{s-t
     }>K^2x_n^2}\abs{t-s}^{-3/2}\abs{\nabla^2_xu(X,s)}\,ds\,dtdX
\]
The terms $I_{2211}$ and $I_{2212}$ can be estimated in a similar spirit using the maximal function. We start with $I_{2212}$:  
\begin{multline*}
I_{2212}\le CK^{-1}
\int_{\Rn_+}x_n\int_{t\in\R}\abs{\dr_nw(X,t)}M\br{\abs{\nabla_xA_\parallel(X,\cdot)}\abs{\nabla_xu(X,\cdot)}}(t)\,dtdX\\
\le CK^{-1}\int_{\Rn_+}x_n\br{\int_{\R}\abs{\dr_nw(X,t)}^2dt}^{1/2}\br{\int_{\R}M\br{\abs{\nabla_xA_\parallel(X,\cdot)}\abs{\nabla_xu(X,\cdot)}}(t)^2dt}^{1/2}dX\\
\le CK^{-1}\int_{\Rn_+}x_n\br{\int_{\R}\abs{\dr_nw(X,t)}^2dt}^{1/2}\br{\int_{\R}\abs{\nabla_xA_\parallel(X,t)}^2\abs{\nabla_xu(X,t)}^2dt}^{1/2}dX
\end{multline*}
by the $L^2$ boundedness of the maximal function.
Using the Cauchy-Schwarz inequality again and then the Carleson condition on the coefficients, we get
\begin{multline*}
I_{2212}\le CK^{-1}
\br{\iint_{\R^{n+1}_+}\abs{\dr_nw(X,t)}^2x_ndtdX}^{1/2}\br{\iint_{\R^{n+1}_+}\abs{\nabla_xA_\parallel(X,t)}^2\abs{\nabla_xu(X,t)}^2x_ndtdX}^{1/2}\\
\le \epsilon \int_{\Rn_+}\int_{\R}\abs{\dr_nw(X,t)}^2x_ndtdX+ C_\epsilon K^{-2}\norm{\abs{\nabla_xA_\parallel}^2\delta}_C\int_{\pom}N(\nabla_xu)^2dtdx. 
\end{multline*}
The term $I_{2211}$ needs a bit more care. Observe first that for almost every $(X,t)$, there holds 
\[\fint\displaylimits_{\abs{s-t}\sim 2^{j}K^2x_n^2}\abs{\nabla_xA_\parallel(X,t)}\abs{\nabla_xu(X,s)}ds\le \abs{\nabla_xA_\parallel(X,t)}M(\nabla_xu(X,\cdot))(t).\]
So we get 
\begin{multline*}
I_{2211}\le CK^{-1}\br{\iint\abs{\dr_nw(X,t)}^2x_ndXdt}^{1/2}\br{\iint\abs{\nabla_xA_\parallel(X,t)}^2M(\nabla_xu(X,\cdot))(t)^2x_ndXdt}^{1/2}\\
\le \epsilon\iint\abs{\dr_nw(X,t)}^2x_ndXdt 
+C_\epsilon K^{-2}\norm{\abs{\nabla_xA_\parallel}^2\delta}_C\int_{\pom}N(F)^2dtdx,
\end{multline*}
where $F(X,t)=M(\nabla_xu(X,\cdot))(t)$. Thanks to Lemma \ref{lem.NM<N}, we can control $I_{2211}$ by the same upper bound for $I_{2212}$.

To estimate $I_{2213}$, we apply the Cauchy-Schwarz inequality and get that
\begin{multline}
I_{2213}\le C\int_{\Rn_+}x_n^2
\br{\int_{t\in\R}\abs{\dr_nw(X,t)}^2dt}^{1/2}\\
\cdot
\bigg\{\int_{t\in\R} \br{\int_{\abs{s-t
     }>K^2x_n^2}\abs{t-s}^{-\frac32}\abs{\nabla^2_xu(X,s)}\,ds}^2dt\bigg\}^{1/2}dX. 
\end{multline}
Since $\int_{\abs{s-t}>K^2x_n^2}\abs{s-t}^{-3/2}ds=4K^{-1}x_n^{-1}$, 
Schur's test gives that 
\[\bigg\{\int_{t\in\R} \br{\int_{\abs{s-t
     }>K^2x_n^2}\abs{t-s}^{-\frac32}\abs{\nabla^2_xu(X,s)}\,ds}^2dt\bigg\}^{1/2}
     \le 4K^{-1}x_n^{-1}\br{\int_{\R}\abs{\nabla^2_xu(X,t)}^2dt}^{1/2}.\]
Then by the Cauchy-Schwarz inequality and Young's inequality, one obtains that
\[
   I_{2213}\le \epsilon\int_{\Rn_+}\int_{\R} \abs{\dr_nw(X,t)}^2x_n\,dt\,dX 
   +C_\epsilon K^{-2}\int_{\Rn_+}\int_{\R} \abs{\nabla^2_xu(X,t)}^2x_n\,dt\,dX. 
\]
This completes the estimate for $I_{221}$. 
\ms

We now turn to $I_{222}$. By the Leibniz integral rule, 
     \begin{multline*}
    \dr_n\br{\int_{\abs{s-t
     }>K^2x_n^2}\frac{A_\parallel(X,t)-A_\parallel(X,s)}{\abs{t-s}^{3/2}}\nabla_xu(X,s)ds}\\
     =\int_{\abs{s-t
     }>K^2x_n^2}\abs{t-s}^{-3/2}\dr_n\Big\{\br{A_\parallel(X,t)-A_\parallel(X,s)}\nabla_xu(X,s)\Big\}ds\\
     -\frac{2K^{-1}}{x_n^2}\br{A_\parallel(X,t)-A_\parallel(X,t+K^2x_n^2)}\nabla_xu(X,t+K^2x_n^2)\\
      -\frac{2K^{-1}}{x_n^2}\br{A_\parallel(X,t)-A_\parallel(X,t-K^2x_n^2)}\nabla_xu(X,t-K^2x_n^2),
     \end{multline*}
     and accordingly $I_{222}$ are split into the sum of three terms, which we call $I_{2221}$, $I_{2222}$ and $I_{2223}$. Observe that the term
     \begin{multline*}
       -I_{2221}=\\
       \int_{\Rn_+}\int_{t\in\R}\int_{\abs{s-t
     }>K^2x_n^2}\abs{t-s}^{-3/2}\dr_n\Big\{\br{A_\parallel(X,t)-A_\parallel(X,s)}\nabla_xu(X,s)\Big\}ds\cdot\nabla_xw(X,t)x_n^2dXdt
     \end{multline*}   
     can be estimated similarly as $I_{221}$, which is further split into 3 terms. One can show  that 
     \begin{multline*}
\abs{I_{2221}}
\le \epsilon\iint_{\Rn_+\times\R}\abs{\nabla_xw(X,t)}^2x_ndXdt
+C_\epsilon K^{-2}\norm{\abs{\dr_nA_\parallel}^2\delta}_{C}\int_{\pom}N(\nabla_xu)(x,t)^2dxdt\\
+C_\epsilon K^{-2}\iint_{\Rn_+\times\R}\abs{\dr_n\nabla_xu(X,t)}^2x_ndXdt.
     \end{multline*}
The terms $I_{2222}$ and $I_{2223}$ are similar. We write
\begin{multline*}
\frac12\,I_{2222}=
K^{-1}\int_{\Rn_+}\int_{t\in\R}\br{\int_{t}^{t+K^2x_n^2}\dr_sA_\parallel(X,s)ds}\nabla_xu(X,t+K^2x_n^2)\cdot\nabla_xw(X,t)dtdX\\
=\int_{\Rn_+}\int_{t\in\R}\br{\int_{t-K^2x_n^2}^{t}\dr_sA_\parallel(X,s)ds}\nabla_xu(X,t)\cdot\nabla_xw(X,t-K^2x_n^2)dtdX
\end{multline*}
by a change of variable. Then we use the localized maximal function introduced in \eqref{def.trcmaxfunc} and use Cauchy-Schwarz inequality to get that
\begin{multline*}
    \frac{1}{2}\abs{I_{2222}}\le K^{-1}\int_{\Rn_+}\int_{\R}M_{K^2x_n^2}\br{\dr_tA_{\parallel}(X,\cdot)}(t)\abs{\nabla_xu(X,t)}\abs{\nabla_xw(X,t-K^2x_n^2)}x_n^2dtdX\\
    \le K^{-1}\br{\iint M_{K^2x_n^2}\br{\dr_tA_{\parallel}(X,\cdot)}(t)^2\abs{\nabla_xu(X,t)}^2x_n^3dtdX}^{1/2}\\
    \cdot
    \br{\iint\abs{\nabla_xw(X,t-K^2x_n^2)}^2x_ndtdX}^{1/2}\\
    \le \epsilon \iint_{\Rn_+\times\R}\abs{\nabla_xw(X,t)}^2x_ndtdX
+C_\epsilon\norm{\abs{\dr_tA_\parallel}^2\delta^3}_{C}\int_{\pom}N(\nabla_xu)(x,t)^2dxdt,
\end{multline*}
where we have used again Corollary \ref{cor.locMaxCarl}.
The term $I_{2223}$ can be estimated almost exactly the same and one gets that 
\[\abs{I_{2223}}\le  \epsilon \iint_{\Rn_+\times\R}\abs{\nabla_xw(X,t)}^2x_ndtdX
+C_\epsilon\norm{\abs{\dr_tA_\parallel}^2\delta^3}_{C}\int_{\pom}N(\nabla_xu)(x,t)^2dxdt.\]

\smallskip
Altogether, we have proved that there exists $C>0$ depending only on ellipticity and the dimension, such that for any $\epsilon\in(0,1)$ and any $K\ge 1$,
\begin{multline*}
    \iint\abs{\nabla w}^2x_ndXdt\le \epsilon\, C\iint\abs{\nabla w}^2x_ndXdt 
    +C_\epsilon K^{-2}\sum_{i=1}^{n-1}\iint_{\om}\abs{\nabla(\dr_iu)}^2x_ndXdt\\
+\br{C_{\epsilon}K^4\norm{\mu_{||}}_C+C_{\epsilon}K^{-2}\norm{\abs{\dr_nA_{||}}^2\delta}_C}\int_{\pom}N(\nabla_xu)^2dxdt
+ C\norm{D^{1/2}_{t}f}_{L^2(\pom)}^2.
\end{multline*}
The estimate \eqref{eq.gradD1/2.1} follows from taking $\epsilon$ small enough so that $\epsilon C<1/2$.

\ep

\ms

Now we show that the integral $\iint_\om\abs{\dr_t u(X,t)}^2x_ndXdt$ can be controlled  as follows.
\begin{lemma}\label{lem.Dtu}
 Let $u$ be a bounded solution to $\mathcal{L}u=0$ in $\om$, and let $w=D^{1/2}_{t}u$. Then there is a constant $C>0$, and for any $\epsilon>0$, there exists $C_\epsilon>0$ such that 
    \begin{multline}\label{Est.Dtusq}
        \iint_\om\abs{\dr_t u(X,t)}^2x_ndXdt\le\epsilon\iint\abs{\nabla^2w}^2x_n^3dtdX+C\br{1+\epsilon^{-1}}\iint\abs{\nabla w}^2x_n\,dtdX\\
+C\iint\abs{\nabla_x^2u}^2x_n\,dtdX
+C(1+\epsilon^{-1})\norm{\mu_{||}}_C\int_{\pom}N(\nabla_xu)^2dxdt.
    \end{multline}
\end{lemma}

\bp
We claim that 
\begin{multline}\label{eq.Dtusqr}
    \iint\abs{\dr_tu}^2x_ndtdX
    =\iint H_t(\dr_nw)\divg(A\nabla w)x_n^2dtdX\\
    -c_0\iint\int_{s\in\R}\frac{A_\parallel(X,t)-A_\parallel(X,s)}{\abs{t-s}^{3/2}}\nabla_xu(X,s)ds\cdot\nabla_x(H_t\dr_nw)x_n^2dtdX=:T_1+T_2,
\end{multline}
where $w=D^{1/2}_{t}u$. Assuming \eqref{eq.Dtusqr} for now, we control $T_1$ and $T_2$ using similar techniques as the proof of Lemma \ref{lem.gradD1/2}, and we omit the details from time to time to avoid redundancy. 

We start with $T_1$. 
By Cauchy-Schwarz,
\begin{multline*}
    \abs{T_1}\le \br{\iint \abs{H_t(\dr_nw)}^2x_ndtdX}^{1/2}\br{\iint\abs{\divg A\nabla w}^2x_n^3dtdX}^{1/2}\\
     \le\br{\iint \abs{H_t(\dr_nw)}^2x_ndtdX}^{1/2}
     \Big\{\br{\iint\abs{\nabla A}^2\abs{\nabla w}^2x_n^3dtdX}^{1/2}+\br{\iint\abs{A}^2\abs{\nabla^2w}^2x_n^3}^{1/2}\Big\}\\
     \le C\br{1+\epsilon^{-1}}\iint\abs{\nabla w}^2x_ndtdX
     +\epsilon\iint\abs{\nabla^2w}^2x_n^3,
\end{multline*}
where in the last inequality we have used Young's inequality and $\abs{\nabla A}x_n\le C$, which is valid thanks to \eqref{D2Bbdd} in Lemma \ref{lem.CarImprov} and Remark \ref{re.pert} (3). 

For $T_2$, we split the integral in $s\in\R$ into $\abs{s-t}<x_n^2$ and $\abs{s-t}\ge x_n^2$ and call the two terms $T_{21}$ and $T_{22}$, respectively. The term $T_{21}$ will be treated similarly as $I_{21}$ in the proof of Lemma \ref{lem.gradD1/2}: 
 \begin{multline*}
   c_0^{-1}\abs{T_{21}}= 
   \abs{\iint\int_{\abs{s-t}<x_n^2}\frac{\int_s^t\dr_\tau A_\parallel(X,\tau)d\tau}{\abs{t-s}^{3/2}}\nabla_xu(X,s)ds\cdot\nabla_x(H_t\dr_nw)x_n^2\,dtdX}\\
   \le2\iint\int_{\abs{s-t}<x_n^2}\frac{M_{x_n^2}\br{\dr_\tau A_{\parallel}(X,\cdot)}(s)}{\abs{t-s}^{1/2}}\abs{\nabla_xu(X,s)}ds\abs{\nabla_xH_t\dr_nw}x_n^2dtdX,
 \end{multline*}
where $M_{x_n^2}\br{\dr_\tau A_{\parallel}(X,\cdot)}(s)$ is the local maximal function defined as in \eqref{def.trcmaxfunc}. We proceed as in term $I_{21}$, that is, we decompose $\set{s\in\R:\abs{s-t}<x_n^2}$ into annuli $2^{-j-1}x_n^2\le\abs{s-t}\le 2^{-j}x_n^2$ for $j\in\N$, and we control each of them by the maximal function at $t$:
\begin{multline*}
    c_0^{-1}\abs{T_{21}}
\lesssim\int_{\Rn_+}x_n^3\int_{\R}M(M_{x_n^2}\br{\dr_\tau A_{\parallel}(X,\cdot)}\abs{\nabla_xu(X,\cdot)})(t)\abs{\nabla_xH_t\,\dr_nw}dtdX\\
\lesssim\br{\iint M_{x_n^2}\br{\dr_\tau A_{\parallel}(X,\cdot)}(t)^2\abs{\nabla_xu(X,t)}^2x_n^3dtdX}^{1/2}\br{\iint\abs{\nabla_x\dr_nw}^2x_n^3}^{1/2}\\
\lesssim\norm{\abs{\dr_tA}^2\delta^3}_{C}^{1/2}\br{\int_{\pom}N(\nabla_xu)^2dxdt}^{1/2}\br{\iint\abs{\nabla_x\dr_nw}^2x_n^3}^{1/2}\\
\le C\epsilon^{-1}\norm{\mu_{||}}_C\int_{\pom}N(\nabla_xu)^2dxdt+\epsilon\iint\abs{\nabla^2w}^2x_n^3dXdt
\end{multline*}
by Corollary \ref{cor.locMaxCarl} and Young's inequality. For $T_{22}$, we move $\nabla_x$ from $H_t\dr_nw$ to get
\begin{multline*}
    c_0^{-1}T_{22}=-\iint\divg_x\br{\int_{\abs{t-s}\ge x_n^2}\frac{A_\parallel(X,t)-A_\parallel(X,s)}{\abs{t-s}^{3/2}}\nabla_xu(X,s)ds}H_t(\dr_nw) x_n^2 dtdX\\
    =-\iint\int_{\abs{t-s}\ge x_n^2}\frac{\nabla_xA_\parallel(X,t)-\nabla_xA_\parallel(X,s)}{\abs{t-s}^{3/2}}\nabla_xu(X,s)ds\,H_t(\dr_nw) x_n^2 dtdX\\
    -\iint\int_{\abs{t-s}\ge x_n^2}\frac{A_\parallel(X,t)-A_\parallel(X,s)}{\abs{t-s}^{3/2}}\nabla^2_xu(X,s)ds\,H_t(\dr_nw) x_n^2 dtdX=:T_{221}+T_{222}.
\end{multline*}
We have 
\begin{multline*}
    \abs{T_{221}}\le\iint\int_{\abs{t-s}\ge x_n^2}\frac{\abs{\nabla_xA_\parallel(X,t)}}{\abs{t-s}^{3/2}}\abs{\nabla_xu(X,s)}ds\abs{H_t(\dr_nw)} x_n^2 dtdX\\
    +\iint\int_{\abs{t-s}\ge x_n^2}\frac{\abs{\nabla_xA_\parallel(X,s)}}{\abs{t-s}^{3/2}}\abs{\nabla_xu(X,s)}ds\abs{H_t(\dr_nw)} x_n^2 dtdX.
\end{multline*}
The first term on the right-hand side can be treated similarly as $I_{2211}$ in the proof of Lemma \ref{lem.gradD1/2} and be controlled by
\begin{multline*}
    C\br{\iint\abs{\nabla_xA(X,t)}^2M(\nabla_xu(X,\cdot))(t)^2x_ndtdX}^{1/2}\br{\iint\abs{\dr_nw}^2x_n}^{1/2}\\
\lesssim\norm{\abs{\nabla_xA}^2\delta}_C^{1/2}\br{\int_{\pom}N(\nabla_xu)^2dxdt}^{1/2}\br{\iint\abs{\dr_nw}^2x_n\,dtdX}^{1/2}\\
\le C\norm{\mu_{||}}\int_{\pom}N(\nabla_xu)^2dxdt + C\iint\abs{\nabla w}^2x_ndXdt
\end{multline*}
using Lemma \ref{lem.NM<N}. The second term on the right-hand side can be treated similarly as $I_{2212}$ and can be controlled by the same bounds as above. 
We are left with term $T_{222}$, which can be treated similarly as $I_{2213}$ and controlled by 
\[
\abs{T_{222}}\lesssim\br{\iint\abs{\nabla_x^2u}^2x_n}^{1/2}\br{\iint\abs{\dr_nw}^2x_n}^{1/2}\le C\iint\abs{\nabla_x^2u}^2x_ndXdt+C\iint\abs{\dr_nw}^2x_n.
\]
Collecting all the estimates, we have \eqref{Est.Dtusq}.\ms

It remains to justify \eqref{eq.Dtusqr}. We write $x_n=\frac12\dr_n(x_n^2)$ and $\dr_t=-D_{t}^{1/2}H_tD_t^{1/2}$ to get that 
\[
    \iint\abs{\dr_t u}^2x_ndtdX=-\frac12\iint D_{t}^{1/2}H_tD_t^{1/2}u\,\dr_tu\,\dr_n(x_n^2)dtdX=-\frac12\iint H_tw\,\dr_tw\,\dr_n(x_n^2)dtdX
\]
by moving one $D_{t}^{1/2}$ to $\dr_t u$. By divergence theorem, we can move $\dr_n$ to get that
\[
\iint\abs{\dr_t u}^2x_ndtdX=\frac12\iint H_t(\dr_nw)\dr_tw\,x_n^2dtdX+\frac12\iint H_tw\,\dr_n\dr_t w\, x_n^2dtdX.
\]
Moving $\dr_t$ and $H_t$, we see that 
\[
\iint H_tw\,\dr_n\dr_t w\, x_n^2dtdX=-\iint H_t\dr_tw\,\dr_nw\,x_n^2dtdX=\iint \dr_tw\,H_t\dr_nw\, x_n^2dtdX,
\]
which shows that 
\[
\iint\abs{\dr_t u}^2x_ndtdX=\iint H_t(\dr_nw)\dr_tw\,x_n^2dtdX.
\]
We then use the PDE for $\dr_t w$, that is, \eqref{eq.BlckD1/2eq}.
Then \eqref{eq.Dtusqr} follows from moving $\divg_x$ to $H_t(\dr_nw)x_n^2$.

\ep

We now apply the Caccioppoli inequality for $\nabla^2(D^{1/2}_tu)$ (Lemma \ref{lem.Caccio}) to the first term on the right-hand side of \eqref{Est.Dtusq} to obtain the following estimate.

\begin{lemma}\label{lem.glbCaccio}
There is $C>0$ such that 
    \begin{multline}\label{eq.glbCaccio}
    \iint\abs{\nabla^2(D_t^{1/2}u)}^2x_n^3dXdt
\le C\iint\abs{\nabla(D^{1/2}_tu)}^2x_ndXdt+ C\iint\abs{\nabla^2u}^2x_ndXdt\\
+C\norm{\mu}_C\norm{N(\nabla u)}_{L^2(\pom)}^2.
\end{multline}
\end{lemma}

\bp
Let $w=D^{1/2}_tu$. 
We control the integral $\iint \abs{\nabla^2w}^2x_n^3$ using \eqref{eq.Caccio}. Let $\W$ be a Whitney decomposition of the upper half-space $\Rn_+$, that is, the cubes $Q\in \W$ satisfy the properties that they are non-overlapping, 
\(\Rn_+ = \bigcup_{Q\in \W} Q\),
and
\(2\ell(Q) \leq \dist(Q,\partial \Rn_+) < 4\ell(Q)\). By \eqref{eq.Caccio}, 
\begin{multline*}
    \iint_{\Rn_+\times\R} \abs{\nabla^2w}^2x_n^3dXdt
    =\sum_{Q\in\W}\int_{(x,x_n)\in Q}\int_{t\in \R} \abs{\nabla^2w}^2x_n^3dXdt
    \le\sum_{Q\in\W} \int_{2Q}\int_{\R}\abs{\nabla w}^2x_ndtdX\\
    +\sum_{Q\in\W}\int_{2Q}\int_{\R}\abs{\nabla A}^2\abs{\nabla u}^2x_n dtdX
        +\sum_{Q\in\W}\int_{2Q}\int_{\R}\abs{\nabla A}^2M(\nabla u(X,\cdot))(t)^2x_ndtdX\\
        +\sum_{Q\in\W}\int_{2Q}\int_{\R}M_{x_n^2}(\dr_\tau\nabla A(X,\cdot))(t)^2\abs{\nabla u}^2x_n^5dtdX
    +\sum_{Q\in\W}\int_{2Q}\int_{\R}\abs{\nabla^2u(X,t)}^2x_ndtdX\\
   \le C\iint_{\R^{n+1}_+} \abs{\nabla w}^2x_ndtdX
    +C\iint_{\R^{n+1}_+}\abs{\nabla A}^2\abs{\nabla u}^2x_n 
        +C\iint_{\R^{n+1}_+}\abs{\nabla A}^2M(\nabla u(X,\cdot))(t)^2x_n\\
        +C\iint_{\R^{n+1}_+}M_{x_n^2}(\dr_\tau\nabla A(X,\cdot))(t)^2\abs{\nabla u}^2x_n^5
    +C\iint_{\R^{n+1}_+}\abs{\nabla^2u(X,t)}^2x_ndtdX
\end{multline*}
as the family $\set{2Q}_{Q\in\W}$ has finite overlaps. By Carleson's inequality and Lemma \ref{lem.NM<N},
\begin{multline*}
      \iint_{\Rn_+\times\R} \abs{\nabla^2w}^2x_n^3dXdt\le 
    C\br{\norm{\abs{\nabla A}^2x_n}_C+\norm{M_{x_n^2}\br{\dr_t\nabla A(X,\cdot)}^2x_n^5}_C}\int N(\nabla u)^2dxdt\\
    +C\iint\abs{\nabla w}^2x_ndtdX+ C\iint\abs{\nabla^2u}^2x_ndtdX.
\end{multline*}
By perturbation theory (Lemma \ref{lem.pert}) and Lemma \ref{lem.CarImprov}, we can assume that $\abs{\dr_t\nabla A(X,t)}^2x_n^5dXdt$ is a Carleson measure with a Carleson norm bounded by $C\norm{\mu}_C$ (see Remark \ref{re.pert} (3)). Then by Corollary \ref{cor.locMaxCarl}, we have that $\norm{M_{x_n^2}\br{\dr_t\nabla A(X,\cdot)}^2x_n^5}_C\le C\norm{\mu}_C$. 
From this \eqref{eq.glbCaccio} follows.
\ep


\section{$S<N$ estimates in $L^p$ for block form matrix $A$}
The goal of this section is to prove the following theorem.
\begin{theorem}\label{thm.S<Np}
Let $\LL=-\dr_t+L$ be a parabolic operator with matrix $A$ in the block form \eqref{block} that satisfies \eqref{E:elliptic}, \eqref{E:1:carl}, and \eqref{E:1:bound}. Let $u$ be an energy solution to $\LL u=0$ in $\om=\Rn_+\times\R$.
For any $p\in(0,\infty)$, there is $C>0$ such that
    \[
    \norm{S(\nabla u)}_{L^p}\le C\norm{N(\nabla u)}_{L^p}. 
    \]
\end{theorem}
\medskip

To prove Theorem \ref{thm.S<Np}, we first derive a $S<N$ estimate on sawtooth domains, which is done using cutoff functions.  

Let $\Psi$ be a cutoff function on $\om$ that satisfies the following properties.
\begin{enumerate}
    \item $0\le \Psi\le1$,
    \item for any $p\in(0,\infty)$, the measure given by the density
    \begin{equation}\label{eq.gradPsiCM}
        \br{\abs{\nabla\Psi(X,t)}^px_n^{p-1}+\abs{\dr_t\Psi(X,t)}^px_n^{2p-1}}dxdx_ndt
    \end{equation}
    is a Carleson measure on $\om$.
\end{enumerate}
\begin{lemma}\label{lem.STSdk}
Let $\LL=-\dr_t+L$ be a parabolic operator with matrix $A$ in the block form \eqref{block}, and let $u$ be an energy solution to $\LL u=0$ in $\om$. Let $\Psi$ be a cutoff function that satisfies the following Carleson condition: the measure $\nu$ given by the density 
\begin{equation}\label{gradPsiCM}
d\nu=\br{\abs{\nabla\Psi}^2x_n+\abs{\dr_n\Psi}+\abs{\dr_t\Psi}x_n}dxdtdx_n \quad\text{ is a Carleson measure}.
\end{equation}
Let $m\ge 3$ be an integer. Then
\begin{equation}\label{eq.STSdk<N}
    \iint \abs{\nabla^2 u}^2\Psi^mx_n\, dXdt
    \le C\br{\norm{\nu}_C+\norm{\abs{\nabla A_{\parallel}}^2x_n}_{C}}\int\abs{N(\Psi^{\frac{m}{2}-1}\nabla u)}^2dxdt.
\end{equation}
More precisely, for any $k\in\N$ with $1\le k\le n$, there holds
\begin{multline*}
    \iint \abs{\nabla(\dr_ku)}^2\Psi^mx_n\, dXdt\\
    \le C\norm{\nu}_C\int\abs{N(\Psi^{\frac{m}{2}-1}\dr_k u)}^2dxdt
    +\norm{\abs{\dr_k A_{\parallel}}^2x_n}_{C}\int\abs{N(\Psi^{\frac{m}{2}-1}\nabla_x u)}^2dxdt.
\end{multline*}
\end{lemma}

\bp
Fix $1\le k\le n$, we write $w_k=\dr_ku$. By the divergence theorem,
\begin{multline}\label{eq.STpf1}
\iint A\nabla w_k\cdot\nabla w_k\br{\Psi^mx_n}=-\iint w_k A\nabla w_k\cdot\nabla(\Psi^mx_n) \underbrace{-\iint\divg(A\nabla w_k)(w\Psi^mx_n)}_{=:J}\\
=-\iint w_k\Psi^m\dr_nw_k-m\iint w_k\Psi^{m-1} x_nA\nabla w_k\cdot\nabla\Psi +J\\
=\frac{m}{2}\iint w_k^2\Psi^{m-1}\dr_n\Psi
-m\iint w_k\Psi^{m-1}x_nA\nabla w_k\cdot\nabla\Psi +J.
\end{multline}
For the term $J$, we use the PDE that $w_k$ satisfies and then integrate by parts to get that
\begin{multline*}
    J=-\iint(\dr_t w_k)w_k\Psi^mx_n+\iint \divg\br{(\dr_kA)\nabla u}w_k\Psi^mx_n\\
    =\frac{m}{2}\iint w_k^2\Psi^{m-1}x_n\dr_t\Psi-\iint\Psi^mx_n\br{\dr_kA_\parallel}\nabla_xu\cdot\nabla_x w_k-m\iint w_k\Psi^{m-1} x_n\br{\dr_kA_\parallel}\nabla_xu\cdot\nabla_x \Psi,
\end{multline*}
where we have used $\dr_k a_{nn}=0$ because $A$ is in block form (so $a_{nn}=1$).
By ellipticity, \eqref{eq.STpf1}, and the Cauchy-Schwarz inequality, one obtains that 
\begin{multline*}
\lambda\iint \abs{\nabla w_k}^2\Psi^mx_n\le \frac{\lambda}{2}\iint\abs{\nabla w_k}^2\Psi^mx_n
+C\iint w_k^2\Psi^{m-2}\abs{\nabla\Psi}^2x_n\\
+ C\iint\abs{\dr_kA_\parallel}^2\abs{\nabla_xu}^2\Psi^mx_n
+\frac{m}{2}\iint w_k^2\Psi^{m-1}\abs{\dr_n\Psi}
+\frac{m}{2}\iint w_k^2\Psi^{m-1}x_n\abs{\dr_t\Psi}.
\end{multline*}
Using the Carleson condition satisfied by $\Psi$ and $A$, and absorbing the first term to the left-hand side, we have 
\begin{multline*}
    \frac{\lambda}{2}\iint \abs{\nabla w_k}^2\Psi^mx_n
    \le C\norm{\nu}_C\int N\br{w_k\Psi^{\frac{m}{2}-1}}^2dxdt 
    +\norm{\abs{\dr_kA_\parallel}^2x_n}_C\int N\br{\abs{\nabla_x u}\Psi^{\frac{m}{2}}}^2 dxdt,
\end{multline*}
as desired.
\ep
\smallskip

\begin{definition}[The cutoff function $\Psi$ adapted to a sawtooth domain]
   Let $\Delta=\Delta_r(x_0,t_0)$ be a (boundary) parabolic ball of radius $r$, that is, $\Delta=\set{(x,t)\in\R^{n-1}\times\R: \abs{x-x_0}+\abs{t-t_0}^{1/2}<r}$. Let $F\subset\Delta$ be a Borel set. Define the sawthooth domain 
   \begin{equation}\label{defeq.STom}
       \om_{a,F}:=\bigcup_{(x,t)\in F}\Gamma_a(x,t),
   \end{equation}
    where $\Gamma_a(x,t):=\set{(y,y_n,s)\in\R^{n+1}_+\times\R: \abs{y-x}+\abs{s-t}^{1/2}<ay_n}$ is the parabolic cone with vertex $(x,t)$ and aperture $a$.
    Let $\eta\in C^\infty(\R_+)$ be such that $0\le\eta\le1$, $\eta=1$ on $[0,1]$, $\eta=0$ on $[2,+\infty)$, and $\abs{\eta'}\le 2$. 
    Let $\Upsilon\in C_0^\infty(\R^{n+1})$ be supported in $B(0,10^{-3})\subset\R^{n+1}$, with $0\le\Upsilon\le 1$, and $\iint_{\R^{n+1}}\Upsilon(y,y_n,t)dydy_ndt=1$. For $\ell>0$, set \[\Upsilon_\ell^a(x,x_n,t):=\frac{1}{(a\ell)^{n+1}\ell}\Upsilon\br{\frac{x}{a\ell},\frac{x_n}{\ell},\frac{t}{(a\ell)^2}}.\]
    We define the cutoff function $\Psi$ as
    \begin{equation}\label{defeq.Psi}
        \Psi(x,x_n,t):=\eta\br{\frac{ax_n}{2r}}
        \br{\1_{\om_{3a/2,F}}*\Upsilon^a_{x_n}}(x,x_n,t).
    \end{equation}
\end{definition}
Let $k\Delta$ denote $\Delta_{kr}$ for any $k>0$.
By construction, $0\le\Psi\le1$, $\Psi=1$ on $\om_{a,F}\cap\br{[0,\frac{2r}{a}]\times4\Delta}$, and $\supp\Psi\subset\om_{2a,F}\cap \br{[0,\frac{4r}{a}]\times 10\Delta}$. Moreover, we can verify that $\Psi$ satisfies the Carleson condition \eqref{eq.gradPsiCM}, with a constant depending on $n$, $a$, and $p$. 

\begin{lemma}[Good-$\lambda$ inequality]\label{lem.gdlambda}
    For any $\nu>0$ and any $\gamma\in(0,1)$, 
    \begin{multline}\label{eq.gl1}
        \left|\set{(x,t)\in\R^{n-1}\times\R:\, S_a(\nabla u)(x,t)>\nu, \, N_{2a}(\nabla u)(x,t)\le\gamma\nu}\right|\\
        \le C\gamma^2\abs{\set{(x,t)\in\R^{n-1}\times\R:\, S_{2a}(\nabla u)(x,t)>\nu/2}}.
    \end{multline}
\end{lemma}

\bp
Fix $\nu>0$ and $\gamma\in(0,1)$. 

We denote $E_1:=\set{(x,t)\in\R^{n-1}\times\R:\, S_a(\nabla u)(x,t)>\nu, \, N_a(\nabla u)(x,t)\le\gamma\nu}$ and 
$E_2:=\set{(x,t)\in\R^{n-1}\times\R:\, S_{2a}(\nabla u)(x,t)>\nu/2}$.
Clearly, 
$E_2$ is an open 
subset of ${\R}^{n}$. When this set is empty, or is all of ${\R}^{n}$, 
estimate \eqref{eq.gl1} is trivial, so we focus on the case when the set in question is 
both nonempty and proper. By Vitali's covering lemma, we can find a non-overlapping collection of parabolic balls $\set{\Delta(x_i,t_i)}_{i\in \N}$ 
such that $2\Delta(x_i,t_i)\subset E_2$, $10\Delta(x_i,t_i)\cap E_2^c\neq\emptyset$ for each $i$, and  that $E_2\subset\bigcup_{i\in \N}5\Delta(x_i,t_i)$. Denote $\Delta_i=5\Delta(x_i,t_i)$, and denote by $r_i$ the radius of $\Delta_i$.
Set $F_i:=E_1\cap\Delta_i$. Note that $E_1\subset\bigcup_{i\in\N}F_i$ as $E_1\subset E_2$. Since $2\Delta_i\cap E_2\neq\emptyset$, there exists $(x',t')\in 2\Delta_i$ such that $S_{2a}(\nabla u)(x',t')\le \nu/2$. 
Observe that for any $(x,t)\in \Delta_i$, $\Gamma_{2a}(x',t')\supset \br{\Gamma_a(x,t)\cap\set{(y,y_n,s): y_n\ge 2r_i/a}}$. Therefore, for any $(x,t)\in F_i$, $S_a^{2r_i/a}(\nabla u)(x,t)>\nu/2$, where $S_a^{2r_i/a}$ is the square function defined for parabolic cones truncated at height $2r_i/a$. It follows that 
\[\abs{F_i}\le \br{\frac{2}{\nu}}^2\int_{F_i}S_a^{2r_i/a}(\nabla u)(x,t)^2dxdt\le \frac{C}{\nu^2}\sum_{k=1}^{n}\iint_{\om_{a,F_i}\cap\set{0\le x_n\le 2r_i/a}}\abs{\nabla(\dr_ku)}^2x_ndxdt,\]
where $\om_{a,F_i}$ is the sawthooth domain defined as in \eqref{defeq.STom}. Let $\Psi_i$ be the cutoff function defined as in \eqref{defeq.Psi} for $\om_{a,F_i}$. Since $\Psi_i=1$ on $\om_{a,F_i}\cap\br{[0,\frac{2r_i}{a}]\times4\Delta_i}$, we can insert $\Psi^m$ in the above integral, where $m\ge3$ is an integer. Then we apply Lemma \ref{lem.STSdk} to get that 
\[
    \abs{F_i}\le\frac{C}{\nu^2}\sum_{k=1}^{n}\iint\abs{\nabla(\dr_ku)}^2\Psi_i^mx_ndxdt
    \le \frac{C}{\nu^2}\int N_a(\Psi_i^{\frac{m}{2}-1}\abs{\nabla u})(x,t)^2dxdt.
\]
Since $\supp\Psi_i\subset\om_{2a,F_i}\times \br{[0,\frac{4r_i}{a}]\times 10\Delta_i}$, the support of $N_a(\Psi_i^{\frac{m}{2}-1}\abs{\nabla u})$ is a subset of $15\Delta_i$. Moreover, (recalling the definition of $\om_{2a,F_i}$) for any $(x,t)\in 15\Delta_i$, there exists some $(x',t')\in F_i$, such that \[N_a(\Psi_i^{\frac{m}{2}-1}\abs{\nabla u})(x,t)\le N_{2a}(\Psi_i^{\frac{m}{2}-1}\abs{\nabla u})(x',t')\le N_{2a}(\abs{\nabla u})(x',t')\le \gamma\nu\]
by the definition of the set $E_1$. That is, we have that 
\[
\abs{F_i}\le \frac{C}{\nu^2}\int_{15\Delta_i} N_a(\Psi_i^{\frac{m}{2}-1}\abs{\nabla u})(x,t)^2dxdt\le C\gamma^2\abs{15\Delta_i}.
\]
Summing in $i\in\N$, this implies that 
\[\abs{E_1}\le C\gamma^2\sum_{i\in\N}\abs{\Delta_i}\le C\gamma^2\abs{\bigcup_{i\in\N}\Delta(x_i,t_i)}\le C\gamma^2\abs{E_2}.\]
\ep



Now that we have the good-$\lambda$ inequality, Theorem \ref{thm.S<Np} can be obtained via a standard argument.
\smallskip

\noindent{\it Proof of Theorem \ref{thm.S<Np}:} 
Let $\gamma\in(0,1)$ be determined later. For any $\nu>0$, denote $E_{\nu}:=\set{(x,t)\in\R^{n-1}\times\R:\, S_a(\nabla u)(x,t)>\nu, \, N_a(\nabla u)(x,t)\le\gamma\nu}$.
We compute
\begin{multline*}
   \norm{S_a(\nabla u)}_{L^p(\pom)}^p=p\int_0^\infty \nu^{p-1}\abs{\set{(x,t)\in\R^{n-1}\times\R: S_a(\nabla u)(x,t)>\nu}}d\nu\\
   \le p\int_0^\infty \nu^{p-1}\abs{E_\nu}d\nu +p\int_0^\infty \nu^{p-1}\abs{\set{N_{2a}(\nabla u)>\gamma\nu}}d\nu.
\end{multline*}
By Lemma \ref{lem.gdlambda} and a change of variable, one has
\begin{multline*}
    \norm{S_a(\nabla u)}_{L^p(\pom)}^p \le C\gamma^2p\int_0^\infty \nu^{p-1}\abs{\set{S_{2a}(\nabla u)>\nu/2}}d\nu +\gamma^{-p}p\int_0^\infty \nu^{p-1}\abs{\set{N_{2a}(\nabla u)>\nu}}d\nu\\
    =C\gamma^2\norm{S_{2a}(\nabla u)}_{L^p}^p+\gamma^{-p}\norm{N_{2a}(\nabla u)}_{L^p}^p \le C\gamma^2\norm{S_{a}(\nabla u)}_{L^p}^p+C\gamma^{-p}\norm{N_{a}(\nabla u)}_{L^p}^p.
\end{multline*}
Choosing $\gamma\in(0,1)$ sufficient small so that $C\gamma^2<1/2$, one can hide $C\gamma^2\norm{S_{a}(\nabla u)}_{L^p}^p$ to the left-hand side and obtain the desired estimate.

\ep

\section{$N<S+A$ estimates for block form matrix $A$}\label{SS:43}

We follow the approach of \cite{KKPT} to prove nontangential maximal function estimates, with some important modifications. 
The arguments in the following section are motivated by the work of \cite{DHM} that introduced 
stopping time arguments using an entire family of Lipschitz graphs on which the nontangential 
maximal function is large in lieu of a single graph constructed via a stopping time argument. 
This approach is necessary as we are using $L^2$ averages of solutions to define the nontangential maximal 
function and hence the knowledge of certain bounds for a solution on a single graph provides no 
information about the $L^2$ averages over interior balls.

The arguments differ from \cite{DHM} in one substantial point, namely that we do not use the \lq\lq pullback" map as it destroys the block-form structure of the matrix $A$ . Instead, we use an integration by parts argument involving the function constructed via the stopping time technique. Here the argument is more akin to \cite{DHP} which used this method for elliptic block-form operators. We also borrow a key idea from \cite{DFM} which introduced a method that avoids using maximal functions in the good-$\lambda$ inequality and hence it thus implies
$N<S+A$ in the $L^p$ sense directly for all $p>1$ (whereas an older method required a separate argument for $p\le 2$).

Finally, we have to modify the approach and adapt it to the parabolic setting. We are only aware of one place where $N<S$ estimates have been established in parabolic settings, namely the  thesis of Rivera-Noriega \cite{RNt} for the Dirichlet problem (but with some errors in that proof). \medskip

As before, we shall work under the assumption that $\Omega=\R^n_+\times\R$. We will only assume the 
block-form structure of the matrix $A$ and impose the Carleson condition on the coefficients, but distinguishing the norm in directions parallel to the boundary).  

The goal of this section is to prove the following theorem.
\begin{theorem}\label{thm.NlessS}
   Let $A$ be a block-form matrix \eqref{block} that satisfies the ellipticity condition \eqref{E:elliptic}, the Carleson condition \eqref{E:1:carl}, and the bound \eqref{E:1:bound}. 
   Let $\LL=-\dr_t+\divg(A\nabla\cdot)$.
   
   Then there exists $\delta>0$ such that if $$\left[\|\mu_\parallel\|_C+\norm{x_n\abs{\nabla_x A}}^2_{L^\infty(\om)}\right]
  \Big(1+\norm{x_n\abs{\nabla A}}^2_{L^\infty(\om)}\Big)<\delta,$$  then for any $p>1$, $a>0$, $f\in \dot L^{p}_{1,1/2}(\pom)\cap \dot H_{\dr_t-\Delta_x}^{1/4}(\pom)$,   there exists a 
constant $C=C(n,p,a,\delta, \lambda,\Lambda)>0$, such that for the energy solution $u$ to \eqref{eq-pp}, there holds
\begin{multline}\label{eq.N<SLp}
    \norm{\tilde{N}(\nabla u)}_{L^p(\pom)}\le C(1+\|\mu\|_C+\|x_n|\nabla A_\parallel|\|_{L^\infty}^2)\left[\norm{S(\nabla u)}_{L^p(\pom)} +\norm{\nabla_x f}_{L^p(\pom)}\right]\\+C \norm{A(\nabla u)}_{L^p(\pom)},
 \end{multline}   
provided we know a priori that $\|\tilde{N}(\nabla u)\|_{L^p(\partial\Omega)}<\infty$.
Here $\tilde{N}$ is the averaged version of the nontangetial maximal function. 
\end{theorem}

The energy solutions $u$ constructed using Lax-Milgram lemma on $\Omega$ earlier a priori belong to the space $\dot{W}^{1,2}(\mathbb R^n_+\times \R)$ and thus $\nabla u\in L^2(\mathbb R^n_+\times \R)$. 
In particular, this implies that 
\begin{equation}\label{def.w}
    w(X,t):=\left(\fiint_{B((X,t),\delta(X,t)/2)}|\nabla u(Y,s)|^2dYds\right)^{1/2}\to 0
\end{equation}
as $x_n\to\infty$.

Recall that we have defined the nontangential cones in \eqref{Gamma2.11}. As we now work on just upper half-space we modify the cones somewhat and define, for a fixed background parameter
$a>0$, the nontangential cone $\gamma_a(q,\tau)$ for a boundary point $(q,\tau)\in\partial(\R^n_+\times \R)$ as follows:
\begin{equation}\label{gamma2.11}
\gamma_a(q,\tau)=\{(X,t)=(x,x_n,t):\, x_n>a^{-1}\|(x-q,t-\tau)\|\}.
\end{equation}
There exists $a',a''$ such that $\Gamma_{a'}(q,\tau)\subset \gamma_a(q,\tau)\subset \Gamma_{a''}(q,\tau)$ where $\Gamma$ are the cones defined in  \eqref{Gamma2.11} and thus the $L^p$ norms of the square and nontangential maximal functions defined with respect to $\gamma$ are comparable to those defined
with respect to $\Gamma$.. But \eqref{gamma2.11} is better for our purposes here, as the special direction $x_n$ is separated from all other variables. We will refer to the parameter $a$ as before the aperture and 
$a^{-1}$ we shall call the slope of the nontangential cone $\gamma_a$. For the remainder of this section we assume that $N$, $S$ and $A$
are defined with respect to the cones $\gamma_a$.

For a constant $\nu>0$, define the set
\begin{equation}\label{E}
E_{\nu,a}:=\big\{x'\in\partial\Omega:\,N_{a}(w)(x')>\nu\big\}.
\end{equation}

These cones, as defined currently have vertices at the boundary, but in order to define the stopping time construction we also need to define interior cones (with vertices inside $\Omega$). In the elliptic setting these interior cones
are defined by simply shifting the boundary cone, that is for $(Q,\tau)=(q,q_n,\tau)\in\Omega$
\begin{equation}\label{oldG}
\gamma_a(q,q_n,\tau)=(0,q_n,0)+\{(X,t)=(x,x_n,t):\, x_n>a^{-1}\|(x-q,t-\tau)\|\}.
\end{equation}
This is the approach also taken in \cite{RN}. However, we prefer to modify this construction and consider \lq\lq enlarged" cones that have an additional feature which we explain below.

Firstly, let us consider the set $S(X,t)=\{(q,\tau)\in \partial(\R^n_+\times \R):\, (X,t)\in\gamma_a(q,\tau)\}$. Given the way the cones are defined, it follows that
\begin{equation}\label{def.SXt}
    S(X,t)=\{(q,\tau):\, \|(x-q,t-\tau)\|<a\,x_n\},
\end{equation}
which is an open parabolic ball centered at $(x,t)$ of \lq\lq radius" $a\,x_n$. We define the cone $\gamma_a$ at $(X,t)$ as the interior of the intersection of all boundary cones that originate in $S(X,t)$. Hence
\begin{equation}\label{newG}
\gamma_a(X,t)=\gamma_a(x,x_n,t)=\mbox{int}\bigcap \{\gamma_a(q,\tau):\, \|(x-q,t-\tau)\|<a\,x_n\}.
\end{equation}

\begin{figure}[htbp]
    \centering
\begin{tikzpicture}[scale=1.1]

    \colorlet{myred}{red!30}
    \colorlet{myblue}{cyan!15}

    \clip (-1.5, -1.5) rectangle (12.5, 10.5);

    \begin{scope}
        \clip (-10, -2) rectangle (20, 9.6); 
        
        \fill[myblue] (-10, 15)
            -- plot[domain=-10:3.5, samples=300] (\x, {2.8 * sqrt(3.5 - \x)})
            -- plot[domain=3.5:20, samples=300] (\x, {2.8 * sqrt(\x - 3.5)})
            -- (20, 15) -- cycle;

        \fill[myred] (-10, 15)
            -- plot[domain=-10:4.5, samples=300] (\x, {2.8 * sqrt(7 - \x)})   
            -- plot[domain=4.5:20, samples=300] (\x, {2.8 * sqrt(\x - 2)})    
            -- (20, 15) -- cycle;
    \end{scope}

    \draw[->, thin] (-1, 0) -- (12, 0) node[above left, inner sep=2pt] {$t$};
    \draw[->, thin] (0, -0.5) -- (0, 10.2) node[below right, inner sep=2pt] {$x_n$};

    \foreach \x in {1,2,3,4,5,6,7,8,9,10,11}
        \draw[very thin, gray] (\x, 0.05) -- (\x, -0.05);
    \foreach \y in {1,2,3,4,5,6,7,8,9,10}
        \draw[very thin, gray] (0.05, \y) -- (-0.05, \y);


    \draw[blue, thick] plot[domain=-1.5:3.5, samples=400] (\x, {2.8 * sqrt(3.5 - \x)});
    \draw[blue, thick] plot[domain=3.5:12.5, samples=400] (\x, {2.8 * sqrt(\x - 3.5)});

    \draw[very thick] plot[domain=-1.5:7, samples=500] (\x, {2.8 * sqrt(7 - \x)});
    
    \draw[very thick] plot[domain=2:12.5, samples=500] (\x, {2.8 * sqrt(\x - 2)});

    
    \draw[<->, >=latex, red, thick] (2, 0) -- (7, 0);
    \node[red, below, inner sep=5pt] at (4.7, 0) {$S(X,t)$};

    \fill[gray!80!black] (2, 0) circle (1.5pt);
    \fill[gray!80!black] (7, 0) circle (1.5pt);
    
    \fill[blue] (3.5, 0) circle (2pt);
    \node[blue, below=2pt] at (3.5, 0) {$(q,\tau)$};

    \fill[black] (4.5, {2.8*sqrt(2.5)}) circle (2.5pt);
    \node[right=3pt] at (4.5, {2.8*sqrt(2.5)}) {$(X,t)$};

    \node at (4.5, 5.4) {$\gamma_a(X,t)$};
    \node[blue] at (1.8, 4.4) {$\gamma_a(q,\tau)$}; 
\end{tikzpicture}
\caption{Novel cones in the $(t,x_n)$ plane.}
\end{figure}
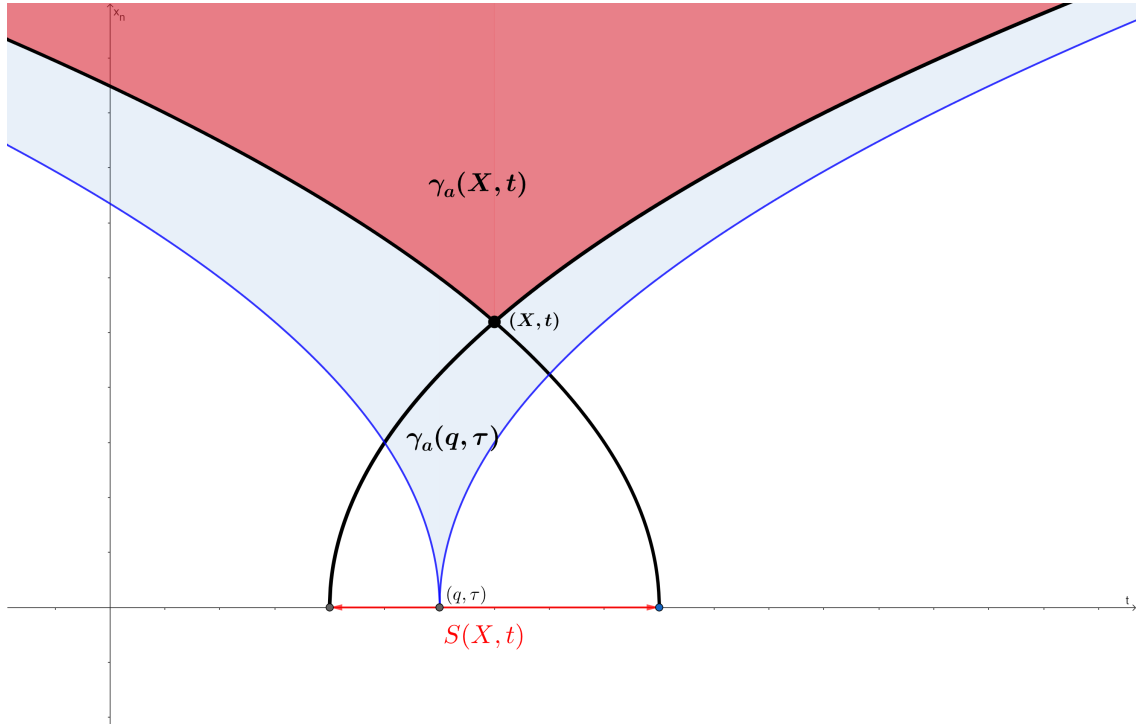


To explain this definition we compare \eqref{oldG} and \eqref{newG}. When the time variable is not present there is no difference between these definitions, and thus there is no reason to consider \eqref{newG}. This is due to the fact that in this case the boundary of the cone \eqref{oldG} is a union of straight lines of slope $a^{-1}$ originating from its vertex and intersections of such cones results in a cone of the same type. 

This is not the case with the time variable. In this variable the cone defined by \eqref{oldG} has a cusp of shape
given by the curve $a^{-1}\sqrt{|t|}$ at zero.
The definition  \eqref{newG} removes this cusp for interior cones at $x_n>0$; it is easy to see that the boundary cones 
$\gamma_a(q,\tau)$ that define $\gamma_a(X,t)$ have bounded slope of their boundary in $t$-variable at height $x_n>0$. This slope is bounded by $\frac{a^{-2}}{2x_n}$ (this is derived from slope of the function $a^{-1}\sqrt{|t|}$). 

Thus it follows that the boundary of the cone $\gamma_a(X,t)$ is given by the graph of a function that is Lipschitz 
with Lipschitz constant $a^{-1}$ in spatial variables and Lipschitz constant $\frac{a^{-2}}{2x_n}$ in the time variable.

Moreover, our new cones \eqref{newG} maintain a key property of the cones from \eqref{oldG}, namely, that
if $(Q,\tau)$ is any point (including a boundary point) and $(X,t)\in \gamma_a(Q,\tau)$ then
$$\overline{\gamma_a(X,t)}\subset \gamma_a(Q,\tau).$$

As in \cite{DHM} we now consider the map $\hbar:\partial\Omega\to\R$ given at each $(x,t)\in\partial\Omega$ defined by 
\begin{equation}\label{h}
\hbar_{\nu,a}(w)(x,t):=\inf\left\{x_n>0:\,\sup_{(Z,s)\in\gamma_{a}(x,x_n,t)}w(Z,s)<\nu\right\}
\end{equation}
At this point we observe that $\hbar_{\nu,a}w(x,t)<\infty$ for all points $(x,t)\in\partial\Omega$. 
This is due to the fact that the $L^2$ averages $w(Z,s)\to 0$ as $z_n\to\infty$, as observed earlier.

\begin{lemma}\label{S3:L5}
Let $w$ be as above. Then the following properties hold.
\vglue2mm

\noindent (i)
The function $\hbar_{\nu, a}(w)$ is Lipschitz in spatial variables with constant $a^{-1}$. It is also 
Lipschitz in time variable above a positive height and the following holds:
\begin{equation}\label{Eqqq-5}
\left|\hbar_{\nu,a}(w)(x,t)-\hbar_{\nu,a}(w)(y,s)\right|\leq a^{-1}|x-y|+\frac{a^{-2}}{2\min\{\hbar_{\nu,a}(w)(x,t),\hbar_{\nu,a}(w)(y,s)\}}|t-s|.
\end{equation}
for all $(x,t),(y,s)\in\partial\Omega$. In particular,
\begin{equation}\label{derh}
|\nabla \hbar_{\nu,a} (x,t)|\le a^{-1},\qquad |\partial_t \hbar_{\nu,a} (x,t)|\le \frac{a^{-2}}{2\hbar_{\nu,a} (x,t)}.
\end{equation}
\vglue2mm

\noindent (ii)
The function $\hbar_{\nu,a}(w)$ satisfies the following: there exists $C_n>0$ that depends only on $n$ such that
\begin{equation}\label{derh'}
    \abs{\hbar_{\nu,a}(w)(x,t)-\hbar_{\nu,a}(w)(y,s)}\le C_na^{-1}\br{\abs{x-y}+\abs{t-s}^{1/2}} 
\end{equation}
for all $(x,t),(y,s)\in\partial\Omega$.
\vglue2mm

\noindent (iii)
Given an arbitrary $(x,t)\in E_{\nu,a}$, where $E_{\nu,a}$ is defined in \eqref{E}, let $x_n:=\hbar_{\nu,a}(w)(x,t)$. Then there exists a 
point $(y,y_n,s)\in\partial\gamma_{a}(x,x_n,t)$ such that $w(y,y_n,s)=\nu$ and $\hbar_{\nu,a}(w)(y,s)=y_n$. 		
\end{lemma}

\begin{proof} (i) and (ii).
The way the function $\hbar_{\nu, a}(w)$ is defined implies the following observation.
Consider two points $(x,t)$, $(y,t)\in \partial\Omega\to\R$ and let $x_n=\hbar_{\nu,a}(w)(x,t)$ and $y_n=\hbar_{\nu,a}(w)(y,s)$ and let $(X,t)=(x,x_n,t)$, $(Y,s)=(y,y_n,s)$.

It is not possible to have $(Y,s)\in \gamma_a(X,t)$ or vice versa $(X,t)\in \gamma_a(Y,s)$. Indeed, if for example $(Y,s)\in \gamma_a(X,t)$ it would follow that 
$$\overline{\gamma_a(Y,s)}\in \gamma_a(X,t)$$
and hence $\sup_{(Z,s)\in\gamma_{a}(Y,s)}w(Z,s)<\nu$ which contradicts with the definition of 
$\hbar_{\nu,a}(w)(y,s)=y_n$. Therefore 
$(Y,s)\notin \gamma_a(X,t)$ and $(X,t)\notin \gamma_a(Y,s)$. This implies that if
we
let $\phi(x,t):\partial\Omega\to \R$ be the function whose graph is the boundary of
 $\partial\gamma_a(X,t)$ then
$$\hbar_{\nu,a}(w)(y,s)\le \phi(x,t)(y,s)\qquad\forall (x,t),\,(y,s)\in\partial\Omega,$$
and thus
\begin{equation}\label{eq.hbar_constr}
    \hbar_{\nu,a}(w)=\inf\{\phi(x,t):\, (x,t)\in\partial\Omega\}.
\end{equation}
Each function $\phi(x,t)$ being a boundary of a cone described above enjoys the Lipschitz bounds
by $a^{-1}$ in spatial and $\frac{a^{-2}}{2x_n}$ is the time variable. From this both claims (i) and (ii) follow.

We omit the proof of part (iii) as it is identical to that of given in \cite{DHM}.
\end{proof}

The next two statements appear in \cite{DHM} and apply in our settings with analogous proofs.

\begin{lemma}\label{l6} 
Assume that $w$ is defined as in \eqref{def.w} and $\Omega=\R^n_+\times\R$. 
Then for any $a>0$ there exists $0<b_0=b_0(a)$ and $\gamma_0=\gamma_0(a)>0$ such that the following holds. 
Having fixed an arbitrary $\nu>0$, $b\ge b_0$, and $\gamma\le\gamma_0$ for each point $(x,t)\in\partial\Omega$ from the set 
\begin{equation}\label{Eqqq-17}
\big\{(x,t):\,N_{a}(w)(x,t)>\nu\mbox{ and }S_{b}(\nabla u)(x,t)+ A_{b}(\nabla u)(x,t)\leq\gamma\nu\big\}
\end{equation}
there exists a boundary ball $R$ with $(x,t)\in 2R$ and such that
\begin{equation}\label{Eqqq-18}
\big|w\big(z,\hbar_{\nu,a}(w)(z,\tau),\tau\big)\big|>\nu/{2}\,\,\text{ for all }\,\,(z,\tau)\in R.
\end{equation}
\end{lemma}

A new feature here (distinct from  \cite{DHM}) is the presence of the area function \eqref{DefArea} that allows us to use Poincar\'e inequality in all variables (including $t$). (Since \cite{DHM} deals with the elliptic case, the $t$-variable was not present there.)

\begin{corollary}\label{S3:L6} 
Let $w$ be defined as in \eqref{def.w}, $\Omega=\R^n_+\times\R$ and fix $a>0$. 
Associated with these, let $b_0,\gamma_0$ be as in Lemma~\ref{l6}. Then there exists a finite 
constant $C=C(n)>0$ with the property that for any $\nu>0$, $b\ge b_0$, $\gamma\in(0,\gamma_0]$, and any point $(x,t)\in E_{\nu,a}$
such that $S_{b}(\nabla u)((x,t))+{ A_{b}(\nabla u)((x,t))}\leq\gamma\nu$, one has
\begin{equation}\label{Eqqq-23}
(M_{\hbar_{\nu,a}}w)\big(x,\hbar_{\nu,a}(x,t),t\big)\geq\,C\nu,
\end{equation}
where $M_{\hbar_{\nu,a}}$ is the Hardy-Littlewood maximal function on the graph of $\hbar_{\nu,a}$.
\end{corollary}

Finally, we are ready to state the key lemma that allows us to formulate a good-$\lambda$ inequality later.
From here we follow more closely the elliptic model case of \cite{DHP}. Notice, that up to this point we have not used any property of $\nabla u$ except a moderate decay $x_n\to\infty$ and the fact that $\nabla u\in L^2_{loc}(\Omega)$. This changes now as we will use the PDE system it satisfies.\vglue1mm

Consider the PDEs satisfied by each $w_m=\partial_m u$ for $m=1,2,
\dots, n-1$. Due to the block form nature of our operator $\mathcal L=-\partial_t+\mbox{div}(A\nabla \cdot)$ we have the following:
\begin{eqnarray}\label{system}
{\mathcal L}w_m&=&-\sum_{i,j=1}^{n-1}\partial_i((\partial_m a_{ij})w_j)\quad\mbox{in }\Omega,\quad m=1,2,\dots,n-1,\\
w_m\Big|_{\partial\Omega}&=&\partial_m f.\nonumber
\end{eqnarray}
Observe that only $w_1,\dots, w_{n-1}$ appears in these equations and hence \eqref{system} is a {\it weakly coupled} fully determined system of $n-1$ equations for the unknown vector valued function 
\[W:=(w_1,w_2\dots,w_{n-1})\]
with boundary datum $W\big|_{\partial\Omega}=\nabla_xf\in L^p$. We call this system {\it weakly coupled} because each $\partial_ma_{ij}$ appearing on the righthand side will be assumed to have a small Carleson measure norm, $\|\mu_\parallel\|_C$. 

Hence, let us write $w_m=v_m+\eta_m$ where each $v_m$ solves the parabolic Dirichlet problem 
$$\mathcal Lv_m=0\mbox{ in }\Omega,\qquad v_m\big|_{\partial\Omega}=\partial_mf\in L^q(\partial\Omega).$$
We denote 
\[
V:=(v_1,v_2,\dots,v_{n-1}).
\]
As $\mathcal L$ is a block form matrix we know this Dirichlet problem is solvable for all $1<q<\infty$ (c.f. Appendix \ref{APA}) and we have the following square and nontangential estimates:
\begin{equation}\label{sss1}
\|S(v_m)\|_{L^q}\approx\|\tilde{N}(v_m)\|_{L^q}\lesssim \|\partial_m f\|_{L^q},\qquad m=1,2,\dots,n-1.
\end{equation}
Thus each $\eta_m$ solves
\begin{eqnarray}\label{system2}
{\mathcal L}\eta_m&=&-\sum_{i,j=1}^{n-1}\partial_i((\partial_m a_{ij})(v_j+\eta_j))\quad\mbox{in }\Omega,\quad m=1,2,\dots,n-1,\\
\eta_m\Big|_{\partial\Omega}&=&0.\nonumber
\end{eqnarray}
Our aim is to establish  nontangential estimates for each $\eta_m$ as well, and thus also for $w_m$.

\begin{lemma}\label{S3:L8-alt1} 
Let $\Omega={\mathbb R}^n_+\times\R$ and let ${\mathcal L}$ be a block-form operator as above.
Suppose $\vec\eta=(\eta_1,\eta_2,\dots,\eta_{n-1})$ is a weak solution of \eqref{system2} in $\Omega$.  For a fixed (sufficiently large $a>0$ determined below), consider an arbitrary function $\hbar:{\mathbb R}^{n-1}\times\R\to \mathbb R$ such that it satisfies
the estimates \eqref{Eqqq-5}, \eqref{derh} and $\hbar\ge 0$. 
Then 
we have the following:\vglue1mm

For any $\varepsilon>0$ there exists $\delta>0$
such that if $\left[\|\mu_{||}\|_C+\|x_n\nabla_xA_{||}\|^2_{L^\infty}\right]\Big(1+\|x_n\nabla A\|^2_{L^\infty}\Big)<\delta$ (where $\mu_{||}$ defined by \eqref{def.mu11}) then
for all arbitrary parabolic surface balls $\Delta_r\subset{\mathbb R}^{n-1}\times\R$ of radius $r$ such that at least one point of $\Delta_r$
the inequality $\hbar(x,t)\le 2r$ holds we have the following estimate for all $m=1,2,\dots,n-1$ and an arbitrary $\vec{c}=(c_1,c_2,\dots,c_{n-1})\in\mathbb R^{n-1}$:

\begin{multline}\label{TTBBMM}
\sum_{m<n}\int_{1/6}^6\int_{\Delta_r}\big|\eta_m\big(x,\theta\hbar(x,t),t\big)-c_m\big|^2\,dx\,dt\,d\theta
\leq C(1+\varepsilon^{-1})\iint_{\cS(\Delta_r,\hbar)}\abs{\nabla\eta_m}^2x_n\, dx_ndtdx\\
+\varepsilon\br{\|\tilde{N}_a\br{(\vec\eta-\vec{c})\1_{\cS(\Delta_r,\hbar)}}\|_{L^2(\Delta_{2r})}^2
+\|\tilde{N}_a(\vec\eta\,\1_{\cS(\Delta_r,\hbar)})\|_{L^2(\Delta_{2r})}^2+\|\tilde{N}_a(V\1_{\cS(\Delta_r,\hbar)})\|_{L^2(\Delta_{2r})}^2}\\
+\frac{C}{r}\iint_{\mathcal{K_\varepsilon}}|\vec\eta-\vec{c}|^{2}\,dX\,dt,
\end{multline}
for some $C\in(0,\infty)$ that only depends on $a,\Lambda,n$ but not on $\vec\eta$, $\vec{c}$, $\varepsilon$ or $\Delta_r$. 
Here, \[\cS(\Delta_r,\hbar):=\set{(x,x_n,t):\, (x,t)\in\Delta_{2r} \text{ and } \frac{\hbar(x,t)}{6}<x_n<c_ar},\]
where $c_a$ is a constant depending only on $a$ (one can take $c_a=26(1+2a^{-1})$ for instance), and $\mathcal K_{\varepsilon}:=\cS(\Delta_r,\hbar)\cap\set{(x,x_n,t): x_n>\varepsilon r}$.

Also for $w_n=\partial_{x_n} u$ we have for an arbitrary $c\in\R$
\begin{multline}\label{TTBBMMx}
\int_{1/6}^6\int_{\Delta_r}\big|w_n\big(x,\theta\hbar(x,t),t\big)-c\big|^2\,dx\,dt\,d\theta
\leq C(1+\varepsilon^{-1})\iint_{\cS(\Delta_r,\hbar)}\abs{\nabla w_n}^2x_n\, dx_ndtdx\\
+C(\norm{\mu}_C+\|x_n|\nabla A_\parallel|\|_{L^\infty}^2)(1+\|x_n|\nabla A_\parallel|\|_{L^\infty}^2)\|\tilde N_a(W\1_{\cS(\Delta_r,\hbar)})\|_{L^2(\Delta_{2r})}^2\\+\varepsilon \|\tilde{N}_a\br{(w_n-c)\1_{\cS(\Delta_r,\hbar)}}\|_{L^2(\Delta_{2r})}^2
+\frac{C}{r}\iint_{\mathcal{K_\varepsilon}}|w_n-c|^{2}\,dX\,dt,
\end{multline}

The cones used to define the nontangential 
maximal functions in this lemma have vertices on $\partial\Omega$.

\end{lemma}

\begin{proof} Let $\Delta_r$ be as in the statement of our Lemma. and assume that $(q,\tau)$ is the center of $\Delta_r$. Let $\zeta$ be a smooth cutoff function of the form $\zeta(x,x_n,t)=\zeta_{0}(x_n)\zeta_{1}(x,t)$ where
\begin{equation}\label{Eqqq-27}
\zeta_{0}= 
\begin{cases}
1 & \text{ in } (-\infty, r_0+r], 
\\
0 & \text{ in } [r_0+2r, \infty),
\end{cases}
\qquad
\zeta_{1}= 
\begin{cases}
1 & \text{ in } \Delta_{r}(q,\tau), 
\\
0 & \text{ in } \mathbb{R}^{n}\setminus \Delta_{2r}(q,\tau)
\end{cases}
\end{equation}
and
\begin{equation}\label{Eqqq-28}
r|\partial_{x_n}\zeta_{0}|+r|\nabla_{x}\zeta_{1}|+r^2|\partial_t\zeta_{1}|\leq c
\end{equation}
for some constant $c\in(0,\infty)$ independent of $r$. Here 
$r_0=6\sup_{(x,t)\in \Delta_r}\hbar(x,t)$. Observe that our assumptions imply that
$$0\le r_0-\theta\hbar(x,t)\le  r_0\le 24(1+a^{-1}) r,\qquad \mbox{for all }(x,t)\in \Delta_{2r},$$
for $\theta\in (1/6,6)$.

We focus on \eqref{TTBBMM} first. 
Our goal is to control the $L^2$ norm of $\eta_m\big(\cdot,\theta\hbar(\cdot)\big)-c_m$.  We fix $m\in\{1,\dots,n-1\}$ and proceed to estimate
\begin{align}
&\hskip -0.20in
\int_{\Delta_{r}(q,\tau)}(\eta_m(x,\theta\hbar(x,t),t)-c_m)^2\,dx\,dt \le \mathcal I:=\int_{\Delta_{2r}(q,\tau)}(\eta_m(x,\theta\hbar(x,t),t)-c_m)^2\zeta(x,\theta\hbar(x,t),t)\,dx\,dt
\nonumber\\[4pt]
&\hskip 0.70in
=-\iint_{\mathcal S(q,\tau,r,r_0,\theta\hbar)}\partial_{x_n}\left[(\eta_m(x,x_n,t)-c_m)^2\zeta(x,x_n,t)\right]\,dx_n\,dx\,dt
\nonumber
\end{align}
where $\mathcal S(q,\tau,r,r_0,\theta\hbar)=\{(x,x_n,t):(x,t)\in \Delta_{2r}(q,\tau)\mbox{ and }\theta\hbar(x,t)<x_n<r_0+2r\}$. Hence:

\begin{align}\nonumber
&\hskip 0.10in
\mathcal I \le-2\iint_{\mathcal S(q,\tau,r,r_0,\theta\hbar)}(\eta_m-c_m)\partial_{x_n}(\eta_m-c_m)\zeta\, dx_n\,\,dt\,dx  
\\[4pt]
&\hskip 0.70in
\quad-\iint_{\mathcal S(q,\tau,r,r_0,\theta\hbar)}(\eta_m-c_m)^2(x,x_n,t)\partial_{x_n}\zeta\,dx_n\,dx\,dt
=:\mathcal{A}+V.\label{u6tg}
\end{align}
We further expand the term $\mathcal A$ as a sum of three terms obtained 
via integration by parts with respect to $x_n$ as follows:
\begin{align}\label{utAA}
\mathcal A &=-2\iint_{\mathcal S(q,\tau,r,r_0,\theta\hbar)}(\eta_m-c_m)\partial_{x_n} 
(\eta_m-c_m)(\partial_{x_n}x_n)\zeta\,d_n\,dt\,dx 
\nonumber\\[4pt]
&=2\iint_{\mathcal S(q,\tau,r,r_0,\theta\hbar)}\left|\partial_{x_n}\eta_m\right|^{2}x_n\zeta\,dx_n\,dt\,dx 
\nonumber\\[4pt]
&\quad +2\iint_{\mathcal S(q,\tau,r,r_0,\theta\hbar)}(\eta_m-c_m)\partial^2_{x_nx_n}(\eta_m-c_m)x_n\zeta\,dx_n\,dt\,dx 
\nonumber\\[4pt]
&\quad +2\iint_{\mathcal S(q,\tau,r,r_0,\theta\hbar)}(\eta_m-c_m)\partial_{x_n}\eta_m\,x_n\partial_{x_n}\zeta\,dx_n\,dt\,dx
\nonumber\\[4pt]
&\quad -2\int_{\mathcal \partial S(q,\tau,r,r_0,\theta\hbar)}(\eta_m-c_m)\partial_{x_n}\eta_m\,x_n\zeta\,\nu_n dS
\nonumber\\[4pt]
&=:I+II+III+IV.
\end{align}

Here the boundary term $IV$ is nonvanishing only on the graph of function $\theta\hbar$.

We start by analyzing the term $II$. As the $\eta_m$ solve the PDE \eqref{system2} then we have for $\eta_m-c_m$:
$${\mathcal L}(\eta_m-c_m)=-\sum_{i,j=1}^{n-1}\partial_i((\partial_m a_{ij})(v_j+\eta_j))$$
and thus using the block-form structure of our matrix:
\begin{equation}\label{S3:T8:E01-x}
\partial^2_{x_nx_n}(\eta_m-c_m)=-\sum_{i,j=1}^{n-1}\partial_i((\partial_m a_{ij})(v_j+\eta_j))-\sum_{i,j=1}^{n-1}\partial_i(a_{ij}\partial_j(\eta_m-c_m))+\partial_t(\eta_m-c_m).
\end{equation}
In turn, this permits us to write the term $II$ as
\begin{align}
II &=-2\sum_{i,j<n}\iint_{\mathcal S(q,\tau,r,r_0,\theta\hbar)}(\eta_m-c_m)\partial_{i}\left({a}_{ij}\partial_{j}\eta_m\right)x_n\zeta\,dx_n\,dt\,dx
\nonumber\\[4pt]
&\quad-2\sum_{i,j<n}\iint_{\mathcal S(q,\tau,r,r_0,\theta\hbar)}(\eta_m-c_m)\partial_{i}\left((\partial_m a_{ij})(v_j+\eta_j)\right)x_n\zeta\,dx_n\,dt\,dx \nonumber\\
&\quad+2\iint_{\mathcal S(q,\tau,r,r_0,\theta\hbar)}(\eta_m-c_m)\partial_{t}(\eta_m-c_m)x_n\zeta\,dx_n\,dt\,dx=\mathcal B+\mathcal C+II_8.
\end{align}
Here we leave analysis of the last term $II_8$ for later and continue with the first two terms which we  both integrate by parts w.r.t. $\partial_i$. This then yields
\begin{align}
\mathcal B+\mathcal C&=2\sum_{i,j<n}\iint_{\mathcal S(q,\tau,r,r_0,\theta\hbar)}a_{ij}\partial_i\eta_m\partial_{j}\eta_m\,x_n\zeta\,dx_n\,dt\,dx
\nonumber\\[4pt]
&+2\sum_{i,j<n}\iint_{\mathcal S(q,\tau,r,r_0,\theta\hbar)}a_{ij}(\eta_m-c_m)\partial_{j}(\eta_m)\,x_n\partial_i\zeta)\,dx_n\,dt\,dx
\nonumber\\[4pt]
&+2\sum_{i,j<n}\iint_{\mathcal S(q,\tau,r,r_0,\theta\hbar)}(\partial_m a_{ij})(\partial_i\eta_m)(v_j+\eta_j)x_n\zeta\,dx_n\,dt\,dx 
\nonumber\\[4pt]
&+2\sum_{i,j<n}\iint_{\mathcal S(q,r,r_0,\theta\hbar)}(\partial_m a_{ij})(\eta_m-c_m)(v_j+\eta_j)x_n(\partial_i \zeta)\,dx_n\,dt\,dx 
\nonumber\\[4pt]
&\quad-2\sum_{i>0}\int_{\partial\mathcal S(q,\tau,r,r_0,\theta\hbar)}(\mbox{boundary terms})x_n\zeta\nu_i\,dS
\nonumber\\[4pt]
&=:II_{1}+II_{2}+II_{3}+II_{4}+II_{5}.\label{TFWW}
\end{align}
The boundary integral (term $II_5$) vanishes everywhere except on the graph of the function $\theta\hbar$. Note that on the graph of $\theta\hbar$,
\[
\nu_i=\frac{-\theta \dr_i \hbar}{J_h}dxdt, \quad dS=J_h\, dxdt \quad\text{for }i=1,\dots,n-1,
\]
where $J_h:=\br{1+(\theta\nabla_x\hbar)^2+(\theta\dr_t\hbar)}^{1/2}$, and so 
$\abs{\nu_i}dS\le\theta\abs{\dr_i\hbar}dxdt\le \theta\, a^{-1}dxdt$. Hence,
\begin{align*}
&|II_5|\nonumber\\
&\le C\sum_{i,j<n}\int_{\Delta_{2r}(q,\tau)}|(\eta_m-c_m)(x,\theta\hbar(x,t),t))\nabla_x(\eta_m)(x,\theta\hbar(x,t),t)\hbar(x,t)\zeta(x,\theta\hbar(x),t) |dt\,dx\nonumber\\
&\,+ C\sum_{i,j<n}\int_{\Delta_{2r}(q,\tau)}|\partial_m a_{ij}||(\eta_m-c_m)(x,\theta\hbar(x,t),t))(\eta_j+v_j)(x,\theta\hbar(x,t),t)\hbar(x,t)\zeta(x,\theta\hbar(x,t),t) |dt\,dx\nonumber\\
&\le \frac16\int_{\Delta_{2r}(q,\tau)}(\eta_m(x,\theta\hbar(x,t),t)-c_m)^2\zeta(x,\theta\hbar(x,t),t)\,dt\,dx\nonumber\\&\quad+C'
\int_{\Delta_{2r}(q,\tau)}|\nabla_x \eta_m(x,\theta\hbar(x,t),t)|^2|\hbar(x,t)|^2\,dt\,dx\\\nonumber
&\quad +\|x_n\nabla_xA_{||}\|_{L^\infty}^2\int_{\Delta_{2r}(q,\tau)}(|\vec\eta+V|^2)(x,\theta\hbar(x,t),t)\,dt\,dx
=\frac13\mathcal I+II_6+II_7.
\end{align*}

Here we have used the Cauchy-Schwarz for the first two terms.
 We can hide the term $\frac16\mathcal I$ on the left-hand side of \eqref{u6tg}, while the second term after integrating $II_6$ in $\theta$ becomes:
\begin{align}
\int_{1/6}^6|II_6|\,d\theta&\le C\int_{1/6}^6\int_{\Delta_{2r}(q,\tau)}|\nabla \eta_m(x,\theta\hbar(x,t),t)|^2 |\hbar(x,t)|^2dt\,dxd\theta.\nonumber\\
&\le C\iint_{\cS(q,\tau,r,r_0,\frac16\hbar)}|\nabla \eta_m(x,x_n,t)|^2 x_n\,dt\,dxdx_n.\label{eq10.31}
\end{align}
The term $II_7$ can be estimated using the nontangential maximal function: for $(x,t)\in\Delta_{2r}$, we have $\abs{\vec\eta+\vec{V}}(x,\theta\hbar(x,t),t)\le N((\vec\eta+\vec{V})\,\1_{\cSS})(x,t)$, and thus, using Lemma \ref{spCarl}
\begin{multline*}
    |II_7|
    \le
\|x_n\nabla_xA_\parallel\|_{L^\infty}^2\Big(1+\|x_n\nabla A_\parallel\|^2_{L^\infty}\Big)\\
\cdot\left(\|\tilde{N}_a(\vec\eta\,\1_{\cSS})\|_{L^2(\Delta_{2r})}^2+\|\tilde{N}_a(V\1_{\cSS})\|_{L^2(\Delta_{2r})}^2\right).
\end{multline*}

It follows that this term will appear in \eqref{TTBBMM} with an $\varepsilon$, if we have $\|x_n\nabla_xA_{||}\|_{L^\infty}^2\Big(1+\|x_n\nabla A\|^2_{L^\infty}\Big)$ sufficiently small.
Observe that the boundary term $IV$ enjoys a similar estimate. Then $\nu_n dS=-\, dt\, dx$ on the graph of $\theta\hbar$ and by Cauchy-Schwarz we again obtain the term $\frac16\mathcal I$ that can be absorbed by the lefthand side
and a term analogous to $II_6$ but with $|\partial_{x_n}\eta_m|^2$ instead of $|\nabla_x\eta_m|^2$ but this does not change
anything in the estimate \eqref{eq10.31} and we obtain an identical bound for it.

Some of the remaining (solid integral) terms that are of the same type we estimate together. Firstly, we have 
\begin{equation}\label{Eqqq-29}
|I+II_1|\le C\iint_{\cSS}\abs{\nabla\eta_m}^2x_n\, dx_ndtdx.
\end{equation}
%
%
Next, since $r|\nabla\zeta|\le c$, if the derivative falls on the cutoff function $\zeta$ we have
\begin{align*}
|II_2+III| &\lesssim \iint_{\cSS}\left|\nabla \eta_m\right||\vec\eta-\vec{c}|\frac{x_n}{r}\,dx_n\,dt\,dx
\nonumber\\[4pt]
&\le \left(\iint_{\cSS}|\vec\eta-\vec{c}|^{2}\frac{x_n}{r^{2}}\,dx_n\,dt\,dx\right)^{1/2} 
\br{\iint_{\cSS}\abs{\nabla\eta_m}^2x_n\,dxdtdx_n}^{1/2}.
\end{align*}
By Fubini's theorem, and the fact that 
$\iint_{\gamma_a(x,t)\cap \cSS}y_n^{-n-1}dydsdy_n\le Cr$, the first integral on the right hand side can be controlled it by the nontangential maximal function:
\begin{multline}\label{eq.Neta-c}
    \iint_{\cSS}|\vec\eta-\vec{c}|^{2}\frac{x_n}{r^{2}}\,dx_n\,dt\,dx
\le \frac Cr\int_{\Delta_{2r}}\iint_{\gamma_a(x,t)\cap\cSS}\abs{\vec\eta-\vec c}^2y_n^{-n-1}dydsdy_n\\
\le C\int_{\Delta_{2r}}\tilde N_a\br{(\vec\eta-\vec c)\1_{\cSS}}^2dxdt.
\end{multline}
Hence,
\[\abs{II_2+III}\le C\norm{\tilde N_a\br{(\vec\eta-\vec c)\1_{\cSS}}}_{L^2(\Delta_{2r})}\br{\iint_{\cSS}\abs{\nabla\eta_m}^2x_n\,dxdtdx_n}^{1/2}.\]

The Carleson condition for $|\nabla A_\parallel|^2x_n$ and the Cauchy-Schwarz inequality imply
\begin{multline*}
   |II_3| \le 2\br{\iint_{\cSS}\abs{\dr_mA}\abs{v_j+\eta_j}^2x_n\,dx_ndxdt}^{1/2}\br{\iint_{\cSS}\abs{\nabla\eta_m}^2x_n\,dx_ndxdt}^{1/2}\\
   \le C \|\mu_{||}\|_C^{1/2}(1+\|x_n|\nabla A|\|_{L^\infty})\br{(\|\tilde {N}_a(\vec\eta\,\1_{\cSS})\|_{L^2(\Delta_{2r})}^2+\|\tilde N_a(V\1_{\cSS})\|^2_{L^2(\Delta_{2r})})^{1/2}}\\
   \cdot\br{\iint_{\cSS}\abs{\nabla\eta_m}^2x_n\,dx_ndxdt}^{1/2}.
\end{multline*}

For the term $II_4$ we use  $r|\nabla\zeta|\le c$.
It follows that  
\begin{equation}\nonumber
|II_4| \lesssim  \|x_n|\nabla_x A_\parallel\|_{L^\infty} \iint_{\cSS}|\vec\eta-\vec{c}||V+\vec\eta|\frac{x_n}{r^{2}}\,dx_n\,dt\,dx.
\end{equation}
An application of Cauchy-Schwarz inequality and \eqref{eq.Neta-c}  implies that
\begin{multline*}
    |II_4| \le C \|x_n|\nabla_x A_\parallel |\|_{L^\infty}
\norm{\tilde N_a\br{(\vec\eta-\vec c)\1_{\cSS}}}_{\Delta_{2r}}\\
\cdot
\br{\norm{\tilde N_a\br{\vec\eta\1_{\cSS}}}_{\Delta_{2r}}+\norm{\tilde N_a\br{V\1_{\cSS}}}_{\Delta_{2r}}}.
\end{multline*}
Again the constant in front of these terms can be made less than $\varepsilon$.

In a similar spirit we analyze the term $II_8$. We write it as
\begin{eqnarray}
II_8&=&
-\iint_{\mathcal S(q,\tau,r,r_0,\theta\hbar)}\partial_{t}[(\eta_m-c_m)^2]x_n\zeta\,dx_n\,dt\,dx\\\nonumber
&=&\iint_{\mathcal S(q,\tau,r,r_0,\theta\hbar)}(\eta_m-c_m)^2x_n(\partial_t\zeta)\,dx_n\,dt\,dx\\\nonumber
&+&\int_{\partial\mathcal S(q,\tau,r,r_0,\theta\hbar)}(\mbox{boundary terms})x_n\zeta\nu_t\,dS=:II_9+II_{10},
\end{eqnarray}
where the boundary terms only arise on the portion of $\partial\mathcal S(q,\tau,r,r_0,\theta\hbar)$ where
$\zeta$ is supported (which means on the graph of the function $\theta\hbar$). $\nu_t$ denotes the co-normal 
in the $t$-direction. 
Note that on the graph of $\theta\hbar$, $\nu_t=-\frac{\dr_t\hbar}{J_h}$ and $dS=J_hdxdt$, where $J_h$ is defined above the computation for $II_{5}$. By \eqref{derh}, we have that $\abs{\nu_t}dS\le \frac{a^{-2}}{2x_n}dxdt$.
It follows that 
$$|II_{10}|\lesssim \frac{1}{2a^2}\int_{\Delta(q,\tau)}|(\eta_m-c_m)^2(x,\theta\hbar(x,t),t)|\zeta\,dt\,dx\le \frac13\mathcal I,
$$
where the last estimate holds if $a$  is chosen sufficiently large. We do so now. The solid integral term $II_9$ can be bounded as follows. Since $|\partial_t\zeta|\lesssim r^{-2}$ we split the area of integration into  two parts, the first one for $x_n/r\le\varepsilon$ and its complement.
On the set where $x_n/r\le\varepsilon$  we get that the integral is bounded by
$$ \varepsilon\iint_{\cSS}|\vec\eta-\vec{c}|^{2}\frac{1}{r}\,dx_n\,dt\,dx\le\varepsilon \|\tilde{N}_a\br{(\vec\eta-\vec{c})\1_{\cSS}}\|_{L^2(\Delta_{2r})}^2
$$
using \eqref{eq.Neta-c}.
The remaining part of the integral is bounded by $\frac{C}{r}\displaystyle\iint_{\mathcal{K_\varepsilon}}|\vec\eta-\vec{c}|^2\,dX\,dt$.

Another term with a similar estimate is $V$; the fact that $\partial_{x_n}\zeta$ vanishes on the set
$(-\infty,r_0+r)\cup(r_0+2r,\infty)$ means that it is supported away from the boundary of $\partial\Omega$ and therefore,
\begin{equation}\label{Eqqq-31}
|V|\le \frac{C}{r}\iint_{\cSS\cap\set{x_n\ge r_0+r}}|\vec\eta-\vec{c}|^{2}\,dx_n\,dt\,dx
\le \frac{C}{r}\iint_{\mathcal K_{\varepsilon}}|\vec\eta-\vec{c}|^{2}\,dx_n\,dt\,dx.
\end{equation}
We put together all terms and integrate in $\theta$. Observe that $\cSS\subset \cS(\Delta_r,\hbar)$, and the above analysis ultimately yields \eqref{TTBBMM}.
Thus the claim follows.\vglue1mm

We only point out the differences in our calculation for \eqref{TTBBMMx}. The function $w_n-c$ for an arbitrary $c\in\mathbb R$ solves the PDE
\begin{equation}\label{system33}
{\mathcal L}(w_n-c)=\sum_{i,j=1}^{n-1}\partial_i((\partial_{n} a_{ij})w_j).
\end{equation}
It follows that the calculation given above can be followed step by step. The only difference is that we cannot claim any smallness of certain numbers appearing in these inequalities since the Carleson measure we are estimating is that of $|\partial_n A_\parallel|^2x_n$, which is not assumed to be small.  This is not
problematic since the terms that it modifies involve $\tilde{N}(W)$ for which we obtain the desired estimates using
\eqref{TTBBMM}.
\end{proof}

\begin{lemma}\label{LGL} Let $\mathcal L$ be an operator as in Lemma \ref{S3:L8-alt1} and fix $a>0$ sufficiently large as in Lemma \ref{S3:L8-alt1}.  
There exists $\kappa\in(0,1)$ that depends only on $n$, constants $b>a$ and $\gamma_0$ that depend on $a$, and $C>0$ that depends only on $n$ and the ellipticity constants such that the following holds.

For any $\varepsilon\in(0,1)$, if $\left[\|\mu_{||}\|_C+\norm{x_n\abs{\nabla_xA_\parallel}}^2_{L^\infty}\right]\Big(1+\|x_n|\nabla A_\parallel|\|^2_{L^\infty}\Big)$ is sufficiently small (depending on $\varepsilon$), then for any $\vec\eta$ that satisfies \eqref{system2}, for any $\gamma\in(0,\gamma_0)$, $\beta>0$, there holds
\begin{equation}    
  \abs{E_{1,\beta}}\le C\gamma^2\abs{\set{(x,t)\in\R^{n-1}\times\R: M(\tilde{N}_a(\vec\eta))(x,t)>\kappa\beta}}.\label{eq.gdLmd}
\end{equation}
where
\begin{multline}\label{def.E1}
    E_{1,\beta}:=\\
\set{(x,t)\in {\mathbb R}^{n-1}\times\R:\, \tilde{N}_a(\vec\eta)>\beta, \varepsilon^{-1/2}[S_b(\vec\eta)+{\varepsilon^{-1}A_b(\vec\eta)}]\le\gamma\beta, 
    \varepsilon^{1/2}\br{\tilde{N}_b(\vec\eta)+\tilde{N}_b(V)}\le \gamma\beta}.
\end{multline}
\end{lemma}

\begin{proof}
Let $b_0$ and $\gamma_0$ be as in Corollary \ref{S3:L6}.
Let $\varepsilon>0$, $\gamma\in(0,\gamma_0)$, $\beta>0$ be fixed. 
Observe that $E_{2,\beta}:=\set{(x,t)\in\R^{n-1}\times\R: M(\tilde{N}_a(\vec\eta))(x,t)>\kappa\beta}$ is an open subset of $\R^{n-1}\times\R$. We can assume that $E_{2,\beta}$ is a nonempty proper subset of $\R^{n-1}\times\R$, as otherwise the estimate \eqref{eq.gdLmd} is trivial. 
By Vitali's covering lemma, we can find a non-overlapping collection of parabolic balls $\set{\Delta(x_i,t_i)}_{i\in \N}$ 
such that $2\Delta(x_i,t_i)\subset E_{2,\beta}$, $10\Delta(x_i,t_i)\cap E_{2,\beta}^c\neq\emptyset$ for each $i$, and  that $E_{2,\beta}\subset\bigcup_{i\in \N}5\Delta(x_i,t_i)$. Denote $\Delta_i:=5\Delta(x_i,t_i)$, and denote by $r_i$ the radius of $\Delta_i$.
Set $F_i:=E_{1,\beta}\cap\Delta_i$. Note that $E_{1,\beta}\subset\bigcup_{i\in\N}F_i$ as $E_{1,\beta}\subset E_{2,\beta}$. 

Let us denote 
\[
 \vec\eta_{L^2}(x,x_n,t):=  \left(\fiint_{B((x,x_n,t),x_n/2)}\abs{\vec\eta(Y,s)}^2dYds\right)^{1/2},
\]
where $B((x,x_n,t),x_n/2)$ is a parabolic Whitney ball centered at $(x,x_n,t)$ in $\Rn_+\times\R$. Take $b\ge b_0$ and let $\hbar:=\hbar_{\beta,a}(\vec\eta_{L^2})$ be as in \eqref{h}. By Corollary \ref{S3:L6}, $(M_{\hbar}(\vec\eta_{L^2}))\big(x,\hbar(x,t),t\big)\geq\,C\beta$ for $(x,t)\in F_i$. We claim that this is also true when $M_{\hbar}(\vec\eta_{L^2})$ is replaced  with the localized maximal function $M_{\hbar, 6ar_i}(\vec\eta_{L^2})$ at scale $\sim r_i$. In fact, since $2\Delta_i\cap E_{2,\beta}^c\neq\emptyset$, we have
\begin{equation}\label{eq.MNsmall}
    M(\tilde N_a(\vec\eta))(y,s)\le\kappa\beta \quad\text{for some }(y,s)\in 2\Delta_i\cap E_{2,\beta}^c.
\end{equation}
From this one deduces that 
\begin{equation}\label{eq.etasmall}
  \vec\eta_{L^2}(x,x_n,t)
  \le 2^{n+1}\kappa\beta \quad\text{for  }(x,t)\in\Delta_i \text{ and }x_n>2r_i.
\end{equation}
In fact, suppose that \eqref{eq.etasmall} is false for some $(Z,\tau)=(z,z_n,\tau)$ with $(z,\tau)\in\Delta_i$ and $z_n>2r_i$. Then $\tilde{N}_a(\vec\eta)(x',t')>2^{n+1}\kappa\beta$ for all $(x',t')\in S(Z,\tau)$, where $S(Z,\tau)$ is defined as in \eqref{def.SXt}, which is a parabolic ball centered at $(z,\tau)$ of radius at least $2ar_i$. Therefore, since $a>1$,
\[
\kappa\beta<\br{\frac{2a}{3+2a}}^{n+1}2^{n+1}\kappa\beta<\fint_{\Delta_{(3+2a)r_i}(y,s)}\tilde{N}_a(\vec\eta)\,dx'dt'\le M(\tilde N_a(\vec\eta))(y,s),
\]
a contradiction to \eqref{eq.MNsmall}.

We choose $\kappa\le 2^{-n-1}$. In view of \eqref{eq.etasmall}, for $(x,t)\in F_i$, $\tilde N_a(\vec\eta)(x,t)=\tilde N_a^{2r_i}(\vec\eta)(x,t)$. It also guarantees that $\hbar\le 2r_i$ on $\Delta_i$. Moreover, it implies that the ``radius'' of the boundary ball $R$ constructed in Lemma \ref{l6} is bounded by $6a r_i$. Therefore, we have justified that 
\begin{equation}\label{eq.Mhrbig}
    \text{ for } (x,t)\in F_i,\ (M_{\hbar,6ar_i}(\vec\eta_{L^2}))\big(x,\hbar(x,t),t\big)\geq\,C\beta.
\end{equation}
We define 
\begin{equation}\label{def.hi}
    \hbar_i(x,t):=\sup\set{\hbar(x,t),\frac12\dist_p\br{(x,t),F_i}},
\end{equation}
and $\cS_i:=\cS(\Delta_{20r_i},\hbar_i)$, where $\cS(\Delta_{20r_i},\hbar_i)$ is defined as in Lemma \ref{S3:L8-alt1}. One can easily check that $\hbar_i$ satisfies  \eqref{Eqqq-5}, \eqref{derh}, and \eqref{derh'}, and enjoys in addition  the property 
\begin{equation}\label{SiProp}
    \text{for any }(x,t)\in40a\,\Delta_i\setminus F_i,\ \exists\, (y,s)\in F_i\text{ such that }
    \quad\cS_i\cap\gamma_a(x,t)\subset \gamma_b(y,s)
\end{equation}
for $b\ge b_0$ sufficiently large.
Property \eqref{SiProp} holds because we ensure that the graph of $h_i$ is sufficiently far away from the boundary even when it is not above $F_i$. 

We choose the vector $\vec c_i:=\fiint_{B_i^0}\vec\eta$, where $B_i^0$ is a ball with radius $\sim r_i$ that is contained in $\cS_i\cap\set{(x,x_n,t): (x,t)\in\Delta_i, x_n>2r_i}$.  By \eqref{eq.etasmall}, $\abs{\vec c_i}\le|\vec\eta_{L^2}|\le 2^{n+1}\kappa\beta$. Then by \eqref{eq.Mhrbig}, we have 
\begin{multline*}
    C^2\beta^2\abs{F_i}\le \int_{F_i}(M_{\hbar,6ar_i}(\vec\eta_{L^2}))\big(x,\hbar(x,t),t\big)^2dxdt\\
    \le \int_{F_i}(M_{\hbar,6ar_i}(\vec\eta-\vec c_i)_{L^2})\big(x,\hbar(x,t),t\big)^2dxdt+2^{2n+2}\kappa^2\beta^2\abs{F_i}.
\end{multline*}
We take $\kappa\le 2^{-n-1}$ small enough so that $2^{2n+2}\kappa^2<C^2/2$, and so we can hide the last term to the left-hand side. Then 
\begin{multline*}
    \frac{C^2}{2}\abs{F_i}\le\frac{1}{\beta^2}\int_{F_i}(M_{\hbar,6ar_i}(\vec\eta-\vec c_i)_{L^2})\big(x,\hbar(x,t),t\big)^2dxdt\\
    \le \frac{1}{\beta^2}\int_{\Delta_i}(M_{\hbar_i,6ar_i}(\vec\eta-\vec c_i)_{L^2})\big(x,\hbar_i(x,t),t\big)^2dxdt.
\end{multline*}
Notice that we changed $\hbar$ to $\hbar_i$ in the last inequality; this is valid because $\hbar_i=\hbar$ on $F_i$. By the $L^2$ boundedness of the localized maximal function, we obtain that
\begin{equation}\label{eq.eta-c}
     C'\abs{F_i}\le\frac{1}{\beta^2}\int_{10\Delta_i}(\vec\eta-\vec c_i)_{L^2}(x,\hbar_i(x,t),t)^2dxdt.
\end{equation}
We claim that 
\begin{equation}\label{eq.eta-c'}
    \int_{10\Delta_i}(\vec\eta-\vec c_i)^2_{L^2}(x,\hbar_i(x,t),t)dxdt\le C\int_{1/6}^6\int_{20\Delta_i}\abs{\vec\eta-\vec c_i}^2(x,\theta\hbar_i(x,t),t)dxdtd\theta. 
\end{equation}
This has been established (in the elliptic case) in \cite[Lemma 5.6]{DHM}. The proof in the parabolic case is analogous, as it only uses properties of the function $\hbar_i$ which are as stated in Lemma \ref{S3:L5}. 
These imply that around a boundary point $(x,t)$ there is a boundary parabolic ball of radius $\sim \hbar_i(x,t)$ such that
$\hbar_i(y,s)\sim\hbar_i(x,t)$ for all $(y,s)$ that belong to this boundary parabolic ball.
Hence it follows that by Lemma \ref{S3:L8-alt1} we get that
\begin{multline}\label{eq.Fiest}
    C'\abs{F_i}\le \frac{C}{\beta^2}(1+\varepsilon^{-1})\iint_{\cS_i}\abs{\nabla\vec\eta}^2x_n\, dx_ndtdx
    +\frac{C}{\beta^2r}\iint_{\mathcal{K_\varepsilon}}|\vec\eta-\vec{c_i}|^{2}\,dX\,dt\\
+\frac{\varepsilon}{\beta^2}\br{\|\tilde{N}_a\br{(\vec\eta-\vec{c_i})\1_{\cS_i}}\|_{L^2(40\Delta_{i})}^2
+\|\tilde N_a(\vec\eta\,\1_{\cS_i})\|_{L^2(40a\Delta_{i})}^2+\|\tilde N_a(V\1_{\cS_i})\|_{L^2(40\Delta_{i})}^2}.
\end{multline}
Our goal is to bound the right-hand side of \eqref{eq.Fiest} by $C\gamma^2\beta^2\abs{\Delta_i}$, which will imply the desired estimate \eqref{eq.gdLmd} by summing over $i$.
To this end, we write 
\[
\iint_{\cS_i}\abs{\nabla\vec\eta}^2x_n\, dx_ndtdx
\le C\int_{40\Delta_i}\iint_{\gamma_a(x,t)}\abs{\nabla\vec\eta}^2\1_{\cS_i}(y,y_n,s)\frac{dydy_nds}{y_n^n}dxdt.
\]    
When $(x,t)\in F_i$, it holds trivially that $\iint_{\gamma_a(x,t)}\abs{\nabla\vec\eta}^2\1_{\cS_i}(y,y_n,s)\frac{dydy_nds}{y_n^n}\le S_b(\vec\eta)(x,t)^2$. 
When $(x,t)\in 40\Delta_i\setminus F_i$, by \eqref{SiProp}, we can find some $(x_0,t_0)\in F_i$ so that 
$\iint_{\gamma_a(x,t)}\abs{\nabla\vec\eta}^2\1_{\cS_i}(y,y_n,s)\frac{dydy_nds}{y_n^n}\le S_b(\vec\eta)(x_0,t_0)^2$. Hence, by the definition of $E_{1,\beta}$, we conclude that 
\begin{equation}\label{trunceta}
\iint_{\gamma_a(x,t)}\abs{\nabla\vec\eta}^2\1_{\cS_i}(y,y_n,s)\frac{dydy_nds}{y_n^n}\le \varepsilon\gamma^2\beta^2 \quad\text{for all }(x,t)\in 40\Delta_i,
\end{equation}
and so 
\begin{equation}\label{eq.sqrSi}
    \iint_{\cS_i}\abs{\nabla\vec\eta}^2x_n\,dx_ndtdx\le C\varepsilon\gamma^2\beta^2\abs{\Delta_i}.
\end{equation}
The terms  $\|\tilde{N}_a(\vec\eta\,\1_{\cS_i})\|_{L^2(40\Delta_{i})}^2$ and $\|\tilde N_a(V\1_{\cS_i})\|_{L^2(40\Delta_{i})}^2$ are controlled similarly. We note that for any $(x_0,t_0)\in 40\Delta_i$, there exists $(x_0,t_0)\in F_i$ such that $\tilde N_a(\vec\eta)\1_{\cS_i}(x,t)\le \tilde N_b(\vec\eta)(x_0,t_0)$ due to \eqref{SiProp}. Therefore, by the definition of $F_i$ we get that
\begin{equation}\label{eq.NSi}
\varepsilon\|\tilde N_a(\vec\eta\,\1_{\cS_i})\|_{L^2(40\Delta_{i})}^2+\varepsilon\|\tilde N_a(V\1_{\cS_i})\|_{L^2(40\Delta_{i})}^2\le C\gamma^2\beta^2\abs{\Delta_i}.
\end{equation}
For the term $\frac{1}{r_i}\iint_{\mathcal{K_\varepsilon}}|\vec\eta-\vec{c_i}|^{2}\,dX\,dt$, we use Poincar\'e inequality in variables $x$ and $t$ simultaneously to get that
\[
\frac{1}{r_i}\iint_{\mathcal{K_\varepsilon}}|\vec\eta-\vec{c_i}|^{2}\,dX\,dt\le Cr_i\iint_{\mathcal{K_\varepsilon}}\br{\abs{\nabla\vec\eta}^2+r_i^2\abs{\partial_t\vec\eta}^2}dX\,dt\]
\[
\le C\left(\varepsilon^{-1}\iint_{\cS_i}\abs{\nabla\vec\eta}^2x_n+\varepsilon^{-3}\iint_{\cS_i}\abs{\partial_t\vec\eta}^2x_n^3\right)\,dx_ndtdx,
\]
where in the last step we recalled that $x_n>\varepsilon r_i$ in $\mathcal{K}_{\varepsilon}$. So this term is again bounded by a bound similar to \eqref{eq.sqrSi}. 

Finally,  the term $\varepsilon\|\tilde {N}_a\br{(\vec\eta-\vec{c_i})\1_{\cS_i}}\|_{L^2(40\Delta_{i})}^2$ is bounded by the sum of $\varepsilon\|\tilde N_a(\vec\eta\,\1_{\cS_i})\|_{L^2(40\Delta_{i})}^2$, which is treated in \eqref{eq.NSi}, and $C\varepsilon\abs{\vec c_i}^2\abs{\Delta_i}$,  which, since $\abs{\vec c_i}\le|\vec\eta_{L^2}|$, has the bound given in \eqref{eq.etasmall}.
\end{proof}

The following good-$\lambda$ estimate for $w_n$ can be established using \eqref{TTBBMMx} in a similar way.

\begin{lemma}\label{lem.Gdl_wn} Let $\mathcal L$ and $a>0$ as in Lemma \ref{LGL}. 
There exists $\kappa\in(0,1)$ that depends only on $n$, constants $b>a$ and $\gamma_0$ that depend on $a$, and $C>0$ that depends only on $n$ and the ellipticity constants such that the following holds.

For any $\varepsilon\in(0,1)$, for any $w_n$ that satisfies \eqref{system33}, for any $\gamma\in(0,\gamma_0)$, $\beta>0$, there holds
\begin{multline*}
   \abs{\set{\tilde{N}_a(w_n)>\beta, \varepsilon^{-1/2}[S_b(w_n)+\varepsilon^{-1}A_b(w_n)]\le\gamma\beta, 
    \varepsilon^{1/2}\tilde{N}_b(w_n)+M_0\,\tilde{N}_b(W)\le \gamma\beta}}\\
    \le  C\gamma^2\abs{\set{(x,t)\in\R^{n-1}\times\R: M(\tilde{N}_a(w_n))(x,t)>\kappa\beta}},
\end{multline*}
where $M_0=(\norm{\mu}_C+\|x_n|\nabla A_\parallel|\|_{L^\infty}^2)^{1/2}(1+\|x_n|\nabla A_\parallel|\|_{L^\infty}^2)^{1/2}$.
\end{lemma}

We can now establish Theorem \ref{thm.NlessS}. Recall that we have assumed a priori  that $\|\tilde{N}(\nabla u)\|_{L^p(\partial\Omega)}<\infty$. Since $\nabla_xu =V+\vec\eta$ and $V$ already enjoys the bound \eqref{sss1},  it follows that $\|\tilde{N}(\vec\eta)\|_{L^p(\partial\Omega)}<\infty$. Consider any $\varepsilon\in (0,1)$ such that for $\left[\|\mu_{||}\|_C+\norm{x_n\abs{\nabla_xA_\parallel}}^2_{L^\infty}\right]\Big(1+\|x_n|\nabla A_\parallel|\|^2_{L^\infty}\Big)<\delta(\varepsilon)<\infty$ Lemma \ref{LGL} holds. Then
\begin{multline*}
   \norm{\tilde N_a(\vec\eta)}_{L^p(\pom)}^p=p\int_0^\infty \nu^{p-1}\abs{\set{(x,t)\in\R^{n-1}\times\R: \tilde N_a(\vec\eta)(x,t)>\nu}}d\nu\\
   \le p\int_0^\infty \nu^{p-1}\abs{E_{1,\nu}}d\nu +p\int_0^\infty \nu^{p-1}\abs{\set{\varepsilon^{-1/2}[S_b(\vec\eta)+{\varepsilon^{-1}A_b(\vec\eta)}]>\gamma\nu}}d\nu\\
+p\int_0^\infty \nu^{p-1}\abs{\set{\varepsilon^{1/2}\br{\tilde{N}_b(\vec\eta)+\tilde{N}_b(V)}>\gamma\nu}}d\nu,
\end{multline*}
where $E_{1,\nu}$ is defined as in \eqref{def.E1}.
By Lemma \ref{LGL} and change of variables,  one has
\begin{multline*}
    \norm{\tilde{N}_a(\vec\eta)}_{L^p(\pom)}^p \le C\gamma^2p\int_0^\infty \nu^{p-1}\abs{\set{M(\tilde N_{a}(\vec\eta))>\kappa\nu}}d\nu \\
    +(\gamma\varepsilon^{1/2}/2)^{-p}p\int_0^\infty \nu^{p-1}\abs{\set{S_{b}(\vec\eta)>\nu}}d\nu
    +(\gamma\varepsilon^{3/2}/2)^{-p}p\int_0^\infty \nu^{p-1}\abs{\set{A_{b}(\vec\eta)>\nu}}d\nu\\
    +(\gamma\varepsilon^{-1/2}/2)^{-p}p\int_0^\infty \nu^{p-1}\abs{\set{\tilde{N}_{b}(\vec\eta)>\nu}}d\nu
    +(\gamma\varepsilon^{-1/2}/2)^{-p}p\int_0^\infty \nu^{p-1}\abs{\set{\tilde{N}_{b}(V)>\nu}}d\nu\\
 = \kappa^{-p}\gamma^2p\int_0^\infty \nu^{p-1}\abs{\set{M(\tilde N_{a}(\vec\eta))>\nu}}d\nu\qquad\qquad\qquad\qquad\qquad\qquad\qquad\qquad\qquad\qquad\\
    +C(\gamma,\varepsilon)[\norm{S_{b}(\vec\eta)}_{L^p}^p+\norm{A_{b}(\vec\eta)}_{L^p}^p+\|\tilde{N}_{b}(V)\|_{L^p}^p]
    +(\gamma\varepsilon^{-1/2}/2)^{-p}\|\tilde{N}_{b}(\vec\eta)\|_{L^p}^p\\
=    \kappa^{-p}\gamma^2\|M(\tilde N_{a}(\vec\eta))\|_{L^p}^p+(\gamma\varepsilon^{-1/2}/2)^{-p}\|\tilde{N}_{b}(\vec\eta)\|_{L^p}^p
+C(\gamma,\varepsilon)[\norm{S_{b}(\vec\eta)}_{L^p}^p+\norm{A_{b}(\vec\eta)}_{L^p}^p+\|\tilde{N}_{b}(V)\|_{L^p}^p].
\end{multline*}
As the maximal function $M$ is $L^p$ bounded for $p>1$ the first term of the last line is further bounded by
$C(p)\kappa^{-p}\gamma^2\|\tilde N_{a}(\vec\eta)\|_{L^p}^p$. 
Choosing $\gamma\in(0,1)$ sufficient small so that $C(p)\kappa^{-p}\gamma^2<1/4$, as well as $\varepsilon\in (0,1)$ so small such that $(\gamma/2)^{-p}\varepsilon^{p/2}\|\tilde{N}_{b}(\vec\eta)\|_{L^p}^p<1/4\|\tilde{N}_{a}(\vec\eta)\|_{L^p}^p$
 one can hide $\|\tilde N_{a}(\vec\eta)\|_{L^p}^p$ on the left-hand side and obtain the estimate:
$$
    \|\tilde{N}_a(\vec\eta)\|_{L^p(\pom)}^p \le 
C(\gamma,\varepsilon)[\norm{S_{b}(\vec\eta)}_{L^p}^p+\norm{A_{b}(\vec\eta)}_{L^p}^p+\|\tilde{N}_{b}(V)\|_{L^p}^p].
$$
As $\gamma,\,\varepsilon$ are fixed now we drop the dependence of the constant on them. Writing $\nabla_x u=\vec\eta+V$ it follows that
$$
    \|\tilde{N}_a(\nabla_x u)\|_{L^p(\pom)}^p \le 
C\left[\norm{S_{b}(\nabla u)}_{L^p}^p+\norm{A_{b}(\nabla u)}_{L^p}^p+
\norm{S_{b}(V)}_{L^p}^p+\norm{A_{b}(V)}_{L^p}^p+\|\tilde{N}_{b}(V)\|_{L^p}^p\right],
$$
and hence
$$
    \|\tilde{N}(\nabla_x u)\|_{L^p(\pom)}^p \le 
C\left[\norm{S(\nabla u)}_{L^p}^p+\norm{A(\nabla u)}_{L^p}^p+
\|\nabla_x f\|_{L^p}^p\right],
$$
given \eqref{sss1}. Note that \eqref{sss1} does not contain an estimate for $\norm{A(V)}_{L^p}$
but that follows form Lemma 4.4 of \cite{DH18} where the Cacciopoli inequality for second gradient is shown. 
The final ingredient in the proof \eqref{eq.N<SLp} is the estimate for $\|\tilde{N}(\partial_{x_n} u)\|_{L^p(\pom)}$.
This relies on Lemma \ref{lem.Gdl_wn}.  Tracking the constant $M_0$, we obtain that 
$$
    \|\tilde{N}(\partial_{x_n} u)\|_{L^p(\pom)} \le 
C(\norm{\mu}^{1/2}_C+\|x_n|\nabla A_\parallel|\|_{L^\infty})(1+\|x_n|\nabla A_\parallel|\|_{L^\infty}) \|\tilde{N}(\nabla_x u)\|_{L^p(\pom)}
$$
$$+C \|{S}(\nabla u)\|_{L^p(\pom)},$$
which gives \eqref{eq.N<SLp}.\qed

\section{Localization on bounded Lipschitz cylinders}

In this section we introduce a localization technique that will allow us to establish solvability of the Regularity problem on Lipschitz cylinders of the form $\Omega=\mathcal O\times\R$ for a bounded Lipschitz domain
$\mathcal O\subset\R^n$ using previously  established solvability on $\R^n_+\times\R$. We note that this localization is not needed in the case when $\mathcal O$ is an unbounded Lipschitz domain
of the form $\{(x',x_n):\, x_n>\phi(x')\}$ for a Lipschitz function $\phi:\R^{n-1}\to\R$, as in such a case we can consider a pullback map $\rho:\R^n_+\times\R\to\mathcal O\times\R$ and transfer the parabolic PDE from $\Omega$ to
$\R^n_+\times\R$ directly without any localization. We will return to this point when we discuss the map $\rho$
is more detail where we will show how to solve the unbounded case.

In what follows, we borrow ideas from Section 5.2 of \cite{DinSa} where a similar decomposition was considered. Recalling the definition of the Lipschitz domain
$\mathcal O\subset\R^n$ from Definition \ref{DefLipDomain} we thus restrict ourselves to the case
 \( \diam(\mathcal{O}) < \infty \).     \vglue1mm

 A simple compactness argument implies that there exists $s_0<1$ such that the $\ell$-cylinders $s_0\Z_j$ still cover $\partial\mathcal O$. Thus by making the scale $r_0$ and diameter $d$ smaller if necessary (at the expense of increased $N$) we may without loss of generality assume that the union of $(1/2)\Z_j$ covers $\partial\mathcal O$.
 
 By rescaling the PDE, it may also be assumed that the $\ell$-cylinders $\Z_j$ in the definition above have diameter $d=1$. Let $N,\, C_0$ be as above and hence 
 there are $N$ such $\ell$-cylinders $(1/2)\Z_j$ needed to cover the boundary $\partial\mathcal O$.
            
It follows that there exists a partition of unity $(\varphi_j)_{j=1}^N$  such that each  $\varphi_j\in C_0^\infty(\R^n)$, is supported in
$(1/2)\Z_j$, and altogether, 
$$\sum_{j=1}^N \varphi_j(x)=1,\qquad \mbox{ for all }\qquad x\in 
\left\{
x\in\mathcal O:\,  \mathrm{dist}\left( x,{\partial\mathcal O}
\right) \leq \frac{1}{2}\right\}.$$
As each $\varphi_j$ is smooth, and there are $N$ of them, it follows that $\|\nabla \varphi_j\|_{L^\infty}\le M$ for some $M>0$ for all $j$. This completes 
the description of the  partition of the spatial component of our domain $\Omega$. \medskip

Denote by  \( \Omega_i := \Base \times (i,i+1] \)
    and let \( \Delta_i  = \partial\Base \times (i,i+1] \) for $i\in\Z$. There exists a nonnegative $C^1$ function $\psi\in C_c^1(\R)$, supp $\psi\subset [-1,2]$ such that for 
$$\psi_i(\cdot)=\psi(\cdot+i), \qquad |\psi'|\lesssim 1,\quad\int_{\R}\psi=1,\qquad \sum_{i\in \Z} \psi_i \equiv 1.$$
Clearly,  \( \| D^{1/2}_t \psi \|_{L^r(\R)} \lesssim 1 \),  for each \( 1<r<\infty \).\medskip

Consider now an arbitrary $f\in \dot{L}^p_{1,1/2}(\partial\Omega)$ for some $p>1$, and the solution to $\LL u=0$ with $u\big|_{\partial\Omega}=f$.
Set
\[ f_{\Delta_i}  := \fint_{{\Delta}_{i}} f. \]
    Finally, let  \( \Psi_k  := \sum_{i \geq k} \psi_i \) and note that \( \psi_i = \Psi_{i+1} - \Psi_{i} \) and 
    $\Psi_k\equiv 1$ on $[i+2,\infty)$.
    
For all $i\in\Z$ and $j\in\{1,2,\dots,N\}$ we denote by \( w_{ij} \) the solution of the boundary value problem
\[ \begin{cases}
    \LL w_{ij} = 0, &\text{in } \Omega, 
    \\
    w_{ij} = (f-f_{\Delta_i})\psi_i\varphi_j &\text{on } \partial\Omega,
\end{cases} \]
    If \( w  := \sum_{i,j} w_{ij} \) then  \( v  := u-w \) solves \(\LL v = 0 \)
    with boundary data 
\[ \sum_{i \in \Z} f_{\Delta_i} \psi_i = \sum_{i \in \Z} f_{\Delta_i} (\Psi_{i+1} - \Psi_{i})
    = \sum_{i \in \Z} (f_{\Delta_{i-1}} -f_{\Delta_i}) \Psi_{i}. \]
Thus we can write \( v = \sum_{i,j} v_{ij} \) where \( v_{ij} \) solves
\[ \begin{cases}
    \LL v_{ij} = 0, &\text{in } \Omega, 
    \\
    v_{ij} = (f_{\Delta_{i-1}} -f_{\Delta_i})\Psi_i\varphi_j &\text{on } \partial\Omega.
\end{cases} \]
It follows that we have decomposed the solution $u$ to $\LL u=0$ with $u\big|_{\partial\Omega}=f$ into
$$u=\sum_{i\in \Z}\sum_{j=1}^N (v_{ij}+ w_{ij}),$$
such that $w_{ij}\big|_{\partial\Omega}$ is supported only on $(\partial\Base\cap (1/2)\Z_j) \times (i-1,i+2)$ and $v_{ij}\big|_{\partial\Omega}$ is supported in $(\partial\Base\cap (1/2)\Z_j) \times (i-1,\infty)$.\vglue1mm

Let us also define $\tilde{v}_{ij}$, such that 
\( \tilde{v}_{ij} \) solves
\[ \begin{cases}
    \LL\tilde{v}_{ij} = 0, &\text{in } \Omega, 
    \\
    \tilde{v}_{ij} = (f_{\Delta_{i-1}} -f_{\Delta_i})\Psi_i (1-\Psi_{i+20})\varphi_j&\text{on } \partial\Omega.
\end{cases} \]
To motivate this, observe first that 
$v_{ij}\big|_{\partial\Omega}(P,\tau)=\tilde{v}_{ij}\big|_{\partial\Omega}(P,\tau)$ for $\tau\le i+15$ and therefore $\tilde N(\nabla v_{ij})(P,\tau)=\tilde N(\nabla \tilde{v}_{ij})(P,\tau)$ for all $\tau\le i+10$, using the assumption on the diameter of the cylinders.
However, $\tilde{v}_{ij}$ is much more similar to $w_{ij}$, as both have compactly supported boundary data and
 the half derivative of 
$\Psi_i (1-\Psi_{i+20})$ has faster decay than the half derivative of $\Psi_i$.
\vglue1mm

Let us study the properties of $N(\nabla v_{ij})$ and $N(\nabla w_{ij})$. Without loss of generality, let
$v$ be either $v_{ij}$ or $w_{ij}$ for some $i\in\Z$. It follows that $v\equiv 0$ for $t\le i-1$. Thus
$\tilde N(\nabla v)(P,\tau)\equiv 0$ for all $\tau\le i-5$.

When $i-5<\tau<i+10$ we may study $\tilde N(\nabla v)(P,\tau)$ by considering the $(R_L)_p$ boundary value problem with datum
supported in a single coordinate patch. We address that later, and this is the heart of the matter. In this time range we may use
$\tilde{v}_{ij}$ instead of $v_{ij}$, as the solutions coincide when $\tau<i+15$.

Finally, for $\tau\ge i+10$ we will use the fact that after summing in $j$, the solutions $\nabla v_i = \sum_j \nabla v_{ij}$ decay exponentially. 
This is a consequence of the following lemma from \cite{DinSa}.

\begin{lemma}[{\cite[Lemma 5.10]{DinSa}}]\label{lemma:Exponential Decay} Let $p\ge 1$ and
assume that the Dirichlet problem for the adjoint PDE $(D)_{p'}^*$ is solvable. (When $p=1$ we assume 
$(D)_{q}^*$ for some $q>1$). Let
 $v$ be a solution to $\LL v=0$ and
    suppose that $v$ is constant on \( \partial\Base \times (i+2,\infty) \).
    There exists an \( \alpha>0 \) such that for all $k>i+9$ we have that
    \begin{align*}
        \| \tilde{N}(\nabla v) \|_{L^p(\Delta_k)}
        \lesssim e^{-\alpha|i-k|} \| {\tilde N}(\nabla v) \|_{L^p(\Delta_{i+3})}.     
    \end{align*}
\end{lemma}
Notice that the assumption that $(D)_{p'}^*$ is solvable is harmless because under the conditions of the matrix $A$ in Theorem \ref{MainT}, the parabolic measure satisfies $\omega_{\LL^*}\in A_\infty(d\sigma)$ for the adjoint operator, which is a consequence of the paper \cite{DPP2} (see also \cite{DDH,DH18} for small Carleson case). The $A_\infty$ condition implies that there is some $p_0\in(1,\infty)$ such that   $(D)_{p'}^*$  is solvable for all  $p\in (0,p_0)$ 
(and thus $p'\in(p_0',\infty)$). We are going to use this fact several times.

Applying Lemma \ref{lemma:Exponential Decay} to $v_i=\sum_j v_{ij}$ and $w_i=\sum_j w_{ij}$ we have for all $k>i+9$:
\begin{equation}\label{eq60}
\int_{\Delta_k} {\tilde N}(\nabla v_i)^p \lesssim e^{-\alpha p|i-k|}\sum_{j}\| {\tilde N}(\nabla v_{ij}) \|_{L^p(\Delta_{i+3})}^p
,\quad \int_{\Delta_k} {\tilde N}(\nabla w_i)^p \lesssim e^{-\alpha p|i-k|}\sum_{j}\| {\tilde N}(\nabla w_{ij}) \|_{L^p(\Delta_{i+3})}^p.
\end{equation}\medskip

With this in place it remains to establish estimates for $\|{\tilde N}(\nabla {v}_{ij}) \|_{L^p(\Delta_{k})}=\|{\tilde N}(\nabla \tilde{v}_{ij}) \|_{L^p(\Delta_{k})}$
and $\|{\tilde N}(\nabla {w}_{ij}) \|_{L^p(\Delta_{k})}$ for $k=i-5,i-4,\dots,i+9$.\medskip

Let us study the boundary data of $\tilde{v}_{ij}$ and $w_{ij}$. We start with $w_{ij}$. We recall Proposition 2.4 of \cite{DinSa}, which is stated in 
greater generality than we consider here (for chord-arc domains). 

\begin{proposition}\label{prop:Product_Rule}
Let $\Omega=\mathcal O\times\R$, where $\mathcal O\subset\R^n$ is a $1$-sided chord-arc domain. Assume that for some $p>1$
 $f\in \dot{L}^p_{1,1/2}(\partial\Omega)$.
 There exists $h\in L^p(\partial\Omega)$ with $\|h\|_{L^p(\partial\Omega)}\lesssim \|f\|_{\dot{L}^p_{1,1/2}(\partial\Omega)}$ for which the following holds:
 
Consider any parabolic ball $B\subset \R^n\times\R$ of radius $r$ centered at the boundary $\partial\Omega$, let $\Delta=B\cap\partial\Omega$ and  $\varphi$ be a $C_0^\infty(\mathbb R^n\times\R)$ cutoff function supported on $3B$ with $\|\partial_t\varphi\|_{L^\infty}\lesssim r^{-2}$ and $\|\nabla_x\varphi\|_{L^\infty}\lesssim r^{-1}$. 
 Then we have the following estimate:
$$\|(f-\textstyle\fint_{\Delta} f)\varphi\|_{\dot{L}^p_{1,1/2}(\partial\Omega)}\lesssim \|h\|_{L^p(32\Delta)}+\|D^{1/2}_t f\|_{L^p(32\Delta)}
\lesssim \|f\|_{\dot{L}^p_{1,1/2}(\partial\Omega)}.$$
\end{proposition}
Applying this result to  
$(1/2)\Delta_{ij}: = ((1/4)\Z_j\cap\partial\mathcal O)\times (i,i+1)$, (and also defining
$\Delta_{ij}:=((1/2)\Z_j\cap\partial\mathcal O)\times (i-1,i+2)$),
we see that
$$
\|(f-\textstyle\fint_{(1/2)\Delta_{ij}} f)\psi_i\varphi_j\|_{\dot{L}^p_{1,1/2}(\partial\Omega)}\lesssim \|h\|_{L^p(16\Delta_{ij})}+\|D^{1/2}_t f\|_{L^p(16\Delta_{ij})},
$$
and hence 
\begin{equation}\label{bdtw}
\|(f-\textstyle\fint_{\Delta_{i}} f)\psi_i\varphi_j\|_{\dot{L}^p_{1,1/2}(\partial\Omega)}\lesssim \|h\|_{L^p(16\Delta_{ij})}+\|D^{1/2}_t f\|_{L^p(16\Delta_{ij})}+|\fint_{(1/2)\Delta_{ij}} f-\fint_{\Delta_{i}} f|,
\end{equation}
where we use the fact that $D^{1/2}_t(\psi_i\varphi_j)=\varphi_jD^{1/2}_t \psi_i$ has $L^p$ norm of size $O(1)$. A similar statement holds for boundary data of $\tilde v_{ij}$ where we get that
\begin{equation}\label{bdtv}
\|(\textstyle\fint_{\Delta_{i+1}} f-\fint_{\Delta_{i}} f)\Psi_i(1-\Psi_{i+20})\varphi_j\|_{\dot{L}^p_{1,1/2}(\partial\Omega)}
\lesssim|\fint_{\Delta_{i+1}} f-\fint_{\Delta_{i}} f|,
\end{equation}
using the fact that $D^{1/2}_t(\Psi_i(1-\Psi_{i+20})\varphi_j)=\varphi_jD^{1/2}_t(\Psi_i(1-\Psi_{i+20}))$ also has $L^p$ norm
of the size $O(1)$ for all $p>1$.

It is further shown in \cite{DinSa} that the difference of averages on the right-hand side of \eqref{bdtv} enjoy estimates in term of the sharp maximal function of $f$. Actually, the way $h$ in Proposition \ref{prop:Product_Rule} was defined in \cite{DinSa} used $f^\sharp$ with the bound $\|M(f^\sharp)\|_{L^p}\le \|h\|_{L^p}$
and hence by (5.17) of \cite{DinSa}:
\begin{equation}\label{bdtv-f}
\|(\textstyle\fint_{\Delta_{i+1}} f-\fint_{\Delta_{i}} f)\Psi_i(1-\Psi_{i+20})\varphi_j\|_{\dot{L}^p_{1,1/2}(\partial\Omega)}
\lesssim\|h\|_{L^p(16\Delta_{i})},
\end{equation}
with analogous estimate for 
\begin{equation}\label{bdtw-f}
\|(f-\textstyle\fint_{(1/2)\Delta_{i}} f)\psi_i\varphi_j\|_{\dot{L}^p_{1,1/2}(\partial\Omega)}\lesssim \|h\|_{L^p(16\Delta_{i})}+\|D^{1/2}_t f\|_{L^p(16\Delta_{ij})}.
\end{equation}

Recall now the following theorem from \cite{DinSa}.

\begin{theorem}[{\cite[Theorem 6.52]{DinSa}}]\label{ZeroPart} Assume that $\omega\in A_{\infty}(d\sigma)$, where $\omega$ is the parabolic measure of the operator $\LL^*=\partial_t+\div(A^T\nabla\cdot)$ on a Lipschitz cylinder $\Omega=\mathcal O\times \R$.
Then for any $p\in (1,\infty]$ the following holds. There exists $C(p)>0$ such that for any 
homogeneous $(1,1/2,p)$-atom associated to a parabolic ball $\Delta\subset \partial\Omega$
 the weak solution $u$ of the equation $\LL u=-\partial_tu+\div(A\nabla u)=0$ with boundary datum $f$ satisfies the estimate
 \begin{equation}\label{eqSato}
\|\tilde N(\nabla u)\|_{L^1(\partial\Omega\backslash 2\Delta)}\leq C(p).
\end{equation}
\end{theorem}

Without going into the details of what a $(1,1/2,p)$-atom is, suffice to say that it is a localised function $f$ with support in a boundary ball $\Delta$, and that
has an appropriately defined scaling of $L^p$ norms of $\nabla f$ and $D^{1/2}_t f$. Given \eqref{bdtv-f}
and \eqref{bdtw-f}, the functions of the left-hand side are localised and while not $(1,1/2,p)$-atoms, yet are multiples of atoms. Hence we can apply Theorem \ref{ZeroPart} to both the $w_{ij}$ and the $\tilde v_{ij}$, respecting the right scaling which takes into account the size of the support of the boundary balls.\medskip


Starting with \eqref{bdtw-f},
the function $(f-\textstyle\fint_{(1/2)\Delta_{i}} f)\psi_i\varphi_j$
is supported in 
$\Delta_{ij}=((1/2)\Z_j\cap\partial\mathcal O)\times (i-1,i+2)$ and hence \eqref{eqSato} will give a bound on 
$\|\tilde N(\nabla w_{ij})\|_{L^1(\partial\Omega\backslash 2\Delta_{ij})}$, where
$2\Delta_{ij}=(\Z_j\cap\partial\mathcal O)\times (i-3,i+4)$.
Given that our atoms are defined on balls of size $O(1)$, which simplifies the matter of scaling, we therefore obtain that 
 \begin{equation}\label{eqSatow}
\|\tilde N(\nabla w_{ij})\|_{L^1(\partial\Omega\backslash 2\Delta_{ij})}\lesssim \|h\|_{L^p(16\Delta_{i})}+\|D^{1/2}_t f\|_{L^p(16\Delta_{ij})},
\end{equation}
and  similarly, for $\tilde{v}_{ij}$ whose boundary data is supported on 
$((1/2)\Z_j\cap\partial\mathcal O)\times (i-1,i+21),$
 \begin{equation}\label{eqSatov}
\|\tilde N(\nabla \tilde{v}_{ij})\|_{L^1(\partial\Omega\backslash [(\Z_j\cap\mathcal O)\cap (i-4,i+25)])}\lesssim \|h\|_{L^p(16\Delta_{i})}.
\end{equation}

It remains to consider the key estimates at the points near the support of the boundary data of $\tilde{v}_{ij}$ and $w_{ij}$, addressed in the following result.

\begin{theorem}\label{tlocal} Let $\Omega=\mathcal O\times\R$ be as above and assume that the coefficients of the matrix $A$ defining the operator $\LL=-\partial_t+\div(A\nabla\cdot)$ satisfy either (1) or (2) of Theorem \ref{MainT}.
Then there exists $p>1$ such that the following holds.

For any boundary ball $\Delta$ of size $O(1)$ such that $\Delta\subset [(1/2)\Z_j\cap\partial\mathcal O]\times (\tau-11,\tau+11)$, where $\Z_j$
is one of the $\ell$-cylinders defining the Lipschitz domain $\mathcal O$, and for any $f\in \dot{L}^p_{1,1/2}(\partial\Omega)$ supported in $\Delta$ and any solution $v$ of $\LL v=0$ in $\om$ with $v=f$ on $\pom$,  
we have that
$$\|\tilde N(\nabla v)\|_{L^p(3\Delta)}\lesssim \|f\|_{\dot{L}^p_{1,1/2}(\partial\Omega)}.$$
\end{theorem}

Note that an important aspect of the above Lemma is that fact that $3\Delta$ belongs to one of the coordinate patches that define the boundary $\partial\Omega$. Thus the support of both $f$ and $\tilde{N}(\nabla v)$ is localized. 

Assuming for the moment this theorem holds, we will pull these pieces together to obtain solvability of the $L^p$ Regularity problem on the bounded Lipschitz cylinder $\mathcal O\times\R$.
Here $p>1$ is the smaller of $p$ as in Theorem \ref{tlocal} and $\omega\in B_{p}(d\sigma)$ where $\omega$
is the parabolic measure of the adjoint operator $\mathcal L^*$. Consider such $p>1$.
\medskip

Firstly, thanks to Theorem \ref{tlocal}, \eqref{bdtv-f}, and \eqref{bdtw-f}, we have that
 \begin{equation}\label{eqSatowloc}
\|\tilde N(\nabla w_{ij})\|_{L^p(3\Delta_{ij})}\lesssim \|h\|_{L^p(16\Delta_{i})}+\|D^{1/2}_t f\|_{L^p(16\Delta_{ij})},
\end{equation}
and  similarly, for $\tilde{v}_{ij}$:
\begin{equation}\label{eqSatovloc}
\|\tilde N(\nabla \tilde{v}_{ij})\|_{L^p((2\Z_j\cap\mathcal O)\cap (i-10,i+50)])}\lesssim \|h\|_{L^p(16\Delta_{i})}.
\end{equation}

Next, we have the $L^1$ estimates \eqref{eqSatow}-\eqref{eqSatov} but these can be improved to $L^p$ thanks to the following reverse H\"older result 
from \cite{DinSa}.


\begin{theorem}\label{thm:Reverse Holder}
    Let \( 1 \le p < \infty \) and suppose that \( (D_{\mathcal L^*})_{p'} \) is solvable
        and suppose that \( u \) is a weak solution of $\mathcal Lu=0$ with \( u=0 \) 
        on \(  \Tilde{\Delta} := 80 \Delta \) where \( \Delta = \Delta_r(P_0,\tau_0) \), with $r\lesssim \mbox{diam}(\mathcal O)$.
    Then
    \[ \Big( \fint_{\Delta} {\tilde N}(\nabla u)^p \Big)^{1/p}
        \lesssim \fint_{\Tilde\Delta} {\tilde N}(\nabla u). \]
\end{theorem}

In particular, consider a cover of the set $\mathcal G=(\Delta_{i-5}\cup\Delta_{i-4}\cup\dots\Delta_{i+9})\setminus3\Delta_{ij}=(\partial\mathcal O\times (i-5,i+10])\setminus3\Delta_{ij}$.
Because $w_{ij}$ on the boundary is supported in $\Delta_{ij}$ each point in this set is at least a distance $r\sim O(1)$ away from this support. Furthermore, the area of this set is also of size $O(1)$ and therefore there will be a fixed finite number $K$ that only depends on the geometry of $\mathcal O$, but not on $w_{ij}$, such that
$\mathcal G$ can be covered by at most $K$ parabolic boundary balls $B_k$ of size $r/80$ with $80B_k\cap 2\Delta_{ij}=\emptyset$. Hence by the result above, on each $B_k$:
    \[  \int_{B_k} {\tilde N}(\nabla w_{ij})^p 
        \lesssim \left(\int_{80B_k} {\tilde N}(\nabla w_{ij})\right)^p\le \left(\int_{\partial\Omega\setminus 2\Delta_{ij}} {\tilde N}(\nabla w_{ij})\right)^p\lesssim \|h\|^p_{L^p(16\Delta_{i})}+\|D^{1/2}_t f\|^p_{L^p(16\Delta_{ij})}, \]
by \eqref{eqSatow}. 
Summing over all $k\le K$ yields
 \[  \int_{(\Delta_{i-5}\cup\Delta_{i-4}\cup\dots\Delta_{i+9})\setminus3\Delta_{ij}} {\tilde N}(\nabla w_{ij})^p 
        \lesssim \|h\|^p_{L^p(16\Delta_{i})}+\|D^{1/2}_t f\|^p_{L^p(16\Delta_{ij})}, \]
and after combining this estimate with \eqref{eqSatowloc} and summing over all $j\le N$ we get that:
\begin{equation}\label{w-loc-final}
 \int_{\Delta_{i-5}\cup\Delta_{i-4}\cup\dots\Delta_{i+9}} {\tilde N}(\nabla w_{i})^p 
        \lesssim \|h\|^p_{L^p(16\Delta_{i})}+\|D^{1/2}_t f\|^p_{L^p(16\Delta_{i})}.
\end{equation}    
Recall that $w_i=\sum_{j\le N}w_{ij}$.
By an analogous argument we get for $\tilde{v}_i=\sum_{j\le N}\tilde{v}_{ij}$ (and hence also for ${v}_{i}$ as they coincide of this portion of the domain):
\begin{equation}\label{v-loc-final}
 \int_{\Delta_{i-5}\cup\Delta_{i-4}\cup\dots\Delta_{i+9}} {\tilde N}(\nabla {v}_{i})^p=\int_{\Delta_{i-5}\cup\Delta_{i-4}\cup\dots\Delta_{i+9}} {\tilde N}(\nabla \tilde{v}_{i})^p 
        \lesssim \|h\|^p_{L^p(16\Delta_{i})}.
\end{equation}    
    
We are ready to add up the estimates to obtain a bound for the $L^p$ norm of $u$. Fix $i\in \Z$
and consider $\tilde N(\nabla u)$ on $\Delta_i$. Since $\tilde N$ is sub-additive, clearly
$$\tilde N(\nabla u)(P,\tau)\le \sum_k \left[\tilde N(\nabla v_k)+\tilde N(\nabla w_k)\right](P,\tau),$$ 
and therefore 
$$\|\tilde N(\nabla u)\|_{L^p(\Delta_i)}\le \sum_k \left[\|\tilde N(\nabla v_k)\|_{L^p(\Delta_i)}+\|\tilde N(\nabla w_k)\|_{L^p(\Delta_i)}\right]$$ 
$$\le \sum_{k=i-9}^{i+5}\left[\|\tilde N(\nabla v_k)\|_{L^p(\Delta_i)}+\|\tilde N(\nabla w_k)\|_{L^p(\Delta_i)}\right]+\sum_{k<i-9}\left[\|\tilde N(\nabla v_k)\|_{L^p(\Delta_i)}+\|\tilde N(\nabla w_k)\|_{L^p(\Delta_i)}\right]$$
$$\lesssim \sum_{k=i-9}^{i+5}\left[\|h\|_{L^p(16\Delta_{k})}+\|D^{1/2}_t f\|_{L^p(16\Delta_{k})}\right]+\sum_{k<i-9}e^{-\alpha|i-k|}\left[\|h\|_{L^p(16\Delta_{k})}+\|D^{1/2}_t f\|_{L^p(16\Delta_{k})}\right],
$$
using \eqref{w-loc-final}, \eqref{v-loc-final} and \eqref{eq60}. We raise both sides to the $p$-th power. It follows that
$$\hspace{-1.5cm}\|\tilde N[\nabla u]\|^p_{L^p(\Delta_i)}\lesssim \sum_{k=i-9}^{i+5}\left[\|h\|^p_{L^p(16\Delta_{k})}+\|D^{1/2}_t f\|^p_{L^p(16\Delta_{k})}\right]$$
$$\qquad\qquad\qquad\qquad+\left(\sum_{k<i-9}e^{-\alpha|i-k|}\left[\|h\|_{L^p(16\Delta_{k})}+\|D^{1/2}_t f\|_{L^p(16\Delta_{k})}\right]\right)^p.$$
For the last term we use H\"older's inequality for series. We split the exponential into two terms
$$e^{-\alpha|i-k|}=e^{-\alpha|i-k|/p'}e^{-\alpha|i-k|/p}$$
to obtain
$$\left(\sum_{k<i-9}e^{-\alpha|i-k|}\left[\|h\|_{L^p(16\Delta_{k})}+\|D^{1/2}_t f\|_{L^p(16\Delta_{k})}\right]\right)^p\le \left(\sum_{k<i-9} \left(e^{-\alpha|i-k|/p'}\right)^{p'}\right)^{p/p'}\times$$
$$\qquad\qquad\qquad\sum_{k<i-9}e^{-\alpha|i-k|}\left[\|h\|_{L^p(16\Delta_{k})}+\|D^{1/2}_t f\|_{L^p(16\Delta_{k})}\right]^p$$
$$\le C(p,\alpha)\sum_{k<i-9}e^{-\alpha|i-k|} \left[\|h\|^p_{L^p(16\Delta_{k})}+\|D^{1/2}_t f\|^p_{L^p(16\Delta_{k})}\right].$$

Now we sum over all indices $i\in \Z$. Since 
$\|\tilde N[\nabla u]\|^p_{L^p(\partial\Omega)}= \sum_i \|N[\nabla u]\|^p_{L^p(\Delta_i)}$ we obtain after rearranging the summation order:
\begin{equation}
\|\tilde N[\nabla u]\|^p_{L^p(\partial\Omega)}\lesssim \sum_{i\in\Z} (1+\sum_{k>9}e^{-\alpha k})\left[\|h\|^p_{L^p(16\Delta_{i})}+\|D^{1/2}_t f\|^p_{L^p(16\Delta_{i})}\right]
\end{equation}
We use the finite overlap of the balls $32\Delta_i$ which allows us to claim that
$$\sum_{i\in \Z}\|h\|^p_{L^p(32\Delta_i)}\lesssim \|h\|^p_{L^p(\partial\Omega)},$$
with a similar statement holding for $D^{1/2}_t f$.
For both $h$ and $D^{1/2}_t f$ we have bounds on their $L^p$ norms by $\|f\|_{\dot{L}^p_{1,1/2}(\partial\Omega)}$. In summary, we have that
\begin{equation}\label{twotz}
\|\tilde N[\nabla u]\|^p_{L^p(\partial\Omega)}\lesssim \|f\|^p_{\dot{L}^p_{1,1/2}(\partial\Omega)},
\end{equation}
proving solvability of the Regularity problem for this particular $p>1$. Hence, with the exception of the proof of Theorem \ref{tlocal} we are done, as this argument extends the solvability result stated in Theorem \ref{MainT} to 
bounded Lipschitz cylinders. The second part of the Theorem \ref{tlocal} claiming duality of ranges of solvability then follows from Lemma \ref{lem.DtoR}.
\qed\medskip

\subsection{Proof of Theorem \ref{tlocal}.}

In the section above we have reduced matters to solvability of Regularity problem for data localized to a parabolic ball which additionally sits inside one of the coordinate patches that describes the boundary $\partial\mathcal O$
times the time direction. In the local coordinates of this particular patch on ($8\Z$) we may assume the following:
\medskip

$\phi:\{x'\in\R^{n-1}:\,|x'|\le 8\}\to \R$ is a Lipschitz function with Lipschitz constant $\ell$ with $\phi(0)=0$
and in the coordinates $(x',x_n,t)$ the portion of domain $\Omega$ that belongs to $8\Z\times\R$ can be written
as
$$\Omega\cap (8\Z\times\R)=\{(x',x_n,t):\, ,|x'|\le 8,\, \phi(x')<x_n<8(\ell+1),\,t\in\R\},$$
with $\partial\Omega\cap( 8\Z\times\R)$ being the graph of function $\phi$ times $\R$.
Let us now choose a new function $\tilde{\phi}:\R^{n-1}\to \R$ with Lipschitz constant at most $3\ell$ having the following properties:

$$\tilde{\phi}(x')={\phi}(x')\quad\mbox{for all }|x'|\le 6, \qquad\mbox{and}\qquad \tilde{\phi}(x')=0\quad\mbox{for all }|x'|\ge 8.$$
This can be done trivially by multiplying $\phi$ by a cutoff function.\medskip

We similarly extend the coefficients of our equation. Let us write the parabolic PDE for a  solution $v$ to $\LL v=0$
on $\Omega=\mathcal O\times\R$
with boundary datum $f$ supported in the ball 
$$\Delta=[(1/2)Z_j\cap\partial\mathcal O] \times (-11,11),$$
where we have without loss of generality shifted the support of our datum to $|t|<11$. In these local coordinates the PDE has the form
$$-\partial_tv+\div(A\nabla v)=0,\qquad\mbox{where $A$ in uniformly elliptic in }[8\Z\cap\mathcal O]\times \R,$$
and satisfies the Carleson condition as in Theorem \ref{MainT}.\medskip

We would like to extend the PDE to the domain $\tilde\Omega:=\{(x',x_n,t):\,x'\in\R^{n-1},\, x_n>\tilde\phi(x'),\,t\in\R\}$ such that 
the PDE will coincide with the original PDE in the set $[4\Z\cap\mathcal O]\times \R$, and so that the new matrix $\tilde{A}$
will be uniformly elliptic and satisfy the Carleson condition of Theorem \ref{MainT} with respect to the boundary of 
the domain $\partial\tilde\Omega$.
This again is trivial, as we may consider a partition of unity $(\eta_1,\eta_2)$ subordinate to the 
cover of $\{(x',x_n):\,x'\in\R^{n-1},\, x_n>\tilde\phi(x')\}$ by the sets $6\Z\cap\mathcal O$ and $\R^n\setminus [4\Z\cap\mathcal O]$ with new matrix $\tilde{A}$ defined by
$$\eta_1(x',x_n)A(x,x_n,t)+\eta_2(x',x_n)I,\qquad\mbox{for all } \{(x',x_n,t):\,x'\in\R^{n-1},\, x_n>\tilde\phi(x'),\,t\in\R\}.$$
When $|x'|\le 4$ and $x_n<4(\ell+1)$ this preserves the property that $\tilde A=A$, while away from $|x'|\ge 8$ or $x_n>8(\ell+1)$ the matrix is simply the identity matrix and thus trivially satisfies both ellipticity and Carleson condition.

Hence it is possible to consider now the solution $u$ to the PDE
\begin{equation}\label{globalPDE}
-\partial_tu+\div(\tilde A\nabla u)=0 \qquad \mbox{on }\tilde\Omega,\quad\mbox{ with }u\big|_{\partial\tilde\Omega}=f.
\end{equation}

Observe that while the solution $v$ is only defined in our coordinates 
on the set $(8\Z\cap \mathcal O)\times\R$, the solution $u$ is defined on the whole set $\tilde\Omega$.
Let us establish first solvability of the PDE \eqref{globalPDE}.

Consider now a bijective bi-Lipschitz map $\Psi: {\mathbb R}^n_+
\to \{(x',x_n):\, x_n>\tilde\phi(x')\}$. 
The natural choice here is the map  due to Dahlberg, Kenig, Ne\v{c}as, Stein
(see for example \cite{D} or \cite{N}) defined as
\begin{equation}\label{map2}
\Psi(x',x_n)=(x',c_0x_n+(\theta_{x_n}*\tilde\phi)(x')),
\end{equation}
where $(\theta_s)_{s>0}$ is smooth compactly supported approximate
identity and $c_0$ can be chosen large enough (depending only on
$\|\nabla\tilde\phi\|_{L^\infty({\mathbb R}^{n-1})}$) so that $\Psi$ is
one to one.

Extend this map to 
\begin{equation}\label{map_rho}
    \rho(x,x_n,t)=(\Psi(x',x_n),t),
\end{equation}
a map from ${\mathbb R}^n_+\times\R$ to $\tilde{\Omega}$, the domain on which we consider \eqref{globalPDE}.

Then
$$U=u\circ\rho,\quad\text{solves the parabolic PDE}\quad \LL_1U=0\text{ on }{\mathbb
R}^n_+\times\R$$ where
$$\LL_1=-\partial_t+\text{div}(B\nabla\cdot),$$
where $B$ is a new matrix obtained from the original matrix $\tilde{A}$. 
 Also, $U=\wt f$ on $\partial\br{\Rn_+\times\R}$ where $\wt f=f\circ\rho|_{x_n=0}$.

This type of mapping has been studied extensively, in the parabolic setting in \cite{HL01}. However, our map is simpler than the one in \cite{HL01}
as it is constant in $t$. Crucially, as shown in \cite{HL01}, the new matrix $B$ as above inherits the Carleson condition that $\tilde{A}$ satisfies with perhaps a larger constant (as the Carleson norm also depends on $\|\nabla\tilde\phi\|_{L^\infty({\mathbb R}^{n-1})}$). Furthermore as $\Psi$ is $t$-independent, the above change of variables does not introduce any first order (drift) term.\medskip

It follows that the PDE for $U$ is of the type we have considered in Theorem \ref{MainT} on $\R^n_+\times\R$
for which the solvability has been established earlier.
 Note that the matrix $B$ inherits its ellipticity from $\tilde{A}$ and hence $A$, and that it also satisfies Carleson condition with norm that only depends on the Carleson condition of the original matrix $A$ and $\ell$. We can then apply Theorem \ref{MainT} for $\Rn\times\R$ to the operator $\LL_1$, which implies that there exists a uniform $p_0=p_0(A,\lambda,\Lambda,\ell, n)>1$ such that for all $1<p<p_0$
the $L^p$ Regularity problem for the operator $\LL_1$ is solvable on $\R^n_+\times\R$ and for some $C=C(A,\lambda,\Lambda,\ell, n,p)>0$, we have that
$$\|\tilde N(\nabla U)\|_{L^p(\partial(\R^n_+\times\R))}\le C\|\wt f\|_{\dot{L}^p_{1,1/2}(\partial(\R^n_+\times\R))}.$$
As $U=u\circ\rho$ where $\rho$ is bi-Lipschitz the same estimate also holds for $u$ on the domain $\tilde\Omega$. Thus
\begin{equation}\label{Rp-u}
\|\tilde N(\nabla u)\|_{L^p(\partial\tilde\Omega)}\le C\|f\|_{\dot{L}^p_{1,1/2}(\partial\tilde\Omega)}.
\end{equation}

We now want to deduce the solvability of the Regularity problem for $v$ from solvability of the related PDE for $u$. Recall that we know that on the set $(4\Z\cap \mathcal O)\times\R$ the two PDEs (for $u$ and $v$)
 we consider coincide and hence we have for $w=u-v$ the following:
$$-\partial_tw+\div(\tilde A\nabla w)=0,\quad\mbox{on}\quad [4\Z\cap\mathcal O]\times \R.$$

We want to further restrict the domain on which we consider $w$. Firstly, clearly $w\equiv 0$ for $t<-11$ as $f$ is not supported there. Secondly, as we may use Theorem \ref{ZeroPart} directly on $v$ on the domain 
$\Omega$. It follows that
\begin{equation}\label{Rp-v}
\|\tilde N(\nabla v)\|_{L^1(\partial\Omega\setminus 2\Delta)}\le C\|f\|_{\dot{L}^p_{1,1/2}(\partial\Omega)}.
\end{equation}
Hence we have all required estimates for $t>2\times 11=22$ from above as we may combine the estimate
\eqref{Rp-v} with Theorem \ref{thm:Reverse Holder} and the exponential decay (Lemma \ref{lemma:Exponential Decay}) to conclude that
\begin{equation}\label{Rp-v2}
\|\tilde N(\nabla v)\|_{L^p(\partial\Omega\cap \{t\ge 22\})}\le C\|f\|_{\dot{L}^p_{1,1/2}(\partial\Omega)}.
\end{equation}
There remains only the interval $-11<t<22$ to consider.

Consider first the set $\mathcal S=(\partial\Omega\setminus 3\Delta)\cap \{-11<\tau<22\}$. When $(x,t)\in\mathcal S$, we have $f\equiv 0$ (and even on an enlarged) portion of boundary),  and hence an argument similar to the one following Theorem \ref{thm:Reverse Holder} applies. 
$\mathcal S$ can be covered by finite collection of balls $B_k$ with $80B_k\cap 2\Delta=\emptyset$
and hence we may again use Theorem \ref{thm:Reverse Holder} and the estimate
\eqref{Rp-v} to obtain
\begin{equation}\label{Rp-v3}
\|\tilde N(\nabla v)\|_{L^p(\mathcal S)}\le C\|f\|_{\dot{L}^p_{1,1/2}(\partial\Omega)}.
\end{equation}
We also consider 
$$\tilde\Omega^s:=(s\Z\cap \mathcal O)\times \R\qquad\mbox{for } 3\le s\le 4,$$
thus for each such $s$ the domain $\tilde\Omega^s$ is an unbounded cylinder and on $\tilde\Omega^s$ the PDE given above for $w$ holds as both $u$ and $v$ solves PDEs with identical coefficients.

Considering the boundaries of $\tilde\Omega^s$ for $3\le s\le 4$ we see that
$$(x,t)\in \partial\tilde\Omega^s\cap \partial \tilde\Omega,\qquad\Longrightarrow\qquad w(x,t)=0,$$
which follows from the fact that $u=v=f$ on the $ \partial \tilde\Omega$ portion of the boundary. Given the definition of $\tilde\Omega^s$, its lateral boundary has two more components:
\begin{equation}\label{wbound}
\{(x',x_n):\, |x'|=s,\, \phi(x')<x_n<s(\ell+1)\}\times\R \quad\mbox{and}\quad \{(x',s(\ell+1)):\,|x'|<s\}\times\R.
\end{equation}
To gain understanding of the function $w$ on \eqref{wbound} we look the domain
$$\tilde\Omega_{0}=\tilde\Omega^4\setminus\overline{\tilde\Omega^3}=$$
$$[\{(x',x_n):\, 3<|x'|<4,\, \phi(x')<x_n<4(\ell+1)\}\cup \{(x',x_n):\,|x'|<4,\, 3(\ell+1)<x_n<4(\ell+1)\}]\times\R.$$

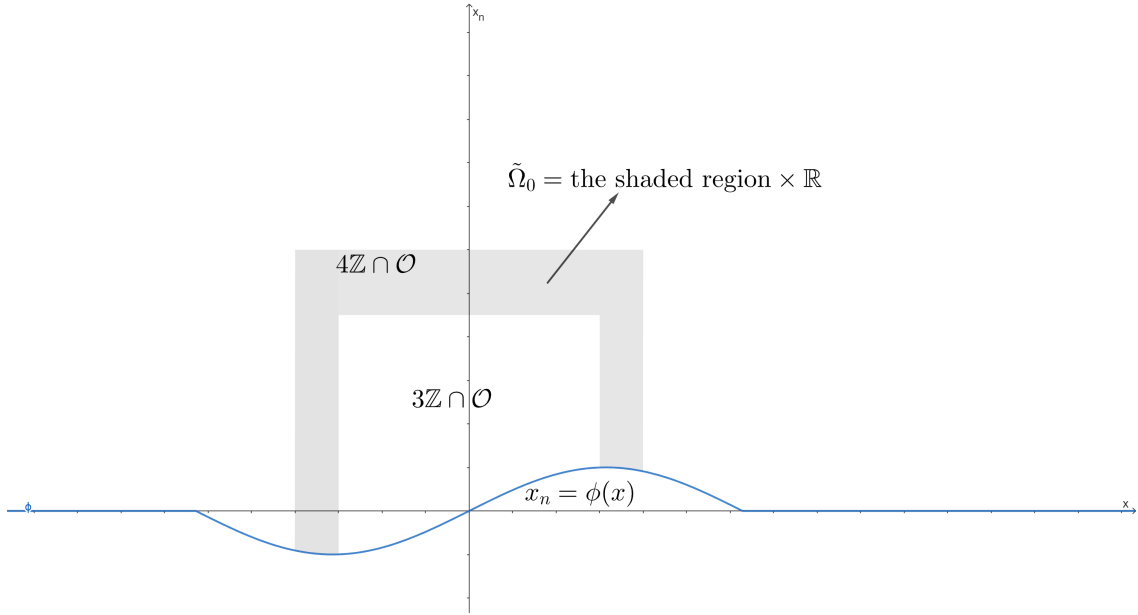
\begin{figure}[htbp]
    \centering
    \begin{tikzpicture}[x=1.5cm, y=1.5cm]
        
        \draw[->, thin] (-5, 0) -- (5, 0) node[above left] {$x$};
        \draw[->, thin] (0, -2) -- (0, 5) node[below right] {$x_n$};
        
        \foreach \x in {-4,-3,-2,-1,1,2,3,4}
            \draw[very thin, gray] (\x, 0.05) -- (\x, -0.05);
        \foreach \y in {-1,1,2,3,4}
            \draw[very thin, gray] (0.05, \y) -- (-0.05, \y);

        \begin{scope}
            \clip (-5, 5) -- (-5, 0) -- (-3.5, 0) 
                .. controls (-2.5, 0) and (-2.5, -1.5) .. (-1.5, -1) 
                .. controls (-0.5, -0.5) and (-0.5, 0) .. (0, 0)
                .. controls (0.5, 0) and (0.5, 0.8) .. (1.5, 0.8) 
                .. controls (2.5, 0.8) and (2.5, 0) .. (3.5, 0)
                -- (5, 0) -- (5, 5) -- cycle;
                
            \fill[gray!20] (-2.5, -2) rectangle (-1.5, 3.5); 
            \fill[gray!20] (1.5, -2) rectangle (2.5, 3.5);   
            \fill[gray!20] (-2.5, 2.5) rectangle (2.5, 3.5); 
        \end{scope}
        
        \draw[blue, thick] (-5, 0) -- (-3.5, 0) 
            .. controls (-2.5, 0) and (-2.5, -1.5) .. (-1.5, -1) 
            .. controls (-0.5, -0.5) and (-0.5, 0) .. (0, 0)
            .. controls (0.5, 0) and (0.5, 0.8) .. (1.5, 0.8)    
            .. controls (2.5, 0.8) and (2.5, 0) .. (3.5, 0)
            -- (5, 0);

        \node[above] at (1.5, 0.8) {$x_n = \phi(x)$};
        
        \node[blue, above] at (-4.5, 0) {\small $\phi$};
        \fill[blue] (-4.5, 0) circle (1pt);

        \node[black] at (-1.5, 3) {$4\mathbb{Z} \cap \mathcal{O}$};
        \node[black] at (-0.5, 1.5) {$3\mathbb{Z} \cap \mathcal{O}$};
        
        \draw[->, >=stealth, thin, gray!80!black] (1.5, 4.3) -- (0.8, 3.0);
        \node[above right] at (1.3, 4.3) {$\tilde{\Omega}_0 = \text{the shaded region} \times \mathbb{R}$};
        
    \end{tikzpicture}
\caption{Sketch of domains in the $(x,x_n)$ plane.}
    \label{fig:region}
\end{figure}

Our goal is to establish that 
$\nabla w\in L^2(\Omega_{0}).$

Given what we know about $u$ and $v$ we clearly have $\nabla w=\nabla u-\nabla v=0$ when $t\le -11$.
Also since \eqref{Rp-u} and \eqref{Rp-v2} holds when $t\ge 22$ we obtain by using
Lemma \ref{lemma:Gradient L2toL1} and the exponential decay (Lemma \ref{lemma:Exponential Decay})
 that 
$$\|\nabla w\|_{L^2(\Omega_{0}\cap\{t\ge 22\})}\le \|\nabla u\|_{L^2(\Omega_{0}\cap\{t\ge 22\})}+\|\nabla v\|_{L^2(\Omega_{0}\cap\{t\ge 22\})}\lesssim 
\|f\|_{\dot{L}^p_{1,1/2}(\partial\Omega)}.$$
On the component of the domain  $\Omega_0$
$\{(x',x_n):\, 3<|x'|<4,\, \phi(x')<x_n<4(\ell+1)\}\times (-11,22)$ we see that the boundary component (that coincides with $\partial\tilde\Omega$) belongs to the set $\mathcal S$ and thus we get that $\nabla u\in L^2$, by  
Lemma \ref{lemma:Gradient L2toL1} when close to the boundary, and by its interior equivalent when inside the domain. The argument for $v$ is similar - instead of \eqref{Rp-v3} we use \eqref{Rp-u}.

Finally the component of $\{(x',x_n):\,|x'|<4,\, 3(\ell+1)<|x_n|<4(\ell+1)\}\times (-11,22)$ of $\Omega_0$ is in the interior of $\tilde\Omega$ at the height of at least $O(1)\sim(\ell+1)$ away from the boundary. It follows that for any $(x',x_n,\tau)$ in this component we can find points $(q,\tau)\in\mathcal S$ such that $(x',x_n,\tau)\in\Gamma_a(q,\tau)$, provided the slope $a$ that defines nontangential cone $\Gamma_a$ is sufficiently small. As this can always be assumed (change of slope only changes constants in the estimates but not the estimates themselves),
$N(\nabla u)(q,\tau)$ controls the $L^2$ norm near the point $(x',x_n,\tau)$. In summary, we have that 
$$\|\nabla w\|_{L^2(\Omega_{0})}\le \|\nabla u\|_{L^2(\Omega_{0})}+\|\nabla v\|_{L^2(\Omega_{0})}\lesssim 
\|f\|_{\dot{L}^p_{1,1/2}(\partial\Omega)}.$$

We prove below (Lemma \ref{lem.Dt1/2bd}) that $D^{1/2}_t w\in L^2(\Omega_0)$ with $\|D^{1/2}_t w\|_{L^2(\Omega_{0})}\lesssim 
\|f\|_{\dot{L}^p_{1,1/2}(\partial\Omega)}$. It then follows that $w$ has trace in  
$\Hdot^{1/4}_{\pd_{t} - \Delta_x}(\partial\Omega_0)$ and in particular as $\partial\tilde\Omega^3\setminus \partial\tilde\Omega\subset \partial\Omega_0$ it also has trace in $\Hdot^{1/4}_{\pd_{t} - \Delta_x}(\partial\tilde\Omega^3)$ as $w=0$ on the remaining portion of its boundary $\partial\tilde\Omega^3\cap \partial\tilde\Omega$.

Recalling the section \ref{RwEs} where we have introduced energy solutions, there exists a unique energy solution  to the PDE $L\tilde{w}=0$ in $\tilde\Omega^3$ such that Tr $\tilde w=$ Tr $w$. Due to uniqueness established there $w=\tilde w$ in $\tilde\Omega^3$ and thus $w$ itself is an energy solution on this domain
enjoying the estimate
\begin{align*}
 \|w\|_{\dot \E} = \bigg(\|\nabla w\|_{\L^2(\tilde\Omega_3)}^2 + \|\HT \dhalf v\|_{\L^2(\tilde\Omega_3)}^2 \bigg)^{1/2} \lesssim \|\mbox{Tr }w\|_{\Hdot^{1/4}_{\pd_{t} - \Delta_x}(\partial\tilde\Omega^3)}
\end{align*}
$$\lesssim \|\nabla w\|_{L^2(\Omega_{0})}+ \|D^{1/2}_t w\|_{L^2(\Omega_{0})}\lesssim \|f\|_{\dot{L}^p_{1,1/2}(\partial\Omega)}.$$
In particular, as $w=0$ on $4\Delta$,
we obtain using Theorem \ref{thm:Reverse Holder} that 
\begin{equation}\label{Rp-vv}
\|\tilde N(\nabla w)\|_{L^p(3\Delta)}\le C\|f\|_{\dot{L}^p_{1,1/2}(\partial\Omega)}.
\end{equation}
Finally, as $v=u-w$ and $u$ has required estimates on $3\Delta$ from \eqref{Rp-u}, we get also that
\begin{equation}\label{Rp-vvv}
\|\tilde N(\nabla v)\|_{L^p(3\Delta)}\le C\|f\|_{\dot{L}^p_{1,1/2}(\partial\Omega)}.
\end{equation}
This combined with \eqref{Rp-v3} yields 
$\|\tilde N(\nabla v)\|_{L^p(\partial\Omega)}\le C\|f\|_{\dot{L}^p_{1,1/2}(\partial\Omega)}.
$ This proves the claim of Theorem \ref{tlocal} (modulo the following lemma):

\begin{lemma}\label{lem.Dt1/2bd} Assume that $\mathcal O_1\subset \mathcal O\subset \R^n$ are two bounded Lipschitz domains
such that the overlap of the boundaries $\partial \mathcal O_1\cap 
\partial\mathcal O$ is a set of positive measure.
Assume that $r=\mbox{diam}(\mathcal O)$ and fix some $\alpha\in (0,1)$.

Then there exists $C>0$ such that for any solution of the PDE $\mathcal Lv=0$ on $\mathcal O\times\R$ vanishing on the set 
\begin{equation}\label{BC}
\{(x,t)\in \partial\mathcal O\times\R:\, \dist(x,\partial\mathcal O_1)<\alpha r,\,t\in\R\}
\end{equation}
 the following estimate holds:
$$\|D^{1/2}_t v\|_{L^2(\mathcal O_1\times\R)}\le C\|\nabla v\|_{L^2(\mathcal O\times\R)}.$$
\end{lemma}

To prove this we consider a cutoff function $\varphi\in C_0^\infty(\mathcal O)$ such that
$\varphi =1$ on the set $\{(X,t)\in\mathcal O:\, \mbox{dist}(X,\mathcal O_1)<\alpha r/2\}$
and vanishes on the set $\{(X,t)\in\mathcal O:\, \mbox{dist}(X,\mathcal O_1)\ge \alpha r\}$.
Also $|\nabla \varphi|\lesssim r^{-1}$.
Notice that this together with the condition \eqref{BC} guarantees that $v\varphi=0$ on $\partial\mathcal O\times\R$
and $v\varphi=v$ on $\mathcal O_1\times\R$.

We clearly have
\begin{equation}\label{split49}
\iint_{\mathcal O_1\times\R}|D^{1/2}_t(v)|^{2}\le
\displaystyle\iint_{\mathbb R^n\times\mathbb R}D^{1/2}_t(v)^{2}\varphi^2=
\displaystyle\iint_{\mathbb R^n\times\mathbb R}(D^{1/2}_tH_tD^{1/2}_tv)H_tv\varphi^2.
\end{equation}
Here $H_t$ is the Hilbert transform in the time variable, $H_t^2=-I$. Since $D^{1/2}_tH_tD^{1/2}_tv=\partial_t v$ we have after using the PDE v satisfies:
$$\iint_{\mathcal O_1\times\R}|D^{1/2}_t(v)|^{2}\le
\displaystyle\iint_{\mathbb R^n\times\mathbb R}\div(A\nabla v)H_tv\varphi^2
=-\displaystyle\iint_{\mathbb R^n\times\mathbb R}A\nabla v\cdot H_t(\nabla v)\varphi^2- 2\displaystyle\iint_{\mathbb R^n\times\mathbb R}A\nabla v \cdot\nabla\varphi (u\varphi).
$$
Recall that $H_t$ is an isometry on $L^2(\mathbb R)$ and hence 

$$\int_{t\in\mathbb R}|H_t(\nabla v\varphi)|^2\,dt=\int_{t\in\mathbb R}|\nabla v\varphi|^2\,dt\le \int_{t\in \R}|\nabla v|^2\varphi^2\,dt.$$
Hence  the righthand side above has a bound (by Cauchy Schwarz and taking into account the support of $\varphi$):
\begin{equation}\label{split51}
\lesssim \iint_{\mathcal O\times\R} |\nabla v|^2 +r^{-2} \iint_{\mathcal O\times\R} |v|^2.
\end{equation}
For the second term we use Sobolev inequality. On a fixed time slice $t=t_0$, since
$v$ vanishes on $A$ (which is a set of positive measure), we get:
$$r^{-2} \int_{\mathcal O} |v|^2\le C \int_{\mathcal O} |\nabla v|^2.$$
Integrating in $t$ and putting all inequalities together then yields:
$$\iint_{\mathcal O_1\times\R}|D^{1/2}_t(v)|^{2}\lesssim  \iint_{\mathcal O\times\R} |\nabla v|^2,$$
which was our claim.\medskip

To establish that $D^{1/2}_t w\in L^2(\Omega_0)$ we take $\mathcal O_1\times\R=\Omega_0$
and $\mathcal O\times\R=\tilde{\Omega}^{4+\varepsilon}\setminus \overline{\tilde{\Omega}^{3-\varepsilon}}$, i.e., $\mathcal O\times\R$ is taken to be a small enlargement of $\Omega_0$. Clearly the argument we gave on $\Omega_0$ also works on this slightly enlarged domain. Observe that  these two domains share the boundary on a piece of $\partial\tilde{\Omega}$ where $w=0$ and thus the assumption of our Lemma is satisfied. From this $D^{1/2}_t w\in L^2(\Omega_0)$ follows.
\qed

\appendix
\section{Solvability of the Dirichlet problem for a block form matrix}\label{APA}

In this section we prove solvability  of the Dirichlet problem $\mathcal Lu=0$ for a matrix $A$ in the block form. 
Here as before $\LL u=-\partial_t+\div(A\nabla\cdot)$ and $\Omega=\R^n_+\times\R$. We work with the operator $\LL$, not $\LL^*$, to maintain consistency with the rest of the paper, even though for our Regularity result we need solvability of the $L^p$ Dirichlet problem for $\LL^*$ for all $p\in (1,\infty)$. However, the simple change of variables $(x,x_n,t)\mapsto (x,x_n,-t)$ transforms the time-forward PDE into a time-backward one, so proving the result for $\LL$ will suffice.


Our objective is to establish solvability of the $L^p$ Dirichlet problem for $\LL$ for all $p$ close to $1$, as that automatically implies solvability for all $q>p$. Thus, without loss of generality, we fix some $1<p< 2$. Recall the definition of the $p$-adapted square function (introduced originally in \cite{DPP}) which is a modification of the usual square function defined earlier in \eqref{DefSquare}
\begin{equation}\label{Defp-Square}
S_{p,a}(w)(q,\tau):=\br{\iint_{\gamma_a(q,\tau)}\abs{\nabla w(X,t)}^2\abs{w(X,t)}^{p-2}x_n^{-n}dXdt}^{1/p} \quad\text{for }(q,\tau)\in\pom,
\end{equation}
where for simplicity we only work in the upper-half space and thus we use the nontagential cones $\gamma_a(q,\tau)$ as in Section \ref{SS:43}. When clear, we drop the subscript $a$ and only write $S_p(w)$.
\eqref{Defp-Square} follows the convention that if both $|w(X,t)|=|\nabla w(X,t)|=0$ the expression
$\abs{\nabla w(X,t)}^2\abs{w(X,t)}^{p-2}$ is considered to be zero.

Since $2-p>0$, $\abs{u(X,t)}^{2-p}\le N(u)^{2-p}(q,\tau)$ for all $(X,t)\in\gamma(q,\tau)$.
With  H\"older's inequality, this gives

\begin{eqnarray}\nonumber
\|S_2(u)\|_{L^p}^p&=&\int_{\pom}\left(\iint_{\gamma(q,\tau)}|\nabla u(X,t)|^2\abs{u(X,t)}^{p-2}\abs{u(X,t)}^{2-p}x_{n}^{-n}\,dX\,dt\right)^{p/2}dq\,d\tau\\\label{S2Sp}
&\le& \int_{\pom}N(u)^{\frac{4-2p}{p}}\left(\iint_{\gamma(q,\tau)}|\nabla u(X,t)|^2\abs{u(X,t)}^{p-2}x_{n}^{-n}\,dX\,dt\right)^{p/2}dq\,d\tau\\\nonumber
&\le& \varepsilon\int_{\pom}N(u)^pdq\,d\tau+C_\varepsilon \int_{\pom}S_p(u)^pdq\,d\tau
=\varepsilon\|N(u)\|_{L^p}^p+C_\varepsilon \|S_p(u)\|_{L^p}^p,
\end{eqnarray}
for all $\varepsilon>0$. Here $N(u)$ is the sup-version of the nontangential maximal function.
Observe that this inequality holds trivially with $\varepsilon=0$ when $p=2$ and thus we can consider any $1<p\le 2$. We prove an analog of Lemma \ref{lem.Sdk} by integrating by parts in the quantity $\|S_p(u)\|_{L^p}^p$.

\begin{lemma}\label{lem.Sdkp}
Let $\LL=-\dr_t+L$ be a parabolic operator with matrix $A$ in the block form \eqref{block}, and let $u$ be a bounded nonnegative energy solution to $\LL u=0$  in the infinite strip $\R^{n-1}\times(0,r)\times\R$. Then for any $r>0$, we have 
\begin{multline}\label{eq.BlckSdkp<Nr}
    \lambda\int_0^{\frac{r}2
    }\int_{\pom} \abs{\nabla u}^2u^{p-2}x_n\,dxdtdx_n+\frac{2C(p)}{r}\int_0^r\int_{\pom}|u|^p dxdtdx_n\\
    \le C(p)\int_{\pom}|u(x,r,t)|^pdxdt
    +C(p)\int_{\pom}|u(x,0,t)|^pdxdt,
\end{multline}
where $C=C(p)>0$ depending only on $p$.
\end{lemma}

The assumption that $u$ is nonnegative
will make it much easier to handle expressions like the first term of \eqref{eq.BlckSdkp<Nr}. In particular, by Harnack, if $u(X,t)>0$ then $u(Y,s)>0$ for all $s>t$, and that means that $u^{p-2}(Y,s)$ is well-defined. Furthermore, the assumption that $u\ge 0$ is jsutified since an $L^p$ function $f$ on the boundary can be split into its positive and negative parts and the corresponding solutions $u^+$, $u^-$ can be considered separately.
 We postpone the proof until we establish the following corollary. 

\begin{corollary}\label{cor.Sdkp}
Let $\LL=-\dr_t+L$ be a parabolic operator with matrix $A$ in the block form \eqref{block}, and let $u$ be a bounded nonnegative energy solution to $\LL u=0$ in $\Omega$. Then 
\begin{equation}\label{eq.BlckSdkp}
    \lambda\iint_{\Omega}
     \abs{\nabla u}^2u^{p-2}x_n\,dXdt\le C    \int_{\pom}|u(x,0,t)|^pdxdt,
\end{equation}
where $C$ depends only $p>1$.
\end{corollary}

\begin{proof} Consider a bounded $f\in \Hdot^{1/4}_{\partial_t-\Delta_x}(\partial\Omega)$. Without loss of generality we may assume that $\|f\|_{L^p(\pom)}<\infty$ as otherwise the claim is trivial. Let $u$ be the energy solution of the equation $\LL u=0$ in $\Omega$ with boundary datum $f$ and consider also the energy solutions
for $1 \leq k \in \mathbb N:$
\begin{equation}\label{Defuk}
\LL u_k=0\mbox{ on }\R^{n-1}\times(0,k)\times\R,\quad u_k(x,0,t)=f(x,t),\quad u_k(x,k,t)=0,\quad\mbox{for all }(x,t)\in\pom.
\end{equation}
For simplicity we extend $u_k$ to $\Omega$ by defining  $u_k(x,x_n,t)=0$ when $x_n>k$, which extends each $u_k$ continuously.
We now consider the limit of $u_k$, as $k \to \infty$. For each $k\ge 1$, 
the Lax-Milgram lemma (c.f. subsection \ref{RwEs}) used on the domain $\Omega^k=\R^{n-1}\times(0,k)\times\R$  gives us
\begin{equation}\label{WACO}
 \|u_k\|_{\dot \E} = \|\nabla u_k\|_{\L^2(\Omega)} + \|\HT \dhalf u_k\|_{\L^2(\Omega)} \le C \|\mbox{Tr }u_k\|_{\Hdot^{1/4}_{\pd_{t} - \Delta_x}(\partial\Omega^k)}= C\|f\|_{\Hdot^{1/4}_{\pd_{t} - \Delta_x}(\partial\Omega)},
\end{equation}
where the constant $C\in(0,\infty)$ depends on the coercivity constant of the 
sesquilinear form \eqref{eq-sesq}
 in $\Omega^k$ (which is uniform in $k$), and $\|A\|_{L^\infty}=\Lambda$. Therefore, this 
bound is uniform in $k$. This uniformity and the fact that $\text{Tr}(u_k)=f$, implies that a weak convergence argument 
yields a sub-sequence convergent to some $v$ with $ \|v\|_{\dot \E}\leq C\|f\|_{\Hdot^{1/4}_{\pd_{t} - \Delta_x}(\partial\Omega)}$ 
and $\text{Tr}(v)=f$. This sub-sequence (which we somewhat imprecisely also call $(u_k)$) is therefore strongly convergent to $v$ in $L^2_{\rm loc}({\R}^n_+\times\R)$ 
by a standard functional analysis argument.
It follows that the $L^2$ averages of $u_k$ converge locally and uniformly to 
the $L^2$ averages of $v$ in $C_{\rm loc}({\R}^n_{+}\times\R)$. Additionally, because each $u_k$ has uniform interior H\"older continuity estimates, the convergence of $u_k\to v$ also holds in $C^\alpha_{\rm loc}({\R}^n_{+}\times\R)$ for some $\alpha>0$. 

Taking the limit in $k$, it now follows that $v$ must be a weak solution to $\mathcal Lv=0$ in $\Omega$. Since $\text{Tr}(v)=f$ and $ \|v\|_{\dot \E}<\infty$, it must hold that $u=v$ since energy solutions are unique. Consider now $u_k-u$, a solution in $\Omega^k$ enjoying the bound
\begin{equation}\label{WACO2}
 \|u_k-u\|_{\dot \E}  \le C \|\mbox{Tr }(u_k-u)\|_{\Hdot^{1/4}_{\pd_{t} - \Delta_x}(\partial\Omega^k)}= C\|\mbox{Tr }u\|_{\Hdot^{1/4}_{\pd_{t} - \Delta_x}(\R^{n-1}\times\{k\}\times\R)}.
\end{equation}
The trace $\|\mbox{Tr }u\|_{\Hdot^{1/4}_{\pd_{t} - \Delta_x}(\R^{n-1}\times\{k\}\times\R)}\to 0$ as $k\to\infty$, otherwise $ \|u\|_{\dot \E}=\infty$, which is false. Thus the weak convergence $\nabla u_k\to\nabla u$ in $L^2$ upgrades to strong convergence. \vglue1mm

Let $K$ be any compact subset of $\Omega$ on which $u\ne 0$. It follows that
$u_k^{p-2}\to u^{p-2}$ in $L^\infty(K)$. Note that the fact that $p<2$ plays no role as $u\ne 0$ and therefore, for some $k\ge k_0$, the solutions $u_k$ are also bounded away from zero since $u_k\to u$ uniformly on $K$.
Together with $\nabla u_k\to\nabla u$ in $L^2$ on $K$, we have
$$ \lambda\iint_{K}
     \abs{\nabla u_k}^2u_k^{p-2}x_n\,dXdt\to  \lambda\iint_{K}
     \abs{\nabla u}^2u^{p-2}x_n\,dXdt.$$
We use \eqref{eq.BlckSdkp<Nr} for each $u_k$ with $r$ so large such that $K\subset \R^{n-1}\times (0,r/2)\times\R$. This means we need to only consider $k\ge r/2$. We obtain that

\begin{multline}
\lambda\iint_{K}
     \abs{\nabla u_k}^2u_k^{p-2}x_n\,dXdt\le 
    \lambda\int_0^{\frac{k}2
    }\int_{\pom} \abs{\nabla u_k}^2u_k^{p-2}x_n\,dxdtdx_n\\
    \le C(p)\left[\int_{\pom}|u_k(x,k,t)|^pdxdt
    +\int_{\pom}|u_k(x,0,t)|^pdxdt\right]=C(p)\int_{\pom}f^pdxdt.
\end{multline}
Taking the limit $k\to\infty$ on both sides therefore yields that
$$\lambda\iint_{K}
     \abs{\nabla u}^2u^{p-2}x_n\,dXdt\le C(p)\int_{\pom}f^pdxdt,$$
on all compact subset $K\subset\Omega$ on which $u\ne 0$. Taking the supremum over all of these sets gives 
\begin{equation}\label{eq.BlckSdkp2}
    \lambda\iint_{\Omega\setminus{\{u\ne 0\}}}
     \abs{\nabla u}^2u^{p-2}x_n\,dXdt\le C    \int_{\pom}f^pdxdt.
\end{equation}
From this, our claim follows, since we use the convention that $\abs{\nabla u}^2u^{p-2}=0$ when both $u$ and $\nabla u$ vanish thus we must have $\abs{\nabla u(X,t)}^2u(X,t)^{p-2}=0$ when $t< t_0$ for some $t_0\in\R\cup\{-\infty\}$ and $u>0$ when $t>t_0$.
\end{proof}

We are ready to prove Lemma \ref{lem.Sdkp}. Pick an arbitrary $\varepsilon>0$. Because $u$ is an energy solution in the strip $\Omega^r=\R^{n-1}\times (0,r)\times \R$ the quantity below must be finite:
$$\iint_{\Omega^r}
     \abs{\nabla u}^2(u+\varepsilon)^{p-2}x_n\,dXdt \le \varepsilon^{p-2}r\|u\|^2_{\dot \E(\Omega^r)}<\infty.$$
Hence using ellipticity we can write

$$\lambda\iint_{\Omega^r}
     \abs{\nabla u}^2(u+\varepsilon)^{p-2}x_n\,dXdt \le \frac1{p-1} \iint_{\Omega^r}
     A\nabla u\cdot \nabla [(u+\varepsilon)^{p-1}]x_n\,dXdt,$$
and similarly for the localized version of this expression. 

We recall the proof of Lemma \ref{lem.Sdk} where a similar calculation is performed. Let $Q_r$, $\zeta$ be as in that proof, then a calculation as in \eqref{eq.pfSdk1}-\eqref{eq.pfSdk2} gives 

\begin{multline}\label{eq.pfSdk1a}
    \int_0^r\int_{Q_{2r}}A\nabla u\cdot \nabla [(u+\varepsilon)^{p-1}](\zeta^2x_n)dxdtdx_n\\
    =
    \frac{r}{p}\int_{Q_{2r}}\dr_n\br{(u+\varepsilon)^p}(x,r,t)\zeta(x,t)^2dxdt
    +\frac1p\int_{Q_{2r}}(u(x,0,t)+\varepsilon)^p\zeta(x,t)^2dxdt\\
    -\frac1p\int_{Q_{2r}}(u(x,r,t)+\varepsilon)^p\zeta(x,t)^2dxdt
    -\frac1p\int_0^r\int_{Q_{2r}}x_nA_\parallel\nabla_x[(u+\varepsilon)^{p}]\cdot\nabla_x(\zeta^2)dX\,dt\\
    -\int_0^r\int_{Q_{2r}}\divg(A\nabla u)(u+\varepsilon)^{p-1}\zeta^2x_n\,dX\,dt.
\end{multline}
For the last term we use the PDE $u$ satisfies to obtain that it equals to 
$\frac1p\int_0^r\int_{Q_{2r}}(u+\varepsilon)^p\partial_t(\zeta^2)x_n\,dX\,dt$ (c.f. the calculation for $I_1$ in the proof of the proof of Lemma \ref{lem.Sdk}).
Observe that on the right-hand side of \eqref{eq.pfSdk1a} all the powers are $p>1$ and thus the limit $\varepsilon\to 0+$ exists. Taking that limit, and using monotone convergence, we conclude that
\eqref{eq.pfSdk1a} gives: 
\begin{multline}\label{eq.pfSdk1b}
    \int_0^r\int_{Q_{2r}}A\nabla u\cdot \nabla [u^{p-1}](\zeta^2x_n)dxdtdx_n\\
    =
    \frac{r}{p}\int_{Q_{2r}}\dr_n\br{u^p}(x,r,t)\zeta(x,t)^2dxdt
    +\frac1p\int_{Q_{2r}}u(x,0,t)^p\zeta(x,t)^2dxdt\\
    -\frac1p\int_{Q_{2r}}u(x,r,t)^p\zeta(x,t)^2dxdt
    -\frac1p\int_0^r\int_{Q_{2r}}x_nA_\parallel\nabla_x[u^{p}]\cdot\nabla_x(\zeta^2)dX\,dt\\
   +\frac1p\int_0^r\int_{Q_{2r}}u^p\partial_t(\zeta^2)x_n\,dX\,dt.
\end{multline}
Now as in the proof of Lemma \ref{lem.Sdk} we consider a partition of unity of the boundary of $\pom$ and sum over it. This yields
\begin{multline}\label{eq.pfSdk3a}
    \int_0^r\int_{\pom}A\nabla u\cdot\nabla[u^{p-1}]\,x_n\,dX\,dt\\  
    = \frac{r}{p}\int_{\pom}\dr_n\br{u^p}(x,r,t)dxdt
    +\frac1p\int_{\pom}u(x,0,t)^pdxdt
    -\frac1p\int_{\pom}u(x,r,t)^pdxdt.
    \end{multline}
Finally, for any fixed $r_0>0$, we integrate \eqref{eq.pfSdk3a} in $r$ variable over $[0,r_0]$ and then divide both sides by $r_0$. This yields
\begin{multline*}
    \lambda(p-1)\int_0^{r_0}\int_{\pom}\br{x_n-\frac{x_n^2}{r_0}}\abs{\nabla u}^2u^{p-2}x_n\,dXdt
    \le \frac1p\int_{\pom}u(x,r_0,t)^pdxdt
    +\frac1p\int_{\pom}u(x,0,t)^pdxdt\\
    -\frac{2}{p\,r_0}\int_0^{r_0}\int_{\pom}u(x,x_n,t)^pdxdtdx_n.
\end{multline*}
Truncating the integral on the left-hand side to $[0,\frac{r_0}{2}]$ (and replacing $r_0$ by $r$) we obtain \eqref{eq.BlckSdkp<Nr} as desired with $C(p)=2(p-1)/p$.\qed

From Lemma \ref{lem.Sdkp} together with \eqref{S2Sp} for $u_k$ defined earlier by 
\eqref{Defuk}, we see that for any $\varepsilon>0,$
\begin{equation}\label{Strunc}
\|S^{k/2}_2(u_k)\|^p_{L^p(\pom)}\le \varepsilon \|N^{k/2}(u_k)\|^p_{L^p(\pom)}+C_\varepsilon \int_{\pom}f^p\,dxdt,
\end{equation}
where the superscript $k/2$ denotes truncation of the corresponding cones at the height $k/2$ and the constant $C_\varepsilon$ is independent of $k$. However, Lemma \ref{lem.Sdkp} also gives us a second estimate that removes these truncations.  The second term of the left-hand side of \eqref{eq.BlckSdkp<Nr}
implies that $\|u_k\|^p_{L^p(\Omega^k)}\le Ck\int_{\pom}f^p\,dxdt$.
Consider now any $(q,\tau)\in \pom$ and the expression
$$\iint_{\gamma(q,\tau)\cap\{x_n\in [k/2,\infty)\}}|\nabla u_k|^2 x_n^{-n}dX\,dt.
$$
As $u_k=0$ for $x_n>k$ the term is nonzero only for $x_n\in [k/2,k]$. Using boundary Cacciopoli (as $u_k$ vanishes on the boundary $x_n=k$) it follows that
$$\iint_{\gamma_a(q,\tau)\cap\{x_n\in [k/2,\infty)\}}|\nabla u_k|^2 x_n^{-n}dX\,dt\le C
\iint_{\gamma_{\tilde{a}}(q,\tau)\cap\{x_n\in [k/3,k)\}}|u_k|^2 x_n^{-2-n}dX\,dt,
$$
for some $\tilde{a}>a$. It then follows that
\begin{multline*}
\left(\iint_{\gamma_a(q,\tau)\cap\{x_n\in [k/2,\infty)\}}|\nabla u_k|^2 x_n^{-n}dX\,dt\right)^{1/2}\le C\left(
\fiint_{\gamma_{\tilde{a}}(q,\tau)\cap\{x_n\in [k/3,k)\}}|u_k|^2 dX\,dt\right)^{1/2}\\
\le C\sup_{\gamma_{\tilde{a}}(q,\tau)\cap\{x_n\in [k/3,k)\}}|u_k|.
\end{multline*}
Hence the square function at the point $(q,\tau)$ above height $k/2$ can be estimated by the nontangential maximal function at the same point above the height $k/3$ with cones of slightly wider aperture. 
We claim that we have
\begin{equation}\label{Nest}
\int_{\pom}\left(\sup_{\gamma_{\tilde{a}}(q,\tau)\cap\{x_n\in[ k/3,k)\}}|u_k|\right)^pdq\,d\tau\le C\int_{\pom}f^p\,dxdt.
\end{equation}
Assuming this holds then, combining it with \eqref{Strunc}, yields
\begin{equation}\label{Strunc-f}
\|S_2(u_k)\|^p_{L^p(\pom)}\le \varepsilon \|N(u_k)\|^p_{L^p(\pom)}+C_\varepsilon \int_{\pom}f^p\,dxdt,
\end{equation}
for a constant independent of $k$. It remains to show \eqref{Nest}. The region $\gamma_{\tilde{a}}(q,\tau)\cap\{x_n\in [k/3,k)\}$ lays inside a slightly larger region $T_k(q,\tau):=\{(x,x_n,t):\,\|x-q,t-\tau\|<2\tilde{a}k,\,\,\,x_n\in (k/4,k) \}$,
from which, by boundary H\"older continuity (Lemma \ref{lem.bdyHolder}),
$$\left(\sup_{\gamma_{\tilde{a}}(q,\tau)\cap\{x_n\in [k/3,k)\}}|u_k|\right)^p\lesssim \fiint_{T_k(q,\tau)}|u_k|^p.
$$
We integrate this over $\pom$. The right-hand side becomes the solid integral $$Ck^{-1}\|u_k\|^p_{L^p(\Omega \cap\{x_n\in (k/4,k)\})}.$$ 
The factor $k^{-1}$ comes from the average over $T_k(q,\tau)$. That is, the average in the $x_n$ variable takes place over an interval $(k/4,k)$ of length $\sim k$ and the integration over $\pom$ removes the average from all remaining variables. Since $\|u_k\|^p_{L^p(\Omega^k)}\le Ck\int_{\pom}f^p\,dxdt$ it follows that \eqref{Nest} holds.

\medskip

Next, we establish the inequality $\|N(u_k)\|_{L^p(\pom)}\lesssim \|S_2(u_k)\|_{L^p(\pom)}$. Working with $u_k$ allows us to use the a decay that we don't have available for $u$, namely, that each $u_k\to 0$ as $x_n\to \infty$ - in fact $u_k=0$ for $x_n>k$.  This decay was one of the ingredients needed in Section \ref{SS:43}.  With the $w$ of defined in that section replaced by $u_k$, we construct the function $\hbar_{\nu,a}$ as in \eqref{h} and claim that Lemma \ref{S3:L5} holds as stated, noting the continuity of $u_k$. Observe that the nature of this stopping-time construction forces $\hbar_{\nu,a}<k$ everywhere. Next we claim Lemma \ref{l6} holds as stated for the set
\begin{equation}\label{Eqqq-17n}
\big\{(x,t):\,N_{a}(u_k)(x,t)>\nu\mbox{ and }S_{b}(u_k)(x,t)\leq\gamma\nu\big\}
\end{equation}
with $(x,t)\in 2R$ and such that
\begin{equation}\label{Eqqq-18n}
\big|u_k\big(z,\hbar_{\nu,a}(w)(z,\tau),\tau\big)\big|>\nu/{2}\,\,\text{ for all }\,\,(z,\tau)\in R.
\end{equation}
This requires a new argument as $S_{b}(u_k)$ only controls $u_k$ in spatial variables but not in the $t$ variable.
For such a point $(x,t)$ in \eqref{Eqqq-17n}, part (iii) of Lemma \ref{S3:L5} gives the existence of a point $(y,y_n,s)\in\gamma_a(x,x_n,t)$ on the graph of $\hbar$ with $u_k(y,y_n,s)=\nu$.
We consider an interior parabolic cylinder $Q_{y_n/2}(y,y_n,s)$. There exists $b=b(a)>0$ such that 
$3/2Q_{y_n/2}(y,y_n,s)\subset\gamma_b(x,t)$ and hence \eqref{Eqqq-18n} will follow if we prove that for all points inside  $Q_{y_n/2}(y,y_n,s)$ the function $u_k>\nu/2$.

Using interior H\"older inequality Lemma \ref{L:int_Holder} (or if $Q_{y_n/2}(y,y_n,s)$ intersects the hyperplane $\{x_n=k\}$ the boundary H\"older) we always have
$$\sup_{(X_1,t_1),(X_2,t_2)\in Q_{y_n/2}(y,y_n,s)}|u_k(X_1,t_1)-u_k(X_2,t_2)|\le C \inf_{c\in\R} \fiint_{3/2Q_{y_n/2}(y,y_n,s)}|u_k-c|.$$
Considering a fixed time slice
$B^{t_0}=3/2Q_{y_n/2}(y,y_n,s)\cap \{t=t_0\}$ we have by Sobolev inequality
$$\inf_{c\in\R} \fiint_{B^{t_0}}|u_k-c|\,dX\le Cy_n\,\fiint_{B^{t_0}}|\nabla u_k|\,dX,$$
which holds when $c$ is an average of $u_k$ over $B^{t_0}$. Let $\eta$ be a smooth nonnegative cutoff
function in spatial variables, supported in $B^{t_0}$ and satisfying $\iint_{B^{t_0}}\eta\,dX=1$ and $|\nabla\eta|\le y_n^{-1}\eta$. Define an average
$$v_{av}(t_0)=\iint_{B^{t_0}}u_k(X,t_0)\eta(X)\,dX.$$
With such choice of $c$, the Sobolev inequality holds, and therefore we have that
\begin{multline}\label{osces}
\sup_{(X_1,t_1),(X_2,t_2)\in Q_{y_n/2}(y,y_n,s)}|u_k(X_1,t_1)-u_k(X_2,t_2)|\\\le C\left[y_n\fint_{s-(y_n/2)^2}^{s+(y_n/2)^2}\,\fiint_{B^{t_0}}|\nabla u_k|\,dX\,dt_0+\sup_{t_0,t_0'\in (s-(y_n/2)^2,s+(y_n/2)^2)}|v_{av}(t_0)-v_{av}(t_0')|\right].
\end{multline}
The first term in the last line is bounded by $CS_b(x,t)\le C\gamma\lambda$ (by Cauchy-Schwarz) and hence we can pick $C\gamma<1/4$.

We claim that the second term enjoys a similar estimate. Calculating as in \cite{Din23}, we multiply the equation $\LL u=0$ by $\eta$ and integrate it over the region
$\mathbb R^n\times [t_0,t_0']$, extending the function $u_k\eta$ by zero outside the support of $\eta$.
Because $\eta$ is a  
function of the spatial variables only, we have
$$\iint_{\mathbb R^n\times\{t_0'\}}u_k\eta\,dY-\iint_{\mathbb R^n\times\{t_0\}}u_k\eta\,dY=
\iint_{\mathbb R^n\times[t_0,t_0']}(\partial_t u_k)\eta\,dY\,dt$$
$$=\iint_{\mathbb R^n\times[t_0,t_0']}\mbox{\rm div}(A\nabla u_k)\eta\,dY\,dt
=-\iint_{\mathbb R^n\times[t_0,t_0']}(A\nabla u_k)\nabla\eta\,dY\,dt.
$$
 The term $\nabla\eta$ in the right-hand side integral is bounded by
 $y_n^{-1}$, and the length of the interval $[t_0,t_0']$  is bounded by $y_n^2$, therefore
$$\left|\iint_{\mathbb R^n\times\{t_0'\}}u_k\eta\,dY-\iint_{\mathbb R^n_+\times\{t_0\}}u_k\eta\,dY\right|\lesssim
C\left[y_n\fint_{s-(y_n/2)^2}^{s+(y_n/2)^2}\,\fiint_{B^{t_0}}|\nabla u_k|\,dX\,dt_0\right],$$
which again has a bound by $CS_b(x,t)\le C\gamma\lambda$. It follows that for sufficiently small $\gamma>0$
we get that 
$$\sup_{(X_1,t_1),(X_2,t_2)\in Q_{y_n/2}(y,y_n,s)}|u_k(X_1,t_1)-u_k(X_2,t_2)|\le\lambda/2,$$
and hence $u_k>\lambda-\lambda/2=\lambda/2$ on $ Q_{y_n/2}(y,y_n,s)$.
Hence we have Lemma \ref{l6} for the set \eqref{Eqqq-17n}.\qed

With similar modification we therefore have Corollary \ref{S3:L6} as well. We now state a version of Lemma 
\ref{S3:L8-alt1}.

\begin{lemma}\label{S3:L8-alt1a} 
Let $\Omega={\mathbb R}^n_+\times\R$ and let ${\mathcal L}$ be a block-form operator as above with $\LL u_k =0$ as in \eqref{Defuk} extended by zero for $x_n>k$.

For a fixed (sufficiently large $a>0$), consider and an arbitrary function $\hbar:{\mathbb R}^{n-1}\times\R\to \mathbb R$ such that it satisfies
the estimates \eqref{Eqqq-5}, \eqref{derh} and $\hbar\in  [0,k)$. 
Then 
we have the following:\vglue1mm

For all arbitrary parabolic surface balls $\Delta_r\subset{\mathbb R}^{n-1}\times\R$ of radius $r$ such that at least one point of $\Delta_r$
the inequality $\hbar(x,t)\le 2r$ holds we have the following estimate for any $c\in\R$:

\begin{multline}\label{TTBBMMz}
\int_{1/6}^1\int_{\Delta_r}\big|u_k\big(x,\theta\hbar(x,t),t\big)-c\big|^2\,dx\,dt\,d\theta
\leq C\iint_{\cS(\Delta_r,\hbar)}\abs{\nabla u_k}^2x_n\, dx_ndtdx\\
+\frac{C}{r}\iint_{\mathcal{K}_1}|u_k-c|^{2}\,dX\,dt,
\end{multline}
for some $C\in(0,\infty)$ that only depends on $a,\Lambda,n$ but not on $k$, $u_k$, $c$ or $\Delta_r$. 
Here, $\cS(\Delta_r,\hbar)$ and $\mathcal{K}_1$ as are in Lemma 
\ref{S3:L8-alt1}.
The cones used to define the nontangential 
maximal functions in this lemma have vertices on $\partial\Omega$.
\end{lemma}

This is much simpler version of the previous Lemma for gradients, but there a new minor issue to be addressed since $u_k$ only satisfies the PDE for $0<x_n<k$. In fact, 
as $\theta\le 1$ and $\hbar<k$ then $\theta\hbar<k$ and hence we are integrating inside the region where our PDE holds. It follows that the proof of Lemma \ref{S3:L8-alt1a} is analogous, but simpler, since $\LL u_k=0$.
Hence the claim simplifies to just \eqref{TTBBMMz}.

Next we prove the following lemma.

\begin{lemma}\label{LGL2} Let $\mathcal L$ and $u_k$ be as in the previous lemma and fix $a>0$ sufficiently large as above.
There exists $\kappa\in(0,1)$ that depends only on $n$, constants $b>a$ and $\gamma_0$ that depend on $a$, and $C>0$ that depends only on $n$ and the ellipticity constants such that the following holds.

For any $\gamma\in(0,\gamma_0)$, $\beta>0$, there holds
\begin{multline}    
  \abs{\set{(x,t)\in {\mathbb R}^{n-1}\times\R:\, {N}_a(u_k)>\beta, S_b(u_k)\le\gamma\beta}}\\\le C\gamma^2\abs{\set{(x,t)\in\R^{n-1}\times\R: M({N}_a(u_k))(x,t)>\kappa\beta}}.\label{eq.gdLmdx}
\end{multline}
\end{lemma}

\begin{proof} We only sketch the proof, highlighting the differences from the proof of Lemma \ref{LGL}.
Let 
\begin{equation}\label{def.E1x}
    E_{1,\beta}:=
\set{(x,t)\in {\mathbb R}^{n-1}\times\R:\, {N}_a(u_k)>\beta, S_b(u_k)\le\gamma\beta}.
\end{equation}
We proceed is in Lemma \ref{LGL} but skip defining the average $\vec\eta_{L^2}$ as we now work with $N$ instead of $\tilde N$. Hence in the proof whenever we see $\vec\eta_{L^2}$ we replace it by $u_k$. Proceeding exactly as in the proof of Lemma \ref{LGL} we arrive to \eqref{eq.eta-c'}. The argument given previously for \eqref{trunceta}. applies as well to $u_k$ and in particular it is useful for the estimate
\eqref{osces}, leading to the claim that: 

\begin{equation}\label{eq.diff}
\left|\int_{20\Delta_i}\abs{u_k-c}^2(x,\hbar_i(x,t),t)dxdt-\int_{20\Delta_i}\abs{u_k-c}^2(x,\theta\hbar_i(x,t),t)dxdt\right|\le C\gamma^2\beta^2|\Delta_i|,
\end{equation}
for $\theta\in [1/6,1]$
which allows us to replace \eqref{eq.eta-c'} by
\begin{equation}\label{eq.eta-c''}
    \int_{10\Delta_i}|u_k-c|^2(x,\hbar_i(x,t),t)dxdt\le C\int_{1/6}^1\int_{20\Delta_i}\abs{u_k-c}^2(x,\theta\hbar_i(x,t),t)dxdtd\theta+C\gamma^2\beta^2|\Delta_i|. 
\end{equation}
We then proceed analogously (with fewer terms to handle) until we reach 
\[
\frac{1}{r_i}\iint_{\mathcal{K}_1}|u_k-c|^{2}\,dX\,dt.
\]
Previously we used the area function to take care of variations in time variable. Instead, 
\eqref{osces} allows us to bound the oscillation of $u$ in terms of truncated version of $S_b(u_k)$. This concludes the argument.
\end{proof}
Under the assumption that $\|N(u_k)\|_{L^p}<\infty$ we then can conclude that, for all $1<p<\infty$, there exists $C=C(n,p,\lambda,\Lambda)>0$ such that
\begin{equation}\label{finalNS}
\|N(u_k)\|^p_{L^p(\pom)}\le C \|S_2(u_k)\|^p_{L^p(\pom)}.
\end{equation}

Let us address briefly why the apriori bound $\|N(u_k)\|_{L^p}<\infty$ is justified. 

Let $f:\pom\to \R$ be a compactly supported Lipschitz function, a class of functions that are dense in $L^p(\pom)$, $p<\infty$. Moreover, as $f$ has a compact Lipschitz extension into $\Omega$, $f$ has a well-defined trace in the space $\Hdot^{1/4}_{\partial_t-\Delta_x}(\partial\Omega)$. Splitting $f=f^+-f^-$ into positive and negative parts, each of which are Lipschitz, we can again consider separately energy solutions with data $f^+$ and $f^-$.

Let $u_k$ be a solution of \eqref{Defuk} with such nonnegative boundary datum (we call again $f$). Let $\Delta$
be ball containing the support of $f$.
As $f$ is bounded, so is $u_k$ by the maximum principle. Hence we see that 
for all $(p,\tau)\in 3\Delta$ we have $N(u_k)(p,\tau)\le K<\infty$. This ensures $L^p$ integrability of $N(u_k)$ on the set $3\Delta$.
Away from $3\Delta$, as $u_k$ is also nonnegative, the decay of $u_k$ is comparable to the decay of the corresponding Green's function (\eqref{eq.Grnf_upbd}) by the comparison principle.

Furthermore, the set $3\Delta\times (0,k)$ is bounded in the domain $\Omega^k$ and by the maximum principle $u_k$ is bounded on the boundary of this set. We pick a pole  $(Y,s)$ with time $s$ some distance before the beginning of the support of $f$ and, as $G>0$ on $\partial (3\Delta\times (0,k))$, there exists a constant $M>0$ for which
$$0<u_k(x,x_n,t)\le M G((x,x_n,t),(Y,s)),\qquad\mbox{for all }t>s, \, x_n\in (0,k),\, \mbox{ and }\, (x,t)\in \pom)\setminus 3\Delta ,$$
Here, $G$ is the Green's function of the whole space given by \eqref{eq.Grnf_upbd}. 
The Green's function $G$ has sufficient decay - exponential in the spatial variables and polynomial in $t$ of order at 
least $C(1+t)^{-1}$ - 
which makes it $t$-integrable in $L^p$ for $p>1$ on an interval $(0,\infty)$),
and thus ensures that $N(u_k)$ will be integrable in $L^p$ on sets of type 
$\R^{n-1}\times (0,k)\times (0,\infty)$.
While this may be a crude bound, it shows that \eqref{finalNS} holds for all $u_k$ with  a $C$ that only depends on $n,p,\lambda$ and $\Lambda$ but not on $k$. 
\medskip

With \eqref{finalNS} established for non-negative $u_k$ and with Lipschitz, nonnegative, compactly supported data, it will follow that for any $1<p\le 2$, there exists $C>0$ independent of $k$, such that

$$\|N(u_k)\|^p_{L^p(\pom)}\le C \|S_2(u_k)\|^p_{L^p(\pom)}\le C\varepsilon \|N(u_k)\|^p_{L^p(\pom)}+CC_\varepsilon \int_{\pom}f^p\,dxdt,$$
where we have used \eqref{Strunc-f}. Choosing $\varepsilon>0$ such that $C\varepsilon<1/2$ by finiteness of 
$\|N(u_k)\|_{L^p}<\infty$ implies that
$$\|N(u_k)\|^p_{L^p(\pom)}\le 2C \|S_2(u_k)\|^p_{L^p(\pom)}\le 2CC_\varepsilon \int_{\pom}f^p\,dxdt,$$
and hence $u_k$ enjoys both square and nontangential maximal function bounds.

Recall that we have already shown that $u_k\to u$ locally uniformly on compact subsets of $\Omega$, where $u$ solves the PDE $\LL u =0$ in $\Omega$ with datum $f$. Taking limits of notangential maximal functions truncated to these compact subsets yields the same bound for $u$, namely that
$$\|N(u)\|^p_{L^p(\pom)}\le 2CC_\varepsilon \int_{\pom}f^p\,dxdt.$$
Similar square function bounds follow by using $\nabla u_k\to\nabla u$ in $L^2(\Omega)$,
and obtaining that, in every compact $K\subset \R^n_+\times \R$,
$$\int_{K\cap\gamma(p,\tau)}|\nabla u_k|^2x_n^{-n}dX\,dt\to \int_{K\cap\gamma(p,\tau)}|\nabla u|^2x_n^{-n}dX\,dt,$$
for all boundary points $(p,\tau)$. Integrating and taking the supremum over all such $K$ yields
$$\|S_2(u)\|^p_{L^p(\pom)}\le 2C_\varepsilon \int_{\pom}f^p\,dxdt.$$

Finally, applying this to $f^+$ and $f^-$ separately gives:

$$\|N(u)\|^p_{L^p(\pom)}\le \|N(u^+)\|^p_{L^p(\pom)}+\|N(u^-)\|^p_{L^p(\pom)}$$
$$\le
2CC_\varepsilon \int_{\pom}(f^+)^p\,dxdt+2CC_\varepsilon \int_{\pom}(f^-)^p\,dxdt\le 
2CC_\varepsilon \int_{\pom}|f|^p\,dxdt,$$
with an analogous bound for $\|S_2(u)\|^p_{L^p(\pom)}$.
This shows the solvability of the $L^p$ Dirichlet problem for $\mathcal L$ on $\Omega$. The compactly supported Lipschitz functions are dense in $L^p(\pom)$ so there exists a unique continuous extension of this solution operator $f\mapsto u=u(f)$ onto all of $L^p$ that satisfies the required bound 
$\|N(u)\|^p_{L^p(\pom)}\approx \|S_2(u)\|^p_{L^p(\pom)}\lesssim\|f\|^p_{L^p(\pom)}$. This solution operator is close to the energy solutions in the following sense. If $f\in L^p(\pom)\cap \Hdot^{1/4}_{\partial_t-\Delta_x}(\partial\Omega)$, then $u$ obtained as 
$u=u(f)$ is also an energy solution with its energy norm $ \|u\|_{\dot \E}$ controlled by the $\Hdot^{1/4}_{\partial_t-\Delta_x}(\partial\Omega)$ norm of $f$ and Tr $u=f$.

\medskip

\bibliographystyle{alpha}

\end{document}